\documentclass[12pt]{amsart}
\usepackage{geometry} 
\usepackage{amscd}
\usepackage{amssymb}
\usepackage{amsmath}
\usepackage{amsthm}

\newtheorem{thm}{Theorem}[section]
\newtheorem{lemma}[thm]{Lemma}
\newtheorem{corollary}[thm]{Corollary}
\newtheorem{prop}[thm]{Proposition}

\theoremstyle{definition}

\newtheorem{rem}[thm]{Remark}

\newtheorem{defn}[thm]{Definition}

\makeatletter

\newcommand{\isom}{\overset{\sim}{\rightarrow}}

\def\ge{{\geqq}}

\title{Deformations of trianguline $B$-pairs and Zariski density of two dimensional 
crystalline representations.}
\author{Kentaro Nakamura}
\date{} 

\begin{document}

\maketitle
\pagestyle{plain}
\footnote{2010 Mathematical Subject Classification 11F80 (primary), 11F85, 11S25 (secondary).

Keywords: $p$-adic Hodge theory, trianguline representations, B-pairs.  

Address:
Department of Mathematics, Hokkaido University, Kita 10, Nishi 8, Kita-Ku, Sapporo, Hokkaido, 060-0810, Japan. 

E-mail address: kentaro@math.sci.hokudai.ac.jp}
\begin{abstract}
The aims of this article are to study the deformation theory of 
trianguline $B$-pairs and to construct a $p$-adic family of two dimensional trianguline 
representations for any $p$-adic field. The deformation theory is the generalization 
of Bella\"iche-Chenevier's and Chenevier's works in the $\mathbb{Q}_p$-case,
 where they used ($\varphi,\Gamma$)-modules over the Robba ring instead of using $B$-pairs. 
 Generalizing and modifying Kisin's theory of $X_{fs}$ for any $p$-adic field, we construct 
 a $p$-adic family of two dimensional trianguline representations.
 As an application of these theories, we prove a theorem concerning the 
 Zariski density of two dimensional 
 crystalline representations for any $p$-adic field, which is a generalization of Colmez's and Kisin's theorem for the $\mathbb{Q}_p$-case.
\end{abstract}
\setcounter{tocdepth}{2}
\tableofcontents

\section{Introduction.}
\subsection{Background.}
Let $p$ be a prime number and let $K$ be a $p$-adic field, i.e. finite 
extension of $\mathbb{Q}_p$. The theory of trianguline representations which form a class of $p$-adic representations 
of $G_K:=\mathrm{Gal}(\overline{K}/K)$, in particular the theory of their $p$-adic families turns out to be very important for the study of the families of $p$-adic Galois representations over 
the eigenvarieties which parametrize some $p$-adic automorphic representations. 
Inspired by Kisin's $p$-adic Hodge theoretic study of Coleman-Mazur eigencurve \cite{Ki03},
 Colmez \cite{Co08}
defined the notion of trianguline representations by using Fontaine's and Kedlaya's theory of ($\varphi,\Gamma$)-modules over the Robba ring in the study of $p$-adic Langlands correspondence for $\mathrm{GL}_2(\mathbb{Q}_p)$. 
Based on their works, Bella\"iche-Chenevier \cite{Bel-Ch09} and Chenevier \cite{Ch09b} studied deformation theory of trianguline representations 
and $p$-adic families of trianguline representations. These theories are the fundamental tools for 
their applications of the eigenvarieties to some number theoretic problems, e.g. a construction of 
non trivial elements in some Selmer groups. 
Because all their studies are limited to the case $K=\mathbb{Q}_p$, we didn't have any results 
concerning the $p$-adic Hodge theoretic properties of eigenvarieties  over a number field $F$ 
except when $F$ is $\mathbb{Q}$ or more generally is a number field in which $p$ splits completely.


On the other hands, in \cite{Na09}, the author of this article generalized
many results of  \cite{Co08} for any $p$-adic field $K$. The author proved some fundamental properties of trianguline representations and then classified two dimensional trianguline representations for any $p$-adic field, where we studied trianguline representations by using 
$B$-pairs, which were defined by Berger in \cite{Be09}, instead of using ($\varphi,\Gamma$)-modules over the Robba ring.

The aim of this article is 
to generalize Kisin's, Bella\"iche-Chenevier's  and Chenevier's works for any $p$-adic field $K$, more precisely, 
to develop  deformation 
theory of trianguline representations and to construct  a $p$-adic family of two dimensional 
trianguline representations for any $p$-adic field $K$. The author expects that these generalizations are also fundamental for applications 
to  $p$-adic Hodge theoretic study of eigenvarieties for more general number fields. 

As an application of these theories, 
we prove some theorems (see Theorem \ref{000} and Theorem\ref{0000} in Introduction)
concerning the Zariski density of two dimensional crystalline representations for any $p$-adic field. 
These results are the generalizations of  a theorem of Colmez and Kisin for $K=\mathbb{Q}_p$, which played 
some crucial roles in the proof of $p$-adic Langlands correspondence for $\mathrm{GL}_2(\mathbb{Q}_p)$ (\cite{Co10}, \cite{Ki10}, \cite{Pa10}).  
In the next article \cite{Na11} which is based on the results of this article, we construct a $p$-adic family of $d$-dimensional trianguline representations 
for any $d\in \mathbb{Z}_{\geqq 1}$  and for any $K$ and prove some theorems concerning 
the Zariski density of $d$-dimensional crystalline representations for any $d$ and $K$.


  \subsection{Overview.}
 
 Here, we explain the contents of each section of this article.
 
 \smallskip
 
 In $\S$ 2, we study the deformation theory of trianguline $B$-pairs, which is the 
 generalization of the studies of \cite{Bel-Ch09}, \cite{Ch09b} for any $p$-adic field. 
 
 In $\S$ 2.1, we recall the definition of $B$-pairs and some fundamental properties of trianguline 
 $B$-pairs proved in \cite{Na09} and then we extend these notions to those over Artin local rings. 
 Let $E$ be a suitable finite extension of $\mathbb{Q}_p$ which is sufficiently large as in Notation below. We recall the definition of $E$-$B$-pairs of $G_K$, which is the $E$-coefficient version of $B$-pairs. Set $\bold{B}_{e}:=\bold{B}_{\mathrm{cris}}^{\varphi=1}$. 
 An $E$-$B$-pair is a pair $W=(W_e,W^+_{\mathrm{dR}})$ where $W_e$ is a finite free $\bold{B}_{e}\otimes_{\mathbb{Q}_p}E$-module 
 with a continuous semi-linear $G_K$-action such that $W^+_{\mathrm{dR}}\subseteq W_{\mathrm{dR}}:=\bold{B}_{\mathrm{dR}}\otimes_{\bold{B}_{e}}W_e$ 
 is a $G_K$-stable $\bold{B}^+_{\mathrm{dR}}\otimes_{\mathbb{Q}_p}E$-lattice of $W_{\mathrm{dR}}$. The category of $E$-representations 
 of $G_K$ is embedded in the category of $E$-$B$-pairs by $V\mapsto W(V):=(\bold{B}_{e}\otimes_{\mathbb{Q}_p}V, \bold{B}^+_{\mathrm{dR}}\otimes_{\mathbb{Q}_p}V)$. 
 We say that an $E$-$B$-pair is split trianguline if $W$ is a successive extension of rank one $E$-$B$-pairs, i.e. $W$ has a filtration 
 $0=W_0\subseteq W_1\subseteq W_2\subseteq \cdots \subseteq W_{n-1}\subseteq W_n=W$ such that $W_i$ is a saturated sub $E$-$B$-pair 
 of $W$ and $W_i/W_{i-1}$ is a rank one $E$-$B$-pair for any $1\leqq i\leqq n$. We say that $W$ is trianguline 
 if $W\otimes_{E}E'$ is a $E'$-split trianguline $E'$-$B$-pair for a finite extension $E'$ of $E$. We say that an $E$-representation $V$ is split trianguline 
 (resp. trianguline) if $W(V)$ is split trianguline (resp. trianguline). By these definitions, to study trianguline $E$-$B$-pairs, we first need to classify 
 rank one $E$-$B$-pairs and then we need to calculate extensions between them, which were studied in \cite{Co08} for $K=\mathbb{Q}_p$ and in 
 \cite{Na09} for general $K$. In $\S$ 2.1, we recall these results which we need to study the deformation theory of trianguline $E$-$B$-pairs.
We define the Artin local ring coefficient version of $B$-pairs. Let $\mathcal{C}_E$ be the category of Artin local $E$-algebras with the residue fields isomorphic to $E$. For $A\in \mathcal{C}_E$, we say that $W:=(W_e,W^+_{\mathrm{dR}})$ is an $A$-$B$-pair if $W_e$ is a finite free $\bold{B}_{e}\otimes_{\mathbb{Q}_p}A$-module with a continuous semi-linear $G_K$-action and 
 $W^+_{\mathrm{dR}}\subseteq W_{\mathrm{dR}}:=\bold{B}_{\mathrm{dR}}\otimes_{\bold{B}_{e}}W_e$ is a $G_K$-stable $\bold{B}^+_{\mathrm{dR}}\otimes_{\mathbb{Q}_p}A$-lattice. We generalize some results of \cite{Na09} for $A$-$B$-pairs.
 
 In $\S$ 2, we study two types of deformations of split trianguline $E$-$B$-pairs.
 In $\S$ 2.2, first we study the 
 usual deformation for all $E$-$B$-pairs, which is the generalization of Mazur's deformation theory
 of $p$-adic Galois representations. 
 Let $W$ be an $E$-$B$-pair and $A\in \mathcal{C}_E$. 
 We say that $(W_A,\iota)$ is a deformation of $W$ over $A$ if $W_A$ is an $A$-$B$-pair and $\iota:W_A\otimes_A E\isom W$ is an 
 isomorphism. We define the deformation functor of $W$, $D_W:\mathcal{C}_E\rightarrow (\mathrm{Sets})$ by 
 $D_W(A):=\{$equivalent classes of deformations $(W_A,\iota)$ of $W$ over $A$ $\}$. We prove the following proposition concerning 
 the pro-representability and the formal smoothness and the dimension formula 
 of $D_W$.  
 
  \begin{prop} (Corollary \ref{13})
    Let $W$ be an $E$-$B$-pair of rank $n$. If $W$ satisfies the following conditions, 
    \begin{itemize}
    \item[(1)]$\mathrm{End}_{G_K}(W)=E$,
    \item[(2)]$\mathrm{H}^2(G_K, \mathrm{ad}(W))=0$, 
    \end{itemize}
    then the functor $D_W$ is pro-representable by a pro-object $R_W$ of $\mathcal{C}_E$ such that 
    $$R_W\isom E[[T_1,\cdots,T_d]]\,\,\text{ for } d:=[K:\mathbb{Q}_p]n^2 +1.$$
    \end{prop}

 In $\S$ 2.3, we study the other more important type of deformations, i.e. the trianguline deformations. Let $W$ be a split trianguline $E$-$B$-pair of rank $n$ and 
 $\mathcal{T}:0\subseteq W_1\subseteq W_2\subseteq \cdots \subseteq W_{n-1}\subseteq W_n=W$ be a fixed triangulation of $W$.  For $A\in \mathcal{C}_E$, 
 we say that $(W_A, \iota, \mathcal{T}_A)$ is a trianguline deformation of $(W,\mathcal{T})$ over $A$ if $(W_A,\iota)$ is a deformation of $W$ 
 over $A$ and $\mathcal{T}_A:0\subseteq W_{1,A}\subseteq \cdots \subseteq W_{n-1,A}\subseteq W_{n,A}=W_A$ is an $A$-triangulation of $W_A$
( i.e. $W_{i,A}$ is a saturated sub $A$-$B$-pair of $W_A$ such that $W_{i,A}/W_{i-1,A}$ is a rank one $A$-$B$-pair for any $i$) such that 
$\iota(W_{i,A}\otimes_A E)=W_i$ for any $i$. We define the trianguline deformation functor $D_{W,\mathcal{T}}:
\mathcal{C}_E\rightarrow (\mathrm{Sets})$ of $(W,\mathcal{T})$ by $D_{W,\mathcal{T}}(A):=\{ $equivalent classes of trianguline deformations $(W_A,\iota,\mathcal{T}_A)$ 
of $(W,\mathcal{T})$ over $A$$\}$. We prove the following proposition concerning the pro-representability and the formal smoothness and the dimension formula of this functor, which is the  generalization of Proposition 3.6 of \cite{Ch09b} for any $p$-adic field. For the notations, see the main body of the article.

\begin{prop}\label{000000}(Proposition \ref{19})
Let $W$ be a split trianguline $E$-$B$-pair of rank $n$ with a triangulation 
$\mathcal{T}:0\subseteq W_1\subseteq \cdots\subseteq W_{n-1}\subseteq W_n=W$. 
We assume that $(W,\mathcal{T})$ satisfies the following conditions, 
\begin{itemize}
\item[(0)]$\mathrm{End}_{G_K}(W)=E$,
\item[(1)]For any $1\leqq i < j\leqq n$,  $\delta_j/\delta_i\not= \prod_{\sigma\in\mathcal{P}}\sigma^{k_{\sigma}}$ 
for any $\{k_{\sigma}\}_{\sigma\in\mathcal{P}}\in \prod_{\sigma\in\mathcal{P}}\mathbb{Z}_{\leqq 0}$
\item[(2)]For any $1\leqq i < j\leqq n$,
$\delta_i/\delta_j\not= |\mathrm{N}_{K/\mathbb{Q}_p}|\prod_{\sigma\in\mathcal{P}}\sigma^{k_{\sigma}}$ 
for any $\{k_{\sigma}\}_{\sigma\in\mathcal{P}}\in \prod_{\sigma\in\mathcal{P}}\mathbb{Z}_{\geqq 1}$, 
\end{itemize}
then $D_{W,\mathcal{T}}$ is pro-representable by a quotient ring  $R_{W,\mathcal{T}}$ of $R_W$ such 
that $$R_{W,\mathcal{T}}\isom E[[T_1,\cdots, T_{d_n}]]\text{ for }d_n:=\frac{n(n+1)}{2}[K:\mathbb{Q}_p]+1.$$
\end{prop}

In $\S$ 2.4, we define the notion of benign $E$-$B$-pairs, which forms a special good class of split trianguline and potentially crystalline 
 $E$-$B$-pairs, and prove a theorem concerning the tangent spaces of the deformation rings of this class. 
In $\S$ 2.4.1, we define the notion of benign $E$-$B$-pairs , in \cite{Ch09b} this class is called generic, in this article we follow the terminology of \cite{Ki10}. Let $W$ be a potentially crystalline $E$-$B$-pair of rank $n$ such that 
$W|_{G_L}$ is crystalline for a finite totally ramified abelian extension $L$ of $K$, which we call 
a crystabelline representation.
We assume that $\bold{D}_{\mathrm{cris}}^{L}(W):=(\bold{B}_{\mathrm{cris}}\otimes_{\bold{B}_{e}}W_e)^{G_L}=K_0\otimes_{\mathbb{Q}_p}Ee_1\oplus 
\cdots \oplus K_0\otimes_{\mathbb{Q}_p}Ee_n$ such that $K_0\otimes_{\mathbb{Q}_p}Ee_i$ are preserved by $(\varphi,\mathrm{Gal}(L/K))$-action 
and that $\varphi^f(e_i)=\alpha_ie_i$ for some $\alpha_i\in E^{\times}$, here $f:=[K_0:\mathbb{Q}_p]$ and $K_0$ is the maximal unramified extension of $\mathbb{Q}_p$ in $K$. 
We denote by $\{k_{1,\sigma}.k_{2,\sigma},\cdots,k_{n,\sigma}\}_{\sigma:K\hookrightarrow \overline{K}}$ the Hodge-Tate weights of $W$ such that 
$k_{1,\sigma}\geqq k_{2,\sigma}\geqq \cdots \geqq k_{n,\sigma}$ for any $\sigma:K\hookrightarrow \overline{K}$. In this paper, we define 
the Hodge-Tate weight of $\mathbb{Q}_p(1)$ by $1$. 
Let $\mathfrak{S}_n$ be the $n$-th permutation group. For any $\tau\in\mathfrak{S}_n$, we can define a filtration on $\bold{D}^L_{\mathrm{cris}}(W)$ 
by $0\subseteq K_0\otimes_{\mathbb{Q}_p}Ee_{\tau(1)}\subseteq K_0\otimes_{\mathbb{Q}_p}Ee_{\tau(1)}\oplus K_0\otimes_{\mathbb{Q}_p}Ee_{\tau(2)}\subseteq 
\cdots \subseteq K_0\otimes_{\mathbb{Q}_p}Ee_{\tau(1)}\oplus \cdots \oplus K_0\otimes_{\mathbb{Q}_p}Ee_{\tau(n-1)}\subseteq \bold{D}^L_{\mathrm{cris}}(W)$ by 
sub $E$-filtered $(\varphi, G_K)$-modules, where the filtration of $K_0\otimes_{\mathbb{Q}_p}Ee_{\tau(1)}\oplus \cdots \oplus K_0\otimes_{\mathbb{Q}_p}Ee_{\tau(i)}$ is 
the one induced from that of $\bold{D}^L_{\mathrm{dR}}(W)=L\otimes_{K_0}\bold{D}^L_{\mathrm{cris}}(W)$. By the equivalence between 
the category of potentially crystalline $B$-pairs and the category of filtered $(\varphi,G_K)$-modules, for any $\tau\in \mathfrak{S}_n$, 
we obtain the triangulation $\mathcal{T}_{\tau}:0\subseteq W_{\tau,1}\subseteq W_{\tau,2}\subseteq \cdots \subseteq W_{\tau,n}=W$ 
such that 
$W_{\tau,i}$ is potentially crystalline and $\bold{D}^L_{\mathrm{cris}}(W_{\tau,i})\isom K_0\otimes_{\mathbb{Q}_p}Ee_{\tau(1)}\oplus \cdots \oplus K_0\otimes_{\mathbb{Q}_p}E
e_{\tau(i)}$ for any $1\leqq i\leqq n$.

Under this situation, we define the notion of benign $E$-$B$-pairs as follows.
\begin{defn}\label{ben}
Let $W$ be a potentially crystalline $E$-$B$-pair of rank $n$ as above. Then we say that $W$ is benign if $W$ satisfies the following;
\begin{itemize}
\item[(1)] For any $i\not= j$, we have $\alpha_i/\alpha_j\not= 1, p^f, p^{-f}$,
\item[(2)] For any $\sigma:K\hookrightarrow \overline{K}$, we have $k_{1,\sigma}>k_{2,\sigma}>\cdots >k_{n-1,\sigma}> k_{n,\sigma}$,
\item[(3)] For any $\tau\in \mathfrak{S}_n$ and $1\leqq i \leqq n$, the Hodge-Tate weights of $W_{\tau,i}$ are $\{k_{1,\sigma},k_{2,\sigma},\cdots,k_{i,\sigma}\}_{\sigma :K\hookrightarrow \overline{K}}$.
\end{itemize}
\end{defn}

In $\S$ 2.4.2, we prove the main theorem of $\S$ 2. Let $W$ be a benign $E$-$B$-pair of rank $n$ as above. For any $\tau\in \mathfrak{S}_n$, we can define 
the trianguline deformation functor $D_{W,\mathcal{T}_{\tau}}$ as above. Let $R_W$ be the universal deformation ring of $D_W$, and 
let $R_{W,\mathcal{T}_{\tau}}$ be the universal deformation ring of $D_{W,\mathcal{T}_{\tau}}$ for each $\tau\in \mathfrak{S}_n$. 
Let denote by $t(R_{W})$ and $t(R_{W,\mathcal{T}_\tau})$ the tangent spaces 
of $R_W$ and $R_{W,\mathcal{T}_{\tau}}$ respectively. For each $\tau\in \mathfrak{S}_n$, $t(R_{W,\mathcal{T}_{\tau}})$ is a sub $E$-vector space of $t(R_W)$. 
The main theorem of $\S$ 2 is the following, which is the generalization of Theorem 3.19 of  \cite{Ch09b} for any 
$p$-adic field.
\begin{thm}$(\text{Theorem } \mathrm{\ref{33}} )$
Let $W$ be a benign $E$-$B$-pair of rank $n$, then we 
 have an equality $$\sum_{\tau\in \mathfrak{S}_n}t(R_{W,\mathcal{T}_{\tau}})=t(R_W).$$
 \end{thm}
 
 This theorem is a crucial local result for the applications to some Zariski density theorems of 
 local or global $p$-adic Galois representations. In fact, using 
 this theorem for $K=\mathbb{Q}_p$, Chenevier \cite{Ch09b} proved a theorem concerning the Zariski density of the unitary automorphic $p$-adic Galois representations 
 in the universal deformation spaces of three dimensional conjugate-selfdual $p$-adic representations of $G_{F}$ for any CM field $F$
 in which $p$ splits completely. Moreover, this theorem is also a crucial result for the proof of 
 Zariski density of crystalline representations in the universal deformation spaces of $p$-adic Galois representations of $p$-adic fields.
 In the rest of this paper $\S$ 3 and $\S4$, we apply this theorem only for the two dimensional case. Using 
 this theorem for $K=\mathbb{Q}_p$, Chenevier \cite{Ch10} recently 
 proved the Zariski density of crystalline representations for higher dimensional and $K=\mathbb{Q}_p$ case.
 In the next paper \cite{Na11}, the author uses this theorem for proving the Zariski density of crystalline representations for higher dimensional and 
 any $p$-adic field case.
 
 \smallskip
 
 In $\S$ 3, we construct some $p$-adic families of two dimensional trianguline representations for any $p$-adic field. To construct these, we generalize Kisin's theory of finite slope subspace $X_{fs}$ in \cite{Ki03} for any $p$-adic field. 
 As in the case of $\mathbb{Q}_p$ (\cite{Ki03}, \cite{Ki10}), this family is essential for the proof of the Zariski density of two dimensional crystalline 
 representations in $\S4$ and for the study of the $p$-adic Hodge theoretic properties of Hilbert modular eigenvarieties.
 
 In $\S$ 3.1, we prove some propositions concerning Banach $G_K$-modules which we need for the construction of 
 $p$-adic families of trianguline representations. In particular, we show that these Banach $G_K$-modules can be 
 obtained naturally from some almost $\mathbb{C}_p$-representations \cite{Fo03}. For us, one of the important properties of these 
 Banach $G_K$-modules is orthonormalizability as Banach modules over some Banach algebras. All these properties follow from some general facts of almost 
 $\mathbb{C}_p$-representations.
 
 In $\S$ 3.2, for any separated rigid analytic space $X$ over $E$ and for any 
 finite free $\mathcal{O}_X$-module $M$ with a continuous $G_K$-action and for any invertible function 
 $Y$ on $X$, we construct a Zariski closed subspace $X_{fs}$ of $X$, which is ``roughly" defined as 
 the subspace of $X$ consisting of the points $x\in X$ such that $\bold{D}_{\mathrm{cris}}^+(M(x))^{\varphi^f=Y(x)}\not= 0$, 
 where  $M(x)$ is the fiber of $M$ at $x$. For the precise characterization 
 of $X_{fs}$, see Theorem \ref{5.8}. This construction is the generalization of 
 Kisin's $X_{fs}$ in $\S5$ of \cite{Ki03} for any $p$-adic field. After obtaining the results in $\S$ 3.1, the construction 
 and the proof is almost all the same as that of  \cite{Ki03}, but a difference is that we need to consider all the 
 embeddings $\tau:K\hookrightarrow \overline{K}$, which makes the situation more complicated. For convenience of the readers or the author, we choose to re-prove this 
 construction in full detail.
 
 In $\S$ 3.3, we apply this construction to the rigid analytic space associated to 
 the universal deformation ring of a two dimensional mod $p$-representation of $G_K$. 
 Let $\bar{\rho}:G_K\rightarrow \mathrm{GL}_2(\mathbb{F})$ be 
 a two dimensional mod $p$ representation of $G_K$, where $\mathbb{F}$ is the residue field of $E$. 
 For simplicity, in this paper, we assume that $\mathrm{End}_{\mathbb{F}[G_K]}(\bar{\rho})=\mathbb{F}$, then there exists the universal deformation ring $R_{\bar{\rho}}$ 
 of $\bar{\rho}$, which is a complete noetherian local $\mathcal{O}$-algebra, where $\mathcal{O}$ is the integer ring of $E$. Let $\mathfrak{X}(\bar{\rho})$ be the rigid 
 analytic space over $E$ associated to $R_{\bar{\rho}}$. The universal deformation $V^{\mathrm{univ}}$ of $\bar{\rho}$ over $R_{\bar{\rho}}$ defines a rank two free $\mathcal{O}_{\mathfrak{X}(\bar{\rho})}$-module $\widetilde{V}^{\mathrm{univ}}$ with a continuous $\mathcal{O}_{\mathfrak{X}(\bar{\rho})}$-linear $G_K$-action. 
The space  $\mathfrak{X}(\bar{\rho})$ parametrizes $p$-adic representations $V$ of $G_K$ whose reductions are isomorphic to 
 $\bar{\rho}$ for some  $G_K$-stable lattices of $V$.
  Let $\mathcal{W}$ be the rigid analytic space over $E$ which represents the functor $D_{\mathcal{W}}:\{$ rigid analytic spaces over $E\} 
  \rightarrow (\mathrm{Sets})$ defined by $D_{\mathcal{W}}(X):=\{\delta:\mathcal{O}_{K}^{\times}\rightarrow \Gamma(X, \mathcal{O}_X^{\times}):$ continuous group homomorphisms $\}$ for each rigid analytic space $X$ over $E$. 
  Let $\delta^{\mathrm{univ}}:\mathcal{O}_K^{\times}\rightarrow \Gamma(\mathcal{W}, \mathcal{O}_{\mathcal{W}}^{\times})$ be the 
  universal homomorphism. If we fix a uniformizer $\pi_K\in K$, there exists a unique character $\widetilde{\delta}^{\mathrm{univ}}
  :G_K^{\mathrm{ab}}\rightarrow \Gamma(\mathcal{W}, \mathcal{O}_{\mathcal{W}}^{\times})$ such that $\widetilde{\delta}^{\mathrm{univ}}\circ \mathrm{rec}_K|_{\mathcal{O}_K^{\times}}=\delta^{\mathrm{univ}}$ and $\widetilde{\delta}^{\mathrm{univ}}\circ \mathrm{rec}_K(\pi_K)=1$, where 
  $\mathrm{rec}_K:K\hookrightarrow G_K^{\mathrm{ab}}$ is the reciprocity map of the local 
  class field theory. 
  We denote by $X(\bar{\rho}):=\mathfrak{X}(\bar{\rho})\times_E\mathcal{W}\times_E \mathbb{G}_{m,E}^{\mathrm{an}}$ and 
  denote by $p_1:X(\bar{\rho})\rightarrow \mathfrak{X}(\bar{\rho})$, $p_2:X(\bar{\rho})\rightarrow \mathcal{W}$ and 
  $p_3:X(\bar{\rho})\rightarrow \mathbb{G}_{m,E}^{\mathrm{an}}$ the canonical projections. 
  For $x\in X(\bar{\rho})$, we denote by $E(x)$ the residue field at $x$  which is a finite extension of $E$. 
  Let $x=[V]\in \mathfrak{X}(\bar{\rho})$ be a point which corresponds to a two dimensional trianguline representation $V$ 
   with a triangulation $\mathcal{T}:0\subseteq W(\delta_1)
  \subseteq W(V\otimes_{E(x)} E')$ for some $E'$, then we define a point 
  $x_{(V,\mathcal{T})}:=([V],\delta_1|_{\mathcal{O}_K}^{\times},\delta_1(\pi_K))\in X(\bar{\rho})$. We define 
  $M:=p_1^*(\widetilde{V}^{\mathrm{univ}})((p_2^*\widetilde{\delta}^{\mathrm{univ}})^{-1})$ a rank two $\mathcal{O}_{X(\bar{\rho})}$-module 
  with a continuous $\mathcal{O}_{X(\bar{\rho})}$-linear $G_K$-action. Let $Y:=p_3^*(T)\in \mathcal{O}_{X(\bar{\rho})}^{\times}$ be the pullback 
  of the canonical coordinate $T$ of $\mathbb{G}_{m,E}^{\mathrm{an}}$. If we apply the construction of $X_{fs}$ to 
  the triple $(X(\bar{\rho}), M, Y)$, we obtain a Zariski closed subspace $X(\bar{\rho})_{fs}$ of $X(\bar{\rho})$, which we denote by $\mathcal{E}(\bar{\rho}):=X(\bar{\rho})_{fs}$. The main result of $\S$ 3 is the following theorem (see Theorem \ref{5.14} and 
  Theorem \ref{5.18} for more precise statements), which is a generalization of Proposition 10.4 and 10.6 of \cite{Ki03}.
  
  \begin{thm}
  $\mathcal{E}(\bar{\rho})$ satisfies the following properties.
  \begin{itemize}
  \item[(1)] For any point $x:=([V_x],\delta_x,\lambda_x)\in \mathcal{E}(\bar{\rho})$, $V_x$ is a trianguline representation.
  \item[(2)] Conversely, if $x=[V_x]\in \mathfrak{X}(\bar{\rho})$ is a point such that $V_x$ is a split trianguline $E(x)$-representation with a triangulation 
  $\mathcal{T}:0\subseteq W(\delta_1)\subseteq W(V_x)$ satisfying all the conditions in 
  Proposition \ref{000000}, then the point $x_{(V_x,\mathcal{T})}\in X(\bar{\rho})$ defined as above is contained in $\mathcal{E}(\bar{\rho})$. 
  \item[(3)] For each point $x_{(V_x,\mathcal{T})}$ as in $(2)$, there exists an isomorphism $\hat{\mathcal{O}}_{\mathcal{E}(\bar{\rho}),x_{(V_x,\mathcal{T})}}\isom 
  R_{V_x,\mathcal{T}}$, in particular $\mathcal{E}(\bar{\rho})$ is smooth of its dimension $3[K:\mathbb{Q}_p]+1$ at such points.
  \end{itemize}
  
  \end{thm}
  
  In the next paper \cite{Na11}, the author will generalize all these results for higher dimensional case.
  
  
  
  \smallskip
  
  In the final section $\S$ 4, as an application of $\S$ 2 (in the two dimensional case) and of $\S$ 3, we prove 
  the following theorems concerning the Zariski density of two 
  dimensional crystalline representations.
  We denote by $\mathfrak{X}(\bar{\rho})_{\mathrm{reg}-\mathrm{cris}}:=\{x=[V_x]\in \mathfrak{X}(\bar{\rho})| 
  V_x $ is crystalline with the Hodge-Tate weights
  $\{k_{1,\sigma},k_{2,\sigma}\}_{\sigma:K\hookrightarrow \overline{K}}$ such that $k_{1,\sigma}\not= k_{2,\sigma}$ for 
  any $\sigma \}$, $\mathfrak{X}(\bar{\rho})_{\mathrm{b}}:=\{x\in \mathfrak{X}(\bar{\rho})| V_x $ is crystalline 
  and benign $\}$. We denote by $\overline{\mathfrak{X}(\bar{\rho})}_{\mathrm{b}}$ the Zariski closure of $\mathfrak{X}(\bar{\rho})_{\mathrm{b}}$ in $\mathfrak{X}(\bar{\rho})$.
  \begin{thm}$(\text{Theorem } \mathrm{\ref{6.14}} )$\label{000}
  If $\mathfrak{X}(\bar{\rho})_{\mathrm{reg}-\mathrm{cris}}$ is non empty, 
  then $\mathfrak{X}(\bar{\rho})_{\mathrm{b}}$ is also non empty and the closure $\overline{\mathfrak{X}(\bar{\rho})}_{\mathrm{b}}$ is a union of irreducible components of $\mathfrak{X}(\bar{\rho})$.
  
  \end{thm}
  
  Moreover, under the following assumptions, we prove the following stronger results concerning the Zariski density.
  
  \begin{thm}$(\text{Theorem } \mathrm{\ref{4.17}}, \text{Theorem } \mathrm{\ref{4.18}})$\label{0000}
Assume the following conditions,
\begin{itemize}
\item[(0)]$\mathrm{End}_{G_K}(\bar{\rho})=\mathbb{F}$,
\item[(1)]$\mathfrak{X}(\bar{\rho})_{\mathrm{reg}-\mathrm{cris}}$ is not empty.
\end{itemize}
Moreover, assume one of the following two conditions (2), (3),
\begin{itemize}
\item[(2)]$\zeta_p\not\in K$ ($\zeta_p$ is a primitive root of unity) and $\bar{\rho}$ satisfies one of the following conditions (i), (ii),
\begin{itemize}
\item[(i)]If $\bar{\rho}$ is absolutely reducible, then $\bar{\rho}\otimes_{\mathbb{F}}\overline{\mathbb{F}}\not\sim \begin{pmatrix} 1 & \ast \\ 0 & \omega \end{pmatrix}\otimes \chi$ for 
any $\chi:G_K\rightarrow \overline{\mathbb{F}}^{\times}$, where $\omega$ is the mod $p$ cyclotomic character,
\item[(ii)]If $\bar{\rho}$ is absolutely irreducible, then 
$[K(\zeta_p):K]\not=2$ or $\bar{\rho}|_{I_K}\otimes_{\mathbb{F}}\overline{\mathbb{F}}\not\sim \begin{pmatrix} \chi_2^i & 0 \\ 0 & \chi_2^{ip^f}\end{pmatrix}$ such that $\chi_2^{i(p^f-1)}=\omega|_{I_K}$, where $\chi_2:I_K\rightarrow 
\overline{\mathbb{F}}^{\times}$ is a fundamental character of the second kind,
\end{itemize}
\item[(3)]$\zeta_p\in K$ and $p\not=2$,
\end{itemize}
then we have an equality $\overline{\mathfrak{X}(\bar{\rho})}_{\mathrm{b}}=\mathfrak{X}(\bar{\rho})$.

\end{thm}

  This theorem generalizes the results of Colmez \cite{Co08} and Kisin \cite{Ki10} for $K=\mathbb{Q}_p$ to the case of any $p$-adic field. 
  As is stated in the above paragraph, Chenevier recently proved similar results for higher dimensional 
  and for $K=\mathbb{Q}_p$, and the author will prove these theorems in full generality (i.e. for higher dimensional 
  and for any $p$-adic field) in the next paper.

\subsection*{Notation.}
Let $p$ be a prime number. Let $K$ be a finite extension of $\mathbb{Q}_p$, $\overline{K}$  be a 
fixed algebraic closure of $K$, $K_0$ be the maximal unramified extension of $\mathbb{Q}_p$ in $K$, 
$K^{\mathrm{nor}}$ be the Galois closure of $K$ in $\overline{K}$. Let $G_K:=\mathrm{Gal}(\overline{K}/K)$ be the 
absolute Galois group of $K$. Let $\mathcal{O}_K$ be the ring of integers of $K$, $\pi_K\in\mathcal{O}_K$ be a uniformizer of $K$, $k:=\mathcal{O}_K/\pi_K\mathcal{O}_K$ be the residue field 
of $K$, $q=p^f:=\sharp k$ be the order of $k$. Denote by $\chi_p:G_K\rightarrow \mathbb{Z}_p^{\times}$ the $p$-adic cyclotomic character 
(i.e. $g(\zeta_{p^n})=\zeta_{p^n}^{\chi(g)}$ for any $p^n$-th roots of unity and for any $g\in G_K$). Let $\mathbb{C}_p:=\hat{\overline{K}}$ 
be the $p$-adic completion of $\overline{K}$, which is an algebraically closed $p$-adically completed field, and $\mathcal{O}_{\mathbb{C}_p}$ be its ring of integers. We denote by $v_p$ the normalized valuation on $\mathbb{C}_p^{\times}$ such that 
$v_p(p)=1$. We denote by $|-|_p:\mathbb{C}_p\rightarrow \mathbb{R}_{\geqq 0}$ the absolute value such that 
$|p|_p=\frac{1}{p}$. Let $E$ be a finite extension of $\mathbb{Q}_p$ in $\overline{K}$ such that 
$K^{\mathrm{nor}}\subseteq E$. In this paper, we use the notation $E$ as a coefficient field of representations. 
We denote by $\mathcal{P}:=\{\sigma:K\hookrightarrow \overline{K}\}=\{\sigma:K\hookrightarrow E\}$ the set of 
$\mathbb{Q}_p$-algebra homomorphisms from $K$ to $\overline{K}$ (or $E$).
Let $\chi_{\mathrm{LT}}:G_K\rightarrow \mathcal{O}_K^{\times}$
be 
the Lubin-Tate character associated with the fixed uniformizer $\pi_K$. Let $\mathrm{rec}_K:K^{\times}
\rightarrow G_K^{\mathrm{ab}}$ be the reciprocity map of local class field theory normalized such that 
$\mathrm{rec}_K(\pi_K)$ is a lifting of the inverse of $q$-th power Frobenius on $k$. We remark that 
$\chi_{\mathrm{LT}}\circ \mathrm{rec}_K:K^{\times}\rightarrow \mathcal{O}_K^{\times}$ satisfies 
$\chi_{\mathrm{LT}}\circ\mathrm{rec}_K(\pi_K)=1$ and $\chi_{\mathrm{LT}}\circ\mathrm{rec}_K|_{\mathcal{O}_K^{\times}}
=\mathrm{id}_{\mathcal{O}_K^{\times}}$.
\subsection*{Acknowledgement.}
The author would like to thank Kenichi Bannai for reading the first version of this article and for helpful comments. 
He also would like to thank Tadashi Ochiai for many valuable and interesting discussions and also
thank Yoichi Mieda, Atsushi Shiho and Takahiro Tsushima for discussing 
about rigid geometry which we use in $\S$ 4 and  also thank Ga\"etan Chenevier for discussing 
about generalization to higher dimensional case and also thank Go Yamashita and Seidai Yasuda for 
pointing out some mistakes in the previous version of this article and for many valuable discussions. This work is supported in part by the Grant-in-aid (NO. S- 23224001) for Scientific Research, JSPS. 
Part of this paper was written at l'Institute Henri Poincar\'e during Galois Trimester.  The author would like to thank his host 
Pierre Colmez and all of the people at Keio University related to the JSPS ITP program which enabled my stay. 
 

\section{Deformation theory of trianguline $B$-pairs.}
\subsection{Review of $B$-pairs.}
\subsubsection{$E$-$B$-pairs}
We start by recalling the definition of $E$-$B$-pairs (\cite{Be09},  \cite{Na09}) and then recall some fundamental 
properties of them established in \cite{Na09}. 
First, we recall some rings of $p$-adic periods [Fo94] which we need for defining $B$-pairs.
Let $\widetilde{\bold{E}}^+:=\varprojlim_n\mathcal{O}_{\mathbb{C}_p}\isom
\varprojlim_n\mathcal{O}_{\mathbb{C}_p}/p\mathcal{O}_{\mathbb{C}_p}$, where the limits are taken with respect to the $p$-th power map. It is known that $\widetilde{\bold{E}}^+$ is a complete valuation ring of characteristic  $p$ whose valuation
is defined by $v(x):=v_{p}(x^{(0)})$ for $x=(x^{(n)})_{n\geqq 0}\in \varprojlim_n\mathcal{O}_{\mathbb{C}_p}$.
We fix a system of $p^n$-th roots of unity $\{\varepsilon^{(n)}\}_{n\ge 0}$ 
such that $\varepsilon^{(0)}=1$, $(\varepsilon^{(n+1) })^p=\varepsilon^{(n)}$ for any $n$, 
$\varepsilon^{(1)}\not= 1$. Then $\varepsilon:=(\varepsilon^{(n)})_{n\geqq 0}$ is an element of $\widetilde{\bold{E}}^+$ satisfying $v(\varepsilon -1)=p/(p-1)$. The topological group 
$G_K$ acts on this ring continuously in natural way. We put  $\widetilde{\bold{A}}^+:=W(\widetilde{\bold{E}}^+)$, where we denote by $W(R)$ the ring of Witt vectors in $R$ for any perfect ring $R$. We put $\widetilde{\bold{B}}^+:=\widetilde{\bold{A}}^+[\frac{1}{p}]$. These rings are equipped with the weak topology and also have a natural continuous $G_K$-action and have a Frobenius 
action $\varphi$. We have a $G_K$-equivariant surjection $\theta:\widetilde{\bold{A}}^+
\rightarrow \mathcal{O}_{\mathbb{C}_p}:\sum_{k=0}^{\infty}p^k[x_k]\mapsto \sum_{k=0}^{\infty} p^k x_k^{(0)}$, where $[\,\, ]:\widetilde{\bold{E}}^+\rightarrow \widetilde{\bold{A}}^+$ is the
Teichm\"uller lift.
Inverting $p$, we obtain a surjection $\widetilde{\bold{B}}^+\rightarrow \mathbb{C}_p$.
We put $\bold{B}_{\mathrm{dR}}^+:=\varprojlim_n\widetilde{\bold{B}}^+/(\mathrm{Ker}(\theta))^n$, which is
a complete discrete valuation ring with the residue field $\mathbb{C}_p$ and 
is equipped with the projective limit topology of the $\mathbb{Q}_p$-Banach spaces
$\widetilde{\bold{B}}^+/(\mathrm{Ker}(\theta))^n$ ($n\geqq 1$) whose $\mathbb{Z}_p$-lattice is the image of the natural map 
$\widetilde{\bold{A}}^+\rightarrow 
\widetilde{\bold{B}}^+/(\mathrm{Ker}(\theta))^n$. 
Let $\bold{A}_{\mathrm{max}}$ be the $p$-adic completion of $\widetilde{\bold{A}}^+[\frac{[\widetilde{p}]]}{p}]$, where $\widetilde{p}:=(p^{(n)})$ is an element in $\widetilde{\bold{E}}^+$ such that $p^{(0)}=p, (p^{(n+1) })^p=p^{(n)}$ for any $n$. We put $\bold{B}_{\mathrm{max}}^+:=\bold{A}_{\mathrm{max}}[\frac{1}{p}]$. $\bold{A}_{\mathrm{max}}$ and  
$\bold{B}^+_{\mathrm{max}}$ have a continuous $G_K$-action and a Frobenius actions $\varphi$. We have a natural 
$G_K$-equivariant embedding $K\otimes_{K_0}\bold{B}^+_{\mathrm{max}}\hookrightarrow \bold{B}^+_{\mathrm{dR}}$.
If we put $t:=\mathrm{log}([\varepsilon])=\sum_{n=1}^{\infty}(-1)^{n-1}\frac{([\varepsilon]-1)^n}{n}$, then we can see that $t\in \bold{A}_{\mathrm{max}}$, 
$\varphi(t)=pt, g(t)=\chi_p(g)t$ for any $g\in G_K$ and $\mathrm{Ker}(\theta)=t\bold{B}_{\mathrm{dR}}^+
\subset \bold{B}_{\mathrm{dR}}^+$ is the maximal ideal. If we put $\bold{B}_{\mathrm{max}}:=\bold{B}_{\mathrm{max}}^+
[\frac{1}{t}], \bold{B}_{\mathrm{dR}}:=\bold{B}_{\mathrm{dR}}^+[\frac{1}{t}]$, we have a natural embedding 
$K\otimes_{K_0}\bold{B}_{\mathrm{max}}\hookrightarrow \bold{B}_{\mathrm{dR}}$. 
We put $\bold{B}_{e}:=\bold{B}_{\mathrm{max}}^{\varphi =1}$ which is equipped with the locally convex inductive limit topology of 
$\bold{B}_{e}=\cup_{n}(\frac{1}{t^n}\bold{B}^+_{\mathrm{max}})^{\varphi=1}$, where the topology on each $(\frac{1}{t^n}\bold{B}^+_{\mathrm{max}})^{\varphi=1}
=\frac{1}{t^n}\bold{B}^{+,  \varphi=p^n}_{\mathrm{max}}$ is induced that of $\bold{B}^+_{\mathrm{max}}$. We put $\mathrm{Fil}^i \bold{B}_{\mathrm{dR}}:=t^i \bold{B}_{\mathrm{dR}}^+$ for any $i\in\mathbb{Z}$.  On $\bold{B}_{\mathrm{dR}}$, we also equipped with the locally convex inductive limit 
topology of $\bold{B}_{\mathrm{dR}}=\varinjlim_n\frac{1}{t^n}\bold{B}^+_{\mathrm{dR}}$.

In this paper, we fix a coefficient field of $p$-adic representations or $B$-pairs.
Hence we start by recalling the definition of $E$-coefficient versions of $p$-adic representations and $B$-pairs.
\begin{defn}\label{a}
An $E$-representation of $G_K$ is a finite dimensional $E$-vector space $V$ with a continuous $E$-linear action of $G_K$. We call $E$-representation for simplicity when there is no risk 
of confusion about $K$.
\end{defn}
\begin{defn} \label{b}
A pair $W:=(W_e,W^+_{\mathrm{dR}})$ is an $E$-$B$-pair if 
\begin{itemize}
\item[(1)] $W_e$ is a finite $\bold{B}_{e}\otimes_{\mathbb{Q}_p}E$-module which is free over $\bold{B}_{e}$ with 
a continuous semi-linear $G_K$-action.
\item[(2)] $W^+_{\mathrm{dR}}$ is a $G_K$-stable finite $\bold{B}^+_{\mathrm{dR}}\otimes_{\mathbb{Q}_p}E$-submodule 
of $\bold{B}_{\mathrm{dR}}\otimes_{\bold{B}_{e}}W_e$ which generates $\bold{B}_{\mathrm{dR}}\otimes_{\bold{B}_{e}}W_e$ as a $\bold{B}_{\mathrm{dR}}$-module.
\end{itemize}
\end{defn}
We have an exact  fully faithful functor $W(-)$ from the category of $E$-representations to the 
category of $E$-$B$-pairs defined by
$$W(V):=(\bold{B}_{e}\otimes_{\mathbb{Q}_p}V, \bold{B}^+_{\mathrm{dR}}\otimes_{\mathbb{Q}_p}V)$$
for any $E$-representation $V$, 
where the fully faithfulness follows from the Bloch-Kato's fundamental short exact sequence,
$$0\rightarrow \mathbb{Q}_p\rightarrow \bold{B}_{e}\oplus \bold{B}^+_{\mathrm{dR}}\rightarrow \bold{B}_{\mathrm{dR}}\rightarrow 0.$$
We remark that $W_e$ is a free $\bold{B}_{e}\otimes_{\mathbb{Q}_p}E$-module and 
$W^+_{\mathrm{dR}}$ is a free $\bold{B}^+_{\mathrm{dR}}\otimes_{\mathbb{Q}_p}E$-module by Lemma 1.7, 1.8 of \cite{Na09}. We define the rank of $W$ by 
$$\mathrm{rank}(W):=\mathrm{rank}_{\bold{B}_{e}\otimes_{\mathbb{Q}_p}E}(W_e).$$
For $E$-$B$-pairs $W_1:=(W_{1,e},W^+_{1,\mathrm{dR}})$ and $W_2:=(W_{2,e}, W^+_{2,\mathrm{dR}})$, 
we define the tensor product of $W_1$ and $W_2$ by  
$$W_1\otimes W_2:=(W_{1,e}\otimes_{\bold{B}_{e}\otimes_{\mathbb{Q}_p}E}W_{2,e}, W^+_{1,\mathrm{dR}}\otimes_{\bold{B}^+_{\mathrm{dR}}\otimes_{\mathbb{Q}_p}E}W^+_{2,\mathrm{dR}})$$
 and define the dual of $W_1$ 
by
$$W_1^{\vee}:=(\mathrm{Hom}_{\bold{B}_{e}\otimes_{\mathbb{Q}_p}E}(W_{1,e}, B_{e}\otimes_{\mathbb{Q}_p}E), W^{+,\vee}_{1,\mathrm{dR}})$$
where we define 
$$W_{1,\mathrm{dR}}^{+,\vee}:=\{f\in \mathrm{Hom}_{\bold{B}_{\mathrm{dR}}\otimes_{\mathbb{Q}_p}E}(\bold{B}_{\mathrm{dR}}
\otimes_{\bold{B}_{e}}W_{1,e}, \bold{B}_{\mathrm{dR}}\otimes_{\mathbb{Q}_p}E)| f(W^+_{1,\mathrm{dR}})\subseteq \bold{B}^+_{\mathrm{dR}}\otimes_{\mathbb{Q}_p}E\})$$
  (remark: there is a mistake in
the definition 1.9 of \cite{Na09}).
The category of $E$-$B$-pairs is not an abelian category. In particular, an inclusion $W_1\hookrightarrow W_2$
does not have a quotient in the category of $E$-$B$-pairs in general. However we can always take the saturation 
$$W^{\mathrm{sat}}_{1}:=(W^{\mathrm{sat}}_{1.e}, W^{+,\mathrm{sat}}_{1,\mathrm{dR}})$$ 
such that $W^{\mathrm{sat}}_1$ sits in 
$W_1\hookrightarrow W^{\mathrm{sat}}_1\hookrightarrow W_2$ and $W_{1,e}=W^{\mathrm{sat}}_{1,e}$ and $W_2/W^{\mathrm{sat}}_{1}$ is an $E$-$B$-pair (see Lemma 1.14 of \cite{Na09}). We say that an inclusion $W_1\hookrightarrow W_2$ is saturated if $W_2/W_1$ is an $E$-$B$-pair, i.e. $W_1=W_1^{\mathrm{sat}}$.

Next, we recall the $p$-adic Hodge theory for $B$-pairs. Let $W=(W_e,W^+_{\mathrm{dR}})$ be an $E$-$B$-pair. 
We define 
$$\bold{D}_{\mathrm{cris}}(W):=(\bold{B}_{\mathrm{max}}\otimes_{\bold{B}_{e}}W_e)^{G_K}, \,\bold{D}^L_{\mathrm{cris}}(W):=
(\bold{B}_{\mathrm{max}}\otimes_{\bold{B}_{e}}W_e)^{G_L}$$ for any finite extension $L$ of $K$ and define
$$\bold{D}_{\mathrm{dR}}(W):=(\bold{B}_{\mathrm{dR}}\otimes_{\bold{B}_{e}}
W_e)^{G_K},\, \bold{D}_{\mathrm{HT}}(W):=(\bold{B}_{\mathrm{HT}}\otimes_{\mathbb{C}_p}(W^+_{\mathrm{dR}}/tW^+_{\mathrm{dR}}))^{G_K},$$
 here 
$\bold{B}_{\mathrm{HT}}:=\mathbb{C}_p[T, T^{-1}]$ on which $G_K$ acts by $g(aT^i):=\chi_p(g)^ig(a)T^i$ for any $g\in G_K, a\in \mathbb{C}_p, i\in \mathbb{Z}$.
\begin{defn}\label{c}
We say that $W$ is crystalline (resp. de Rham, resp. Hodge-Tate) if 
$\mathrm{dim}_{K_{\ast}}\bold{D}_{\mathrm{\ast}}(W)=[E:\mathbb{Q}_p]\mathrm{rank}(W)$ for $\ast= \mathrm{cris}$ (resp. $\ast=\mathrm{dR}$, resp. 
$\ast=\mathrm{HT}$), where $K_{\ast}=K_0$ when $\ast=\mathrm{cris}$ and $K_{\ast}=K$ when $\ast=\mathrm{dR},\mathrm{HT}$.
We say that $W$ is potentially crystalline if $\mathrm{dim}_{L_0}(\bold{D}^L_{\mathrm{cris}}(W))=[E:\mathbb{Q}_p]\mathrm{rank}(W)$ for 
a finite extension $L$ of $K$, where $L_0$ is the maximal unramified extension of $\mathbb{Q}_p$ in $L$.
\end{defn}

\begin{defn}\label{d}
Let $L$ be a finite Galois extension of $K$. Set $G_{L/K}:=\mathrm{Gal}(L/K)$.
We say that $D$ is an $E$-filtered ($\varphi,G_{L/K}$)-module over $K$ if 
\begin{itemize}
\item[(1)] $D$ is a finite free $L_0\otimes_{\mathbb{Q}_p} E$-module with a $\varphi$-semi-linear 
action $\varphi_D:D\isom D$ and a semi-linear action of $G_{L/K}$ such
 that $\varphi_D$ and  the action of $G_{L/K}$ commute, where ($\varphi$-)semi-linear means that 
$\varphi_D(a\otimes b\cdot x)=\varphi(a)\otimes b\cdot \varphi_D(x)$, $g(a\otimes b\cdot x)=g(a)\otimes b\cdot g(x)$ for 
any $a\in L_0, b\in E, x\in D,  g\in G_{L/K}$,
\item[(2)] $D_L:=L\otimes_{L_0}D$ has a separated and exhausted decreasing $G_{L/K}$-stable filtration $\{\mathrm{Fil}^iD_L\}_{i\in \mathbb{Z}}$ by $L\otimes_{\mathbb{Q}_p}E$-submodules.
\end{itemize}
\end{defn}

Let $W$ be a potentially crystalline $E$-$B$-pair such that $W|_{G_L}$ is crystalline 
for a finite Galois extension $L$ of $K$, then we define an $E$-filtered ($\varphi, G_{L/K}$)-module's 
structure on $\bold{D}^L_{\mathrm{cris}}(W)$ as follows. First, $\bold{D}^L_{\mathrm{cris}}(W)$ has a Frobenius action 
induced from that on $\bold{B}_{\mathrm{max}}$ and has a $G_{L/K}$-action induced from those on $\bold{B}_{\mathrm{max}}$ and 
$W_e$. 
We define a filtration on $L\otimes_{L_0}\bold{D}^L_{\mathrm{cris}}(W)=L\otimes_{K}\bold{D}_{\mathrm{dR}}(W)$ by 
$$\mathrm{Fil}^i(L\otimes_{L_0}\bold{D}^{L}_{\mathrm{cris}}(W)):=(L\otimes_{K}\bold{D}_{\mathrm{dR}}(W))\cap t^iW^+_{\mathrm{dR}}$$
for any $i\in \mathbb{Z}$.

Let $D:=L_0e$ be a rank one $\mathbb{Q}_p$-filtered ($\varphi, G_{L/K}$)-module with a base $e$, then 
we define $t_N(D):=v_p(\alpha)$ where $\varphi_D(e)=\alpha\cdot e$ and define $t_H(D):=i$ such that 
$\mathrm{Fil}^iD_L/\mathrm{Fil}^{i+1}D_L\not= 0$. For general $D$ of rank $d$, we define $t_N(D):=t_N(\wedge^dD)$, 
$t_H(D):=t_H(\wedge^dD)$. We say that $D$ is weakly admissible if $t_N(D)=t_H(D)$ and $t_N(D')\geqq t_H(D')$ for any 
sub $\mathbb{Q}_p$-filtered ($\varphi,G_{L/K}$)-module $D'$ of $D$.
\begin{thm}\label{e}
Let $L$ be a finite Galois extension of $K$, then we have the following results.
\begin{itemize}
\item[(1)] The functor $W\mapsto \bold{D}^L_{\mathrm{cris}}(W)$ gives an equivalence of categories between 
the category of potentially crystalline $E$-$B$-pairs which are crystalline if restricted to $G_L$ and the 
category of $E$-filtered $(\varphi,G_{L/K})$-modules over $K$.
\item[(2)] Restricting the above functor to $E$-representations, the functor $V\mapsto \bold{D}^{L}_{\mathrm{cris}}(V)$ 
gives an equivalence of categories between the category of potentially crystalline $E$-representations which are 
crystalline if restricted to $G_L$  and the category of weakly admissible $E$-filtered $(\varphi,G_{L/K})$-modules over $K$.
\end{itemize}
\end{thm}
\begin{proof}
See Proposition 2.3.4 and Theorem 2.3.5 of \cite{Be09} or Theorem 1.18 of \cite{Na09}.
\end{proof}

Next, we recall the definition of trianguline $E$-$B$-pairs, whose deformation theory we study in detail in this chapter. 
\begin{defn}\label{f}
Let $W$ be an $E$-$B$-pair of rank $n$, then we say that $W$ is split trianguline if 
there exists a filtration 
$$\mathcal{T}:0\subseteq W_1\subseteq W_2\subseteq \cdots \subseteq W_n=W$$
 by sub $E$-$B$-pairs such that $W_{i}$ is saturated in $W_{i+1}$ and $W_{i+1}/W_{i}$ is a rank one 
$E$-$B$-pair for any $0\leqq i\leqq n-1$. We say that $W$ is trianguline if $W\otimes_{E}E'$, the base change of $W$  to $E'$, 
is a split trianguline $E'$-$B$-pair for a finite extension $E'$ of $E$. 
\end{defn} 

By this definition, to study split trianguline $E$-$B$-pairs, it is important to classify rank one $E$-$B$-pairs 
and calculate extension classes of rank one $E$-$B$ pairs, which were studied in \cite{Na09}. 
We recall some  results concerning these.

\begin{thm}\label{g}
There exists a canonical one to one correspondence $\delta\mapsto W(\delta)$ between the set of 
continuous homomorphisms $\delta:K^{\times}\rightarrow E^{\times}$ and  the set of isomorphism classes
of rank one $E$-$B$-pairs. 
\end{thm}
\begin{proof}
See  Proposition 3.1 of \cite{Co08} for $K=\mathbb{Q}_p$ and Theorem 1.45 of \cite{Na09} for general $K$.
For the construction of $W(\delta)$, see $\S$1.4 of \cite{Na09}.
\end{proof}

This correspondence is compatible with the local class field theory, i.e. for any unitary homomorphism 
$\delta:K^{\times}\rightarrow \mathcal{O}_E^{\times}$, if we take  the character $\widetilde{\delta}:G^{\mathrm{ab}}_K\rightarrow \mathcal{O}_E^{\times}$ 
satisfying $\widetilde{\delta}\circ \mathrm{rec}_K=\delta$, then we have a canonical  isomorphism $$W(\delta)\isom W(E(\widetilde{\delta})).$$ 
This correspondence is also compatible with tensor products and with duals, i.e for continuous homomorphisms $\delta_1, \delta_2:K^{\times}
\rightarrow E^{\times}$, we have  canonical isomorphisms $$W(\delta_1)\otimes W(\delta_2)\isom W(\delta_1\delta_2)\text{ and } W(\delta_1)^{\vee}
\isom W(\delta_1^{-1}).$$ 

There are some important examples of rank one $E$-$B$-pairs which we recall now.
For any $\{k_{\sigma}\}_{\sigma\in\mathcal{P}}\in \prod_{\sigma\in \mathcal{P}}\mathbb{Z}$, we define a homomorphism 
$$\prod_{\sigma\in\mathcal{P}}\sigma^{k_{\sigma}}:K^{\times}\rightarrow E^{\times}:y\mapsto \prod_{\sigma\in \mathcal{P}}
\sigma(y)^{k_{\sigma}},$$  then we have an isomorphism 
$$W(\prod_{\sigma\in\mathcal{P}}\sigma^{k_{\sigma}})\isom(\bold{B}_{e}\otimes_{\mathbb{Q}_p}E, 
\oplus_{\sigma\in\mathcal{P}}t^{k_{\sigma}}\bold{B}^+_{\mathrm{dR}}\otimes_{K,\sigma}E)$$ (see  Lemma 2.12 of \cite{Na09}).
 Let  $N_{K/\mathbb{Q}_p}:K^{\times}\rightarrow \mathbb{Q}_p^{\times}$ be the norm and 
 $\allowbreak|-|:\mathbb{Q}_p^{\times}\rightarrow \mathbb{Q}^{\times}\hookrightarrow E^{\times}$ be the $p$-adic absolute value 
 such that $|p|=\frac{1}{p}$, and we define $|N_{K/\mathbb{Q}_p}|:K^{\times}\rightarrow E^{\times}$  the composite
 of $N_{K/\mathbb{Q}_p}$ and $|-|$, then we have an isomorphism 
 $$W(|N_{K/\mathbb{Q}_p}|\prod_{\sigma\in \mathcal{P}}\sigma)\isom W(E(\chi_p)),$$ which is the $E$-$B$-pair
 associated to the $p$-adic cyclotomic character $\chi_p$.  
 Next, we recall the definition and some properties of Galois cohomology of $E$-$B$-pairs. 
 For an $E$-$B$-pair $W:=(W_e, W^+_{\mathrm{dR}})$, we put $W_{\mathrm{dR}}:=\bold{B}_{\mathrm{dR}}\otimes_{\bold{B}_{e}}W_e$. 
 We have natural inclusions $W_e\hookrightarrow W_{\mathrm{dR}}$ and $W^+_{\mathrm{dR}}
 \hookrightarrow W_{\mathrm{dR}}$. We define the Galois cohomology $\mathrm{H}^i(G_K, W)$ of $W$ as the cohomology of the continuous cochains of $G_K$ with values in the complex 
 $$W_e\oplus W^+_{\mathrm{dR}}\rightarrow W_{\mathrm{dR}}:(x,y)\mapsto 
 x-y,$$
  see the appendix of this article or $\S$2.1 of \cite{Na09} for the precise definition. As in the case 
  of usual $p$-adic representations, we have the following isomorphisms 
  of $E$-vector spaces
  $$\mathrm{H}^0(G_K, W)
 \isom\mathrm{Hom}_{G_K}(\bold{B}_E, W),\, \mathrm{H}^1(G_K, W)\isom\mathrm{Ext}^1(\bold{B}_E, W),$$
  where $\bold{B}_E:=(\bold{B}_{e}\otimes_{\mathbb{Q}_p}E, 
 \bold{B}^+_{\mathrm{dR}}\otimes_{\mathbb{Q}_p}E)$ is the trivial $E$-$B$-pair and $\mathrm{Hom}_{G_K}(-,-)$ is the group of homomorphisms
 of $E$-$B$-pairs and $\mathrm{Ext}^1(-,-)$ is the extension class group in the category of $E$-$B$-pairs. If 
 $V$ is an $E$-representation of $G_K$, we have a canonical isomorphism 
 $$\mathrm{H}^i(G_K, V)\isom \mathrm{H}^i(G_K, W(V)),$$ which follows from 
 the Bloch-Kato's fundamental short exact sequence.
 Moreover, we have the following theorem, the Euler-Poincar\'e characteristic  formula and the Tate duality for 
 $B$-pairs. 
 \begin{thm}\label{h}
 Let $W$ be an $E$-$B$-pair. 
 \begin{itemize}
 \item[(1)] For $i=0,1,2$, $\mathrm{H}^i(G_K, W)$ is finite dimensional over $E$ and $\mathrm{H}^i(G_K, W)=0$ for 
 $i\not= 0,1,2$.
 \item[(2)] $\sum_{i=0}^2(-1)^{i-1} \mathrm{dim}_E\mathrm{H}^i(G_K, W)=[K:\mathbb{Q}_p] \mathrm{rank}(W)$.
 \item[(3)]For any $i=0,1,2$, there is a natural perfect pairing defined by cup 
 product,
 \item[] \[
              \begin{array}{ll}
              \mathrm{H}^i(G_K, W)\times\mathrm{H}^{2-i}(G_K, W^{\vee}(\chi_p))&\rightarrow \mathrm{H}^2(G_K, W\otimes W^{\vee}(
 \chi_p)) \\
  &\rightarrow \mathrm{H}^2(G_K, W(E(\chi_p))\allowbreak \isom E,
  \end{array}
  \]
 \item[] where the last isomorphism is the Tate's trace map.
 \end{itemize}
 \end{thm}
 \begin{proof}
 See Theorem \ref{5-6} and Theorem \ref{5-7} in the appendix.
 \end{proof}
 \begin{rem}
 In \cite{Li08}, Liu proved all these results for the cohomology of $(\varphi,\Gamma)$-modules over the Robba ring. 
 In the appendix of this article, we first prove the finiteness and the Euler-Poincar\'e formula (Theorem \ref{5-6}) 
 for the Galois cohomology of $B$-pairs using the theory of almost $\mathbb{C}_p$-representations. Then, 
 we prove the Tate duality (Theorem \ref{5-7} ) for the Galois cohomology of $B$-pairs using the 
 Liu's argument. After establishing these properties, 
 we prove the comparison results (Theorem \ref{5.100}) between the cohomology of $(\varphi,\Gamma)$-modules with that of the corresponding 
 $B$-pairs.
 
  \end{rem}
 Using these formulae, we obtain the following dimension formulae of Galois cohomologies of rank one $E$-$B$-pairs.
 \begin{prop}\label{i}
 Let $\delta:K^{\times}\rightarrow E^{\times}$ be a continuous homomorphism, then we have$:$
 \begin{itemize}
 \item[(1)] $\mathrm{H}^0(G_K, W(\delta))\isom E $ if $\delta=\prod_{\sigma\in \mathcal{P}}\sigma^{k_{\sigma}}$ 
 such that $\{k_{\sigma}\}_{\sigma\in \mathcal{P}}\in \prod_{\sigma\in \mathcal{P}}\mathbb{Z}_{\leqq 0}$
 , and $\mathrm{H}^0(G_K, W(\delta))=0$ otherwise.
 \item[(2)] $\mathrm{H}^2(G_K, W(\delta))\isom E$ if $\delta=|N_{K/\mathbb{Q}_p}|\prod_{\sigma\in\mathcal{P}}\sigma^{k_{\sigma}}$ 
 such that  $\{k_{\sigma}\}_{\sigma\in \mathcal{P}}\in \prod_{\sigma\in \mathcal{P}}\mathbb{Z}_{\geqq 1}$, and $\mathrm{H}^2(G_K, W(\delta))=0$ otherwise.
 \item[(3)] $\mathrm{dim}_E\mathrm{H}^1(G_K, W(\delta))=[K:\mathbb{Q}_p] +1$ if $\delta=\prod_{\sigma\in \mathcal{P}}\sigma
 ^{k_{\sigma}}$ such that $\{k_{\sigma}\}_{\sigma\in \mathcal{P}}\in \prod_{\sigma\in \mathcal{P}}\mathbb{Z}_{\leqq 0}$ or $\delta=
 |N_{K/\mathbb{Q}_p}|\prod_{\sigma\in \mathcal{P}}\sigma^{k_{\sigma}}$ such that $\{k_{\sigma}\}_{\sigma\in \mathcal{P}}\in \prod_{\sigma\in \mathcal{P}}\mathbb{Z}_{\geqq 1}$, and $\mathrm{dim}_E\mathrm{H}^1(G_K, W(\delta))=[K:\mathbb{Q}_p]$ otherwise.
 \end{itemize}
 \end{prop}
 \begin{proof}
 See Theorem 2.9 and Theorem 2.22 of \cite{Co08} for $K=\mathbb{Q}_p$. For general $K$, the results can be proved by using 
 Proposition 2.14 and Proposition 2.15 of \cite{Na09} and Tate duality for $B$-pairs.
 \end{proof}

\subsubsection{$B$-pairs over Artin local rings}
Here, we define $B$-pairs over Artin local rings, which we need to define the notion of deformations of $E$-$B$-pairs.
Let $\mathcal{C}_E$ be the category of Artin local $E$-algebra $A$ with the residue field $E$. The morphisms in $\mathcal{C}_E$ are given by 
local $E$-algebra homomorphisms. For $A\in \mathcal{C}_E$, we denote by  $\mathfrak{m}_A$ the maximal ideal of $A$. We define the $A$-coefficient version of $B$-pairs as follows.
\begin{defn}\label{1}
We call a pair $W:=(W_e, W^+_{\mathrm{dR}})$ an $A$-$B$-pair of $G_K$ if 
\begin{itemize}
\item[(1)] $W_e$ is a finite $\bold{B}_{e}\otimes_{\mathbb{Q}_p}A$-module which is flat over $A$ 
and is free over $\bold{B}_{e}$, with a continuous semi-linear $G_K$-action.
\item[(2)] $W^+_{\mathrm{dR}}$ is a finite generated $\bold{B}^+_{\mathrm{dR}}\otimes_{\mathbb{Q}_p}A$-submodule of $\bold{B}_{\mathrm{dR}}\otimes_{\bold{B}_{e}}W_e$ which is stable by the $G_K$-action and which generates 
$\bold{B}_{\mathrm{dR}}\otimes_{\bold{B}_{e}}W_e$ as a $\bold{B}_{\mathrm{dR}}\otimes_{\mathbb{Q}_p}A$-module such that 
$W^+_{\mathrm{dR}}/tW^+_{\mathrm{dR}}$ is flat over $A$.
\end{itemize}
For an $A$-$B$-pair $W:=(W_e,W^+_{\mathrm{dR}})$, we put $W_{\mathrm{dR}}:=\bold{B}_{\mathrm{dR}}
\otimes_{\bold{B}_{e}}W_e$.
\end{defn}
We simply call an $A$-$B$-pair if there is no risk of confusing about $K$.

\begin{lemma}\label{2}
Let $W:=(W_e,W^+_{\mathrm{dR}})$ be an $A$-$B$-pair. Then 
$W_e$ is a finite free $\bold{B}_{e}\otimes_{\mathbb{Q}_p}A$-module,
$W^+_{\mathrm{dR}}$ is a finite free $\bold{B}^+_{\mathrm{dR}}\otimes_{\mathbb{Q}_p}A$-module 
and $W^+_{\mathrm{dR}}/tW^+_{\mathrm{dR}}$ is a finite free $\mathbb{C}_p\otimes_{\mathbb{Q}_p}A$-module.
\end{lemma}
\begin{proof}
First, we prove the assertion for $W_e$. Because the submodule $\mathfrak{m}_A W_e\subseteq W_e$ is a $G_K$-stable finite generated 
torsion free $\bold{B}_{e}$-module and because $\bold{B}_{e}$ is a B\'ezout domain by Proposition 1.1.9 of \cite{Be08}, $\mathfrak{m}_AW_e$ is a finite free $\bold{B}_{e}$-module by Lemma 2.4 of \cite{Ke04}. 
By Lemma 2.1.4 of \cite{Be08}, the cokernel $W_e\otimes_A E$ is also a finite free $\bold{B}_{e}$-module 
(with an $E$-action). By Lemma 1.7 of \cite{Na09}, $W_e\otimes_AE$ is a finite free $\bold{B}_{e}\otimes_{\mathbb{Q}_p}E$-module of 
some rank $n$.
We take a $\bold{B}_{e}\otimes_{\mathbb{Q}_p}A$-linear morphism $f:(\bold{B}_{e}\otimes_{\mathbb{Q}_p}A)^n\rightarrow W_e$ which 
is a lift of a $\bold{B}_{e}\otimes_{\mathbb{Q}_p}E$-linear isomorphism $(\bold{B}_{e}\otimes_{\mathbb{Q}_p}E)^n\isom W_e\otimes_A E$.
Because $A$ is Artinian, $f$ is surjective. Because $W_e$ is $A$-flat, we have $\mathrm{Ker}(f)\otimes_AE=0$, 
hence $\mathrm{Ker}(f)=0$. Hence $W_e$ is a free $\bold{B}_{e}\otimes_{\mathbb{Q}_p}A$-module.

Next, we prove that $W^+_{\mathrm{dR}}$ is a free $\bold{B}^+_{\mathrm{dR}}\otimes_{\mathbb{Q}_p}A$-module.
Because $W_e$ is a free $\bold{B}_{e}\otimes_{\mathbb{Q}_p}A$-module, $W_{\mathrm{dR}}$ is a free $\bold{B}_{\mathrm{dR}}\otimes_{\mathbb{Q}_p}A$-module, in particular this is flat over $A$. Because $W^+_{\mathrm{dR}}/tW^+_{\mathrm{dR}}$ is a flat $A$-module, $W_{\mathrm{dR}}/W^+_{\mathrm{dR}}$ is also  a flat $A$-module. Hence $W^+_{\mathrm{dR}}$ is also flat over $A$. 
By the $A$-flatness of $W_{\mathrm{dR}}/W^+_{\mathrm{dR}}$, we have an inclusion 
$W^+_{\mathrm{dR}}\otimes_A E\hookrightarrow W_{\mathrm{dR}}\otimes_{A} E$, hence 
$W^{+}_{\mathrm{dR}}\otimes_A E$ is a finite generated torsion free $\bold{B}^+_{\mathrm{dR}}$-module, hence is a
free $\bold{B}^+_{\mathrm{dR}}$-module. By Lemma 1.8 of \cite{Na09}, we can show that 
$W^+_{\mathrm{dR}}$ is a free $\bold{B}^+_{\mathrm{dR}}\otimes_{\mathbb{Q}_p}A$-module in the same way as in the case of $W_e$.
The freeness over $\mathbb{C}_p\otimes_{\mathbb{Q}_p}A$ of $W^+_{\mathrm{dR}}/tW^+_{\mathrm{dR}}$ follows from the
$\bold{B}^+_{\mathrm{dR}}\otimes_{\mathbb{Q}_p}A$-freeness of $W^+_{\mathrm{dR}}$.

\end{proof}
\begin{defn}
Let $W=(W_e, W^+_{\mathrm{dR}})$ be an $A$-$B$-pair. We define the rank of $W$ by 
$\mathrm{rank}(W):=\mathrm{rank}_{\bold{B}_{e}\otimes_{\mathbb{Q}_p}A}(W_e)$.
\end{defn}

\begin{defn}\label{3}
Let $f:A\rightarrow A'$ be a morphism in $\mathcal{C}_E$, and let $W=(W_e,W_{\mathrm{dR}})$ be an $A$-$B$-pair.
We define the base change of $W$ to $A'$ by 
$$W\otimes_A A':=(W_e\otimes_A A',W^+_{\mathrm{dR}}\otimes_A A').$$
By Lemma $\ref{2}$, we can easily see that this is an $A'$-$B$-pair.
\end{defn}

\begin{defn}\label{4}
Let $W_1=(W_{e,1},W^+_{\mathrm{dR},1})$, $W_2=(W_{e,2}, W^+_{\mathrm{dR},2})$ be $A$-$B$-pairs.
We define the tensor product of $W_1$ and $W_2$ by $W_1\otimes W_2:=(W_{e,1}\otimes_{\bold{B}_{e}\otimes_{\mathbb{Q}_p}A}W_{e,2}, 
W^+_{\mathrm{dR},1}\otimes_{\bold{B}^+_{\mathrm{dR}}\otimes_{\mathbb{Q}_p}A}W^+_{\mathrm{dR},2})$ , 
and define the dual of $W_1$ by $W_1^{\vee}:=(\mathrm{Hom}_{\bold{B}_{e}\otimes_{\mathbb{Q}_p}A}(W_{e,1},\allowbreak \bold{B}_{e}\otimes_{\mathbb{Q}_p}A), 
W^{+,\vee}_{\mathrm{dR},1})$. Here, $W^{+,\vee}_{\mathrm{dR},1}:=\{f\in\mathrm{Hom}_{\bold{B}_{\mathrm{dR}}\otimes_{\mathbb{Q}_p}A}(
W_{\mathrm{dR},1}, \bold{B}_{\mathrm{dR}}\otimes_{\mathbb{Q}_p}A)|f(W^{+}_{\mathrm{dR},1})\subseteq \bold{B}^+_{\mathrm{dR}}\otimes_{\mathbb{Q}_p}A\}$. By Lemma $\ref{2}$, we can easily see that these are $A$-$B$-pairs.
\end{defn}

Next, we classify rank one $A$-$B$-pairs.
Let $\delta:K^{\times}\rightarrow A^{\times}$ be a continuous homomorphism, then we define a rank one $A$-$B$-pair $W(\delta)$ as follows.
Let $\bar{u}\in E^{\times}$ be the image of $u:=\delta(\pi_K)\in A^{\times}$ by the canonical 
projection $A\rightarrow A/\frak{m}_A=E$.
We define a homomorphism $\delta_0:K^{\times}\rightarrow A^{\times}$ such that $\delta_0|_{\mathcal{O}_K^{\times}}
=\delta|_{\mathcal{O}_K^{\times}}$, $\delta_0(\pi_K)=u/\bar{u}$. Because $u/\bar{u}\in 1+\mathfrak{m}_A$ ,  $(u/\bar{u})^{p^n}$ ($n\rightarrow \infty$)
converges to $1\in A^{\times}$. If we fix an isomorphism $K^{\times}\isom \mathcal{O}_K^{\times}\times \mathbb{Z}
:v\pi_K^n\mapsto (v,n)$ ($v\in\mathcal{O}_K^{\times}$), then $\delta_0$ uniquely extends to a continuous homomorphism $\delta_0':\mathcal{O}_K^{\times}\times 
\hat{\mathbb{Z}}\rightarrow \mathcal{O}_K^{\times}\times \mathbb{Z}_p\rightarrow A^{\times}$, where 
the first map is induced by the natural projection $\hat{\mathbb{Z}}\rightarrow \mathbb{Z}_p$. By the local class field theory, then there exists a unique character $\widetilde{\delta}_0:G_K^{\mathrm{ab}}\rightarrow A^{\times}$ such 
that $\delta_0=\widetilde{\delta}_0\circ\mathrm{rec}_K$, where $\mathrm{rec}_K:K^{\times}\rightarrow G_K^{\mathrm{ab}}$ is the reciprocity map which is normalized as in Notation. Using $\widetilde{\delta}_0$, we define an
\'etale rank one $A$-$B$-pair $W(A(\widetilde{\delta}_0))$, which is the $A$-$B$-pair associated to the rank one $A$-representation $A(\widetilde{\delta}_0)$.
Next, we define a non-\'etale rank one $A$-$B$-pair by using $\bar{u}\in E^{\times}$. 
For this, we first define a rank one $E$-filtered $\varphi$-module $D_{\bar{u}}:=K_0\otimes_{\mathbb{Q}_p}Ee_{\bar{u}}$ such that 
$\varphi^f(e_{\bar{u}}):=\bar{u}e_{\bar{u}}$ and $\mathrm{Fil}^0(K\otimes_{K_0}D_{\bar{u}}):=K\otimes_{K_0}D_{\bar{u}}$, $\mathrm{Fil}^1(K\otimes_{K_0}D_{\bar{u}}):=0$.
From this, we obtain the rank one crystalline $E$-$B$-pair $W(D_{\bar{u}})$ such that $\bold{D}_{\mathrm{cris}}(W(D_{\bar{u}}))\isom D_{\bar{u}}$ which is 
pure of slope $\frac{v_p(\bar{u})}{f}$.
By tensoring these, we define a rank one $A$-$B$-pair $W(\delta)$ by $W(\delta):=(W(D_{\bar{u}})\otimes_EA)\otimes W(A(\widetilde{\delta}_0))$,
 which is pure of slope $\frac{v_p(\bar{u})}{f}$.
 
 The following proposition is the $A$-coefficient version of Theorem 1.45 of \cite{Na09}.
 
 \begin{prop}\label{5}
 This construction $\delta\mapsto W(\delta)$ does not depend on the 
 choice of uniformizer $\pi_K$, and gives a bijection between 
 the set of continuous homomorphisms $\delta:K^{\times}\rightarrow A^{\times}$ and 
 the set of isomorphism classes of rank one $A$-$B$-pairs.
 \end{prop}
 \begin{proof}
 The independence of the choice of uniformizer and the injection can be  proved in the same 
 way as in the proof of Theorem 1.45 of \cite{Na09}. We prove the surjection. Let $W$ be a rank one $A$-$B$-pair.
 As an $E$-$B$-pair, $W$ is a successive extension of the rank one $E$-$B$-pair $W\otimes_A E$.
 By Lemma 1.42 of \cite{Na09}, $W\otimes_AE$ is pure of slope $\frac{n}{fe_E}$ for some $n\in \mathbb{Z}$. Then, $W$ is also pure of slope $\frac{n}{fe_E}$ by Theorem 1.6.6 of \cite{Ke08}.  
 We define a rank one $E$-filtered $\varphi$-module $D_{\pi_E^n}:=K_0\otimes_{\mathbb{Q}_p}Ee_{\pi_E^n}$ 
 in the same way as for $D_{\bar{u}}$, where $\pi_E$ is a uniformizer of $E$. Then, $W\otimes (W(D_{\pi_E^n}) \otimes_E A)^{\vee}$ is pure of slope zero by Lemma 
 1.34 of \cite{Na09}. Hence, there exists $\widetilde{\delta}':G_K^{\mathrm{ab}}\rightarrow A^{\times}$ such that 
 $W\otimes (W(D_{\pi_E^n})\otimes_E A)^{\vee}\isom W(A(\widetilde{\delta}'))$. We put $\delta':=\widetilde{\delta}'\circ\mathrm{rec}_K:
 K^{\times}\rightarrow A^{\times}$ and define $\delta:K^{\times}\rightarrow A^{\times}$ such that $\delta|_{\mathcal{O}_K^{\times}}
 :=\delta'|_{\mathcal{O}_K^{\times}}$ and $\delta(\pi_K):=\delta'(\pi_K)\pi_E^n$. Then, we have an isomorphism 
 $W\isom W(\delta)$, which can be easily seen from the construction of $W(\delta)$.

 \end{proof}
 By the local class field theory, we have a canonical bijection $\delta\mapsto A(\widetilde{\delta})$ from  the set of unitary continuous homomorphisms 
 from $K^{\times}$ to $A^{\times}$ (here,``unitary" means that the image of the 
 composition of $\delta$ with the 
 projection $A^{\times}\rightarrow E^{\times}$ is contained in $\mathcal{O}_E^{\times}$) 
  to the set of isomorphism class of rank one $A$-representations of $G_K$, where $\widetilde{\delta}:G_K^{ab}\rightarrow 
  A^{\times}$ is the continuous homomorphism such that $\delta=\widetilde{\delta}\circ\mathrm{rec}_K$.
  By the definition of $W(\delta)$  and by the above proof, it is easy to see that there exists an isomorphism 
  $W(\delta)\isom W(A(\widetilde{\delta}))$ for any unitary homomorphism $\delta:K^{\times}\rightarrow A^{\times}$.
 Moreover, it is easy to see that, for any continuous homomorphisms $\delta_1,\delta_2: K^{\times}
 \rightarrow A^{\times}$, we have isomorphisms $W(\delta_1)\otimes W(\delta_2)\isom W(\delta_1\delta_2)$ and $W(\delta_1)^{\vee}
 \isom W(\delta_1^{-1})$. 
  
Next, we generalize the functor $\bold{D}_{\mathrm{cris}}$ to potentially crystalline $A$-$B$-pairs.
First, we define the $A$-coefficient version of filtered $(\varphi,G_K)$-modules.
Let $L$ be a finite Galois extension of $K$, we denote by $G_{L/K}:=\mathrm{Gal}(L/K)$.
\begin{defn}
Let $A$ be an object of $\mathcal{C}_E$. 
We say that $D$ is an $A$-filtered $(\varphi,G_{L/K})$-module of $K$ if $D$ satisfies 
the following conditions.
\begin{itemize}
\item[(1)]$D$ is a finite $L_0\otimes_{\mathbb{Q}_p}A$-module which is free as an 
$A$-module with 
a $\varphi$-semi-linear action $\varphi:D\isom D$.
\item[(2)]$D_L:=L\otimes_{L_0}D$ has a decreasing filtration $\mathrm{Fil}^iD_L$ by 
$L\otimes_{\mathbb{Q}_p}A$-submodules such that $\mathrm{Fil}^kD_L=0$ and 
$\mathrm{Fil}^{-k}D_L=D_L$ for sufficiently large $k$ and that 
$\mathrm{Fil}^kD_L/\mathrm{Fil}^{k+1}D_L$ are free $A$-modules for any $k$.
\item[(3)]$G_{L/K}$ acts on $D$ by $L_0\otimes_{\mathbb{Q}_p}A$-semi-linear automorphism 
which commutes with the action of $\varphi$ and preserves the filtration.
\end{itemize}
\end{defn}
\begin{rem}
Using  the $\varphi$-structure on $D$, we can easily see that $D$ is a free $L_0\otimes_{\mathbb{Q}_p}A$-module.
\end{rem}

Let $W_A:=(W_{A,e},W^+_{A,\mathrm{dR}})$ be an $A$-$B$-pair such that $W_A|_{G_L}$ is crystalline as an $E$-$B$-pair
for a finite 
Galois extension $L$ of $K$. As in the case of $E$-$B$-pairs, we define 
$\bold{D}^L_{\mathrm{cris}}(W_A):=(\bold{B}_{\mathrm{max}}\otimes_{\bold{B}_{e}}W_e)^{G_L}$ with a $\varphi$-action induced from 
that on $\bold{B}_{\mathrm{max}}$, then the natural map $L\otimes_{L_0}\bold{D}^L_{\mathrm{cris}}(W_A)\rightarrow 
\bold{D}^L_{\mathrm{dR}}(W_A):=(\bold{B}_{\mathrm{dR}}\otimes_{\bold{B}_{e}}W_e)^{G_L}$ is isomorphism. 
We define $\mathrm{Fil}^k\bold{D}^L_{\mathrm{dR}}(W_A):=\bold{D}^L_{\mathrm{dR}}(W_A)\cap t^kW^+_{\mathrm{dR}}$ for 
any $k\in \mathbb{Z}$. These are naturally equipped with a $G_{L/K}$-action.
\begin{lemma}\label{s1}
In the above situation, $\bold{D}^L_{\mathrm{cris}}(W_A)$ is an $A$-filtered $(\varphi,G_{L/K})$-module of $K$.
\end{lemma}
\begin{proof}
It suffices only to show the $A$-freeness of $\bold{D}^L_{\mathrm{cris}}(W_A)$, $\mathrm{Fil}^k\bold{D}^L_{\mathrm{dR}}(W_A)$, 
$\mathrm{Fil}^k\bold{D}^L_{\mathrm{dR}}\allowbreak (W_A)/\mathrm{Fil}^{k+1}\bold{D}^L_{\mathrm{dR}}(W_A)$. 
Here, we only prove the $A$-freeness of $\bold{D}^L_{\mathrm{cris}}(W_A)$, other cases can be proved in a similar way.
By the exactness of $\bold{D}^L_{\mathrm{cris}}$ for $E$-$B$-pairs which are crystalline when restricted to $G_L$, we have 
a natural isomorphism $\bold{D}^L_{\mathrm{cris}}(W_A)\otimes_A N\isom \bold{D}^L_{\mathrm{cris}}(W_A\otimes_A N)$  for any finite $A$-module $N$. 
From this, for any $A$-linear injection $N_1\hookrightarrow N_2$ of finite $A$-modules, we have 
an inclusion $\bold{D}^L_{\mathrm{cris}}(W_A)\otimes_A N_1=\bold{D}^L_{\mathrm{cris}}(W_A\otimes_A N_1)\hookrightarrow 
\bold{D}^L_{\mathrm{cris}}(W_A\otimes_A N_2)=\bold{D}^L_{\mathrm{cris}}(W_A)\otimes_A N_2$ because 
$W_{A,e}$ is $A$-flat. Hence, $\bold{D}^L_{\mathrm{cris}}(W_A)$ is $A$-flat.
\end{proof}

Conversely, let $D$ be  an $A$-filtered $(\varphi,G_{L/K})$-module of $K$, then we define 
$W_e(D):=(\bold{B}_{\mathrm{cris}}\otimes_{L_0}D)^{\varphi=1}$. We have 
a natural isomorphism $\bold{B}_{\mathrm{dR}}\otimes_{\bold{B}_{e}}W_e(D)\isom \bold{B}_{\mathrm{dR}}\otimes_{L}D_L$. 
We define $W^+_{\mathrm{dR}}(D):=\mathrm{Fil}^0(\bold{B}_{\mathrm{dR}}\otimes_LD_L)\subseteq 
\bold{B}_{\mathrm{dR}}\otimes_{\bold{B}_{e}}W_e(D)$. We write $W(D):=(W_e(D), W^+_{\mathrm{dR}}(D))$ which is 
a potentially crystalline $E$-$B$-pair with an $A$-action.
\begin{lemma}\label{s2}
In the above situation, $W(D)$ is a potentially crystalline 
$A$-$B$-pair.
\end{lemma}
\begin{proof}
It suffices to show the $A$-flatness of $W_e(D)$ and $W^+_{\mathrm{dR}}(D)/tW^+_{\mathrm{dR}}(D)$. 
We can prove these in the same way as in Lemma $\ref{s1}$ by using the
exactness of the functor $W(D)$ and the $A$-flatness of $D$ and $\mathrm{Fil}^kD_L/\mathrm{Fil}^{k+1}D_L$ for any 
$k$.
\end{proof}

\begin{corollary}\label{s3}
For any $A\in\mathcal{C}_E$, the functor $\bold{D}^L_{\mathrm{cris}}$ gives an equivalence of categories between 
the category of potentially crystalline $A$-$B$-pairs which are crystalline if restricted to $G_L$ and 
the category of $A$-filtered $(\varphi,G_{L/K})$-modules of $K$.
\end{corollary}
\begin{proof}
This follows from Lemma $\ref{s1}$ and Lemma $\ref{s2}$ and Theorem $\ref{e}$.
\end{proof}

Next, we prove some lemmas which will be used in later sections.

Let $W_A$ be an $A$-$B$-pair which is not potentially crystalline in general and $L$ be a finite Galois extension 
of $K$, then we can define $\bold{D}^L_{\mathrm{cris}}(W_A)$ in the same way as in the case where $W_A$ is potentially crystalline. 
This is an $E$-filtered $(\varphi,G_{L/K})$-module, but this may not be an 
$A$-filtered $(\varphi,G_{L/K})$-module in general.

\begin{lemma}\label{s3-5}
Let $W$ be an $E$-$B$-pair. Let $D\subseteq \bold{D}_{\mathrm{cris}}(W)$ be 
a rank one sub $E$-filtered $\varphi$-module whose filtration is induced 
from that of $\bold{D}_{\mathrm{cris}}(W)$, then there exists a natural saturated inclusion 
$W(D)\hookrightarrow W$.

\end{lemma}
\begin{proof}
Twisting $W$ by a suitable crystalline character of the form $\prod_{\sigma\in\mathcal{P}}\sigma(\chi_{\mathrm{LT}})^{k_{\sigma}}$, 
we may assume that $\mathrm{Fil}^0(D_{K})=D_{K}$ and $\mathrm{Fil}^1(D_{K})=0$, where we put 
$D_{K}:=K\otimes_{K_0}D$.
We have natural inclusions $W(D)_e=(\bold{B}_{\mathrm{max}}\otimes_{K_0}D)^{\varphi=1}\subseteq 
(\bold{B}_{\mathrm{max}}\otimes_{K_0}\bold{D}_{\mathrm{cris}}(W))^{\varphi=1}\subseteq (\bold{B}_{\mathrm{max}}\otimes_{\bold{B}_{e}}W_e)^{\varphi=1}
=W_e$ and, under the above assumption, $W^+_{\mathrm{dR}}(D)\allowbreak =\mathrm{Fil}^0(\bold{B}_{\mathrm{dR}}\otimes_{K}D_{K})=
\bold{B}^+_{\mathrm{dR}}\otimes_KD_{K}\subseteq \mathrm{Fil}^0(\bold{B}_{\mathrm{dR}}\otimes_{K}\bold{D}_{\mathrm{dR}}(W))\subseteq 
W^+_{\mathrm{dR}}$, which define an inclusion 
$W(D)\hookrightarrow W$. Hence, it suffices to show that this inclusion is saturated, i.e. suffices to show that we have $\bold{B}^+_{\mathrm{dR}}\otimes_{K}D_{K}
=(\bold{B}_{\mathrm{dR}}\otimes_KD_{K})\cap W^+_{\mathrm{dR}}$. We can write $(\bold{B}_{\mathrm{dR}}\otimes_KD_{K})\cap W^+_{\mathrm{dR}}=
\oplus_{\sigma\in \mathcal{P}}\frac{1}{t^{k_{\sigma}}}\bold{B}^+_{\mathrm{dR}}\otimes_{K,\sigma}D_{K,\sigma}$ for some $k_{\sigma}\in \mathbb{Z}_{\geqq 0}$, where we decompose $D_{K}$ by 
$D_{K}\isom \oplus_{\sigma\in \mathcal{P}}D_{K}\otimes_{K\otimes_{\mathbb{Q}_p}E, \sigma\otimes id_E}E=:\oplus_{\sigma\in \mathcal{P}}D_{K,\sigma}$. 
If $k_{\sigma}\geqq 1$ for some $\sigma\in \mathcal{P}$, then $D_{K,\sigma}\subseteq t^{k_{\sigma}}W^+_{\mathrm{dR}}$. Because the filtration on $D$ is induced from 
$\bold{D}_{\mathrm{cris}}(W)$, this implies that $\mathrm{Fil}^{k_{\sigma}}D_{K,\sigma}=D_{K,\sigma}$, this contradicts 
to $\mathrm{Fil}^1D_{K,\sigma}=0$. Hence $k_{\sigma}=0$ for any $\sigma\in \mathcal{P}$, which implies 
that $\bold{B}^+_{\mathrm{dR}}\otimes_{K}D_{K}
=(\bold{B}_{\mathrm{dR}}\otimes_KD_{K})\cap W^+_{\mathrm{dR}}$.

\end{proof}
\begin{lemma}\label{s4}
Let $W_A$ be an $A$-$B$-pair. 
Let $D\subseteq \bold{D}_{\mathrm{cris}}(W_A)$ be a sub $E$-filtered $\varphi$-module which is 
an $A$-filtered $\varphi$-module
of rank one, where the filtration on $D$ is the one induced from that of 
$\bold{D}_{\mathrm{cris}}(W_A)$. We assume that the natural map 
$D\otimes_A E\rightarrow \bold{D}_{\mathrm{cris}}(W_A\otimes_A E)$ remains an injection.
Then, we have a natural injection of $A$-$B$-pairs $W(D)\hookrightarrow W_A$ such that 
the cokernel $W_A/W(D)$ is also an $A$-$B$-pair.
\end{lemma}
\begin{proof}
In the same way as in the above proof, 
we have a natural injection $W(D)\hookrightarrow W_A$. Because the natural map 
$D\otimes_A E\rightarrow \bold{D}_{\mathrm{cris}}(W_A\otimes_A E)$ is an injection, 
we obtain an injection $W(D)\otimes_AE \isom W(D\otimes_A E)\hookrightarrow W_A\otimes_A E$ and this injection is saturated by the above lemma.
Hence, it suffices to show that if $W_1\hookrightarrow W_2$ is an inclusion of $A$-$B$-pairs such that 
$W_1\otimes_{A}E\rightarrow W_{2}\otimes_AE$ remains to be injective and saturated, then the cokernel $W_2/W_1$ exists and 
is an $A$-$B$-pair.
We put $W_{3,e}$, $W^+_{3,\mathrm{dR}}$ the cokernels of 
$W_{1,e}\hookrightarrow W_{2,e}$, $W^+_{1,\mathrm{dR}}\hookrightarrow W^+_{2,\mathrm{dR}}$ respectively.
By Lemma 2.2.3 (i) of \cite{Bel-Ch09}, these are $A$-flat. Hence, it suffices to show that these are 
free over $\bold{B}_{e}$, $\bold{B}^+_{\mathrm{dR}}$ respectively. We can prove this claim in the same way as in Lemma 2.2.3 (iii) of \cite{Bel-Ch09}.


\end{proof}

\subsection{Deformations of $B$-pairs}

In this subsection, we develop the deformation theory of $B$-pairs, which 
is a natural generalization of Mazur's deformation theory of $p$-adic Galois 
representations. 

\begin{defn}\label{6}
Let $A$ be an object in $\mathcal{C}_E$, and let $W$ be an $E$-$B$-pair.
We say that a pair $(W_A,\iota)$ is a deformation of $W$ over $A$ 
if $W_A$ is an $A$-$B$-pair and $\iota: W_A\otimes_AE\isom W$
is an isomorphism of $E$-$B$-pairs. Let $(W_A,\iota)$, $(W'_A,\iota')$ be 
two deformations of $W$ over $A$. Then, we say that $(W_A,\iota)$ and $(W'_A,\iota')$ are 
equivalent if there exists an isomorphism $f:W_A\isom W'_A$ of $A$-$B$-pairs which 
satisfies $\iota=\iota'\circ\bar{f}$, where $\bar{f}:W_A\otimes_A E\isom W'_A\otimes_AE$ 
is the isomorphism naturally induced by $f$.
\end{defn}
\begin{defn} \label{7}
Let $W$ be an $E$-$B$-pair, then we define the deformation functor 
$D_W$ from the category $\mathcal{C}_E$ to the category of sets by 
$$D_W(A):=\{ \text{ equivalent classes }  (W_A,\iota) \text{ of deformations of }  W \text{ over } A \}$$
for $A\in \mathcal{C}_E$.
\end{defn} 
We simply denote by $W_A$ if there is no risk of confusing about $\iota$.

Next, we prove the pro-representability 
of the functor $D_W$ under suitable conditions.
For this, we recall Schlessinger's criterion for pro-representability of functors from 
$\mathcal{C}_E$ to the category of sets. 
We call a morphism $f:A'\rightarrow A$ in $\mathcal{C}_E$ a small extension if it is surjective and 
the kernel $\mathrm{Ker}(f)=(t)$ is generated by a nonzero single element $t\in A'$ 
and $\mathrm{Ker}(f)\cdot \mathfrak{m}_{A'}=0$. $E[\varepsilon]$ is the ring defined by $E[\varepsilon]:=E[X]/(X^2):\varepsilon\mapsto \overline{X}$.

\begin{thm}\label{8}
Let $F$ be a functor from $\mathcal{C}_E$ to the category of sets such that 
$F(E)$ consists of a single element. For morphisms $A'\rightarrow A$, $A"\rightarrow A$ in 
$\mathcal{C}_E$, consider the natural map 
\begin{itemize}
\item[(1)] $F(A'\times_{A}A")\rightarrow F(A')\times_{F(A)}F(A")$,
\end{itemize}
then $F$ is pro-representable if and only if $F$ satisfies properties $(H_1)$,
$(H_2)$, $(H_3)$, $(H_4)$ below:
\begin{itemize}
\item[$(H_1)$] $(1)$ is surjective if $A"\rightarrow A$ is surjective.
\item[$(H_2)$] $(1)$ is bijective when $A=E$ and $A"=E[\varepsilon]$.
\item[$(H_3)$] $\mathrm{dim}_E(t_F)<\infty$ $($ where $t_F:=F(E[\varepsilon])$ and, under 
the condition $(H_2)$, it is known that $t_F$ naturally has a structure of an $E$-vector space$)$.
\item[$(H_4)$] $(1)$ is  bijective if $A'=A"$ and $A'\rightarrow A$ is a small extension.
\end{itemize}
\end{thm}
\begin{proof}
See \cite{Schl68} or $\S 18$ of \cite{Ma97}.
\end{proof}

Using this criterion, we prove the pro-representability of 
$D_W$.

\begin{prop}\label{9}
Let $W$ be an $E$-$B$-pair.
If $\mathrm{End}_{G_K}(W)=E$, then $D_W$ is pro-representable by a complete 
noetherian local $E$-algebra $R_W$ with the residue field $E$.
\end{prop}

To prove this proposition, we first prove some lemmas.

\begin{lemma}\label{10}
Let $\mathrm{ad}(W):=\mathrm{Hom}(W,W)(\isom W\otimes W^{\vee})$ be the internal endomorphism of $W$,
then there exists an isomorphism of $E$-vector spaces
$$D_W(E[\varepsilon])\isom \mathrm{H}^1(G_K,\mathrm{ad}(W)).$$
\end{lemma}
\begin{proof}
Let $W_{E[\varepsilon]}:=(W_{E[\varepsilon],e}, W^+_{E[\varepsilon],\mathrm{dR}})$ 
be a deformation of $W$ over $E[\varepsilon]$. From this, we define an element 
in $\mathrm{H}^1(G_K,\mathrm{ad}(W))$ as follows. Because we have natural isomorphisms 
$\varepsilon W_{E[\varepsilon],e}\isom 
W_e$ and $W_{E[\varepsilon],e}/\varepsilon W_{E[\varepsilon],e} \isom W_e$ (here we put $W:=(W_e,W^+_{\mathrm{dR}})$),
we have a natural exact sequence of $\bold{B}_{e}\otimes_{\mathbb{Q}_p}E[G_K]$-modules
\begin{equation*}
0\rightarrow W_e\rightarrow W_{E[\varepsilon],e}\rightarrow W_e\rightarrow 0.
\end{equation*}
We fix an isomorphism of $\bold{B}_{e}\otimes_{\mathbb{Q}_p}E$-modules 
$W_{E[\varepsilon],e}\isom W_e e_1\oplus W_ee_2$ such that first factor $W_e e_1$is equal to $\varepsilon W_{E[\varepsilon]}$ 
as $\bold{B}_{e}\otimes_{\mathbb{Q}_p}E[G_K]$-module and that the above natural projection maps the second factor $W_e e_2$ to $W_e$ by $xe_2\mapsto x$ for any $x\in W_e$.  We define a continuous one cocycle by
$$c_e: G_K\rightarrow \mathrm{Hom}_{\bold{B}_{e}\otimes_{\mathbb{Q}_p}E}(W_e,W_e)\text{ by } 
g(ye_2):=c_e(g)(gy)e_1+gye_2$$
 for  any $g\in G_K$ and $y\in W_e$.  
 For $W^+_{\mathrm{dR}}$,  we fix an isomorphism $W^+_{E[\varepsilon],\mathrm{dR}}\isom W^+_{\mathrm{dR}}e_1\oplus 
 W^+_{\mathrm{dR}}e'_2$ as in the case of $W_e$, then we define a one cocycle by
 $$c_{\mathrm{dR}}:G_K\rightarrow \mathrm{Hom}_{\bold{B}^+_{\mathrm{dR}}\otimes_{\mathbb{Q}_p}E}(W^+_{\mathrm{dR}},W^+_{\mathrm{dR}})\text{ by  }g(ye'_2):=c_{\mathrm{dR}}(g)(gy)e_1+gye'_2$$
  for 
 any $g\in G_K$ and $y\in W^+_{\mathrm{dR}}$. Next, we define $c\in\mathrm{Hom}_{\bold{B}_{\mathrm{dR}}\otimes_{\mathbb{Q}_p}E}
 (W_{\mathrm{dR}},W_{\mathrm{dR}})$ as follows. Tensoring $W_{E[\varepsilon],e}$ and $W^+_{E[\varepsilon],\mathrm{dR}}$ with 
 $\bold{B}_{\mathrm{dR}}$ over $\bold{B}_{e}$ or $\bold{B}^+_{\mathrm{dR}}$, we have 
 an isomorphism $f:W_{\mathrm{dR}}e_1\oplus W_{\mathrm{dR}}e_2\isom W_{E[\varepsilon],\mathrm{dR}}\isom 
 W_{\mathrm{dR}}e_1\oplus W_{\mathrm{dR}}e'_2$ of $\bold{B}_{\mathrm{dR}}\otimes_{\mathbb{Q}_p}E$-modules. 
 We define 
 $$c:W_{\mathrm{dR}}\rightarrow W_{\mathrm{dR}}\text{ by } f(ye_2):=c(y)e_1+ye'_2$$
  for any $y\in W_{\mathrm{dR}}$.
 By the definition, the triple ($c_e$, $c_{\mathrm{dR}}$, $c$) satisfies 
 $$c_e(g)-c_{\mathrm{dR}}(g)=gc-c\text{ in }\mathrm{Hom}_{\bold{B}_{\mathrm{dR}}
 \otimes_{\mathbb{Q}_p}E}(W_{\mathrm{dR}},W_{\mathrm{dR}})$$ for any $g\in G_K$, i.e. the triple $(c_e,c_{\mathrm{dR}},c)$ defines an 
 element of $ \mathrm{H}^1(G_K,\mathrm{ad}(W))$ by the definition of Galois cohomology of $B$-pairs ($\S$ 2.1 of \cite{Na09}), then it is standard to check that this definition is independent of 
 the choice of a fixed isomorphism $W_{E[\varepsilon],e}\isom W_ee_1\oplus W_ee_2$, etc, and it is easy to check that 
 this map defines an isomorphism $D_W(E[\varepsilon])\isom \mathrm{H}^1(G_K,\mathrm{ad}(W))$.

\end{proof}

\begin{lemma}\label{11}
Let $W_A$ be a deformation of $W$ over $A$.
If $\mathrm{End}_{G_K}(W)=E$, then $\mathrm{End}_{G_K}(W_A)=A$.
\end{lemma}

\begin{proof}
We prove this lemma by induction on the length of $A$.
When $A=E$, this is trivial. We assume that the lemma is proved for the rings of length $n$ and assume that 
$A$ is of length $n+1$. We take a small extension 
$f:A\rightarrow A'$. Because we have $\mathrm{End}_{G_K}(W)=\mathrm{H}^0(G_K, 
W^{\vee}\otimes W)$, we have the following short exact sequence
\begin{equation*}
0\rightarrow \mathrm{Ker}(f)\otimes_E\mathrm{End}_{G_K}(W)
\rightarrow \mathrm{End}_{G_K}(W_A)\rightarrow 
\mathrm{End}_{G_K}(W_A\otimes_A A').
\end{equation*}
From this and the induction hypothesis, we have 
\[
\begin{array}{ll}
\mathrm{length}(\mathrm{End}_{G_K}(W_A))&\leqq 
\mathrm{length}(\mathrm{End}_{G_K}(W_A\otimes_A A'))+
\mathrm{length}(\mathrm{Ker}(f)\otimes_E \mathrm{End}_{G_K}(W))\\
&=\mathrm{length}(A')+1=\mathrm{length}(A).
\end{array}
\]
 On the other hand, 
we have a natural inclusion $A\subseteq \mathrm{End}_{G_K}(W_A)$.
Comparing the length, we obtain an equality $A=\mathrm{End}_{G_K}(W_A)$.
\end{proof}

\begin{proof} (of proposition \ref{9})
Let $W$ be an $E$-$B$-pair of rank $n$ satisfying that 
$\mathrm{End}_{G_K}(W)=E$. 
For this $W$, we check the conditions $(H_i)$ ($i=1\sim4$) of Schlessinger's criterion.
First, by Lemma $\ref{10}$, we have 
$$\mathrm{dim}_E(D_{W}(E[\varepsilon]))=
\mathrm{dim}_E(\mathrm{H}^1(G_K,\mathrm{ad}(W)))<\infty,$$ hence $(H_3)$ is satisfied.
Next we check the condition $(H_1)$. Let  $f:A'\rightarrow A$, $g:A"\rightarrow A$ be morphisms in 
$\mathcal{C}_E$ such that $g$ is a surjection. Let $([W_{A'}],[W_{A"}])$ be an element in 
$D_W(A')\times_{D_W(A)}D_W(A")$. We take deformations $W_{A'}:=(W_{A', e},W^{ +}_{A',  \mathrm{dR}})
$, $W_{A"}:=(W_{A", e},W^{+}_{A" ,\mathrm{dR}})$ over $A'$ and $A"$ which are representatives of equivalent classes 
$[W_{A'}]$ and $[W_{A"}]$ respectively, then we have an isomorphism $h:W_{A'}\otimes_{A'}A\isom W_{A"}\otimes_{A"}A=:W_A
:=(W_{A, e}, W^+_{A,\mathrm{dR}})$ 
which defines an equivalent class in $D_W(A)$. We fix a basis $e_1,\cdots,e_n$ of $W_{A' ,e}$ as a $\bold{B}_{e}\otimes_{\mathbb{Q}_p}A'$
-module and denote by $\overline{e}_1,\cdots,\overline{e}_n$ the basis of $W_{A' ,e}\otimes_{A'}A$ induced from $e_1,\cdots,e_n$.
By the surjectivity of $g:A"\rightarrow A$ and by the $A"$-flatness of $W_{A", e}$, we can take a basis 
$\widetilde{e}_1,\cdots,\widetilde{e}_n$ of $W_{A", e}$ such that the basis 
$\bar{\widetilde{e}}_1,\cdots,\bar{\widetilde{e}}_n$ of $W_{A", e}\otimes_{A"}A$ induced from $\widetilde{e}_1,\cdots,\widetilde{e}_n$ 
satisfies $h(\overline{e}_i)=\bar{\widetilde{e}}_i$ for any $i$. If we define $W^{'"}_e$ by 
$$W^{'"}_e:=W_{A', e}\times_{W_{A ,e}}W_{A" ,e}
:=\{(x,y)\in W_{A', e}\times W_{A" ,e}| h(\bar{x})=\bar{y}\}, $$ then 
$W^{'"}_e$ is a free $\bold{B}_{e}\otimes_{\mathbb{Q}_p}(A'\times_A A")$-module
 with a basis $(e_1,\widetilde{e}_1), \cdots, (e_n,\widetilde{e}_n)$. In the same way, we define $W^{'"+}_{\mathrm{dR}}$ by
 $W^{'"+}_{\mathrm{dR}}:=W^{+}_{A' ,\mathrm{dR}}\times_{W^+_{A, \mathrm{dR}}} W^{+}_{A", \mathrm{dR}}$,  
 which is a free $\bold{B}^{+}_{\mathrm{dR}}\otimes_{\mathbb{Q}_p}(A'\times_A A")$-module. If we put
 $W_{A'}\times_{W_A} W_{A"}:=(W^{'"}_e,W^{'" +}_{\mathrm{dR}})$, then this is a $(A'\times_A A")$-$B$-pair which is a deformation 
 of $W$ over $A'\times_A A"$ such that the equivalent class $[W_{A'}\times_{W_A}W_{A"}]\in D_W(A'\times_A A")$ maps 
 $([W_{A'}], [W_{A"}])\in D_W(A')\times_{D_W(A)}D_W(A")$. Hence, we have checked the condition $(H_1)$.
 
 Finally, we prove that if $g:A"\rightarrow A$ is a surjectition, then the natural map 
 $D_W(A'\times_AA")\rightarrow D_W(A')\times_{D_W(A)}D_W(A")$ is bijective, which 
 proves the conditions $(H_2)$ and $(H_4)$, hence proves the pro-representability of $D_W$.
 Let $W^{'"}_1$, $W^{'"}_2$ be deformations of $W$ over $A'\times_A A"$ such 
 that $[W^{'"}_1\otimes_{A'\times_AA"}A']=[W^{'"}_2\otimes_{A'\times_{A}A"}A']$ in $D_W(A')$ and 
 $[W^{'"}_1\otimes_{A'\times_AA"}A"]=[W^{'"}_2\otimes_{A'\times_{A}A"}A"]$ in $D_W(A")$.
 Under this situation, we want to show that $[W^{'"}_1]=[W^{'"}_2]$ in $D_W(A'\times_A A")$. 
 We put $W_{1 A'}:=W^{'"}_1\otimes_{A'\times_A A"}A'$, $W_{1 A"}:=W_1^{'"}\otimes_{A'\times_A A"}A"$
 , $W_{1 A}:=W_1^{'"}\otimes_{A'\times_A A"} A$, and similarly for $W_{2 A'}$, $W_{2 A"}$, $W_{2 A}$,
 then we have natural isomorphisms $W^{'"}_1\isom W_{1 A'}\times_{W_{1 A}} W_{1 A"}$ and $W_2^{'"}
 \isom W_{2 A'}\times_{W_{2 A}} W_{2 A"}$ defined as in the previous paragraph. Because 
 we have $[W_{1 A'}]=[W_{2 A'}]$ and $[W_{1 A"}]=[W_{2 A"}]$, we have isomorphisms 
 $h':W_{1 A'}\isom W_{2 A'}$ and $h":W_{1 A"}\isom W_{2 A"}$. By the base change to $A$, 
 we obtain an automorphism $\bar{h}'\circ \bar{h}^{" -1}:W_{2 A}\isom W_{1 A}\isom W_{2 A}$. 
 By Lemma $\ref{11}$ and by the surjectivity of $g:A^{" \times}\rightarrow A^{\times}$, we can find an automorphism 
 $\widetilde{h}:W_{2 A"}\isom W_{2 A"}$ such that $\bar{\widetilde{h}}=\bar{h}'\circ \bar{h}^{" -1}$. 
 If we define a morphism 
 $$h^{'"}:W_{1 A'}\times_{W_{1 A}} W_{1 A"}\rightarrow 
 W_{2 A"}\times_{W_{2 A}} W_{2 A'}: (x,y)\mapsto (h_1(x), \widetilde{h}\circ h_2(y)),$$ then we can see that this is 
 well-defined and is isomorphism. Hence, we finish to prove the proposition.
 \end{proof}
 
 \begin{prop}\label{12}
 Let $W:=(W_e, W^+_{\mathrm{dR}})$ be an $E$-$B$-pair of rank $n$.
  If $\mathrm{H}^2(G_K, \mathrm{ad}(W))\allowbreak =0$, then the functor 
 $D_W$ is formally smooth.
 \end{prop}
 \begin{proof}
 Let $A'\rightarrow A$ be a small extension in $\mathcal{C}_E$,  
 we denote the kernel by $I\subseteq A'$.
 Let $W_A:=(W_{e,A}, W^+_{\mathrm{dR},A})$ be a deformation of $W$ 
 over $A$, then it suffices to show that there exists an $A'$-$B$-pair $W_{A'}$ 
 such that $W_{A'}\otimes_{A'}A\isom W_A$. 
 We fix a basis of $W_{e,A}$ as a $\bold{B}_{e}\otimes_{\mathbb{Q}_p}A$-module.
 Using this basis and the $G_K$-action on $W_{e,A}$, 
 we obtain a continuous one cocycle $$\rho_e:G_K\rightarrow \mathrm{GL}_n(\bold{B}_{e}\otimes_{\mathbb{Q}_p}A).$$
  In the same way, if we fix a basis of $W^+_{\mathrm{dR},A}$ as a $\bold{B}^+_{\mathrm{dR}}\otimes_{\mathbb{Q}_p}A$-module, 
  we obtain a continuous one cocycle $$\rho_{\mathrm{dR}}:G_K\rightarrow \mathrm{GL}_n(\bold{B}^+_{\mathrm{dR}}\otimes_{\mathbb{Q}_p}A).$$
   From the canonical isomorphism $W_{e,A}\otimes_{\bold{B}_{e}}\bold{B}_{\mathrm{dR}}\isom W^+_{\mathrm{dR},A}
   \otimes_{\bold{B}^+_{\mathrm{dR}}}
   \bold{B}_{\mathrm{dR}}$, we obtain a matrix $P\in\mathrm{GL}_n(\bold{B}_{\mathrm{dR}}\otimes_{\mathbb{Q}_p}A)$ which satisfies 
   $$P\rho_e(g)g(P)^{-1}=\rho_{\mathrm{dR}}(g)\text{ for any }g\in G_K.$$ 
   We fix an $E$-linear section $s:A\rightarrow A'$ of $A'\rightarrow A$ and fix a lifting 
   $\widetilde{P}\in \mathrm{GL}_n(\bold{B}_{\mathrm{dR}}\otimes_{\mathbb{Q}_p}A')$ of $P$. Using this section, 
   we obtain  continuous liftings 
   $$\widetilde{\rho}_e:=s\circ \rho_e:G_K\rightarrow \mathrm{GL}_n(\bold{B}_{e}\otimes_{\mathbb{Q}_p}A')$$ of $\rho_e$ 
   and $$\widetilde{\rho}_{\mathrm{dR}}:=s\circ \rho_{\mathrm{dR}}:G_K\rightarrow \mathrm{GL}_n(\bold{B}^+_{\mathrm{dR}}\otimes_{\mathbb{Q}_p}A')$$ of 
   $\rho_{\mathrm{dR}}$.
   Using these liftings, we define 
   $$c_e:G_K\times G_K\rightarrow I\otimes_E \mathrm{Hom}_{\bold{B}_{e}\otimes_{\mathbb{Q}_p}E}(W_e, W_e)$$ 
   by 
   
   \[
\begin{array}{ll}
\widetilde{\rho}_e(g_1g_2)g_1(\widetilde{\rho}_e(g_2))^{-1}\widetilde{\rho}_e(g_1)^{-1}=I_n+c_{e}(g_1,g_2)&\in 
   I_n+I \otimes_{A'}\mathrm{M}_n(\bold{B}_{e}\otimes_{\mathbb{Q}_p}A')\\
   &=I_n+I\otimes_{E}\mathrm{Hom}_{\bold{B}_{e}\otimes_{\mathbb{Q}_p}E}
   (W_e,W_e)
   \end{array}
   \]
    for any $g_1,g_2\in G_K$, where $I_n$ is the identity matrix. In the same way, we define 
    $$c_{\mathrm{dR}}:G_K\times G_K \rightarrow I\otimes_E\mathrm{Hom}_{\bold{B}^+_{\mathrm{dR}}\otimes_{\mathbb{Q}_p}E}
   (W^+_{\mathrm{dR}}, W^+_{\mathrm{dR}})$$ by 
   $$\widetilde{\rho}_{\mathrm{dR}}(g_1g_2)g_1(\widetilde{\rho}_{\mathrm{dR}}(g_2))^{-1}\widetilde{\rho}_{\mathrm{dR}}(g_1)^{-1}
   =I_n+c_{\mathrm{dR}}(g_1,g_2).$$  We define 
   $$c:G_K\rightarrow I\otimes_E\mathrm{Hom}_{\bold{B}_{\mathrm{dR}}\otimes_{\mathbb{Q}_p}E}(W_{\mathrm{dR}}, W_{\mathrm{dR}})$$ 
   by 
   $$\widetilde{P}\widetilde{\rho}_e(g)g(\widetilde{P})^{-1}\widetilde{\rho}_{\mathrm{dR}}(g)^{-1}=I_n+ c(g)\in I_n+I\otimes_E\mathrm{Hom}_{\bold{B}_{\mathrm{dR}}\otimes_{\mathbb{Q}_p}E}(W_{\mathrm{dR}}, W_{\mathrm{dR}}).$$  These $c_e$ and $c_{\mathrm{dR}}$ are continuous two cocycles, i.e. these satisfy 
   $$g_1c_{\ast}(g_2, g_3)-c_{\ast}(g_1g_2,g_3)+c_{\ast}(g_1, g_2g_3)
   -c_{\ast}(g_1,g_2)=0$$ for any $g_1,g_2,g_3\in G_K$ ($\ast=e, \mathrm{dR}$). Moreover, we can check that $c_e$ and $c_{\mathrm{dR}}$ and $c$ satisfy 
   $$c_e(g_1,g_2)-c_{\mathrm{dR}}(g_1,g_2)=g_1(c(g_2))-c(g_1g_2)+c(g_1)$$ for any 
   $g_1,g_2,g_3\in G_K$, here we note that the isomorphism $\mathrm{Hom}_{\bold{B}_{e}\otimes_{\mathbb{Q}_p}E}(W_e, W_e)
   \otimes_{\bold{B}_{e}}\bold{B}_{\mathrm{dR}}\isom \mathrm{Hom}_{\bold{B}^+_{\mathrm{dR}}\otimes_{\mathbb{Q}_p}E}(W^+_{\mathrm{dR}},W^+_{\mathrm{dR}})\otimes_{\bold{B}^+_{\mathrm{dR}}}\bold{B}_{\mathrm{dR}}$ is given by $c\mapsto \bar{P}^{-1}c\bar{P}$, where $\bar{P}\in \mathrm{GL}_n(\bold{B}_{\mathrm{dR}}\otimes_{\mathbb{Q}_p}E)$
    is the reduction of $P\in \mathrm{GL}_n(\bold{B}_{\mathrm{dR}}\otimes_{\mathbb{Q}_p}A)$. By the definition of Galois cohomology 
    of $B$-pairs, these mean that the triple $(c_e,c_{\mathrm{dR}}, c)$ defines an element $[(c_e,c_{\mathrm{dR}},c)]$ in $I\otimes_E\mathrm{H}^2(G_K, \mathrm{ad}(W))$ . 
    We can show that $[(c_e,c_{\mathrm{dR}},c)]$ doesn't depend on the choice of $s$ or $\widetilde{P}$, i.e. it depends only on $W_A$.
    Under the assumption that $\mathrm{H}^2(G_K,\mathrm{ad}(W))=0$, there exists a triple $(f_e,f_{\mathrm{dR}}, f)
    $, where  $f_e:G_K\rightarrow I\otimes_{E}\mathrm{Hom}_{\bold{B}_{e}\otimes_{\mathbb{Q}_p}E}(W_e, W_e)$ and 
    $f_{\mathrm{dR}}:G_K\rightarrow I\otimes_E\mathrm{Hom}_{\bold{B}^+_{\mathrm{dR}}\otimes_{\mathbb{Q}_p}E}(W^+_{\mathrm{dR}}, W^+_{\mathrm{dR}})$ are continuous maps and $f\in I\otimes_E\mathrm{Hom}_{\bold{B}_{\mathrm{dR}}\otimes_{\mathbb{Q}_p}E}(W_{\mathrm{dR}}, W_{\mathrm{dR}})$, satisfying that $$c_e(g_1,g_2)=g_1f_e(g_2)-f_e(g_1g_2)+f_e(g_1)$$ and $$c_{\mathrm{dR}}(g_1,g_2)=g_1f_{\mathrm{dR}}(g_2)
    -f_{\mathrm{dR}}(g_1g_2)+f_{\mathrm{dR}}(g_1)$$ and $$c(g_1)=f_{\mathrm{dR}}(g_1)-\bar{P}^{-1}f_{e}(g_1)\bar{P}+(g_1f-f)$$ 
    for any $g_1,g_2\in G_K$.  
    Using these, we define new liftings $\rho'_{e}:G_K\rightarrow \mathrm{GL}_n(\bold{B}_{e}\otimes_{\mathbb{Q}_p}A')$ by 
    $$\rho'_e(g):=(1+f_e(g))\widetilde{\rho}_e(g),$$ 
    and $\rho'_{\mathrm{dR}}(g):G_K\rightarrow \mathrm{GL}_n(\bold{B}^+_{\mathrm{dR}}\otimes_{\mathbb{Q}_p}A')$ by $$\rho'_{\mathrm{dR}}(g):=(1+f_{\mathrm{dR}}(g))\widetilde{\rho}_{\mathrm{dR}}(g),$$ and define a matrix 
    $$P':=(1+f)\widetilde{P}\in \mathrm{GL}_n(\bold{B}_{\mathrm{dR}}\otimes_{\mathbb{Q}_p}A').$$ Then, we can check 
    that these satisfy the equalities $$\rho'_e(g_1g_2)=\rho'_e(g_1)g_1(\rho'_{e}(g_2))\text{ and }\rho'_{\mathrm{dR}}(g_1g_2)
    =\rho'_{\mathrm{dR}}(g_1)g_1(\rho'_{\mathrm{dR}}(g_2))$$ and $$P'\rho'_e(g_1)g_1(P')^{-1}=\rho'_{\mathrm{dR}}(g_1)$$ for 
    any $g_1,g_2\in G_K$.  By the definition of $A'$-$B$-pair, these equalities mean that the triple 
    $(\rho'_e,\rho'_{\mathrm{dR}},P')$ defines 
    an $A'$-$B$-pair which is a lift of $W_A$, which proves the proposition.
    \end{proof}
    
    \begin{corollary}\label{13}
    Let $W$ be an $E$-$B$-pair of rank $n$.
    If $\mathrm{End}_{G_K}(W)=E$ and $\mathrm{H}^2(G_K, \mathrm{ad}(W))=0$, 
    then the functor $D_W$ is pro-representable by $R_W$ such that 
    $$R_W\isom E[[T_1,\cdots,T_d]]\,\,\text{ for } d:=[K:\mathbb{Q}_p]n^2 +1.$$
    \end{corollary}
    \begin{proof}
    The existence and the formal smoothness of $R_W$ follow from Proposition $\ref{9}$ and Proposition $\ref{12}$.
    For its dimension, we have 
\begin{multline*}
    \mathrm{dim}_ED_W(E[\varepsilon])=\mathrm{dim}_E\mathrm{H}^1(G_K, \mathrm{ad}(W))\\
                                                    =[K:\mathbb{Q}_p]n^2
    +\mathrm{dim}_E\mathrm{H}^0(G_K,\mathrm{ad}(W))+\mathrm{dim}_E\mathrm{H}^2(G_K,\mathrm{ad}(W)) 
    =[K:\mathbb{Q}_p]n^2
    +1
    \end{multline*}
     by Theorem $\ref{h}$ and Lemma $\ref{10}$.
    \end{proof}
 
\subsection{Trianguline  deformations of trianguline $B$-pairs}

In this subsection, we define the trianguline deformation functor for 
split trianguline $E$-$B$-pairs and prove the pro-representability and the formal smoothness 
under some conditions, and calculate the dimension of the universal 
deformation ring of  this functor. These are the generalizations of Bella\"iche-Chenevier's works 
in the $\mathbb{Q}_p$-case. In $\S$ 2 of \cite{Bel-Ch09}, Bella\"iche-Chenevier 
proved all these results in the $\mathbb{Q}_p$-case by using $(\varphi,\Gamma)$-modules over the Robba ring and Colmez's 
theory of trianguline representations \cite{Co08}. 
We generalize their results by using 
$B$-pairs and the theory of trianguline representations for any $p$-adic field (\cite{Na09} or $\S$ 2.1). 

We first define the notion of split  trianguline $A$-$B$-pairs  as follows.
\begin{defn}
Let $W$ be an $A$-$B$-pairs of rank $n$. We say that $W$ is a split trianguline 
$A$-$B$-pair if there exists a sequence of sub $A$-$B$-pairs 
$$\mathcal{T}: 0=W_0\subseteq W_1\subseteq W_2\subseteq\cdots \subseteq W_{n-1}
\subseteq W_n=W$$ such that  $W_i$ is saturated in $W_{i+1}$ and the 
quotient $W_{i+1}/W_i$ is a rank one $A$-$B$-pair for any $0\leqq i\leqq 
n-1$.

By Proposition $\ref{5}$, there exists a continuous homomorphisms $\delta_i:K^{\times}\rightarrow A^{\times}$ 
such that $W_i/W_{i-1}\isom W(\delta_i)$ for each $1\leqq i\leqq n$.
We say that the ordered set $\{\delta_i\}_{i=1}^n$ is the parameter of the triangulation 
$\mathcal{T}$.
\end{defn}

Next, we define the trianguline deformation functor. 
Let $W$ be a split trianguline $E$-$B$-pair of rank $n$. We fix a triangulation 
$$\mathcal{T}: 0\subseteq W_1\subseteq\cdots\subseteq W_{n-1}\subseteq 
W_n=W$$ of $W.$ Under this situation, we define the trianguline deformation 
as follows.

\begin{defn}
Let $A$ be an object in $\mathcal{C}_E$.
We say that $(W_A,\iota, \mathcal{T}_A)$ is a trianguline deformation of $(W,\mathcal{T})$ over $A$
if $(W_A,\iota)$ is a deformation of $W$ over $A$ and $W_A$ is a split trianguline $A$-$B$-pair with 
a triangulation 
$$\mathcal{T}_A:0\subseteq W_{1,A}\subseteq \cdots \subseteq W_{n,A}=W_A$$ 
 such that  $\iota(W_{i,A}\otimes_A E)=W_i$ for any $1\leqq i\leqq n$.
Let $(W_A,\iota,\mathcal{T}_A)$ and $(W'_A,\iota',\mathcal{T}'_A)$ be two trianguline 
deformations of $(W,\mathcal{T})$ over $A$. We say that $(W_A,\iota,\mathcal{T}_A)$ and 
$(W'_A,\iota',\mathcal{T}'_A)$ are equivalent if there exists an isomorphism of $A$-$B$-pairs 
$f:W_A\isom W'_A$ satisfying that $\iota=\iota'\circ \bar{f}$ and $f(W_{i,A})=W'_{i,A}$ for 
any $1\leqq i\leqq n$.
\end{defn}
\begin{defn}
Let $W$ be a split trianguline $E$-$B$-pair with a triangulation $\mathcal{T}$. 
We define the trianguline deformation functor $D_{W,\mathcal{T}}$ 
from the category $\mathcal{C}_E$ to the category of sets by 
\begin{multline*}
D_{W,\mathcal{T}}(A):=\{\text{equivalent classes } (W_A,\iota,\mathcal{T}_A) \text{ of }\\
\text{trianguline deformations of } (W,\mathcal{T})\text{ over }A  \}.
\end{multline*}
for $A\in\mathcal{C}_E$.
\end{defn}

By definition, we have a natural map of functors from $D_{W,\mathcal{T}}$ to $D_{W}$ by forgetting 
the triangulations, i.e. defined by 
$$D_{W,\mathcal{T}}(A)\rightarrow D_{W}(A): [(W_A,\iota,\mathcal{T}_A)]
\mapsto [(W_A,\iota)].$$
 In general, $D_{W,\mathcal{T}}$ is not a subfunctor of 
$D_W$ by this map, i.e. a deformation $W_A$ can have many liftings of the triangulation $\mathcal{T}$.
Here, we give a sufficient condition for $D_{W,\mathcal{T}}$ to be a subfunctor of $D_W$.
Let $\{\delta_i\}_{i=1}^n$ be the parameter of triangulation $\mathcal{T}$.
\begin{lemma}\label{14}
Assume that $\delta_j/\delta_i\not= \prod_{\sigma\in\mathcal{P}}\sigma^{k_{\sigma}}$  for any $1\leqq i<j\leqq n$ 
and $\{k_{\sigma}\}_{\sigma\in\mathcal{P}}\in \prod_{\sigma\in\mathcal{P}} \mathbb{Z}_{\leqq 0}$, then 
the functor $D_{W,\mathcal{T}}$ is a subfunctor of $D_W$.
\end{lemma}
\begin{proof}
Let $W_A$ be a deformation of $W$ over $A$, let $0\subseteq W_{A ,1}\subseteq \cdots \subseteq W_{A ,n-1}\subseteq W_{A}$ 
and $0\subseteq W'_{A ,1}\subseteq \cdots \subseteq W'_{A, n-1}\subseteq W_{A}$ be two triangulations which are 
lifts of $\mathcal{T}$. It suffices to show the equalities $W_{A, i}=W'_{A, i}$ for all $i$. By induction, it suffices 
to show the equality $W_{A ,1}= W'_{A ,1}$. To prove this, we first consider $\mathrm{Hom}_{G_K}(W_{1,A}, W_A)$. 
This is equal to  $\mathrm{H}^0(G_K, W^{\vee}_{1, A}\otimes W_A)$. Because $\mathrm{H}^0(G_K, -)$ is left exact 
and because $\mathrm{H}^0(G_K, W(\delta))=0$ for any $\delta:K^{\times}\rightarrow E^{\times}$ such that 
$\delta\not= \prod_{\sigma\in \mathcal{P}}\sigma^{k_{\sigma}}$ for any $\{k_{\sigma}\}_{\sigma\in\mathcal{P}}
\in \prod_{\sigma\in\mathcal{P}}\mathbb{Z}_{\leqq 0}$ by Proposition $\ref{i}$, we have 
$$\mathrm{H}^0(G_K, W_{1, A}^{\vee}\otimes 
(W_{i+1, A}/W_{i, A}))\allowbreak =\mathrm{H}^0(G_K, W_{1, A}^{\vee}\otimes( W'_{i+1, A}/W'_{i, A}))=0$$
 for any $i\geqq 1$. 
Hence, we obtain equalities  $$\mathrm{Hom}_{G_K}(W_{1, A}, W_{1,A})\allowbreak =\mathrm{Hom}_{G_K}(W_{1,A}, \allowbreak W_A)=\mathrm{Hom}_{G_K}(W_{1, A}, 
W'_{1,A}).$$ This means that the given inclusion $W_{1, A}\hookrightarrow W_A$ factors through $W'_{1, A}\hookrightarrow W_A$. 
By the same reason, the inclusion $W'_{1,A}\hookrightarrow W_A$ also factors through $W_{1, A}\hookrightarrow W_{A}$. 
Hence, we obtain an equality $W_{1,A}=W'_{1, A'}$.
\end{proof}

Next, we prove relative representability of $D_{W,\mathcal{T}}$. Before doing this, we need to define the following 
functor which is the $B$-pair version of Lemma 2.3.8 of \cite{Bel-Ch09}.
Let $W=(W_e,W^+_{\mathrm{dR}})$ be an $E$-$B$-pair. Then we define a functor by 
$$F(W):=\{x\in W_e\cap W^+_{\mathrm{dR}}| \exists n\in\mathbb{Z}_{\geqq 1}, (\sigma_1-1)(\sigma_2-1)\cdots (\sigma_n-1)x=0 
, \forall \sigma_1,\cdots,\sigma_n\in G_K\}$$ which is an $E$-vector space with a $G_K$-action, hence 
$F$ is a left exact functor form the category of $E$-$B$-pairs to that of $E[G_K]$-modules. By this definition, we 
obtain the following lemma, which is the $B$-pair version of Lemma 2.3.8 of \cite{Bel-Ch09}.
\begin{lemma}\label{14.5}
Let $\delta:K^{\times}\rightarrow E^{\times}$ be a continuous homomorphism, then 
$F(W(\delta))\not= 0$ if and only if $\mathrm{H}^0(G_K, W(\delta))\not=0$.
\end{lemma}
\begin{proof}
The proof is essentially the same as that of  Lemma 2.3.8 of \cite{Bel-Ch09}, so we omit it.
\end{proof}

Using this lemma, we prove the relative representability of $D_{W,\mathcal{T}}$.

\begin{prop}\label{15}
Let $W$ be a trianguline representations with a triangulation $\mathcal{T}$ 
such that the parameter $\{\delta_i\}_{\i=1}^n$ satisfies $\delta_j/\delta_i\not=\prod_{\sigma\in \mathcal{P}}
\sigma^{k_{\sigma}}$ for any $1\leqq i<j\leqq n$ and $\{k_{\sigma}\}_{\sigma\in\mathcal{P}}\in \prod_{\sigma\in\mathcal{P}}\mathbb{Z}_{\leqq 0}$,
then the natural map of functors $D_{W,\mathcal{T}}\rightarrow D_{W}$ is relatively representable.
\end{prop}
\begin{proof}
By $\S 23$ of \cite{Ma97}, it suffices to check that the map $D_{W,\mathcal{T}}\rightarrow D_W$  satisfies the fallowing three conditions (1), (2), (3).
\begin{itemize}
\item[(1)] For any map $A\rightarrow A'$ in $\mathcal{C}_E$ and $W_{A}\in D_{W,\mathcal{T}}(A)$, we have
 $W_A\otimes_A A'\in D_{W,\mathcal{T}}(A')$. 
 \item[(2)]For any maps $A'\rightarrow A$, $A"\rightarrow A$ in $\mathcal{C}_E$
  and $W^{'"}\in D_W(A'\times_A A")$, if $W^{'"}\otimes_{A'\times_A A"} A'\in D_{W,\mathcal{T}}(A')$ and 
  $W^{'"}\otimes_{A'\times_A A"}A"\in D_{W,\mathcal{T}}(A")$, then we also have $W^{'"}\in D_{W,\mathcal{T}}(A'\times_A A")$.
  \item[(3)] For any inclusion $A\hookrightarrow A'$ in $\mathcal{C}_E$ and $W_{A}\in D_W(A)$, if $W_A\otimes_A A'\in 
   D_{W,\mathcal{T}}(A')$, then we have $W_A\in D_{W,\mathcal{T}}(A)$. 
   \end{itemize}

   The condition (1) is trivial. For (2), let $W^{'"}\in D_{W}(A'\times_A A")$ be a deformation such that $W_{A'}:=W^{'"}\otimes_{A'\times_A A"}A'
   \in D_{W,\mathcal{T}}(A')$ and $W_{A"}:=W^{'"}\otimes_{A'\times_A A"}A"\in D_{W,\mathcal{T}}(A")$. We put $W_A:=W^{'"}\otimes_{A'\times_A A"}A$. In the same way 
    as in the proof of Proposition $\ref{9}$, we have an isomorphism 
   $W^{'"}\isom W_{A'}\times_{W_A} W_{A"}$. 
   By Lemma $\ref{14}$, the triangulations of $W_A$ induced from $W_{A'}$ and $W_{A"}$ coincide, hence 
   these triangulations induce a triangulation of $W^{'"}\isom W_{A'}\times_{W_A} W_{A"}$, i.e. $W^{'"}\in D_{W,\mathcal{T}}(A'\times_A A")$.
   
   Finally, we check the condition (3). The proof is essentially the same as that of Proposition 2.3.9 
   of \cite{Bel-Ch09}, but here we give the proof for convenience of readers. 
   Let $W\in D_W(A)$ and $A\hookrightarrow A'$ be an inclusion such that 
   $W_A\otimes_A A'\in D_{W,\mathcal{T}}(A')$. Let $0\subseteq W_{1,A'}\subseteq\cdots\subseteq W_{n-1,A'}\subseteq W_{A}
   \otimes_A A'$ be a triangulation which is a lifting of $\mathcal{T}$. By induction on the rank of $W$, it suffices to show that there exists 
   a rank one sub $A$-$B$-pair $W_{1,A}\subseteq W_A$ such that $W_{1,A}\otimes_{A}A'=W_{1,A'}$ and that $W_A/W_{1,A}$ is an $A$-$B$-pair. By Proposition \ref{5}, there exists a continuous homomorphism $\delta_{1,A'}:K^{\times}\rightarrow 
   A^{' \times}$ such that $W_{1,A'}\isom W(\delta_{1,A'})$. Twisting $W$ by $\delta_1^{-1}$, we may assume that 
   $\delta_{1,A'}\equiv 1$ (mod $\mathfrak{m}_{A'}$). Under this assumption, we apply the functor $F(-)$. 
   Because $\delta_{1,A'}$ is unitary,  there exists a continuous character $\widetilde{\delta}_{1,A'}:G_K^{\mathrm{ab}}\rightarrow 
   A^{' \times}$ such that $W(\delta_{1,A'})\isom W(A'(\widetilde{\delta}_{1,A'}))=(\bold{B}_{e}\otimes_{\mathbb{Q}_p}A'(\widetilde{\delta}_{1,A'}), 
   \bold{B}^+_{\mathrm{dR}}\otimes_{\mathbb{Q}_p}A'(\widetilde{\delta}_{1,A'}))$, hence we have $W_{1,A',e}\cap W^+_{1,A',\mathrm{dR}}
   =A'(\widetilde{\delta}_{1,A'})$. Moreover, because the image of $\delta_{1,A'}$ is in $1+\mathfrak{m}_{A'}$, we also have 
   $F(W_{1,A'})=A'(\widetilde{\delta}_{1,A'})$. Next, because $(W_A\otimes_A A')/W_{1,A'}$ is a successive extension of $W(\delta_{i}\delta^{-1}_{1})$ ($i\geqq 2$) as an $E$-$B$-pair, 
   the left exactness of $F$ implies that $F((W_A\otimes_A A')/W_{1,A'})=0$ by Lemma \ref{14.5}.
   Applying $F$ to the short exact sequence $0\rightarrow 
   W_{1,A'}\rightarrow W_A\otimes_A A'\rightarrow (W_A\otimes_A A')/W_{1,A'}\rightarrow 0$, we obtain 
   $A'(\widetilde{\delta}_{1,A'})\isom F(W_{1,A'})=F(W_A\otimes_A A')$. In the same way, we obtain 
   $E=F(W_1)=F(W)$. Then, by the left exactness and by considering the length, 
   we can show that $F(W_A)$ is a free $A$-module of rank one and that the natural map $F(W_A)\rightarrow F(W)$ induced by the 
   natural quotient map $W_A\rightarrow W$ is a surjection and that the natural map $F(W_A)\otimes_A A'\rightarrow 
   F(W_A\otimes_A A')$ is isomorphism. If we define the character $\widetilde{\delta}_{1,A}:G_K^{\mathrm{ab}}\rightarrow A^{\times}$ such that 
   $F(W_A)\isom A(\widetilde{\delta}_{1,A})$ and define $W_{1,A}$ as the image of the natural map $(\bold{B}_{e}\otimes_{\mathbb{Q}_p}F(W_A), 
   \bold{B}^+_{\mathrm{dR}}\otimes_{\mathbb{Q}_p}F(W_A))\rightarrow W_A$ induced form $F(W)\hookrightarrow W_{A,e}, F(W)\hookrightarrow 
   W^+_{A,\mathrm{dR}}$, then we can check that $W_{1,A}$ is a rank one $A$-$B$-pair and that the quotient $W_A/W_{1,A}$ is 
   also an $A$-$B$-pair and that $W_{1,A}\otimes_A A'\isom W_{1,A'}$, which proves the condition (3), hence we finish to prove the proposition.

     \end{proof}
     \begin{corollary}\label{16}
     Let $W$ be a trianguline $E$-$B$-pair with a triangulation $\mathcal{T}$. 
     Assume that $\mathrm{End}_{G_K}(W)=E$
      and that the parameter $\{\delta_i\}_{i=1}^n$ of $\mathcal{T}$ satisfies
      $\delta_j/\delta_i\not= \prod_{\sigma\in \mathcal{P}}\sigma^{k_{\sigma}}$  for any $1\leqq i<j\leqq n$ and $\{k_{\sigma}\}_{\sigma\in \mathcal{P}}
      \in \prod_{\sigma\in\mathcal{P}}\mathbb{Z}_{\leqq 0}$, 
      then the functor $D_{W,\mathcal{T}}$ is pro-representable by a quotient $R_{W,\mathcal{T}}$ of $R_W$.
      \end{corollary}
      \begin{proof} This follows from Proposition $\ref{9}$ and Proposition $\ref{15}$.
      \end{proof}
      
 Next, we prove  the formal smoothness of the functor $D_{W,\mathcal{T}}$.
 \begin{prop}\label{17}
 Let $W$ be a trianguline $E$-$B$-pair of rank $n$ with a triangulation $\mathcal{T}$.
 Assume  that the parameter $\{\delta_i\}_{i=1}^{n}$ of $\mathcal{T}$ satisfies $\delta_i/\delta_j\not=
  |\mathrm{N}_{K/\mathbb{Q}_p}|\prod_{\sigma\in\mathcal{P}}\sigma^{k_{\sigma}} $  for any $1\leqq i< j\leqq n$ and 
  $\{k_{\sigma}\}_{\sigma\in\mathcal{P}}\in \prod_{\sigma\in\mathcal{P}}\mathbb{Z}_{\geqq 1}$, 
  then the functor $D_{W,\mathcal{T}}$ is formally smooth.
  \end{prop}
  \begin{proof}
  We prove this proposition by induction of the rank of $W$. 
  When $W$ is  of rank one, then we have $D_{W,\mathcal{T}}=D_W$ and 
 $\mathrm{ad}(W)$ is the trivial $E$-$B$-pair. 
 Hence $\mathrm{H}^2(G_K,\mathrm{ad}(W))\allowbreak=0$ by Proposition $\ref{i}$, and 
 $D_{W,\mathcal{T}}$ is formally smooth by Proposition $\ref{12}$.
 Let's assume that the proposition is proved for all trianguline $E$-$B$-pairs of rank less or equal $n-1$. Let $W$ be an $E$-$B$-pair of rank $n$ with a triangulation 
  $\mathcal{T}:0\subseteq W_1\subseteq \cdots \subseteq W_{n-1}\subseteq W_n=W$ 
  whose parameter $\{\delta_i\}_{i=1}^n$ satisfies the condition in the proposition. 
  Let $A'\rightarrow A$ be a small extension in $\mathcal{C}_E$, and let 
  $W_A$ be a trianguline deformation of $(W,\mathcal{T})$ with a triangulation 
  $\mathcal{T}_A:0\subseteq W_{1,A}\subseteq \cdots \subseteq W_{n-1,A}\subseteq W_{n,A}=W_A$ 
  which is a lift of $\mathcal{T}$. Then, it suffices to show that there exists a split trianguline $A'$-$B$-pair 
  $W_{A'}$ with a triangulation $0\subseteq W_{1,A'}\subseteq \cdots \subseteq W_{n-1.A'}\subseteq 
  W_{n,A'}=W_{A'}$ which is a lift of $W_A$ and $\mathcal{T}_A$. We take such a lift as follows. 
  Applying the induction hypothesis to $W_{n-1}$, there exists a trianguline $A'$-$B$-pair $W_{n-1,A'}$ 
  with a triangulation $0\subseteq W_{1,A'}\subseteq \cdots \subseteq W_{n-2,A'}\subseteq W_{n-1,A'}$ 
  which is a lift of $W_{n-1,A}$ and $0\subseteq W_{1,A}\subseteq\cdots\subseteq W_{n-1,A}$. 
  We put $\mathrm{gr}_nW_A:=W_{A}/W_{n-1,A}$. By the rank one case and by Proposition $\ref{5}$, there exists 
  a continuous homomorphism $\delta_{n,A'}:K^{\times}\rightarrow A^{' \times}$ such that the rank one $A'$-$B$-pair 
  $W(\delta_{n,A'})$ satisfies $W(\delta_{n,A'})\otimes_{A'}A=W(\delta_{n,A})\isom \mathrm{gr}_nW_A$, 
  where $\delta_{n,A}:K^{\times}\rightarrow A^{\times}$ is the composition of $\delta_{n,A'}$ with $A'\rightarrow A$. 
  We see the isomorphism class $[W_A]$ as an element in $\mathrm{Ext}^1(W(\delta_{n,A}), W_{n-1,A})\isom \mathrm{H}^1(G_K,
  W_{n-1,A}(\delta_{n,A}^{-1}))$. If we take the long exact sequence associated 
  to 
   \begin{equation*}
  0\rightarrow I\otimes_E W_{n-1}(\delta_n^{-1})\rightarrow W_{n-1,A'}(\delta^{-1}_{n,A'})\rightarrow 
  W_{n-1,A}(\delta^{-1}_{n,A})\rightarrow 0,
  \end{equation*}
  where $I\subseteq A'$ is the kernel of $A'\rightarrow A$,  then we obtain a long exact sequence
  \begin{multline*}
  \cdots \rightarrow \mathrm{H}^1(G_K,W_{n-1,A'}(\delta^{-1}_{n,A'}))\rightarrow \mathrm{H}^1(G_K, W_{n-1,A}(\delta^{-1}_{n,A}))\\
                                        \rightarrow I\otimes_E \mathrm{H}^2(G_K, W_{n-1}(\delta^{-1}_n))\rightarrow\cdots 
  \end{multline*}
  By the assumption on $\{\delta_i\}_{i=1}^n$ and by Proposition $\ref{i}$, we have 
  $\mathrm{H}^2(G_K, W_{n-1}(\delta^{-1}_n))\allowbreak=0$. 
  Hence, we can take a lift $[W_{A'}]\in\mathrm{Ext}^1(W(\delta_{n,A'}),W_{n-1,A'})\isom\mathrm{H}^1(G_K, W_{n-1,A'}\allowbreak(\delta^{-1}_{n,A'}))$ of $[W_A]$, which proves the proposition.
  \end{proof}
  
  Next, we calculate the dimension of $D_{W,\mathcal{T}}$. For this, we interpret the tangent space
  $D_{W,\mathcal{T}}(E[\varepsilon])$ in terms of Galois cohomology of $B$-pair as in Lemma $\ref{10}$.
  Let $W$ be a trianguline $E$-$B$-pair with a triangulation $\mathcal{T}:0\subseteq W_1\subseteq \cdots\subseteq 
  W_{n-1}\subseteq W_n=W$, then we define an $E$-$B$-pair $\mathrm{ad}_{\mathcal{T}}(W)$ by 
  $$\mathrm{ad}_{\mathcal{T}}(W):=\{f\in\mathrm{ad}(W)| f(W_i)\subseteq W_i \text{ for any } 1\leqq i\leqq n \}.$$ 
  \begin{lemma}\label{18}
  Let $W$ be a trianguline $E$-$B$-pair, then there exists 
  a canonical bijection of sets $$D_{W,\mathcal{T}}(E[\varepsilon])
  \isom \mathrm{H}^1(G_K, \mathrm{ad}_{\mathcal{T}}(W)).$$ 
  In particular, if $D_{W,\mathcal{T}}$ has a canonical structure of $E$-vector space 
  $($see the condition $(2)$ in Schlessinger's criterion $\ref{8}$$)$, then this bijection is an $E$-linear isomorphism.
  \end{lemma}
  \begin{proof}
The construction of the map $D_{W,\mathcal{T}}(E[\varepsilon])\rightarrow 
\mathrm{H}^1(G_K,\mathrm{ad}_{\mathcal{T}}(W))$ is the same as in the proof of 
Lemma $\ref{10}$. We put $\mathrm{ad}_{\mathcal{T}}(W):=(\mathrm{ad}_{\mathcal{T}}(W_e), \mathrm{ad}_{\mathcal{T}}(W^+_{\mathrm{dR}}))$. Let $W_{E[\varepsilon]}:=(W_{e,E[\varepsilon]},W^+_{\mathrm{dR},E[\varepsilon]})$ be a trianguline deformation 
of $(W,\mathcal{T})$ over $E[\varepsilon]$ with a lifted triangulation $\mathcal{T}_{E[\epsilon]}:0\subseteq W_{1,E[\varepsilon]},\subseteq\cdots\subseteq 
W_{n-1,E[\varepsilon]}\subseteq W_{n,E[\varepsilon]}=W_{E[\varepsilon]}$. Then, we can take a splitting 
$W_{e,E[\varepsilon]}=W_e e_1\oplus W_ee_2$ as a filtered $\bold{B}_{e}\otimes_{\mathbb{Q}_p}E$-module such that 
$W_e e_1=\varepsilon W_{e,E[\varepsilon]}$ and that the natural map $W_ee_2\hookrightarrow W_{e,E[\varepsilon]}
\rightarrow W_{e,E[\varepsilon]}/\varepsilon W_{e,E[\varepsilon]}\isom W_e$ sends  $ye_2$ to 
$y$ for any $y\in W_e$.
If we define $c_e:G_K\rightarrow \mathrm{Hom}_{\bold{B}_{e}\otimes_{\mathbb{Q}_p}E}(W_e, W_e)$ in the same way as in the proof 
of Lemma $\ref{10}$, we can check that the image of $c_e$ is contained in $\mathrm{ad}_{\mathcal{T}}(W_e)$. In the same way,  
we can define $c_{\mathrm{dR}}:G_K\rightarrow \mathrm{ad}_{\mathcal{T}}(W^+_{\mathrm{dR}})$ from a 
filtered splitting $W^+_{\mathrm{dR},E[\varepsilon]}=W^+_{\mathrm{dR}}e_1\oplus W^+_{\mathrm{dR}}e'_2$. 
Moreover, we can define $c\in \mathrm{ad}_{\mathcal{T}}(W_{\mathrm{dR}})$ by $ye_2=c(y)e_1+ye'_2$ for $y\in W_{\mathrm{dR}}$.
Then, the map $$D_{W,\mathcal{T}}(E[\varepsilon])\rightarrow \mathrm{H}^1(G_K, \mathrm{ad}_{\mathcal{T}}(W)):[(W_{E[\varepsilon]},
\mathcal{T}_{E[\varepsilon]})]\mapsto [(c_e,c_{\mathrm{dR}},c)]$$ defines a bijection, and this is an 
$E$-linear isomorphism when $D_{W,\mathcal{T}}(E[\varepsilon])$ 
has a canonical structure of an $E$-vector space.
\end{proof}

We calculate the dimension of $R_{W,\mathcal{T}}$.
\begin{prop}\label{19}
Let $W$ be a split trianguline $E$-$B$-pair of rank $n$ with a triangulation 
$\mathcal{T}:0\subseteq W_1\subseteq \cdots\subseteq W_{n-1}\subseteq W_n=W$. 
We assume that $(W,\mathcal{T})$ satisfies the following conditions, 
\begin{itemize}
\item[(0)]$\mathrm{End}_{G_K}(W)=E$,
\item[(1)]$\delta_j/\delta_i\not= \prod_{\sigma\in\mathcal{P}}\sigma^{k_{\sigma}}$ for any $1\leqq i < j\leqq n$ and $\{k_{\sigma}\}_{\sigma\in\mathcal{P}}\in \prod_{\sigma\in\mathcal{P}}\mathbb{Z}_{\leqq 0}$,
\item[(2)]$\delta_i/\delta_j\not= |\mathrm{N}_{K/\mathbb{Q}_p}|\prod_{\sigma\in\mathcal{P}}\sigma^{k_{\sigma}}$ for any $1\leqq i < j\leqq n$ and $\{k_{\sigma}\}_{\sigma\in\mathcal{P}}\in \prod_{\sigma\in\mathcal{P}}\mathbb{Z}_{\geqq 1}$, 
\end{itemize}
then the universal trianguline deformation ring $R_{W,\mathcal{T}}$ is a quotient ring of $R_W$ such 
that $$R_{W,\mathcal{T}}\isom E[[T_1,\cdots, T_{d_n}]]\text{ for }d_n:=\frac{n(n+1)}{2}[K:\mathbb{Q}_p]+1.$$
\end{prop}
\begin{proof}
By Proposition $\ref{15}$ and Proposition $\ref{17}$ and Lemma $\ref{18}$, it suffices to show that $\mathrm{dim}_E\mathrm{H}^1(G_K, \mathrm{ad}_{\mathcal{T}}(W))=d_n$. 
We prove this by induction on the rank $n$ of $W$. 
When $n=1$, then $\mathrm{ad}_{\mathcal{T}}(W)=\mathrm{ad}(W)$ is the trivial 
$E$-$B$-pair, hence the proposition 
follows from Proposition $\ref{i}$. Let $(W,\mathcal{T})$ be a split trianguline $E$-$B$-pair of rank $n$ satisfying all the conditions in the propositionas. Put $\mathcal{T}_{n-1}:0\subseteq 
W_{1}\subseteq \cdots \subseteq W_{n-2}\subseteq W_{n-1}$ the triangulation of $W_{n-1}(\subseteq W)$ which is  induced from $\mathcal{T}$. 
Then, for any $f\in \mathrm{ad}_{\mathcal{T}}(W)$, the restriction of $f$ 
to $W_{n-1}$ is an element of $\mathrm{ad}_{\mathcal{T}_{n-1}}(W_{n-1})$ and this defines a short exact sequence 
of $E$-$B$-pair
\begin{equation*}
0\rightarrow \mathrm{Hom}(W(\delta_n), W)\rightarrow \mathrm{ad}_{\mathcal{T}}(W)\rightarrow 
\mathrm{ad}_{\mathcal{T}_{n-1}}(W_{n-1})\rightarrow 0.
\end{equation*}
From this, we obtain
$$\mathrm{rank}(\mathrm{ad}_{\mathcal{T}}(W))=\mathrm{rank}(\mathrm{ad}_{\mathcal{T}_{n-1}}(W_{n-1}))
+n=1+2+\cdots +n=\frac{n(n+1)}{2}$$
by induction.
 By Theorem $\ref{h}$, 
it suffices to show that $\mathrm{H}^0(G_K,\mathrm{ad}_{\mathcal{T}}\allowbreak(W))=E$ and $\mathrm{H}^2(G_K,\mathrm{ad}_{\mathcal{T}}(W))=0$
. For $H^0$, this follows from the following natural inclusions
$$E\subseteq \mathrm{H}^0(G_K, \mathrm{ad}_{\mathcal{T}}(W))\subseteq \mathrm{H}^0(G_K, 
\mathrm{ad}(W))=E.$$ We prove $\mathrm{H}^2(G_K, \mathrm{ad}_{\mathcal{T}}(W))=0$ by 
induction of the rank of $W$. 
When $n=1$, this follows from Proposition $\ref{i}$. When $W$ is of rank $n$, from the above short exact sequence, we obtain the following 
 long exact sequence 
 $$\cdots \rightarrow \mathrm{H}^2(G_K, \mathrm{Hom}(W(\delta_n), W))\rightarrow \mathrm{H}^2(G_K, \mathrm{ad}_{\mathcal{T}}(W))
 \rightarrow \mathrm{H}^2(G_K, \mathrm{ad}_{\mathcal{T}_{n-1}}(W_{n-1}))\rightarrow 0.$$
Because we have $\mathrm{H}^2(G_K, \mathrm{Hom}\allowbreak (W(\delta_n), W))=0$ by Proposition $\ref{i}$ and by the assumption on $\{\delta_i\}_{i=1}^n$, we obtain the equality $\mathrm{H}^2(G_K, \mathrm{ad}_{\mathcal{T}}(W))=0$ by induction, which proves the proposition. 
\end{proof}

\subsection{Deformations of benign $B$-pairs}

In this final subsection, we study benign representations or more generally benign $B$-pairs  
which is a class of  potentially crystalline and trianguline $B$-pairs and have some very good properties for trianguline 
deformations and play a crucial role in the problem of the Zariski density of modular Galois (or crystalline) 
representations in some deformation spaces of global (or local) $p$-adic Galois representations.
This class was defined by Kisin in the case where $K=\mathbb{Q}_p$ and the rank is $2$ in \cite{Ki03} and \cite{Ki10}.
He studied some deformation theoretic properties of this class in \cite{Ki03} and used these in a crucial way in his 
proof of Zariski density of two dimensional crystalline representations of $G_{\mathbb{Q}_p}$.
For higher dimensional and the $\mathbb{Q}_p$-case, Bella\"iche-Chenevier \cite{Bel-Ch09} and Chenevier \cite{Ch09b} were 
the first ones who noticed the importance of benign representations in the study of $p$-adic families 
of trianguline representations. In particular, Chenevier \cite{Ch09b} (where he calls ``generic" instead of benign)
discovered and proved a crucial theorem concerning the tangent spaces of the universal deformation 
rings of benign representations. 
In fact, by using this theorem, Chenevier \cite{Ch09b} proved some theorems concerning 
the Zariski density of modular Galois representations in some deformation spaces of global $p$-adic representations. 

The aim of this subsection is to generalize the definition of benign representations and the 
Chenevier's theorem for any $K$.

\subsubsection{Benign $B$-pairs}
Let $P(X)\in \mathcal{O}_K[X]$ be a polynomial such that $P(X)\equiv \pi_K X$ (mod $X^2$) 
and $P(X)\equiv X^q$ (mod $\pi_K$), where $q:=p^f$ and $f:=[K_0:\mathbb{Q}_p]$. We take the Lubin-Tate formal group 
law $\mathcal{F}$ over $\mathcal{O}_K$ such that $[\pi_K]=P(X)$, where $[-]:\mathcal{O}_K\isom \mathrm{End}(\mathcal{F})$.  We denote by $K_n$  the abelian extension of $K$ 
generated by $[\pi_K^n]$-torsion points of $\mathcal{F}(\overline{K})$ for any $n$, then we have a canonical isomorphism $\chi_{\mathrm{LT},n}:\mathrm{Gal}(K_n/K)\isom( \mathcal{O}_K^{\times}/
\pi^n\mathcal{O}_K)^{\times}$. We put $K_{\mathrm{LT}}:=\cup_{n=1}^{\infty}K_n$ and $G_n:=\mathrm{Gal}(K_n/K)$.

In \cite{Ki10}, \cite{Bel-Ch09} or \cite{Ch09b} etc, benign representation is defined as a special class of crystalline representations.
But, as we show in the sequel, we can easily generalize the main theorem to some potentially crystalline representations.
Hence, before defining benign representations, we first define the following class of potentially crystalline representations.

\begin{defn}
Let $W$ be an $E$-$B$pair. We say that $W$ is crystabelline if $W|_{G_{L}}$ is a crystalline $E$-$B$-pair of 
$G_{L}$ for a finite abelian extension $L$ of $K$.
\end{defn}
\begin{rem}
Because a finite abelian extension $L$ of $K$ is contained in $K_mL'$ for some $m\geqq 0$ and for 
a finite unramified extension $L'$ of $K$, by using Hilbert 90, we can easily show that $W$ is crystabelline if and only if 
$W|_{G_{K_m}}$ is crystalline for some $m\geqq 1$.
\end{rem}

Let $W$ be a crystabelline $E$-$B$-pair of rank $n$ such that $W|_{G_{K_m}}$ is crystalline for some $m$. Because $K_m$ is totally ramified over $K$, $\bold{D}^{K_m}_{\mathrm{cris}}(W):=(\bold{B}_{\mathrm{max}}\otimes_{\mathbb{Q}_p}W)^{G_{K_m}}$  is 
a free $K_0\otimes_{\mathbb{Q}_p}E$-module of rank $n$. 
We take an embedding $\sigma:K_0\hookrightarrow \overline{E}$. This defines a map 
$\sigma:K_0\otimes_{\mathbb{Q}_p}E\rightarrow \overline{E}:x\otimes y\rightarrow \sigma(x)y$.
Using this map, we define the $\sigma$-component  $$\bold{D}^{K_m}_{\mathrm{cris}}(W)_{\sigma}:=\bold{D}^{K_m}_{\mathrm{cris}}(W)
\otimes_{K_0\otimes_{\mathbb{Q}_p}E, \sigma}\overline{E},$$ this has a $\overline{E}$-linear 
$\varphi^f$-action and a $\overline{E}$-linear $G_m$-action. Let $\{\alpha_1,\cdots,\alpha_n\}$ be the solution in $\overline{E}$ 
(with multiplicities) of the characteristic polynomial $\mathrm{det}_{\overline{E}}(T\cdot \mathrm{id}-\varphi ^f |_{\bold{D}^{K_m}_{\mathrm{cris}}
(W)_{\sigma}})\in \overline{E}[T]$. Because the actions of $\varphi^f$ and $G_m$ commute, each generalized $\varphi^f$-eigenvector subspace
 of $\bold{D}^{K_m}_{\mathrm{cris}}(W)_{\sigma}$ is preserved by the action of $G_m$. 
 Hence we can take a $\overline{E}$-basis $\{e_{1,\sigma},\cdots,e_{n,\sigma}\}$ of $\bold{D}^{K_m}_{\mathrm{cris}}(W)_{\sigma}$ such that $e_{i,\sigma}$ is a generalized eigenvector of 
 $\varphi^f$ for the eigenvalue $\alpha_i\in \overline{E}^{\times}$ and $G_m$ acts on $e_{i,\sigma}$ by a character $\widetilde{\delta}_i:G_m\rightarrow \overline{E}^{\times}$ for any $i$.  We change the numbering of $\{\alpha_1,\cdots,\alpha_n\}$ so that the basis $e_{1,\sigma}, e_{2,\sigma},\cdots,e_{n,\sigma}$ gives a $\varphi^f$-Jordan decomposition of $\bold{D}^{K_m}_{\mathrm{cris}}(W)_{\sigma}$ by this order. 
 Because we have $\{\sigma, \varphi^{-1}\sigma,\cdots, \varphi^{-(f-1)}\sigma\}=\mathrm{Hom}_{\mathbb{Q}_p}(K_0, \overline{E})$ and 
$$\varphi^i:\bold{D}^{K_m}_{\mathrm{cris}}(W)_{\sigma}\isom \bold{D}^{K_m}_{\mathrm{cris}}(W)_{\varphi^{-i}\sigma}:x\otimes y\mapsto \varphi^i(x)\otimes y$$ ($x\in \bold{D}^{K_m}_{\mathrm{cris}}(W)$ and $y\in \overline{E}$) is a $\overline{E}[\varphi^f,G_m]$-isomorphism, 
the set $\{\alpha_1,\cdots,\alpha_n\}$ doesn't depend on the choice of $\sigma:K_0\hookrightarrow \overline{E}$. 
If we put $$e_i:=e_{i,\sigma}+\varphi(e_{i,\sigma})+\cdots +\varphi^{f-1}(e_{i,\sigma})\in \bold{D}^{K_m}_{\mathrm{cris}}(W)\otimes_{E}\overline{E},$$ we have $$\bold{D}^{K_m}_{\mathrm{cris}}(W)\otimes_E\overline{E}=K_0\otimes_{\mathbb{Q}_p}\overline{E}e_1\oplus\cdots \oplus K_0\otimes_{\mathbb{Q}_p}\overline{E}e_n$$ such that the subspace $K_0\otimes_{\mathbb{Q}_p}\overline{E}e_1\oplus\cdots\oplus K_0\otimes_{\mathbb{Q}_p}\overline{E}e_i$ is preserved by 
$\varphi$ and the action of $G_m$ for any $i$. Moreover, if we take a sufficiently large finite extension $E'$ of $E$, then we have $e_i\in \bold{D}^{K_m}_{\mathrm{cris}}(W)\otimes_E E'$ and $$\bold{D}^{K_m}_{\mathrm{cris}}(W)\otimes_E E'=
K_0\otimes_{\mathbb{Q}_p}E'e_1\oplus \cdots\oplus K_0\otimes_{\mathbb{Q}_p}E'e_n$$ and $\alpha_i\in E'$ and $\widetilde{\delta}_i:G_m\rightarrow E^{'\times}$ for any $i$.

Using these arguments, we first study a relation between crystabelline $E$-$B$-pairs and trianguline $E$-$B$-pairs.
\begin{lemma}\label{20.5}
Let $W$ be an $E$-$B$-pair of rank $n$. The following conditions are equivalent,
\begin{itemize}
\item[(1)]$W$ is crystabelline,
\item[(2)]$W$ is trianguline $($i.e. $W\otimes_E E'$ is a
split trianguline $E'$-$B$-pair for a finite extension $E'$ of $E$$)$ and potentially crystalline.
\end{itemize}
\end{lemma}
\begin{proof}
First we assume that $W$ is crystabelline. By the above argument, for a sufficiently large finite extension 
$E'$ of $E$, we have $\bold{D}^{K_m}_{\mathrm{cris}}(W)\otimes_E E'=K_0\otimes_{\mathbb{Q}_p}E'e_1\oplus 
\cdots \oplus K_0\otimes_{\mathbb{Q}_p}E'e_n$ as above and $K_0\otimes_{\mathbb{Q}_p}E'e_1
\oplus \cdots\oplus K_0\otimes_{\mathbb{Q}_p}E'e_i$ is a sub $E'$-filtered ($\varphi,G_m$)-module 
of $\bold{D}^{K_m}_{\mathrm{cris}}(W\otimes_EE')$ for any $i$. Hence $W\otimes_E E'$ is 
split trianguline and potentially crystalline by Theorem $\ref{e}$.

Next we assume that $W$ is trianguline and potentially crystalline. By extending the coefficient, we may assume that $W$ is split trianguline. 
We take a triangulation $0\subseteq W_1\subseteq \cdots \subseteq W_n=W$ of $W$. 
Because the sub or quotient $B$-pairs of crystalline $B$-pairs are again crystalline, $W_i$ and $W_i/W_{i-1}$ are 
all potentially crystalline. Because  $W_i/W_{i-1}$ is of rank one, $W_i/W_{i-1}|_{G_m}$ is crystalline for any $i$ and for any sufficiently large $m$. We claim that $W|_{G_m}$ is also crystalline. 
We prove this claim by induction on the rank $n$ of $W$. When $n=1$, this is trivial. 
We assume that the claim is proved for the rank $n-1$ case, hence $W_{n-1}|_{G_m}$ is crystalline.
If we put $W/W_{n-1}\isom W(\delta_n)$, we have $[W]\in \mathrm{H}^1(G_K, W_{n-1}(\delta_n^{-1}))$. By the assumption, 
there exists a finite Galois extension $L$ of $K_m$ such that $[W]$ is contained in 
$\mathrm{Ker}(\mathrm{H}^1(G_K, W_{n-1}(\delta_n^{-1}))\rightarrow \mathrm{H}^1(G_L, \bold{B}_{\mathrm{max}}\otimes_{\bold{B}_{e}}
(W_{n-1}(\delta_n^{-1}))_e))$.  Hence, it suffices to show that the natural map $\mathrm{H}^1(G_{K_m}, \bold{B}_{\mathrm{max}}\otimes_{\bold{B}_{e}}(W_{n-1}(\delta_n^{-1}))_e)\rightarrow \mathrm{H}^1(G_L, \bold{B}_{\mathrm{max}}\otimes_{\bold{B}_{e}}(W_{n-1}(\delta_n^{-1}))_e)$ is an injection. By the inflation 
 restriction sequence, the kernel of this map is $\mathrm{H}^1(\mathrm{Gal}(L/K_m), \bold{D}^L_{\mathrm{cris}}(W_{n-1}(\delta_n^{-1})))=0$. 
 Hence $W|_{G_m}$ is crystalline, i.e. $W$ is crystabelline.

\end{proof}

From now on, we consider a crystabelline $E$-$B$-pair $W$ of rank $n$ satisfying that 
$\bold{D}^{K_m}_{\mathrm{cris}}(W)= K_0\otimes_{\mathbb{Q}_p}Ee_1\oplus \cdots \oplus 
K_0\otimes_{\mathbb{Q}_p}Ee_n$ such that 
$K_0\otimes_{\mathbb{Q}_p}Ee_i$ is preserved by ($\varphi,G_m$) and $\varphi^f(e_i)=\alpha_ie_i$ for 
some $\alpha_i\in E^{\times}$ such that $\alpha_i\not= \alpha_j$ for any $i\not= j$. 
Let $\mathfrak{S}_n$ be the $n$-th permutation group.
For any $\tau\in \mathfrak{S}_n$, we define a filtration on $\bold{D}^{K_m}_{\mathrm{cris}}(W)$ 
by $E$-filtered ($\varphi,G_m$)-modules as follows,

$$\mathcal{F}_{\tau}:0\subseteq F_{\tau,1}\subseteq \cdots\subseteq F_{\tau, n-1}
\subseteq F_{\tau,n}=\bold{D}^{K_m}_{\mathrm{cris}}(W)$$  such that 
$$F_{\tau,i}:=K_0\otimes_{\mathbb{Q}_p}Ee_{\tau(1)}\oplus \cdots\oplus K_0\otimes_{\mathbb{Q}_p}Ee_{\tau(i)}$$ for any $1\leqq i \leqq n$
, where the filtration on $F_{\tau,i}$ is induced from that on $\bold{D}^{K_m}_{\mathrm{cris}}(W)$.
We put $\mathrm{gr}_{\tau,i}\bold{D}^{K_m}_{\mathrm{cris}}(W):=F_{\tau,i}/F_{\tau,i-1}$ for any 
$1\leqq i\leqq n$. By Theorem $\ref{e}$, there exists a 
filtration $$\mathcal{T}_{\tau}:0\subseteq W_{\tau,1}\subseteq\cdots \subseteq W_{\tau,n-1}\subseteq W_{\tau,n}=W$$ 
such that $W_{\tau,i}|_{G_m}$ is crystalline and $$\bold{D}^{K_m}_{\mathrm{cris}}(W_{\tau,i})=F_{\tau,i}.$$
For any $i$,  $W_{\tau,i}/W_{\tau,i-1}$ is a rank one crystabelline $E$-$B$-pair such that 
$\bold{D}^{K_m}_{\mathrm{cris}}(W_{\tau,i}/W_{\tau,i-1})\allowbreak\isom \mathrm{gr}_{\tau,i}\bold{D}^{K_m}_{\mathrm{cris}}(W)$. By Lemma 4.1 of \cite{Na09} and by its proof, 
there exists a set of integers $\{k_{(\tau,i),\sigma}\}_{\sigma\in\mathcal{P}}$ and 
a homomorphism $\widetilde{\delta}_{i}:K^{\times}\rightarrow E^{\times}$  satisfying $\widetilde{\delta}_{i}|_{1+\pi_K^m\mathcal{O}_K}=1$ and $\widetilde{\delta}_{i}(\pi_K)=1$, 
such that 
$$W_{\tau,i}/W_{\tau,i-1}\isom W(\delta_{\alpha_{\tau(i)}}\widetilde{\delta}_{\tau(i)}\prod_{\sigma\in\mathcal{P}}\sigma^{k_{(\tau,i),\sigma}})$$
for any $1\leqq i\leqq n$, where $\delta_{\alpha_i}:K^{\times}\rightarrow E^{\times}$ is the homomorphism such that 
$\delta_{\alpha_i}|_{\mathcal{O}_K^{\times}}=1$ and $\delta_{\alpha_i}(\pi_K)=\alpha_i$. For any $\sigma\in\mathcal{P}$, the set $\{k_{(\tau,1),\sigma},k_{(\tau,2),\sigma},\cdots,k_{(\tau,n),\sigma}\}$ is independent of $\tau\in \mathfrak{S}_n$  because these numbers are 
the $\sigma$-components of the Hodge-Tate weights of $W$. We denote this set (with multiplicities) by $\{k_{1,\sigma},k_{2,\sigma},\cdots, k_{n,\sigma}\}$ 
such that $k_{1,\sigma}\geqq k_{2,\sigma}\geqq\cdots\geqq k_{n,\sigma}$ for any $\sigma\in\mathcal{P}$. Under these notations, we define the notion of benign $E$-$B$-pair as follows.
\begin{defn}\label{26}
Let $W$ be a rank $n$ crystabelline $E$-$B$-pair as above. 
We say that $W$ is a benign 
$E$-$B$-pair if the following conditions hold:
\begin{itemize}
\item[(1)] For any $i\not= j$, we have $\alpha_i/\alpha_j\not=1, p^f, p^{-f}$.
\item[(2)] For any $\sigma\in\mathcal{P}$, we have $k_{1,\sigma}>k_{2,\sigma}>\cdots >k_{n-1,\sigma}>k_{n,\sigma}$.
\item[(3)] For any 
$\tau\in \mathfrak{S}_n$ and $\sigma\in \mathcal{P}$, we have 
$k_{(\tau,i),\sigma}=k_{i,\sigma}$ for any $1\leqq i\leqq n$.
\end{itemize}
\end{defn}

\begin{rem}
By definition, if $W$ is a benign, then we have $$W_{\tau,i}/W_{\tau,i-1}
\isom W(\delta_{\alpha_{\tau(i)}}\widetilde{\delta}_{\tau(i)}\prod_{\sigma\in\mathcal{P}}\sigma^{k_{i,\sigma}})$$ 
for any $\tau\in\mathfrak{S}_n$ and $1\leqq i\leqq n$.
\end{rem}
\begin{lemma}\label{21}
Let $W$ be a benign $E$-$B$-pair.
If $W_1$ is a saturated sub $E$-$B$-pair of $W$, then 
$W_1$ and $W/W_1$ are also benign $E$-$B$-pairs.
\end{lemma}
\begin{proof}
This follows from the definition and the fact that 
all the sub or quotient $E$-$B$-pairs of crystabelline $E$-$B$-pairs 
are crystabelline $E$-$B$-pairs.
\end{proof}

\subsubsection{Deformations of benign $B$-pairs}
\begin{lemma}\label{21.5}
Let $W$ be a potentially crystalline $E$-$B$-pair satisfying the condition 
$(1)$ of Definition $\mathrm{\ref{26}}$, then 
we have $\mathrm{H}^2(G_K,\mathrm{ad}(W))=0$ and
$(W,\mathcal{T}_{\tau})$ satisfies the conditions in Proposition $\ref{19}$ except the condition 
$\mathrm{End}_{G_K}(W)=E$ for any $\tau\in \mathfrak{S}_n$. 
\end{lemma}
\begin{proof}
That $\mathrm{H}^2(G_K,\mathrm{ad}(W))=0$ follows from the condition (1) of Definition \ref{26} and from (2) of Proposition \ref{i} 
because $\mathrm{ad}(W)$ is split trianguline whose graded components are of the forms $(W_{\tau,i}/W_{\tau,i-1})\otimes (W_{\tau,j}/W_{\tau,j-1})^{\vee}$ for any fixed 
$\tau\in \mathfrak{S}_n$.
 Other statements follow from the condition (1) of Definition \ref{26}. 
\end{proof}

\begin{lemma}\label{22}
Let $W$ be a benign $E$-$B$-pair of rank $n$, 
then $\mathrm{End}_{G_K}(W)=E$.
\end{lemma}
\begin{proof}
We prove this by induction on $n$, the rank of $W$. 
If  $n=1$, $\mathrm{End}_{G_K}(W)=\mathrm{H}^0(G_K, (\bold{B}_{e},\bold{B}^+_{\mathrm{dR}}))=E$.
We assume that the lemma is proved for $n-1$. Let 
$W$ be a benign $E$-$B$-pair of rank $n$. 
We take an element  $\tau\in\mathfrak{S}_n$ and consider the filtration 
$\mathcal{T}_{\tau}:0\subseteq W_{\tau,1}\subseteq\cdots \subseteq W_{\tau,n-1}\subseteq W_{\tau,n}=W$.
By Lemma \ref{21}, $W_{\tau,n-1}$ is benign of rank $n-1$, hence 
we have $\mathrm{End}_{G_K}(W_{\tau,n-1})=E$ by induction.  
Let $f:W\rightarrow W$ be a non-zero morphism of $E$-$B$-pairs. By (1) of Definition \ref{26} 
and by Proposition $\ref{i}$, we have $\mathrm{Hom}_{G_K}(W_{\tau,n-1}, W/W_{\tau,n-1})=0$. Hence we have 
 $f(W_{\tau,n-1})
\subseteq W_{\tau,n-1}$. Because we have $\mathrm{End}_{G_K}(W_{\tau,n-1})=E$, 
then we have $f|_{W_{\tau,n-1}}=a\cdot\mathrm{id}_{W_{\tau,n-1}}$ 
for some $a\in E$. If $a=0$, then $f:W\rightarrow W$ factors through a non-zero morphism $f':W/W_{\tau,n-1}\rightarrow W$. 
Because $\mathrm{Hom}_{G_K}(W/W_{\tau,n-1}, W_{\tau,n-1})=0$ by (1) of Definition \ref{26} and by Proposition $\ref{i}$, 
the natural map $\mathrm{Hom}_{G_K}(W/W_{\tau,n-1},W)\hookrightarrow \mathrm{Hom}_{G_K}(W/W_{\tau,n-1}, W/W_{\tau,n-1})=E$
is injective, hence the composition of $f'$ with the natural projection $W\rightarrow W/W_{\tau,n-1}$ induces  an isomorphism 
$W/W_{\tau,n-1}\isom W/W_{\tau,n-1}$. This implies that the short exact sequence 
$0\rightarrow W_{\tau,n-1}\rightarrow W\rightarrow W/W_{\tau,n-1}\rightarrow 0$ splits. 
If we take a section $s:W/W_{\tau,n-1}\hookrightarrow W$, then we can choose $\tau'\in\mathfrak{S}_n$ 
such that $W_{\tau',1}=s(W/W_{\tau,n-1})$, then
this $\tau'$ doesn't satisfy the condition (3) in the definition of benign $B$-pairs. It's contradiction. 
Hence the above $a$ must not be zero. If $a\not=0$, consider the map $f-a\cdot\mathrm{id}_W\in \mathrm{End}_{G_K}(W)$, 
then the same argument as above implies that $f=a\cdot\mathrm{id}_W$. Hence we obtain the equality $\mathrm{End}_{G_K}(W)=E$.
\end{proof}
\begin{corollary}\label{23}
Let $W$ be a benign $E$-$B$-pair of rank $n$. 
The functor $D_{W}$ is pro-representable by $R_W$ which is 
formally smooth of its dimension $n^2[K:\mathbb{Q}_p]+1$. 
For any $\tau\in\mathfrak{S}_n$, the functor $D_{W,\mathcal{T}_{\tau}}$ is pro-representable 
by a quotient $R_{W,\mathcal{T}_{\tau}}$ of $R_W$ which is formally smooth 
of its dimension $\frac{n(n+1)}{2}[K:\mathbb{Q}_p]+1$.
\end{corollary}
\begin{proof}
This follows from Proposition $\ref{19}$.
\end{proof}

Next, we want to consider the relation between $R_{W}$ and $R_{W,\mathcal{T}_{\tau}}$ for 
all $\tau\in \mathfrak{S}_n$. 
In particular, we want to compare the tangent space of $R_{W}$ and the sum of tangent spaces 
of $R_{W,\mathcal{T}_{\tau}}$ for all $\tau\in \mathfrak{S}_n$. 
For this, we first need to recall  the potentially crystalline  deformation functor.
\begin{defn}
Let $W$ be a potentially crystalline $E$-$B$-pair. 
We define the potentially crystalline deformation functor $D_{W}^{\mathrm{cris}}$ which is a subfunctor of $D_W$ 
defined by $$D_{W}^{\mathrm{cris}}(A):=\{[W_A]\in D_{W}(A)| W_A \text{ is potentially crystalline }\}$$
for $A\in \mathcal{C}_E$.
\end{defn}
\begin{lemma}\label{24}
Let $W$ be a potentially crystalline $E$-$B$-pair. 
If $\mathrm{End}_{G_K}(W)=E$, then $D_{W}^{\mathrm{cris}}$ is pro-representable by 
a quotient $R_{W}^{\mathrm{cris}}$ of $R_W$ which is formally smooth of its dimension 
equal to $$\mathrm{dim}_E(\bold{D}_{\mathrm{dR}}(\mathrm{ad}(W))/\mathrm{Fil}^0\bold{D}_{\mathrm{dR}}(\mathrm{ad}(W))) +
\mathrm{dim}_E(\mathrm{H}^0(G_K, \mathrm{ad}(W))). $$
\end{lemma}
\begin{proof}
For the pro-representability, by Proposition $\ref{12}$, it suffices to relatively representability of 
$D_{W}^{\mathrm{cris}}\hookrightarrow D_W$ as in the proof of Proposition $\ref{15}$. 
In this case, the conditions (1) and (2) are trivial and (3) follows from the fact 
that any sub $E$-$B$-pair of a potentially crystalline $E$-$B$-pair is again potentially crystalline. 
The formal smoothness follows from Proposition 3.1.2 and Lemma 3.2.1 of \cite{Ki08} by using the deformations 
of filtered ($\varphi, G_K$)-modules. For the dimension, we take a finite Galois extension $L$ of $K$ such that 
$W|_{G_L}$ is crystalline. By the same argument as in the proof of Lemma $\ref{20.5}$, 
 any $W_A\in D^{\mathrm{cris}}_W(A)$ is crystalline when restricted to $G_L$. It's easy to check 
that the map $D_{W}(E[\varepsilon])\isom \mathrm{H}^1(G_K, \mathrm{ad}(W))$ 
induces an isomorphism 
$D_{W}^{\mathrm{cris}}(E[\varepsilon])\isom \mathrm{Ker}(\mathrm{H}^1(G_K, 
\mathrm{ad}(W))\rightarrow \mathrm{H}^1(G_L, \bold{B}_{\mathrm{max}}\otimes_{\bold{B}_{e}} \mathrm{ad}(W)_e))$. In the same way as in 
the proof of Lemma $\ref{20.5}$, the natural map $\mathrm{H}^1(G_K, \bold{B}_{\mathrm{max}}\otimes_{\bold{B}_{e}}\mathrm{ad}(W)_e)\rightarrow \mathrm{H}^1(G_L, \bold{B}_{\mathrm{max}}\otimes_{\bold{B}_{e}}\mathrm{ad}(W)_e)$ is an 
injection. Hence we obtain an isomorphism $$D^{\mathrm{cris}}_W(E[\varepsilon])\isom \mathrm{Ker}(\mathrm{H}^1(G_K, \mathrm{ad}(W))
\rightarrow \mathrm{H}^1(G_K, \bold{B}_{\mathrm{max}}\otimes_{\bold{B}_{e}}\mathrm{ad}(W)_e)).$$ We can calculate the dimension of this 
group in the same way as in the proof of  Proposition 2.7 \cite{Na09}.
\end{proof}
\begin{corollary}\label{25}
Let $W$ be a benign $E$-$B$-pair of rank $n$,
then $R_{W}^{\mathrm{cris}}$ is formally smooth of its dimension 
$\frac{(n-1)n}{2}[K:\mathbb{Q}_p]+1$.
\end{corollary}
\begin{proof}
This follows from Lemma $\ref{22}$ and the equality 
$$\mathrm{dim}_E(\bold{D}_{\mathrm{dR}}(\mathrm{ad}(W))/\mathrm{Fil}^0\bold{D}_{\mathrm{dR}}(\mathrm{ad}(W)))
=\frac{(n-1)n}{2}[K:\mathbb{Q}_p],$$ which follows from the condition (2) in Definition $\ref{26}$.
\end{proof}

\begin{defn}\label{34}
Let $W$ be a benign $E$-$B$-pair of rank $n$ such that $W|_{G_{K_m}}$ is crystalline. 
For any $\tau\in\mathfrak{S}_n$, we define a rank one saturated 
crystabelline $E$-$B$-pair $W'_{\tau}\subseteq W$ such that $\bold{D}^{K_m}_{\mathrm{cris}}(W'_{\tau})=K_0\otimes_{\mathbb{Q}_p}E
e_{\tau(n)} \subseteq \bold{D}^{K_m}_{\mathrm{cris}}(W)$ and define a subfunctor 
$D^{\mathrm{cris}}_{W,\tau}$ of $D_W$ by 
\begin{multline*}
D^{\mathrm{cris}}_{W,\tau}(A):=\{[W_A]\in D_W(A)|\text{ there exists a rank one crystabelline 
saturated}\\
\text{ sub }A\text{-}B\text{-pair } W'_A\subseteq W_A \text{ such that } W'_A\otimes_A E=W'_{\tau} \}.
\end{multline*}
\end{defn}

\begin{lemma}\label{27}
Under the above condition. The functor $D^{\mathrm{cris}}_{W,\tau}$ is pro-representable by 
a quotient $R^{\mathrm{cris}}_{W,\tau}$ of $R_W$ which is formally smooth and of its dimension 
$n(n-1)[K:\mathbb{Q}_p] +1$.
\end{lemma}
\begin{proof}
The relatively representability of $D^{\mathrm{cris}}_{W,\tau}\hookrightarrow D_W$
 and the formal smoothness easily follows from the combination of proofs  
 of Proposition $\ref{15}$ and Proposition $\ref{17}$ and Lemma $\ref{24}$.
 Here, we only prove the dimension formula. 
 Let $\mathrm{ad}_{\tau}(W):=\{f\in\mathrm{ad}(W)| f(W'_{\tau})\subseteq 
 W'_{\tau}\}$, then we have the following short exact sequence
 \begin{equation*}
 0\rightarrow \mathrm{Hom}(W/W'_{\tau}, W)\rightarrow \mathrm{ad}_{\tau}(W)
 \rightarrow \mathrm{ad}(W'_{\tau})\rightarrow 0.
 \end{equation*}
 Taking the long exact sequence, we obtain
 the following short exact sequence
 \begin{equation*}
 0\rightarrow \mathrm{H}^1(G_K, \mathrm{Hom}(W/W'_{\tau},W))\rightarrow 
 \mathrm{H}^1(G_K, \mathrm{ad}_{\tau}(W))\rightarrow \mathrm{H}^1(G_K, 
 \mathrm{ad}(W'_{\tau}))\rightarrow 0
 \end{equation*}
 by Proposition $\ref{i}$.
 We define a subspace $\mathrm{H}^1_{f,\tau}(G_K, \mathrm{ad}_{\tau}(W))$ of 
 $\mathrm{H}^1(G_K,\allowbreak \mathrm{ad}_{\tau}(W))$
 as the inverse image of $\mathrm{H}^1_f(G_K,\mathrm{ad}(W'_{\tau}))(\subseteq
 \mathrm{H}^1(G_K,\mathrm{ad}(W'_{\tau})))$. 
 Hence we obtain
 a short exact sequence
 \begin{equation*}
 0\rightarrow \mathrm{H}^1(G_K, \mathrm{Hom}(W/W'_{\tau},W))
 \rightarrow \mathrm{H}^1_{f,\tau}(G_K, \mathrm{ad}_{\tau}(W))
 \rightarrow \mathrm{H}^1_f(G_K, \mathrm{ad}(W'_{\tau}))\rightarrow 0.
 \end{equation*}
 In the same way as in Lemma $\ref{24}$, we can show that the natural isomorphism
  $D_{W}(E[\varepsilon])\isom\mathrm{H}^1(G_K, \mathrm{ad}(W))$ 
 induces an isomorphism $$D^{\mathrm{cris}}_{W,\tau}(E[\varepsilon])\isom \mathrm{H}^1_{f,\tau}(G_K, 
 \allowbreak \mathrm{ad}_{\tau}(W)).$$ 
 By Theorem $\ref{h}$ and Proposition $\ref{i}$, we obtain the equality
 $$\mathrm{dim}_E\mathrm{H}^1(G_K, \mathrm{Hom}\allowbreak (W/W'_{\tau},W))=n(n-1)[K:\mathbb{Q}_p].$$ 
 Because $\mathrm{ad}(W'_{\tau})$ is the trivial $E$-$B$-pair, we have $\mathrm{dim}_E\mathrm{H}^1_{f}(G_K, \mathrm{ad}(W'_{\tau}))
 \allowbreak=1$ by Proposition 2.7 of \cite{Na09}. Hence we obtain the equality
 $$\mathrm{dim}_E\mathrm{H}^1_{f,\tau}(G_K, \mathrm{ad}_{\tau}(W))\allowbreak =n(n-1)[K:\mathbb{Q}_p] +1.$$
 This proves that the dimension of $R_{W,\tau}^{\mathrm{cris}}$ is $n(n-1)[K:\mathbb{Q}_p]+1$.

\end{proof}
\begin{lemma}\label{28}
Let $W$ be a benign $E$-$B$-pair of rank $n$. 
Let $W_A$ be a deformation of $W$ over $A$ which is potentially crystalline, 
then $[W_A]\in D_{W,\mathcal{T}_{\tau}}(A)$ and $[W_A]\in D^{\mathrm{cris}}_{W,\tau}(A)$ for 
any $\tau\in \mathfrak{S}_n$.
\end{lemma}
\begin{proof}
Let $W_A$ be as above. If $W|_{G_{K_m}}$ is crystalline, 
then $W_A|_{G_{K_m}}$ is crystalline by the proof of Lemma $\ref{24}$. 
Hence it suffices to show that $\bold{D}^{K_m}_{\mathrm{cris}}(W_A)$ is of the form $\bold{D}^{K_m}_{\mathrm{cris}}(W_A)=
K_0\otimes_{\mathbb{Q}_p}Ae_1\oplus K_0\otimes_{\mathbb{Q}_p}Ae_2\oplus\cdots 
\oplus K_0\otimes_{\mathbb{Q}_p}Ae_n$ such that $K_0\otimes_{\mathbb{Q}_p}Ae_i$ is 
preserved by ($\varphi,G_{m}$) and $\varphi^f(e_i)=\widetilde{\alpha}_ie_i$ for a lift $\widetilde{\alpha}_i\in A^{\times}$ 
of $\alpha_i\in E^{\times}$ for any $1\leqq i\leqq n$.  To prove this claim, we first note that 
$\bold{D}^{K_m}_{\mathrm{cris}}(W_A)$ is a free $K_0\otimes_{\mathbb{Q}_p}A$-module of rank $n$ and 
$\bold{D}^{K_m}_{\mathrm{cris}}(W_A)\otimes_A E\isom \bold{D}^{K_m}_{\mathrm{cris}}(W)$ by 
Proposition 1.3.4 and Proposition 1.3.5 \cite{Ki09}. Then, for any $\sigma:K_0\hookrightarrow E$, $\bold{D}^{K_m}_{\mathrm{cris}}(W_A)_{\sigma}$ is of the form $\bold{D}^{K_m}_{\mathrm{cris}}(W_A)_{\sigma}=
Ae_{1,\sigma}\oplus \cdots \oplus Ae_{n,\sigma}$ such that 
$\varphi^f(e_{i,\sigma})\equiv \alpha_ie_{i,\sigma}$ (mod $\mathfrak{m}_A$) for any $1\leqq i\leqq n$. 
By an easy linear algebra, we can take an $A$-basis $e'_{1,\sigma}, e'_{2,\sigma}, \cdots, e'_{n,\sigma}$ of 
$\bold{D}^{K_m}_{\mathrm{cris}}(W_A)_{\sigma}$ such that $\varphi^f(e'_{i,\sigma})=\widetilde{\alpha}_ie'_{i,\sigma}$ 
for a lift $\widetilde{\alpha}_i\in A^{\times}$ of $\alpha_i$ for any $i$. Because the actions of $\varphi$ and $G_{m}$ commute and we have $\alpha_i\not=\alpha_j$, $Ae'_{i,\sigma}$ is stable by $G_{m}$. If we take $e_i:=e'_{i,\sigma}+\varphi(e'_{i,\sigma})+\cdots 
\varphi^{f-1}(e'_{i,\sigma})\in \bold{D}^{K_m}_{\mathrm{cris}}(W_A)$, then $\bold{D}^{K_m}_{\mathrm{cris}}(W_A)$ can be written by $\bold{D}^{K_m}_{\mathrm{cris}}(W_A)=
K_0\otimes_{\mathbb{Q}_p}Ae_1\oplus \cdots \oplus K_0\otimes_{\mathbb{Q}_p}Ae_n$ satisfying 
the property of the claim.

\end{proof}

\begin{lemma}\label{29}
Let $W$ be a benign $E$-$B$-pair of rank $n$ such that 
$W|_{G_{K_m}}$ is crystalline, and let $\tau\in \mathfrak{S}_n$.
Let $[W_A]\in D_{W,\mathcal{T}_{\tau}}(A)$ be a trianguline deformation over $A$ 
with a lifting $0\subseteq W_{1,A}\subseteq W_{2,A}\subseteq \cdots W_{n,A}=W_A$ of 
the triangulation $\mathcal{T}_{\tau}$. 
If $W_{i,A}/W_{i-1,A}$ is Hodge-Tate for any $1\leqq i\leqq n$, then $W_A|_{G_{K_m}}$ is crystalline.
\end{lemma}
\begin{proof}
First, we prove that $(W_{i,A}/W_{i-1,A})|_{G_{K_m}}$ is crystalline with the Hodge-Tate weights $\{k_{i,\sigma}\}_{\sigma\in \mathcal{P}}$. 
Because $W_{i,A}/W_{i-1,A}$ is written as a successive extension of $W_{\tau,i}/W_{\tau,i-1}$, $W_{i,A}/W_{i-1,A}$
has the Hodge-Tate weights $\{k_{i,\sigma}\}_{\sigma\in\mathcal{P}}$ with multiplicity. Twisting $W_A$ by the crystalline character $\delta_{\alpha_{\tau(i)}}^{-1}
\prod_{\sigma\in\mathcal{P}}\sigma^{-k_{i,\sigma}}:K^{\times}\rightarrow A^{\times}$, we may assume that 
$W_{i,A}/W_{i-1, A}$ is an \'etale Hodge-Tate $A$-$B$-pair of rank one with the Hodge-Tate weight zero and is 
a deformation of an \'etale potentially unramified $E$-$B$-pair $W(\widetilde{\delta}_{\tau(i)})$. By Sen's theorem 
(\cite{Se73} or Proposition 5.24 of \cite{Be02}),
$W_{i,A}/W_{i-1,A}$ is potentially unramified, hence there exists a unitary homomorphism $\delta:K^{\times}\rightarrow A^{\times}$ 
such that $\delta|_{\mathcal{O}_K^{\times}}$ has a finite image and $W_{i,A}/W_{i-1,A}\isom W(\delta)$ and 
$\delta$ is a lift of $\widetilde{\delta}_{\tau(i)}$. Because $(1+\mathfrak{m}_A)\cap A^{\times}_{\mathrm{torsion}}=\{1\}$, then 
we have $\delta|_{\mathcal{O}_K^{\times}}=\widetilde{\delta}_{\tau(i)}|_{\mathcal{O}_K^{\times}}:\mathcal{O}_K^{\times}\rightarrow A^{\times}$, hence $W_{i,A}/W_{i-1,A}|_{G_{K_m}}$ is crystalline.
Next, we prove that $W_A|_{G_{K_m}}$ is crystalline by induction on the rank of $W$. 
When $n=1$, we just have proved this. 
Assume that the lemma is proved for $n-1$, then 
$W_{n-1, A}|_{G_{K_m}}$ is crystalline. If we put $W_A/W_{n-1,A}\isom W(\delta_{A,n})$, then we have 
$[W_A]\in \mathrm{H}^1(G_K, W_{n-1,A}(\delta_{A,n}^{-1}))$. By considering the Hodge-Tate weights of 
$W_{A}$ and the condition (3) of Definition $\ref{26}$, we have $\mathrm{Fil}^0\bold{D}_{\mathrm{dR}}(W_{n-1,A}(\delta_{A,n}^{-1}))=0$. 
Comparing the dimensions, we obtain the equality $\mathrm{H}^1_f(G_K,W_{A,n-1}(\delta_{A,n}^{-1}))=\mathrm{H}^1(G_K, 
W_{A,n-1}(\delta_{A,n}^{-1}))$ by Proposition 2.7 of \cite{Na09}. In particular, $W_A|_{G_{K_m}}$ is crystalline.

\end{proof}

\begin{lemma}\label{30}
Let $W$ be a benign $E$-$B$-pair of rank $n$, and let $W_1$ be a rank one crystabelline sub $E$-$B$-pair of $W$. Then, the saturation $W_{1}^{sat}:=(W^{sat}_{1,e}, W^{+, sat}_{1,\mathrm{dR}})$ of $W_1$ in $W$ is 
crystabelline and the natural map $\mathrm{Hom}_{G_K}( W_{1}^{sat}, W)
\rightarrow \mathrm{Hom}_{G_K}(W_1,W)$ induced by the natural 
inclusion $W_1\hookrightarrow W_{1}^{sat}$ is isomorphism 
between one dimensional $E$-vector spaces.
\end{lemma}
\begin{proof}
Because we have $W_{1,e}=W^{sat}_{1,e}$ by Lemma 1.14 of \cite{Na09}, so $W_{1}^{sat}$ is crystabelline.
By the definition of benign $E$-$B$-pairs, the Hodge-Tate weights of $W_{1}^{sat}$ 
are $\{k_{1,\sigma}\}_{\sigma\in\mathcal{P}}$. Consider the following short exact sequence 
of complexes of $G_K$-modules defined in p.890 of \cite{Na09}
\begin{equation*}
0\rightarrow \mathrm{C}^{\bullet}(W\otimes (W_{1}^{sat})^{\vee})\rightarrow \mathrm{C}^{\bullet}(W
\otimes W_{1}^{\vee})\rightarrow  
((W\otimes W_{1}^{\vee})^+_{\mathrm{dR}}/
(W\otimes(W_{1}^{sat})^{\vee})^+_{\mathrm{dR}}) [0]\rightarrow 0,
\end{equation*}
where, for an $E$-$B$-pair $W$, we denote by $C^{\bullet}(W)$ the complex 
$$ \mathrm{C}^{0}(W):=W_e\oplus W^+_{\mathrm{dR}}\xrightarrow{(x,y)\mapsto x-y} W_{\mathrm{dR}}=:\mathrm{C}^{1}(W).$$
From this, we obtain an exact sequence
\begin{multline*}
0\rightarrow \mathrm{H}^0(G_K, W\otimes (W_1^{sat})^{\vee})\rightarrow 
\mathrm{H}^0(G_K, W\otimes W_1^{\vee}) \\
 \rightarrow \mathrm{H}^0(G_K, (W\otimes W_1^{\vee})^+_{\mathrm{dR}}/
(W\otimes(W_1^{sat})^{\vee})^+_{\mathrm{dR}})\rightarrow \cdots
\end{multline*}
By the condition (3) in Definition $\ref{26}$, we have 
$$\mathrm{dim}_E\mathrm{H}^0(G_K, (W\otimes (W_1^{sat})^{\vee})^+_{\mathrm{dR}})
=\mathrm{dim}_E\mathrm{H}^0(G_K, (W\otimes W_1^{\vee})^+_{\mathrm{dR}})=n[K:\mathbb{Q}_p].$$ Hence, by a standard argument 
of the theory of $\bold{B}^+_{\mathrm{dR}}$-representations, we obtain the equality $\mathrm{H}^0(G_K, (W\otimes W_1^{\vee})^+_{\mathrm{dR}}/(W\otimes (W_1^{sat})^{\vee})^+_{\mathrm{dR}})=0$. Hence the map $\mathrm{H}^0(G_K, W\otimes (W_1^{sat})^{\vee}))\isom \mathrm{H}^0(G_K, W\otimes (W_1)^{\vee})
$ is isomorphism. Finally, for the dimension, 
we have $\mathrm{dim}_E\mathrm{H}^0(G_K, W\otimes (W_1^{sat})^{\vee})=1$ by the condition (1) in Definition $\ref{26}$ and by Proposition $\ref{i}$ of \cite{Na09}.

\end{proof}
\begin{lemma}\label{31}
Let $W$ be a benign $E$-$B$-pair, and let $W_A$ be a deformation 
of $W$ over $A$. If there exists a crystabelline sub $A$-$B$-pair 
$W_{1,A}\subseteq W_A$  of rank one such 
that the base change $W_1:=W_{1,A}\otimes_A E\hookrightarrow W_A\otimes_{A}E$ remains to be injective, 
then the saturation $W_{1,A}^{sat}$ of $W_{1,A}$ in $W_A$ as an $E$-$B$-pair is 
a crystabelline $A$-$B$-pair and $W_A/W^{sat}_{1,A}$ is an $A$-$B$-pair and 
$W_{1,A}^{sat}\otimes_A E\isom W_1^{sat} (\subseteq W)$.
\end{lemma}
\begin{proof}
First, by Proposition 2.14 of \cite{Na09}, there exists $\{l_{\sigma}\}_{\sigma\in\mathcal{P}}
\in \prod_{\sigma\in\mathcal{P}}\mathbb{Z}_{\leqq 0}$ such that  $W_1^{sat}\isom W_1\otimes W(\prod_{\sigma\in\mathcal{P}}
\sigma^{l_{\sigma}})$. We claim that the inclusion 
$$\mathrm{Hom}_{G_K}(W_{1,A}\otimes W_A(\prod_{\sigma\in\mathcal{P}}\sigma^{l_{\sigma}}), W_A)
\rightarrow \mathrm{Hom}_{G_K}(W_{1,A}, W_A)$$ induced by the natural inclusion $W_{1,A}\hookrightarrow W_{1,A}
\otimes W_A(\prod_{\sigma\in\mathcal{P}}\sigma^{l_{\sigma}})$ is isomorphism and 
that these groups are rank one free $A$-modules. By the same argument as in Lemma \ref{30}, 
the cokernel of $\mathrm{Hom}_{G_K}(W_{1,A}\otimes W_A(\prod_{\sigma\in\mathcal{P}}\sigma^{l_{\sigma}}), W_A)
\rightarrow \mathrm{Hom}_{G_K}(W_{1,A}, W_A)$ is contained in $\mathrm{H}^0(G_K, (W_A\otimes W_{1,A}^{\vee})^+_{\mathrm{dR}}/
(W_A\otimes W_{1,A}^{\vee}\otimes W_A(\prod_{\sigma\in \mathcal{P}}\sigma^{-l_{\sigma}}))^+_{\mathrm{dR}})$, which is zero 
by the proof of Lemma $\ref{30}$. Hence the natural inclusion 
$$\mathrm{Hom}_{G_K}(W_{1,A}\otimes W_A(\prod_{\sigma\in\mathcal{P}}\sigma^{l_{\sigma}}), W_A)
\isom \mathrm{Hom}_{G_K}(W_{1,A}, W_A)$$ is isomorphism. Next, we prove
that $\mathrm{Hom}_{G_K}(W_{1,A},\allowbreak W_A)$ is a free $A$-module of rank one by 
induction on the length of $A$.
When $A=E$, this claim is proved in Lemma $\ref{30}$. 
Assume that $A$ is of length $n$ and assume that the claim is 
proved for $W_{A'}:=W\otimes_A A'$ for a small extension $A\rightarrow A'$. We denote by $I$ the kernel of $A\rightarrow A'$, $W_{1,A'}:=W_{1,A}\otimes_A A'$. From the exact sequence
\begin{equation*}
0\rightarrow I\otimes_E \mathrm{Hom}_{G_K}(W_1, W)
\rightarrow \mathrm{Hom}_{G_K}(W_{1,A}, W_A)\rightarrow \mathrm{Hom}_{G_K}(W_{1,A'}, W_{A'})
\end{equation*}
 and the induction hypothesis, we have 
$\mathrm{length}\mathrm{Hom}_{G_K}(W_{1,A}, W_A)\leqq 
\mathrm{length}A.$ On the other hand, the fact that the given inclusion $\iota:W_{1,A}\hookrightarrow W_A$ remains to be
injective after tensoring $E$ and the fact that $\mathrm{dim}_E\mathrm{Hom}_{G_K}(W_1, W)=1$ and the induction hypothesis 
imply that the map $A\rightarrow \mathrm{Hom}_{G_K}(W_{1,A}, W_A): a\mapsto a\cdot\iota $ is an injection. Hence we obtain the equality
$\mathrm{length}(A)= \mathrm{length} \mathrm{Hom}_{G_K}(W_{1,A}, W_A)$. These facts prove the claim for $A$.
From this claim, the given inclusion $\iota:W_{1,A}\hookrightarrow W_A$  factors through a map 
$$\widetilde{\iota}:W_{1,A}':=W_{1, A}\otimes W_A(\prod_{\sigma}\sigma^{l_{\sigma}})
\rightarrow W_A.$$ Because the injectiveness of morphisms of $B$-pairs is determined only by the $W_e$-part of $B$-pairs and we have $W_{1,A,e}=(W_{1,A}\otimes W_A(\prod_{\sigma\in\mathcal{P}}\sigma^{l_{\sigma}}))_e$, the map $\widetilde{\iota}$ is also an injection.
Under this situation, we claim that the map $\widetilde{\iota}$ gives an isomorphism $W_{1,A}^{sat}\isom W_{1,A}'$ and satisfies all the properties in this lemma. By induction on the length of $A$, 
we may assume that this claim is proved for $A'$.  First, we prove that $W_{A}/W_{1,A}'$ is an $E$-$B$-pair. To prove this, by Lemma 2.1.4 of \cite{Be08}, 
it suffices to show that $W_{A,\mathrm{dR}}^{+}/W_{1,A,\mathrm{dR}}^{'+}$ is a free $\bold{B}_{\mathrm{dR}}^+$-module. By the snake lemma, we have the following short exact sequence
\begin{equation*}
0\rightarrow I\otimes_E W^+_{\mathrm{dR}}/W^{sat +}_{1,\mathrm{dR}}\rightarrow 
W^+_{A,\mathrm{dR}}/W^{' +}_{1, A,\mathrm{dR}}\rightarrow W^+_{A',\mathrm{dR}}/(W_{1,A,\mathrm{dR}}^{' +}\otimes_A A')\rightarrow 0.
\end{equation*}
From this and the induction hypothesis, $W^+_{A,\mathrm{dR}}/W^{' +}_{1,A,\mathrm{dR}}$ is a free $\bold{B}^+_{\mathrm{dR}}$-module.
Finally, we prove the $A$-flatness of $W_{A}/W_{1,A}^{'}$. This follows from the fact that  the map 
$\widetilde{\iota}\otimes \mathrm{id}_E:W_{1,A}^{'}\otimes_A E\hookrightarrow W_{A}\otimes_A E$ is saturated, which can be seen 
from the proof of the above first claim. Hence $W_A/W_{1,A}^{'}$ is an $A$-$B$-pair. We finish the proof of the lemma.

\end{proof}

\begin{lemma}\label{32}
Let $W$ be a benign $E$-$B$-pair of rank $n$.
For any $\tau\in\mathfrak{S}_n$, we have $$D_{W,\mathcal{T}_{\tau}}(E[\epsilon])
+D^{\mathrm{cris}}_{W,\tau}(E[\varepsilon])=D_{W}(E[\varepsilon]).$$
\end{lemma}
\begin{proof}
By Corollary $\ref{23}$ and  Lemma $\ref{24}$ and Lemma $\ref{27}$, we obtain an equality
$$\mathrm{dim}_ED_{W}(E[\varepsilon])\allowbreak +\mathrm{dim}_ED_W^{\mathrm{cris}}(E[\varepsilon])
=\mathrm{dim}_E D_{W,\mathcal{T}_{\tau}}(E[\varepsilon])+\mathrm{dim}_ED^{\mathrm{cris}}_{W,\tau}(E[\varepsilon]).$$ 
Hence, it suffices to show that we have an equality
$$D_{W,\mathcal{T}_{\tau}}(A)\cap D^{\mathrm{cris}}_{W,\tau}(A)=D_{W}^{\mathrm{cris}}(A)$$ for 
any $A\in \mathcal{C}_E$. We first note that the inclusion $D^{\mathrm{cris}}_W(A)\subseteq D_{W,\mathcal{T}_{\tau}}(A)\cap D_{W,\tau}(A)$ 
follows from Lemma $\ref{28}$. 

We prove the opposite inclusion $D_{W,\mathcal{T}_{\tau}}(A)\cap D^{\mathrm{cris}}_{W,\tau}(A)\subseteq D_{W}^{\mathrm{cris}}(A)$ by induction on
the rank $n$ of $W$. When  $n=1$, this is trivial.
Let $W$ be of rank $n$ and assume that the lemma is proved for $n-1$.
Take any $[W_A]\in D_{W,\mathcal{T}_{\tau}}(A)\cap D^{\mathrm{cris}}_{W,\tau}(A)$. By the definitions of $D_{W,\mathcal{T}_{\tau}}$ and $D^{\mathrm{cris}}_{W,\tau}$, then there exist an $A$-triangulation 
$0\subseteq W_{1,A}\subseteq W_{2,A}\subseteq \cdots \subseteq W_{n-1,A}\subseteq W_{n,A}=W_A$ 
such that $W_{i,A}\otimes_A E\isom W_{\tau,i}$ for any $i$ and a saturated crystabelline rank one $A$-$B$-pair 
$W'_{1,A}\subseteq W_{A}$ such that $W'_{1,A}\otimes_A E\isom W'_{\tau}$. 
We first claim that the composition of $W'_{1,A}\hookrightarrow W_A$ with $W_A\rightarrow W_A/W_{1,A}$ 
is  an injection. Because $\mathrm{Ker}(W'_{1,A}\rightarrow W_A/W_{1,A})$ is a sub $E$-$B$-pair of 
$W_{1, A}$ and we have 
$\mathrm{Hom}_{G_K}(\mathrm{Ker}(W'_{1, A}\rightarrow W_A/W_{1,A}), W_{1,A})=0$ by the condition (1) of  Definition $\ref{26}$ and Proposition 2.14 of \cite{Na09}, hence the composition map $W'_{1,A}\rightarrow 
W_A/W_{1,A}$ is an injection. Hence, the saturation $(W'_{1,A})^{sat}$ of $W'_{1,A}$ in $W_A/W_{1,A}$ 
is a rank one crystabelline $A$-$B$-pair satisfying the similar conditions as those of $W'_{1,A}\hookrightarrow W_A$ by Lemma $\ref{31}$. 
Hence, $W_A/W_{1,A}$ is crystabelline by induction, and the Hodge-Tate weights of 
$W_A/W_{1,A}$ are $\{k_{2,\sigma},k_{3,\sigma},\cdots,k_{n,\sigma}\}_{\sigma\in\mathcal{P}}$ (with multiplicity $[A:E]$) by the condition (3) of Definition $\ref{26}$. 
Moreover, $W'_{1,A}$ has the Hodge-Tate weights $\{k_{1,\sigma}\}_{\sigma\in\mathcal{P}}$ 
(with multiplicity)  by (3) of Definition $\ref{26}$. Because $k_{1,\sigma}\not= k_{i,\sigma}$ for any $i\not=1$ and there is no extension between the objects with different 
Hodge-Tate weights by a theorem of Tate, so $W_A$ is a Hodge-Tate $E$-$B$-pair. Hence, 
$W_A$ is crystabelline by Lemma $\ref{29}$, and we have that $[W_A]\in D_{W}^{\mathrm{cris}}(A)$.

\end{proof}

\begin{defn}
For $R=R_W$ or $R_{W,\mathcal{T}_{\tau}}$, we denote by  
$$t(R):=\mathrm{Hom}_E(\mathfrak{m}_R/\mathfrak{m}^2_R,E)$$ the tangent 
space of $R$. We have a natural inclusion $t(R_{W,\mathcal{T}_{\tau}})\hookrightarrow 
t(R_W)$ for each $\tau\in \mathfrak{S}_n$.
\end{defn}
The following theorem is the main theorem of $\S$ 2, which is crucial for the applications to some 
problems on the Zariski density of crystalline points or modular points.
This theorem was first discovered by Chenevier (Theorem 3.19 of \cite{Ch09b}) for $K=\mathbb{Q}_p$.

\begin{thm}\label{33}
Let $W$ be a benign $E$-$B$-pair of rank $n$. 
We have an equality $$\sum_{\tau\in\mathfrak{S}_n}t(R_{W,\mathcal{T}_{\tau}})=t(R_W).$$
\end{thm}
\begin{proof}
We prove this theorem by induction on $n$. 
When $n=1$, then the theorem is trivial.
Assume that the theorem is true for the rank $n-1$. 
Let $W$ be a benign $E$-$B$-pair of  rank $n$.
We take an element $\tau\in\mathfrak{S}_n$. We define a subfunctor $D_{W,\tau}$ of $D_W$ by 
\begin{multline*}
D_{W,\tau}(A):=\{[W_A]\in D_{W}(A)| \text{ there exists a rank one sub } A\text{-}B\text{-pair }W_{1,A}\subseteq W_A\\
\text{ such 
that the quotient }W_A/W_{1,A}\text{ is an }A\text{-}B\text{-pair and }W_{1,A}\otimes_A E\isom W'_{\tau}\},
\end{multline*}
 where $W'_{\tau}$ is defined in Definition $\ref{34}$.
 Then, $D^{\mathrm{cris}}_{W,\tau}$ is a subfunctor of $D_{W,\tau}$,  
and we can show in the same way that $$D_{W,\tau}(E[\varepsilon])\isom \mathrm{H}^1(G_K, \mathrm{ad}_{\tau}(W)),$$
 where we define $$\mathrm{ad}_{\tau}(W)
:=\{f \in \mathrm{ad}(W)| f(W'_{\tau})\subseteq W'_{\tau} \}.$$ 
By Lemma $\ref{32}$, we obtain an equality 
$$\mathrm{H}^1(G_K, \mathrm{ad}_{\tau}(W))+
\mathrm{H}^1(G_K, \mathrm{ad}_{\mathcal{T}_{\tau}}(W))=\mathrm{H}^1(G_K, \mathrm{ad}(W)).$$ 
Because we have 
a natural inclusion $D_{W,\mathcal{T}_{\tau'}}\subseteq D_{W,\tau}$ for any $\tau'\in\mathfrak{S}_{\tau,n}:=\{\tau'\in\mathfrak{S}_n| \tau'(1)=\tau(n)\}$, we have a natural map $\mathrm{H}^1(G_K, \mathrm{ad}_{\mathcal{T}_{\tau'}}(W))
\rightarrow \mathrm{H}^1(G_K, \mathrm{ad}_{\tau}(W))$ for each such $\tau'$. Therefore,  it suffices to prove that 
the map $$\oplus_{\tau'\in\mathfrak{S}_{\tau,n}}\mathrm{H}^1(G_K, \mathrm{ad}_{\mathcal{T}_{\tau'}}(W))\rightarrow 
\mathrm{H}^1(G_K, \mathrm{ad}_{\tau}(W))$$ is a surjection. We prove this surjection as follows. 
By the definition, we have the following short exact sequences of $E$-$B$-pairs for each $\tau'\in\mathfrak{S}_{\tau,n}$,
\begin{equation}
0\rightarrow \mathrm{Hom}(W/W'_{\tau}, W)\rightarrow \mathrm{ad}_{\tau}(W)\rightarrow \mathrm{ad}(W'_{\tau})
\rightarrow 0,
\end{equation}
and
\begin{equation}
0\rightarrow \{f\in\mathrm{ad}_{\mathcal{T}_{\tau'}}(W)| f(W'_{\tau})=0\}\rightarrow \mathrm{ad}_{\mathcal{T}_{\tau'}}(W)\rightarrow 
\mathrm{ad}(W'_{\tau})\rightarrow 0.
\end{equation}
Moreover, we have
\begin{equation}
0\rightarrow \mathrm{Hom}(W/W'_{\tau}, W'_{\tau})\rightarrow \mathrm{Hom}(W/W'_{\tau}, W)\rightarrow \mathrm{ad}(W/W'_{\tau})\rightarrow 0,
\end{equation}
and
\begin{equation}
0\rightarrow \mathrm{Hom}(W/W'_{\tau}, W'_{\tau})\rightarrow \{f\in\mathrm{ad}_{\mathcal{T}_{\tau'}}(W)|f(W'_{\tau})=0\}
\rightarrow \mathrm{ad}_{\mathcal{T}_{\bar{\tau}'}}(W/W'_{\tau})\rightarrow 0,
\end{equation}
where, for $\tau'\in\mathfrak{S}_{\tau,n}$, we denote by 
$$\mathcal{T}_{\bar{\tau}'}:0\subseteq W_{\tau',2}/W'_{\tau}\subseteq W_{\tau',3}/W'_{\tau}
\subseteq \cdots \subseteq W_{\tau',n-1}/W'_{\tau}\subseteq W/W'_{\tau}$$ the triangulation of 
$W/W'_{\tau}$ induced from $\mathcal{T}_{\tau'}$. We have $\mathrm{H}^2(G_K, \mathrm{Hom}(W/W'_{\tau}, W'_{\tau}))=0$ by the condition $(1)$ of Definition $\ref{26}$ and 
Proposition $\ref{i}$, and $\mathrm{H}^2(G_K, \mathrm{ad}_{\mathcal{T}_{\tau}}(W))\allowbreak=0$ by the proof of Proposition 
$\ref{19}$. Hence, from the short exact sequence (4) above, we obtain the equality $\mathrm{H}^2(G_K, \{f\in\mathrm{ad}_{\tau'}(W)|f(W'_{\tau})=0\})=0$.
From this and from (1) and (2) above, it suffices to show that 
the map $$\oplus_{\tau'\in \mathfrak{S}_{\tau,n}}\mathrm{H}^1(G_K, \{f\in\mathrm{ad}_{\tau'}(W)|f(W'_{\tau})=0\})
\rightarrow \mathrm{H}^1(G_K, \mathrm{Hom}(W/W'_{\tau}, W))$$ is a surjection. By (3) and (4) above and the fact that $\mathrm{H}^2(G_K, \mathrm{Hom}(W/W'_{\tau}, W'_{\tau}))=0$, 
this surjectivity follows from the surjectivity of the map 
$$\oplus_{\tau'\in\mathfrak{S}_{\tau,n}}\mathrm{H}^1(G_K, \mathrm{ad}_{\mathcal{T}_{\bar{\tau}'}}(W/W'_{\tau}))\rightarrow \mathrm{H}^1(G_K, \mathrm{ad}(W/W'_{\tau})),$$ which  is the induction hypothesis. Hence we have finished the proof of the theorem

\end{proof}

\section{Construction of $p$-adic families of two dimensional trianguline representations.}
In this section, we generalize the Kisin's theory of the finite slope subspace 
for any $p$-adic field, and then construct $p$-adic families of two dimensional 
triangulline representations, which are crucial for the density of two dimensional 
crystalline representations.

\subsection{Almost $\mathbb{C}_p$-representations}
In the first subsection, we prove some propositions concerning 
Banach $G_K$-modules, which we need for the generalization of 
the Kisin's theory of the finite slope subspace.

We first recall some rings of Lubin-Tate's $p$-adic periods defined by Colmez \cite{Co02} and 
the definition of almost $\mathbb{C}_p$-representations defined by Fontaine \cite{Fo03}. 

Let $\pi_K$ be a fixed uniformizer of $K$. Let $P(X)\in \mathcal{O}_K[X]$ be a monic polynomial of degree $q:=p^f$ such 
that $P(X)\equiv \pi_K X$ (mod $X^2$) and $P(X)\equiv X^q$ (mod $\pi_K$). Let $\mathcal{F}_{\pi_K}$ be the Lubin-Tate formal group law over $\mathcal{O}_K$ on which the multiplication by $\pi_K$ is given by $[\pi_K]=P(X)$. 
Let $\chi_{\mathrm{LT}}:G_K\rightarrow \mathcal{O}_K^{\times}$ be the Lubin-Tate character associated to $\pi_K$.
Put  $\bold{A}_{\mathrm{inf},K}:=\bold{A}^+\otimes_{\mathcal{O}_{K_0}}\mathcal{O}_K$, which is equipped with the weak topology on which 
$G_K$ acts continuously. The ring $\bold{A}_{\mathrm{inf},K}$ also has a $\mathcal{O}_K$-linear continuous $\varphi_K:=\varphi^f$-action. 
Any element of $\bold{A}_{\mathrm{inf},K}$ can be written uniquely of the form $\sum_{k=0}^{\infty}[x_k]\pi_K^k$ ($x_k\in \widetilde{\bold{E}}^+$). 
Put $\bold{B}_{\mathrm{inf},K}:=\bold{A}_{\mathrm{inf},K}[p^{-1}]$. By Lemma 8.3 of  \cite{Co02}, for each $x\in \widetilde{\bold{E}}^+$, 
there exists a unique element $\{x\}\in \bold{A}_{\mathrm{inf},K}$ such that $\{x\}$ is a lift of $x$ and $\varphi_K(\{x\})=[\pi_K](\{x\})(=P(\{x\}))$. 
 We fix a set $\{\omega_n\}_{n\geqq 0}$ such that $\omega_1\in \mathfrak{m}_{\overline{K}}$ is a primitive $[\pi_K]$-torsion point of $\mathcal{F}_{\pi_K}(\overline{K})$ and 
 $[\pi_K](\omega_{n+1})=\omega_n$ for any $n\geqq 0$, then $(\bar{\omega}_n)_{n\geqq 0}$ defines an element in $\widetilde{\bold{E}}^+\isom \varprojlim_{n}\mathcal{O}_{\overline{K}}/\pi_K\mathcal{O}_{\overline{K}}$ where the projective limit is given by the $q$-th power Frobenius map. We define $\omega_K:=\{(\bar{\omega}_n)_{n \geqq 0}\}\in \bold{A}_{\mathrm{inf},K}$. By the definition of $\{-\}$ and the uniqueness of $\{-\}$, 
 the actions of $G_K$ and $\varphi_K$ on $\omega_K$ are given by $g(\omega_K)=[\chi_{\mathrm{LT}}(g)](\omega_K)$ for $g\in G_K$ (which converges for the weak topology)
 and $\varphi_K(\omega_K)=[\pi_K](\omega_K)$. Take a subset $\{\pi_n\}_{n\geqq 0}\subseteq \mathcal{O}_{\overline{K}}$ such that $\pi_0=\pi_K$ and 
 $\pi_{n+1}^{q}=\pi_n$ for any $n$, and put  $\widetilde{\pi}_K:=(\bar{\pi}_n)_{n\geqq 0}\in \widetilde{\bold{E}}^+$. Define $\bold{A}_{\mathrm{max},K}:=\widehat{\bold{A}_{\mathrm{inf},K}[\frac{[\widetilde{\pi}_K]}{\pi_K}]}$ 
 the $p$-adic 
 completion of $\bold{A}_{\mathrm{inf},K}[\frac{[\widetilde{\pi}_K]}{\pi_K}]$. 
 Define $\bold{B}^+_{\mathrm{max},K}=\bold{A}_{\mathrm{max},K}[p^{-1}]$, which is a $K$-Banach space with continuous actions of $G_K$ and $\varphi_K$.  By the definition, we have a canonical isomorphism 
 $K\otimes_{K_0}\bold{B}^+_{\mathrm{max},\mathbb{Q}_p}\isom \bold{B}^+_{\mathrm{max},K}$ (Remark 7.13 of \cite{Co02}). 
 By Lemma 8.8 and Proposition 8.9 of \cite{Co02}, there exists a power series $F_K(X)\in K[[X]]$ which is the Lubin-Tate's logarithm
 such that $F_K(X)\circ [a]=aF_K(X)$ for any $a\in \mathcal{O}_K$, and $t_K:=F_K(\omega)$ converges in $\bold{A}_{\mathrm{max},K}$ such 
 that $\varphi_K(t_K)=\pi_Kt_K$, $g(t_K)=\chi_{\mathrm{LT}}(g)t_K$ for  $g\in G_K$. We define $\bold{B}_{\mathrm{max},K}:=\bold{B}^+_{\mathrm{max},K}[t_K^{-1}]$. 
  We define $\bold{B}^+_{\mathrm{dR}}:=\varprojlim_{n}\bold{B}^+_{\mathrm{inf},K}/
( \mathrm{Ker}(\theta))^n$ which is equipped with the projective limit topology of $K$-Banach spaces $\{\bold{B}^+_{\mathrm{inf},K}/(\mathrm{Ker}(\theta))^n\}_{n\geqq1}$ whose 
$\mathcal{O}_K$-lattice is defined as the image of $\bold{A}_{\mathrm{inf},K}\rightarrow \bold{B}^+_{\mathrm{inf},K}/(\mathrm{Ker}(\theta))^n$. 
By Proposition 7.12 of \cite{Co02}, this $\bold{B}^+_{\mathrm{dR}}$ is canonically topologically isomorphic to the usual $\bold{B}^+_{\mathrm{dR}}$ defined in $\S2$.  We define $\bold{B}_{\mathrm{dR}}:=\bold{B}^+_{\mathrm{dR}}
[t^{-1}]=\bold{B}^+_{\mathrm{dR}}[t_K^{-1}]$.  

Using these preliminaries, we define a functor from the category of $\varphi_K$-modules to 
the category of almost $\mathbb{C}_p$-representations defined by Fontaine. We can see this 
construction as a very special case of a generalization of Berger's results \cite{Be09} to the
 case of Lubin-Tate period rings.

\begin{defn}
We say that $D$ is a $\varphi_K$-module over $K$ if $D$ is a finite dimensional 
$K$-vector space with a $K$-linear isomorphism $\varphi_K:D\isom D$.
\end{defn}
Let $D$ be a $\varphi_K$-module over $K$, we extend the action of $\varphi_K$ to $\widehat{K^{\mathrm{ur}}}\otimes_KD$ by 
$\varphi_K(a\otimes x):=\varphi_K(a)\otimes\varphi_K(x)$, where $\widehat{K^{\mathrm{ur}}}$ is the $p$-adic completion of 
the maximal unramified extension $K^{\mathrm{ur}}$ of $K$ and $\varphi_K\in \mathrm{Gal}(\widehat{K^{\mathrm{ur}}}/K)$ is the lift of 
$q$-th Frobenius in $\mathrm{Gal}(\overline{\mathbb{F}}/\mathbb{F}_q)$. The Dieudonn\'e-Manin theorem 
gives a decomposition 
$$\widehat{K^{\mathrm{ur}}}\otimes_{K}D=\oplus_{s\in\mathbb{Q}}D_{s},$$ where for any 
$s=\frac{a}{h}\in\mathbb{Q}$ such that 
$(a,h)\in \mathbb{Z}\times\mathbb{Z}_{\geqq 1}$ are co-prime, $D_{s}$ is zero or a finite direct sum of $D_{a,h}:=\widehat{K^{\mathrm{ur}}}e_1\oplus \widehat{K^{\mathrm{ur}}}e_2\oplus 
\cdots \oplus\widehat{K^{\mathrm{ur}}}e_h$ such that $\varphi_K(e_1)=e_2,\varphi_K(e_2)=e_3,\cdots, \varphi_K(e_{h-1})=e_h, \varphi_K(e_h)=\pi_K^ae_1$. We define the set of slopes of $D$ as the set of $s\in\mathbb{Q}$ such that $D_s\not=0$.
We define a $\widehat{K^{\mathrm{ur}}}$-semi-linear $G_K$-action on $\widehat{K^{\mathrm{ur}}}\otimes_KD$ by 
$g(a\otimes x):=g(a)\otimes x$ for $g\in G_K$, $a\in \widehat{K^{\mathrm{ur}}}$, $x\in D$, then $D_s$ is preserved by this $G_K$-action 
for any $s\in \mathbb{Q}$ because the actions of $G_K$ and $\varphi_K$ commute each other. For $s=\frac{a}{h}$, if we define 
$D'_{s}:=\{x\in D_s| \varphi^h_K(x)=\pi_K^a x \}$, then we have $D_s=\widehat{K^{\mathrm{ur}}}\otimes_{K^{\mathrm{ur}}_h}D'_{s}$ and 
$D'_{s}$ is preserved by $G_K$ and $\varphi_K$, where $K^{\mathrm{ur}}_h$ is the unramified extension of $K$ of degree $h$.


The notion of almost $\mathbb{C}_p$-representations was defined by Fontaine \cite{Fo03}.

\begin{defn}
Let $U$ be a $\mathbb{Q}_p$-Banach space equipped with a continuous 
$\mathbb{Q}_p$-linear $G_K$-action. We say that $U$ is an almost $\mathbb{C}_p$-representation 
if there exists $\mathbb{Q}_p$-representations $V_1$, $V_2$ and an integer $d\in \mathbb{Z}_{\geqq 0}$ such that 
$V_1$ is a sub $G_K$-module of $U$ and $V_2$ is a sub $G_K$-module of $\mathbb{C}_p^d$ and there exists an 
isomorphism $U/V_1\isom \mathbb{C}_p^d/V_2$ as $\mathbb{Q}_p$-Banach $G_K$-modules. 
\end{defn}
\begin{rem}\label{3.3}
By Theorem C of \cite{Fo03}, $\bold{B}^+_{\mathrm{dR}}/t^k\bold{B}^+_{\mathrm{dR}}$ is an almost $\mathbb{C}_p$-representation for any $k\geqq 0$. By Theorem B of \cite{Fo03}, for any continuous 
$\mathbb{Q}_p$-linear $G_K$-morphism  $f:U_1\rightarrow U_2$ between 
almost $\mathbb{C}_p$-representations $U_1$, $U_2$, it is known that $\mathrm{Ker}(f)$ and $\mathrm{Coker}(f)$ (as $\mathbb{Q}_p[G_K]$-modules) 
are almost $\mathbb{C}_p$-representations and $\mathrm{Im}(f)$ is also an almost $\mathbb{C}_p$-representation 
which is a closed subspace of $U_2$. 
\end{rem}

Let $D$ be a $\varphi_K$-module over $K$. We 
prove that $X_0(D):=(\bold{B}^+_{\mathrm{max},K}\otimes_KD)^{\varphi_K=1}$ 
is an almost $\mathbb{C}_p$-representation.

\begin{lemma}\label{5.4}
Let $D$ be a $\varphi_K$-module over $K$.
\begin{itemize}
\item[(1)] $X_0(D)$ is an almost $\mathbb{C}_p$-representation.
\item[(2)]If any slope $s$ of $D$ satisfies $s>0$, then $X_0(D)=0$.
\end{itemize}
\end{lemma}
\begin{proof}
The proof is similar to that of Proposition 2.2 of \cite{Be09}. 
If we denote by $\widehat{K^{\mathrm{ur}}}\otimes_{K}D=\oplus_{s\in\mathbb{Q}}D_s$ as above, then we have 
$\bold{B}^+_{\mathrm{max},K}\otimes_KD=\oplus_{s\in \mathbb{Q}}\bold{B}^+_{\mathrm{max},K}\otimes_{\widehat{K^{\mathrm{ur}}}}D_{s}$ 
as a  $\varphi_K$-module. For $s=\frac{a}{h}$,  $\bold{B}^+_{\mathrm{max},K}\otimes_{\widehat{K^{\mathrm{ur}}}}D_{s}=\bold{B}^+_{\mathrm{max},K}\otimes_{K^{\mathrm{ur}}_h}D'_{s}$ is preserved by the actions of $G_K$ and 
$\varphi_K$. Hence, it suffices to show that, for any $s=\frac{a}{h}$, $(\bold{B}^+_{\mathrm{max},K}\otimes_{K_h^{\mathrm{ur}}}D'_{s})^{\varphi_K=1}$ is an almost 
$\mathbb{C}_p$-representation and is zero if $a>0$. By the definition of $D'_s$, we have a canonical inclusion 
$(\bold{B}^+_{\mathrm{max},K}\otimes_{K_h^{\mathrm{ur}}}D'_{s})^{\varphi_K=1}\subseteq \bold{B}_{\mathrm{max},K}^{+, \varphi^h_K=\pi_K^{-a}}
\otimes_{K^{\mathrm{ur}}_h}D'_s$.
By 8.5 of \cite{Co02}, we have $\bold{B}_{\mathrm{max},K}^{+, \varphi_K^h=\pi_K^{-a}}=0$ for $a>0$, and a short exact sequence 
$$0\rightarrow K_{h}^{\mathrm{ur}}t_{K_{h}^{\mathrm{ur}}}^{-a}\rightarrow \bold{B}_{\mathrm{max},K}^{+, \varphi_K^h=\pi_K^{-a}}\rightarrow \bold{B}^+_{\mathrm{dR}}/t^{-a}\bold{B}^+_{\mathrm{dR}}\rightarrow 0$$
for 
$a\leqq 0$, 
 where 
$t_{K_{h}^{\mathrm{ur}}}\in \bold{B}^+_{\mathrm{max},K}=\bold{B}^+_{\mathrm{max},K_{h}^{\mathrm{ur}}}$ is defined from the triple ($K_{h}^{\mathrm{ur}},\pi_K, \varphi_K^h$) in the same way 
as in the definition of $t_K$ defined from $(K,\pi_K,\varphi_K)$. Moreover, because $\bold{B}^+_{\mathrm{dR}}/t^{-a}\bold{B}^+_{\mathrm{dR}}\otimes_{K_h^{\mathrm{ur}}}D'_s$ is 
a $\bold{B}^+_{\mathrm{dR}}$-representation, so this is also an almost $\mathbb{C}_p$-representation by Theorem 5.13 of \cite{Fo03}.
Hence,  $\bold{B}_{\mathrm{max},K}^{+, \varphi^h_K=\pi_K^{-a}}
\otimes_{K^{\mathrm{ur}}_h}D'_s$ is also an almost $\mathbb{C}_p$-representation. Because we have an equality 
$$(\bold{B}_{\mathrm{max},K}^{+}
\otimes_{K^{\mathrm{ur}}_h}D'_s)^{\varphi_K=1}=\mathrm{Ker}(\varphi_K-1:\bold{B}_{\mathrm{max},K}^{+, \varphi^h_K=\pi_K^{-a}}
\otimes_{K^{\mathrm{ur}}_h}D'_s\rightarrow \bold{B}_{\mathrm{max},K}^{+, \varphi^h_K=\pi_K^{-a}}
\otimes_{K^{\mathrm{ur}}_h}D'_s),$$ so $(\bold{B}_{\mathrm{max},K}^{+}
\otimes_{K^{\mathrm{ur}}_h}D'_s)^{\varphi_K=1}$ is also an almost $\mathbb{C}_p$-representation by Remark \ref{3.3}.
\end{proof}
 
As an application of this lemma, we obtain the following corollary. 
We fix an embedding $\sigma:K\hookrightarrow E$. For a
$K$-vector space $M$ and an $E$-vector space $N$, we denote by $M\otimes_{K,\sigma}N$ the tensor product 
of $M$ and $N$ over $K$, where we view $N$ as a $K$-vector space by the map $\sigma:K\hookrightarrow E$.
\begin{corollary}\label{5.5}
Let $\alpha\in E^{\times}$ be a non-zero element, 
then $(\bold{B}^+_{\mathrm{max},K}\otimes_{K,\sigma}E)^{\varphi_K=\alpha}$ is 
an almost $\mathbb{C}_p$-representation.
For any positive integer $k$ such that $k> e_Kv_p(\alpha)$, where $e_K$ is the absolute ramified index of $K$, the natural map 
$$(\bold{B}^+_{\mathrm{max},K}\otimes_{K,\sigma}E)^{\varphi_K=\alpha}\rightarrow (\bold{B}^+_{\mathrm{dR}}/
t^k\bold{B}^+_{\mathrm{dR}})\otimes_{K,\sigma} E$$ is an injection, Moreover, if we denote the cokernel of this inclusion by $U_k$, 
we have the following short exact sequence of $E$-Banach almost $\mathbb{C}_p$-representations,
\begin{equation*}
0\rightarrow (\bold{B}^+_{\mathrm{max},K}\otimes_{K,\sigma}E)^{\varphi_K=\alpha}\rightarrow (\bold{B}^+_{\mathrm{dR}}/t^k\bold{B}^+_{\mathrm{dR}})
\otimes_{K,\sigma}E\rightarrow U_k\rightarrow 0.
\end{equation*}
\end{corollary}
\begin{proof}
For $\alpha\in E^{\times}$, we define a $\varphi_K$-module $D_{\alpha}$ over $K$ by 
$D_{\alpha}:=Ee$ such that 
$\varphi_K(ae)=\alpha^{-1} ae$ for any $a\in E$. $D_{\alpha}$ has a unique slope $-e_Kv_p(\alpha)$ and we have a natural 
isomorphism $$X_0(D_{\alpha})\isom (\bold{B}^+_{\mathrm{max},K}\otimes_{K,\sigma}E)^{\varphi_K=\alpha}.$$ Hence, 
$(\bold{B}^+_{\mathrm{max},K}\otimes_KE)^{\varphi=\alpha}$ is a non-zero almost $\mathbb{C}_p$-representation by Lemma \ref{5.4}. 
Moreover, using Proposition 8.10 of \cite{Co02}, we can show that we have 
an equality $$(\bold{B}^+_{\mathrm{max},K}\otimes_{K,\sigma}E)^{\varphi_K=\alpha}
\cap(t^k\bold{B}^+_{\mathrm{dR}}\otimes_{K,\sigma}E)=(t_K^k\bold{B}^+_{\mathrm{max},K}\otimes_{K,\sigma}E)^{\varphi_K=\alpha}$$
 for any $k\geqq0$, which is isomorphic 
to $X_0(D_{\alpha\sigma(\pi_K)^{-k}})$. Because, $D_{\alpha\sigma(\pi_K)^{-k}}$ has a unique slope $(k-e_Kv_p(\alpha))$, so 
we have $X_0(D_{\alpha\sigma(\pi_K)^{-k}})=0$ when $k>e_Kv_p(\alpha)$ 
by Lemma \ref{5.4}. This implies that the natural map 
$$(\bold{B}^+_{\mathrm{max},K}\otimes_{K,\sigma}E)^{\varphi_K=\alpha}
\rightarrow (\bold{B}^+_{\mathrm{dR}}/t^k\bold{B}^+_{\mathrm{dR}})\otimes_{K,\sigma}E$$ is an injection. Because both of these are almost $\mathbb{C}_p$-representations which are also $E$-Banach spaces, hence the cokernel $U_k$ is also an $E$-Banach almost $\mathbb{C}_p$-representation
 by Remark \ref{3.3}.
\end{proof}

For two $K$-Banach spaces $M_1$ and $M_2$, we denote by $M_1\hat{\otimes}_KM_2$ the complete tensor product of 
$M_1$ and $M_2$ over $K$. Let $R$ be a complete topological $E$-algebra. We say that $R$ is an $E$-Banach algebra if there exists 
a map $|-|_R:R\rightarrow \mathbb{R}_{\geqq 0}$ which satisfies the following, 
\begin{itemize}
\item[(1)]$|1|_R=1$, $|x|_R=0$ if and only if $x=0$,
\item[(2)]$|x+y|_R\leqq \mathrm{max}\{|x|_R, |y|_R\}$,
\item[(3)]$|xy|_R\leqq |x|_R|y|_R$ and $|ax|_R=|a|_p|x|_R$
\end{itemize}
 for any $x,y\in R$, $a\in E$,
and if  the topology of $R$ is defined by the metric induced from $|-|_R$.

\begin{lemma}\label{5.6}
Let $R$ be an $E$-Banach algebra, and let $\alpha\in R$ be an element of $R$ such that $\alpha-1$ is 
topologically nilpotent, then there exists $u\in (\widehat{K^{\mathrm{ur}}}\hat{\otimes}_{K,\sigma}R)^{\times}$ such that $\varphi_K(u)=\alpha u$.
\end{lemma}
\begin{proof}
The proof is the same as that of Lemma 3.6 of \cite{Ki03}.

\end{proof}

Here, we recall some terminologies concerning   Banach modules from $\S2$ of \cite{Bu07}.
 Let $R$ be an $E$-Banach algebra and let $M$ be a topological $R$-module. 
 We say that $M$ is a Banach $R$-module if  $M$ is 
 a complete topological $R$-module with a map $|-|:M\rightarrow \mathbb{R}_{\geqq 0}$ 
 which satisfies the following,
 \begin{itemize}
 \item[(1)]$|m|=0$ if and only if $m=0$, 
 \item[(2)]$|m+n|\leqq \mathrm{max}\{|m|, |n|\}$, 
 \item[(3)]$|am|\leqq |a|_R|m|$  ($|-|_R$ is a fixed $E$-Banach norm on $R$ as above)
 \end{itemize}
  for any $m,n\in M$, $a\in R$, and if the topology on $M$ is defined by the metric induced from $|-|$. 
 Let $M$ be a Banach $R$-module. We say that $M$ is orthonormalizable if there exist a map 
 $|-|:M\rightarrow \mathbb{R}_{\geqq 0}$ as above and a subset $\{e_i\}_{i\in I}$ of $M$ such that, 
 for any $m\in M$,
 \begin{itemize}
 \item[(1)]there exists a unique $\{a_i\}_{i\in I}$ ($a_i\in R$) such that 
 $a_i\rightarrow 0$ $(i\rightarrow \infty)$ and $m=\sum_{i\in I}a_ie_i$, 
 \item[(2)]we have $|m|=
 \mathrm{max}_{i\in I}\{|a_i|_R\}$.
 \end{itemize}
 We say that a Banach $R$-module $M$ has the property $(\mathrm{Pr})$ if there exists a Banach $R$-module $N$ such that 
 $M\oplus N$ is orthonormalizable.

The following proposition is also a generalization of Corollary 3.7 of \cite{Ki03} which will play a crucial role in the 
next subsection.
\begin{prop}\label{5.7}
Let $R$ be an $E$-Banach algebra and let $Y\in R^{\times}$ be a unit of $R$. Assume that 
there exists a finite Galois extension $E'$ of $E$ and $\lambda\in (R\otimes_E E')^{\times}$ 
such that $E'[\lambda]\subseteq R\otimes_{E}E'$ is an \'etale $E'$-algebra 
and that $Y\lambda^{-1}-1$ is topologically nilpotent in $R\otimes_E E'$. 
Then, for any sufficiently large $k\in \mathbb{Z}_{>0}$, there exists 
a Banach $R$-module $U_k$ with the property $(Pr)$ which is equipped with a continuous $R$-linear $G_K$-action 
such that there exists a $G_K$-equivariant short exact sequence of  Banach $R$-modules with the property (Pr) 
\begin{equation*}
0\rightarrow (\bold{B}^+_{\mathrm{max},K}\hat{\otimes}_{K,\sigma}R)^{\varphi_K=Y}
\rightarrow \bold{B}^+_{\mathrm{dR}}/t^k\bold{B}^+_{\mathrm{dR}}\hat{\otimes}_{K,\sigma}R
\rightarrow U_k\rightarrow 0
\end{equation*}
which is compatible with any base change, i.e. for any continuous homomorphism
 $f:R\rightarrow R'$ of $E$-affinoid algebras, the complete tensor product 
 of the above exact sequence with $R'$ is 
 equal to 
 \begin{equation*}
0\rightarrow (\bold{B}^+_{\mathrm{max},K}\hat{\otimes}_{K,\sigma}R')^{\varphi_K=f(Y)}
\rightarrow \bold{B}^+_{\mathrm{dR}}/t^k\bold{B}^+_{\mathrm{dR}}\hat{\otimes}_{K,\sigma}R'
\rightarrow U_k\hat{\otimes}_{R}R'\rightarrow 0.
\end{equation*}
\end{prop}
\begin{proof}
Decompose $E'[\lambda]=\prod_{i\in I}E_i\subseteq R\otimes_E E'$ such 
that each $E_i$ is a finite extension of $E'$, and denote by  $\lambda_i\in E_i$ the image 
of $\lambda$ in $E_i$, then we obtain a decomposition $R\otimes_E E'=\prod_{i\in I}R_i$ such 
that $E_i\subseteq R_i$ for any $i$. By Corollary \ref{5.5}, we have a short exact 
sequence of $E_i$-Banach spaces
$$ 0\rightarrow (\bold{B}^+_{\mathrm{max},K}\otimes_{K,\sigma}E_i)^{\varphi_K=\lambda_i}
\rightarrow \bold{B}^+_{\mathrm{dR}}/t^k\bold{B}^+_{\mathrm{dR}}\otimes_{K,\sigma}E_i\rightarrow U_{k,i}\rightarrow 0.$$
for any $k\in\mathbb{Z}_{>0}$ such that $k>e_Kv_p(\lambda_i)$ for any $i$.
Hence, if we take the complete tensor product over $E_i$ of this sequence with $R_i$, 
we obtain a short exact sequence of orthonormalizable $R_i$-Banach spaces
$$0\rightarrow (\bold{B}^+_{\mathrm{max},K}\hat{\otimes}_{K,\sigma}R_i)^{\varphi_K=\lambda_i}
\rightarrow \bold{B}^+_{\mathrm{dR}}/t^k\bold{B}^+_{\mathrm{dR}}\hat{\otimes}_{K,\sigma}R_i\rightarrow U_{k,i}\hat{\otimes}_{E_i}R_i\rightarrow 0.$$
By the assumption, the element $Y\lambda_i^{-1}-1$ is topologically nilpotent in $R_i$, hence we have 
an element $u_i\in( \widehat{K^{\mathrm{ur}}}\hat{\otimes}_{K,\sigma}R_i)^{\times}$ such that $\varphi_K(u_i)=Y\lambda_i^{-1} u_i$ by Lemma \ref{5.6}. Multiplying 
$u_i$ to the above short exact sequence, we obtain a short exact sequence
$$0\rightarrow (\bold{B}^+_{\mathrm{max},K}\hat{\otimes}_{K,\sigma}R_i)^{\varphi_K=Y}\rightarrow 
\bold{B}^+_{\mathrm{dR}}/t^k\bold{B}^+_{\mathrm{dR}}\hat{\otimes}_{K,\sigma}R_i\rightarrow (U_{k,i}\hat{\otimes}_{E_i}R_i)\rightarrow 0.$$
Summing up for all $i\in I$, we obtain
$$0\rightarrow (\bold{B}^+_{\mathrm{max},K}\hat{\otimes}_{K,\sigma}(R\otimes_{E}E'))^{\varphi_K=Y}
\rightarrow \bold{B}^+_{\mathrm{dR}}/t^k\bold{B}^+_{\mathrm{dR}}\hat{\otimes}_{K,\sigma} (R\otimes_E E')\rightarrow \oplus_{i\in L}(U_{k,i}
\hat{\otimes}_{E_i}R_i)\rightarrow 0,$$
which we can see as an exact sequence of $R$-Banach modules with the property ($Pr$).
Taking the $\mathrm{Gal}(E'/E)$-fixed part, we obtain a short exact sequence 
$$0\rightarrow (\bold{B}^+_{\mathrm{max},K}\hat{\otimes}_{K,\sigma}R)^{\varphi_K=Y}\rightarrow 
\bold{B}^+_{\mathrm{dR}}/t^k\bold{B}^+_{\mathrm{dR}}\hat{\otimes}_{K,\sigma}R\rightarrow U_k\rightarrow 0$$satisfying all the conditions in the proposition, where $U_k:=(\oplus_{i\in I}
(U_{k,i}\hat{\otimes}_{E_i}R_i))^{\mathrm{Gal}(E'/E)}$. Finally, the base change property of this exact sequence is clear from the proof.

\end{proof}

Let $V$ be an $E$-representation, then we define 
$$\bold{D}^+_{\mathrm{cris}}(V):=(\bold{B}^+_{\mathrm{max}}\otimes_{\mathbb{Q}_p}V)^{G_K},\,\mathrm{Fil}^0\bold{D}_{\mathrm{cris}}(V):=\bold{D}_{\mathrm{cris}}(V)\cap \mathrm{Fil}^0\bold{D}_{\mathrm{dR}}(V),$$ where we recall that $\bold{B}^+_{\mathrm{max}}=\bold{B}^+_{\mathrm{max},\mathbb{Q}_p}$. 
Then, we have a natural inclusion 
$\bold{D}^+_{\mathrm{cris}}(V)\subseteq \mathrm{Fil}^0\bold{D}_{\mathrm{cris}}(V)$, which is not an equality in general.

\begin{lemma}\label{5.7.5}
Let $\alpha\in E^{\times}$ be a non zero element. If a $\varphi$-submodule $D$ of $
\bold{D}_{\mathrm{cris}}(V)^{\varphi^f=\alpha}$ 
is contained in $\mathrm{Fil}^0\bold{D}_{\mathrm{dR}}(V)$, then $D$ is also contained in $\bold{D}^+_{\mathrm{cris}}(V)^{\varphi^f=\alpha}$.

\end{lemma}

\begin{proof}
It suffices to show that 
if an element $x\in (\bold{B}_{\mathrm{max}}\otimes_{\mathbb{Q}_p}E)^{\varphi^f=\alpha}$ satisfies 
that $\varphi^i(x)\in \bold{B}^+_{\mathrm{dR}}\otimes_{\mathbb{Q}_p}E$ for any $i\in \mathbb{Z}_{\geqq 0}$, then 
$x\in (\bold{B}^+_{\mathrm{max}}\otimes_{\mathbb{Q}_p}E)^{\varphi^f=\alpha}$. 
If we write $x=\frac{a}{t^n}$ for some $a\in (\bold{B}^+_{\mathrm{max}}\otimes_{\mathbb{Q}_p}E)^{\varphi^f=\alpha p^{fn}}$ and $n\geqq 0$, then we have $\frac{\varphi^i(a)}{p^{ni}t^n}=\varphi^i(x)\in \bold{B}^+_{\mathrm{dR}}\otimes_{\mathbb{Q}_p}E$ for any $0\leqq i\leqq f-1$. 
Hence, we have $$\varphi^i(a)\in (\bold{B}^+_{\mathrm{max}}\otimes_{\mathbb{Q}_p}E)^{\varphi^f=\alpha p^{nf}}\cap t^{n}\bold{B}^+_{\mathrm{dR}}\otimes_{\mathbb{Q}_p}E
=(t_{K_0}^n\bold{B}^+_{\mathrm{max}}\otimes_{\mathbb{Q}_p}E)^{\varphi^f=\alpha p^{nf}},$$ where the last equality follows from 
Proposition 8.10 of \cite{Co02}. Hence, we can write $a=\varphi^{-i}(t_{K_0}^n)a_i$ for some 
$a_i\in \bold{B}^+_{\mathrm{max}}\otimes_{\mathbb{Q}_p}E$ for any $1\leqq i\leqq f-1$.  Then we can write by 
$a=(\prod_{i=0}^{f-1}\varphi^{-i}(t_{K_0}^n))a'$ for some $a'\in \bold{B}^+_{\mathrm{max}}\otimes_{\mathbb{Q}_p}E$ by Lemma 8.18 of \cite{Co02}. 
Because $\prod_{i=0}^{f-1}\varphi^{-i}(t_{K_0})\in K_0^{\times}t$ by Lemma 8.17 of \cite{Co02}, we have $x=\frac{a}{t^n}\in \bold{B}^+_{\mathrm{max}}\otimes_{\mathbb{Q}_p}E$.

\end{proof}

\subsection{Construction of the finite slope subspace for general $p$-adic field}
In this subsection, we generalize the theory of the  finite slope subspace $X_{fs}$ for any $p$-adic field. Using Proposition \ref{5.7}, the construction and the proof is essentially the same as those 
for $K=\mathbb{Q}_p$, but there is a difference that we need to consider all the embeddings $\sigma:K\hookrightarrow E$, which makes the theory more complicated. Hence, for convenience of readers (and the author), here we reprove the Kisin's theory in detail in our generalized setting.

Let $X$ be a separated rigid analytic space (in the sense of Tate) over $E$. 
Let $M$ be a free $\mathcal{O}_X$-module of rank $d$ for some $d\geqq 1$ equipped with 
a continuous $\mathcal{O}_X$-linear $G_K$-action, where ``continuous" means that, for any admissible open affinoid $U=\mathrm{Spm}(R)$ of $X$, the action of $G_K$ on $\Gamma(U, M)$ is continuous on which the topology is naturally induced by that of $R$. We denote by $M^{\vee}:=\mathrm{Hom}_{\mathcal{O}_X}(M,\mathcal{O}_X)$ the $\mathcal{O}_X$-dual of $M$. For 
$x\in X$, we denote by $\mathcal{O}_{X,x}$ the local ring at $x$, $\mathfrak{m}_x$ 
the maximal ideal at $x$, $E(x)$ the residue field at $x$ which is a finite extension of $E$. 
We denote by $M(x)$ the fiber of $M$ at $x$, which is a $d$-dimensional 
$E(x)$-representation of $G_K$. 

Under this situation, we briefly recall Sen's theory \cite{Se88} of the analytic variations of Hodge-Tate weights following \cite{Be-Co08} and \cite{Ch09a}. By Lemma 3.18 of \cite{Ch09a}, we can take an affinoid covering $ \{U_i\}_{i\in I}$ of $X$ such that $M_i:=\Gamma(U_i,M)$ is a free $R_i:=\Gamma(U_i,\mathcal{O}_X)$-module with a $G_K$-stable finite free $R_i^0$-module $M_i^0$ such that $M_i=M^0_i[p^{-1}]$ for a model 
$R^0_{i}\subseteq R_i$ for any $i\in I$. Here, for an affinoid $A$, a model is defined as a topologically finite generated complete
$\mathcal{O}_E$-subalgebra of $A$ which generates $A$ after inverting $p$. 
Then, we can apply the results of \cite{Be-Co08} to $M_i$ (and $M^0_i$) for any 
$i\in I$. By Proposition 4.1.2 of \cite{Be-Co08}, there exists a unique monic polynomial $P_{M_i}(T)\in K\otimes_{\mathbb{Q}_p}R_i[T]$ of dimension $d$, which is defined as the characteristic polynomial 
of Sen's operator on $\bold{D}_{\mathrm{Sen}}^{L}(M_i)$ for a sufficiently large finite extension $L$ of $K$, such that 
the specialization of $P_{M_i}(T)$ at $x$ gives the Sen's polynomial $P_{M(x)}(T)\in K\otimes_{\mathbb{Q}_p}E(x)[T]$ of $M(x)$  for any $x\in \mathrm{Spm}(R_i)$.
By the uniqueness of $\bold{D}_{\mathrm{Sen}}^{L}(M_i)$, $\{P_{M_i}(T)\}_{i\in I}$ glue together to 
a monic polynomial $P_M(T)\in K\otimes_{\mathbb{Q}_p}\mathcal{O}_X[T]$. By the canonical decomposition $K\otimes_{\mathbb{Q}_p}E=\oplus_{\sigma\in \mathcal{P}}E:a\otimes b\mapsto (\sigma(a)b)_{\sigma}$, $P_M(T)$ decomposes into the $\sigma$-components
$P_M(T)=(P_M(T)_{\sigma})_{\sigma\in\mathcal{P}}\in \oplus_{\sigma\in \mathcal{P}} \mathcal{O}_X[T]$.

From now on, we assume that the constant term of $P_M(T)_{\sigma}$ is zero for any $\sigma\in\mathcal{P}$. 
We denote by $P_M(T)_{\sigma}=TQ_{\sigma}(T)$ for $Q_{\sigma}(T)\in \mathcal{O}_X[T]$.

Before stating the theorem, we recall some terminologies 
 of rigid geometry from $\S$ 5 of \cite{Ki03} which we need to characterize the finite slope subspace $X_{fs}$.
 
 Let $X=\mathrm{Spm}(R)$ be an affinoid over $E$, and let
 $U$ be an admissible open in $X$. We say that 
 $U$ is scheme theoretically dense in $X$ if there exists a
 Zariski open $V\subseteq \mathrm{Spec}(R)$ which is 
 dense in $\mathrm{Spec}(R)$ for the Zariski topology 
 and $U=V^{\mathrm{an}}$, where $V^{\mathrm{an}}$ is 
the associated rigid space of $V$. 
For any rigid analytic space $X$ over $E$, an admissible 
open $U$ of $X$ is said to be scheme theoretically dense 
in $X$ if there exists an admissible affinoid covering $\{U_i\}_{i\in I}$ of $X$ 
such that $U\cap U_i$ is scheme theoretically dense in $U_i$ for any $i\in I$.
The typical example is the following. For any $f\in \Gamma(X, \mathcal{O}_X)$, 
we set $X_f:=\{x\in X| f(x)\not=0\}$ which is an admissible open in $X$. 
If $f$ is a non-zero divisor, then $X_f$ is scheme theoretically dense in $X$.

Next, let $Y\in \Gamma(X, \mathcal{O}_X^{\times})$ be an invertible function on $X$, and let $R$ be an affinoid algebra over $E$. 
We say that a morphism $f:\mathrm{Spm}(R)\rightarrow X$ is $Y$-small 
if there exist a finite extension $E'$ of $E$ and $\lambda\in (R\otimes_{E}E')^{\times}$ such 
that $E'[\lambda]\subseteq R\otimes_EE'$ is a finite \'etale $E'$-algebra and  
$Y\lambda^{-1}-1\in R\otimes_EE'$ is topologically nilpotent. A typical example 
of $Y$-small morphism is following. For any $x\in X$ and $n\in \mathbb{Z}_{\geqq 1}$, the natural 
map $\mathrm{Spm}(\mathcal{O}_{X,x}/\mathfrak{m}_{x}^n)\rightarrow X$ is $Y$-small for any 
$Y\in \Gamma(X,\mathcal{O}_X)^{\times}$.

The following theorem is the generalization of Proposition 5.4 of \cite{Ki03} for general $K$, which 
states the existence and the characterization of the finite slope subspace $X_{fs}$. This theorem 
is the most important for the construction of $p$-adic families of 
trianguline representations in the next subsection. 
For an $E$-affinoid algebra $R$, we set $\bold{B}^+_{\mathrm{dR}}\hat{\otimes}_{\mathbb{Q}_p}R:=\varprojlim_k\bold{B}^+_{\mathrm{dR}}/t^k\bold{B}^+_{\mathrm{dR}}\hat{\otimes}_{\mathbb{Q}_p}R$ which is equipped with the projective limit topology.

\begin{thm}\label{5.8}
Let $X$ be a separated rigid analytic space over $E$, and let $M$ be a free $\mathcal{O}_X$-module 
of rank $d$ with a continuous $\mathcal{O}_X$-linear $G_K$-action. Let $Y\in \Gamma(X, \mathcal{O}_X^{\times})$ 
be an invertible function.
We assume that the constant term of $P_M(T)_{\sigma}$ is zero for any $\sigma\in \mathcal{P}$.
Then, there exists a unique Zariski closed subspace $X_{fs}$ of $X$ satisfying the following conditions,
\begin{itemize}
\item[(1)]$X_{fs, Q_{\sigma}(i)}$ is scheme theoretically dense in $X_{fs}$ for any $\sigma\in \mathcal{P}$ and $i\in \mathbb{Z}_{\leqq 0}$,
\item[(2)]for any $Y$-small map $f:\mathrm{Spm}(R)\rightarrow X$ 
which factors through $X_{Q_{\sigma}(i)}$ for any 
$\sigma\in\mathcal{P}$ and $i\in\mathbb{Z}_{\leqq0}$, the following two conditions are equivalent,
\begin{itemize}
\item[(i)]$f:\mathrm{Spm}(R)\rightarrow X$ factors through $X_{fs}$,
\item[(ii)] any $R$-linear $G_K$-equivariant  map $h:M^{\vee}\otimes_{\mathcal{O}_X}R\rightarrow \bold{B}^+_{\mathrm{dR}}\hat{\otimes}_{\mathbb{Q}_p}R$ factors through $h':M^{\vee}\otimes_{\mathcal{O}_X}R\rightarrow K\otimes_{K_0}(\bold{B}^+_{\mathrm{max}}\hat{\otimes}_{\mathbb{Q}_p}R)^{\varphi^f=Y}$.
\end{itemize}
\end{itemize}
\end{thm}

As in \cite{Ki03}, we prove this theorem by several steps.
We first prove the following lemma.
\begin{lemma}\label{5.9}
Let $X$, $M$ be as above. Let $X'$ be a separated rigid analytic space over $E$, and let
$f:X'\rightarrow X$ be a flat $E$-morphism. If there exists a Zariski closed subspace $X_{fs}\subseteq X$ 
which satisfies $(1)$ and $(2)$ of the above theorem, then 
$X'_{fs}:=X_{fs}\times_XX'\subseteq X'$ also satisfies $(1)$ and $(2)$ for $X'$, $M':=f^*{M}$ and $Y':=f^{*}(Y)\in\Gamma(X', \mathcal{O}_{X'}^{\times})$.
\end{lemma}
\begin{proof}
The condition (1) is satisfied for $X'_{fs}$ because the notion of scheme theoretically dense is preserved by 
flat base changes and we have $f^*(P_M(T))=P_{f^*M}(T)$. That $X'_{fs}$ satisfies (2) is trivial.
\end{proof}
We next prove the uniqueness of $X_{fs}$.
\begin{lemma}
If two Zariski closed subspaces $X_1$ and $X_2$ of $X$ satisfy the conditions $(1)$ and $(2)$, then 
$X_1=X_2$.
\end{lemma}
\begin{proof}
For any admissible open $U\subseteq X$ and $i=1,2$, $X_i\cap U\subseteq U$ satisfies (1) and (2) for 
$U$ by Lemma \ref{5.9} because the inclusion $U\hookrightarrow X$ is flat. Hence, it suffices to prove that  $X_1\cap U_i=X_2\cap U_i$ for any $i\in I$ for an admissible covering $\{U_i\}_{i\in I}$ of 
$X$. Hence we may assume that 
$X=\mathrm{Spm}(R)$ is an affinoid. We denote by $X_1=\mathrm{Spm}
(R/I_1)$, $X_2=\mathrm{Spm}(R/I_2)$ for some ideals $I_1, I_2\subseteq R$. Set
$X_3:=\mathrm{Spm}(R/I_1\cap I_2)$, then we claim that $X_3$ also satisfies the conditions 
(1) and (2) . For (1), we have inclusions $R/I_j\hookrightarrow R/I_j[\frac{1}{Q_{\sigma}(i)}]$ for 
any $j=1,2$, $\sigma\in \mathcal{P}$ and $i\in \mathbb{Z}_{\leqq 0}$ by the assumption, hence we also have an inclusion 
$R/I_1\cap I_2\hookrightarrow R/I_1\cap I_2[\frac{1}{Q_{\sigma}(i)}]$ for any $\sigma\in \mathcal{P}, i\in 
\mathbb{Z}_{\leqq0}$, which proves that $X_3$ satisfies (1). To prove that $X_3$ satisfies (2), we take a $Y$-small morphism 
$f:\mathrm{Spm}(R')\rightarrow X$ which factors $f:\mathrm{Spm}(R')\rightarrow 
X_{Q_{\sigma}(i)}$ for any $\sigma\in \mathcal{P}, i\in \mathbb{Z}_{\leqq 0}$. Set $Y':=f^*(Y)\in R^{' \times}$. 
If $f$ satisfies (ii) of (2), then $f$ factors through $X_1$ and $X_2$ by definition, hence also factors through $X_3$ because we have $X_1,X_2\subseteq X_3$. 
Next we assume that 
$f$ satisfies (i) of (2).
We have the following canonical decompositions 
\begin{multline*}
K\otimes_{K_0} (\bold{B}^+_{\mathrm{max}}\hat{\otimes}_{\mathbb{Q}_p}R')
^{\varphi^f=Y'}=((K\otimes_{K_0}\bold{B}^+_{\mathrm{max}})\hat{\otimes}_{\mathbb{Q}_p}R')^{\varphi^f=Y'}\\
=(\bold{B}^+_{\mathrm{max},K}\hat{\otimes}_{\mathbb{Q}_p}R')^{\varphi^f=Y'}=\oplus_{\sigma\in \mathcal{P}}(\bold{B}^+_{\mathrm{max},K}\hat{\otimes}_{K,\sigma}R')^{\varphi_K=Y'}
\end{multline*}
and 
$$\bold{B}^+_{\mathrm{dR}}\hat{\otimes}_{\mathbb{Q}_p}R'=\oplus_{\sigma\in \mathcal{P}}\bold{B}^+_{\mathrm{dR}}\hat{\otimes}_{K,\sigma}R'.$$ 
Hence, it suffices to show that any $G_K$-equivariant $R'$-linear map 
$h:M^{\vee}\otimes_{\mathcal{O}_X}R'\rightarrow \bold{B}^+_{\mathrm{dR}}\hat{\otimes}_{K,\sigma}R'$ factors through 
$M^{\vee}\otimes_{\mathcal{O}_X}R'\rightarrow (\bold{B}^+_{\mathrm{max},K}\hat{\otimes}_{K,\sigma}R')^{\varphi_K=Y'}$ for any $\sigma\in \mathcal{P}$. 
Because $Q_{\sigma}(i)$ is invertible in 
$R'$ for 
any $\sigma\in \mathcal{P}$ and $i\in \mathbb{Z}_{\leqq 0}$ by the assumption, the natural map 

$$(\bold{B}^+_{\mathrm{dR}}\hat{\otimes}_{K,\sigma}M\otimes_{\mathcal{O}_X}R')^{G_K}\isom (\bold{B}^+_{\mathrm{dR}}/t^k\bold{B}^+_{\mathrm{dR}}\hat{\otimes}_{K,\sigma} M\otimes_{\mathcal{O}_X}R')^{G_K}$$ is isomorphism for any $\sigma\in\mathcal{P}$ and $k\in\mathbb{Z}_{\geqq 1}$ by Proposition 2.5 of \cite{Ki03}. 
Hence, it suffices to show that, 
for some $k\in\mathbb{Z}_{\geqq 1}$, any $G_K$-equivariant map $h:M^{\vee}\otimes_{\mathcal{O}_X}R'\rightarrow (\bold{B}^+_{\mathrm{dR}}/t^k\bold{B}^+_{\mathrm{dR}}\hat{\otimes}_{K,\sigma}R')$ factors through $M^{\vee}\otimes_{\mathcal{O}_X}R'\rightarrow (\bold{B}^+_{\mathrm{max},K}\hat{\otimes}_{K,\sigma}R')^{\varphi_K=Y'}$. 
We choose $k\in\mathbb{Z}_{\geqq 1}$ sufficiently large such that there exists a short exact sequence
\begin{equation}0\rightarrow (\bold{B}^+_{\mathrm{max},K}\hat{\otimes}_{K,\sigma}R')^{\varphi_K=Y'}\rightarrow 
\bold{B}^+_{\mathrm{dR}}/t^k\bold{B}^+_{\mathrm{dR}}\hat{\otimes}_{K,\sigma}R'\rightarrow U_{k,\sigma}\rightarrow 0
\end{equation} of 
Banach $R'$-modules compatible with any base change as in Proposition \ref{5.7}.
If we set $\mathrm{Spm}(R'_i):=f^{-1}(X_{i})\subseteq \mathrm{Spm}(R')$ for $i=1,2$, then we have an inclusion 
$R'\hookrightarrow R'_1\oplus R'_2$ because $f$ factors through $X_3$. From these arguments, the above short exact sequence (5) can be embedded in the following short exact 
sequence
$$0\rightarrow \oplus^2_{i=1}(\bold{B}^+_{\mathrm{max},K}\hat{\otimes}_{K,\sigma}R'_i)^{\varphi_K=Y'}\rightarrow 
\oplus^2_{i=1}\bold{B}^+_{\mathrm{dR}}/t^k\bold{B}^+_{\mathrm{dR}}\hat{\otimes}_{K,\sigma}R'_i\rightarrow \oplus^2_{i=1}
U_{k,\sigma}\hat{\otimes}_{R'}R'_i\rightarrow 0.$$
Then, the composition of $h$ with 
$\bold{B}^+_{\mathrm{dR}}/t^k\bold{B}^+_{\mathrm{dR}}\hat{\otimes}_{K,\sigma}
R'\hookrightarrow \oplus^2_{i=1}\bold{B}^+_{\mathrm{dR}}/t^k\bold{B}^+_{\mathrm{dR}}\hat{\otimes}_{K,\sigma}R'_i$ factors 
through $M^{\vee}\otimes_{\mathcal{O}_X}R'\hookrightarrow \oplus^2_{i=1}(\bold{B}^+_{\mathrm{max},K}\hat{\otimes}_{K,\sigma}R'_i)^{\varphi_K=Y'}$ by  the definition of $X_i$.
 Hence, $h$ also factors through $M^{\vee}\otimes_{\mathcal{O}_X}R'\rightarrow (\bold{B}^+_{\mathrm{max},K}\hat{\otimes}_{K,\sigma}
R_i')^{\varphi_K=Y'}$ by a diagram chase. Hence, $X_3$ also satisfies (1) and (2). 
 
Therefore, to prove the lemma, we may assume that $X_1\subseteq X_2$. We put $W\subseteq X_2$ the support of 
$I_1/I_2$ with the reduced structure. If $x\in X_2$ satisfies $Q_{\sigma}(i)(x)\not=0$ for any $\sigma\in \mathcal{P}$ and $i\in\mathbb{Z}_{\leqq 0}$, then the natural map $\mathrm{Spm}(\mathcal{O}_{X_2,x}/\mathfrak{m}_x^n)\rightarrow 
X_2$ which is $Y$-small factors through $\mathrm{Spm}(\mathcal{O}_{X_2,x}/\mathfrak{m}_x^n)\rightarrow X_1$  for any $n\geqq 1$ by the definition on $X_1$ and $X_2$. 
This implies that there exists a map $\mathcal{O}_{X_1,x}\rightarrow \hat{\mathcal{O}}_{X_2,x}$ such that the composition 
of this with the natural map $\mathcal{O}_{X_2,x}\rightarrow \mathcal{O}_{X_1,x}$ is the natural map $\mathcal{O}_{X_2,x}
\rightarrow \hat{\mathcal{O}}_{X_2,x}$. This implies that the natural quotient map $\mathcal{O}_{X_2,x}
\rightarrow \mathcal{O}_{X_1,x}$ is isomorphism, hence we have $x\not\in W$. Hence, we obtain
an inclusion $W\subseteq \cup_{\sigma\in\mathcal{P},i\in\mathbb{Z}_{\leqq 0}}\{x\in X_2| Q_{\sigma}(i)(x)=0\}$. 
By Lemma 5.7 \cite{Ki03}, then there exists a $Q\in \Gamma(X_2,\mathcal{O}_{X_2})$ which is a finite product of $Q_{\sigma}(i)$ such 
that $X_{2,Q}\subseteq X_2\setminus W=X_1\setminus W\subseteq X_1 \subseteq X_2$. Then, the condition 
(1) for $X_2$ implies that $X_1=X_2$. We finish to prove the lemma.

\end{proof}

Assume that there exists an admissible affinoid covering $\{U_i\}_{i\in I}$ of $X$ such that the subspace $U_{i,fs}\subseteq U_i$ exists for any $i\in I$. 
By the uniqueness of $X_{fs}$, $\{U_{i,fs}\}_{i\in I}$ glue to a Zariski closed subspace $X'_{fs}\subseteq X$ satisfying that 
$X'_{fs}\cap U_i=U_{i,fs}$ for any $i\in I$. 
\begin{lemma}\label{5.11}
In the above situation, $X'_{fs}\subseteq X$ satisfies the conditions $(1)$ and $(2)$ in the theorem, i.e. we have $X'_{fs}=X_{fs}$.
\end{lemma}
\begin{proof}
That $X'_{fs}$ satisfies (1) is trivial. We show that $X'_{fs}$ satisfies (2).
Let $f:\mathrm{Spm}(R)\rightarrow X$ be a $Y$-small map which factors through $X_{Q_{\sigma}(i)}$ 
for any $\sigma\in \mathcal{P}$ and $i\in \mathbb{Z}_{\leqq 0}$.  Because $X$ is separated, 
$f^{-1}(U_{i})$ is an affinoid for any $i\in I$. Set $\mathrm{Spm}(R_i):=f^{-1}(U_{i})$.
First, we show that (i) implies (ii). 
We assume that $f$ factors through $X'_{fs}$. Let $h:M^{\vee}\otimes_{\mathcal{O}_X}R\rightarrow \bold{B}^+_{\mathrm{dR}}\hat{\otimes}_{K,\sigma}R$ 
be a $R$-linear $G_K$-equivariant map. By Proposition 2.5 of \cite{Ki03}, it suffices to show that $h:M^{\vee}\otimes_{\mathcal{O}_X}R
\rightarrow \bold{B}^+_{\mathrm{dR}}\hat{\otimes}_{K,\sigma}R\rightarrow \bold{B}^+_{\mathrm{dR}}/t^k\bold{B}^+_{\mathrm{dR}}\hat{\otimes}_{K,\sigma}R$ 
factors through $M^{\vee}\otimes_{\mathcal{O}_X}R\rightarrow (\bold{B}^+_{\mathrm{max},K}\hat{\otimes}_{K,\sigma}R)^{\varphi_K=Y}$  for some 
 $k\in\mathbb{Z}_{\geqq 1}$. 
We choose $k\in\mathbb{Z}_{\geqq 1}$ such that there exists a short exact sequence $$0\rightarrow (\bold{B}^+_{\mathrm{max},K}\hat{\otimes}_{K,\sigma}R)^{\varphi_K=Y}\rightarrow 
\bold{B}^+_{\mathrm{dR}}/t^k\bold{B}^+_{\mathrm{dR}}\hat{\otimes}_{K,\sigma}R\rightarrow U_{k,\sigma}\rightarrow 0$$
of 
 Banach $R$-modules as in Proposition \ref{5.7}.
By the base change property, this short exact sequence can be embedded into the following exact sequence
$$0\rightarrow \prod_{i\in I}(\bold{B}^+_{\mathrm{max},K}\hat{\otimes}_{K,\sigma}R_i)^{\varphi_K=Y}\rightarrow 
\prod_{i\in I}\bold{B}^+_{\mathrm{dR}}/t^k\bold{B}^+_{\mathrm{dR}}\hat{\otimes}_{K,\sigma}R_i \rightarrow \prod_{i\in I}
U_{k,\sigma}\hat{\otimes}_R R_i\rightarrow 0.$$

By the assumption, the map $M^{\vee}\otimes_{\mathcal{O}_X}R\xrightarrow{h} \bold{B}^+_{\mathrm{dR}}/t^k\bold{B}^+_{\mathrm{dR}}\hat{\otimes}_{K,\sigma}R
\rightarrow \bold{B}^+_{\mathrm{dR}}/t^k\bold{B}^+_{\mathrm{dR}}\hat{\otimes}_{K,\sigma}R_i$ factors through 
$M^{\vee}\otimes_{\mathcal{O}_X}R\rightarrow (\bold{B}^+_{\mathrm{max},K}\hat{\otimes}_{K,\sigma}R_i)^{\varphi_K=Y}$ for any $i\in I$. Hence, 
$ h:M^{\vee}\otimes_{\mathcal{O}_X}R\rightarrow 
\bold{B}^+_{\mathrm{dR}}/t^k\bold{B}^+_{\mathrm{dR}}\hat{\otimes}_{K,\sigma}R$ also factors through $M^{\vee}\otimes_{\mathcal{O}_X}R\rightarrow 
(\bold{B}^+_{\mathrm{max},K}\hat{\otimes}_{K,\sigma}\allowbreak R)^{\varphi_K=Y}$ by a diagram chase of the above two exact sequences.

Next, we assume that, for any $\sigma\in \mathcal{P}$, any 
$R$-linear $G_K$-equivariant map $h:M^{\vee}\otimes_{\mathcal{O}_X}R\rightarrow \bold{B}^+_{\mathrm{dR}}\hat{\otimes}_{K,\sigma}R$ 
factors through $M^{\vee}\otimes_{\mathcal{O}_X}R\rightarrow (\bold{B}^+_{\mathrm{max},K}\hat{\otimes}_{K,\sigma}R)^{\varphi_K=Y}$.  
Because we have $Q_{\sigma}(j)\in R^{\times}$ for any $\sigma\in \mathcal{P}$ and $j\in \mathbb{Z}_{\leqq 0}$, the natural map 
$$(\bold{B}^+_{\mathrm{dR}}/t^k\bold{B}^+_{\mathrm{dR}}\hat{\otimes}_{K,\sigma}M\otimes_{\mathcal{O}_X}R)^{G_K}
\otimes_R R_i\isom (\bold{B}^+_{\mathrm{dR}}/t^k\bold{B}^+_{\mathrm{dR}}\hat{\otimes}_{K,\sigma}M\otimes_{\mathcal{O}_X}R_i)^{G_K}$$ is isomorphism 
for any $k\geqq 1$ and $i\in I$ by Corollary 2.6 of \cite{Ki03}. Hence, any $R_i$-linear $G_K$-equivariant map $h_i:M^{\vee}\otimes_{\mathcal{O}_X}R\rightarrow 
\bold{B}^+_{\mathrm{dR}}/t^k\bold{B}^+_{\mathrm{dR}}\hat{\otimes}_{K,\sigma}R_i$  factors through $M^{\vee}\otimes_{\mathcal{O}_X}R\rightarrow 
(\bold{B}^+_{\mathrm{max},K}\hat{\otimes}_{K,\sigma}R_i)^{\varphi_K=Y}$ for any $i\in I$. This implies that 
$f|_{\mathrm{Spm}(R_i)}:\mathrm{Spm}(R_i)\rightarrow U_i$ factors through $U_{i,fs}$ for any 
$i\in I$. Hence, $f$ also factors through $X'_{fs}$. 
\end{proof}

By this lemma, it suffices to construct $X_{fs}$ when 
$X=\mathrm{Spm}(R)$ is affinoid. Moreover, in the same way as in (5.9) of \cite{Ki03}, we may assume that $|Y|$ satisfies $|Y||Y^{-1}|< \frac{1}{|\pi_K|_p}$ for a norm $|-|:R\rightarrow \mathbb{R}_{\geqq 0}$ which defines the topology of $R$ as in $\S3.2$. 
Then, we construct $X_{fs}\subseteq \mathrm{Spm}(R)$ as follows. We first
construct an ideal of $R$ which defines $X_{fs}$. 
Let $\lambda\in \overline{E}$ be an element such 
that $|Y^{-1}|^{-1}\leqq |\lambda|_p\leqq |Y|$, and let $E'$ be a finite Galois extension of $E$ which contains $\lambda$. 
By Corollary \ref{5.5},
we can take a sufficiently large $k\in\mathbb{Z}_{\geqq 1}$ such that there exists a 
short exact sequence of $E'$-Banach spaces
$$0\rightarrow (\bold{B}^+_{\mathrm{max},K}\hat{\otimes}_{K,\sigma}E')^{\varphi_K=\sigma(\pi_K)\lambda}
\rightarrow \bold{B}^+_{\mathrm{dR}}/t^k\bold{B}^+_{\mathrm{dR}}\hat{\otimes}_{K,\sigma}E'\rightarrow U_{k,\sigma,\lambda}\rightarrow 0$$
for any $\lambda$, $E'$ as above and $\sigma\in \mathcal{P}$. Fix such a $k\geqq1$ until the end of the proof of this theorem.
For any $x\in \widetilde{\bold{E}}^+$ such that $v(x)>0$, we define  an element 
$$P(x,\frac{Y}{\sigma(\pi_K)\lambda}):=
\sum_{n\in \mathbb{Z}}\varphi_K^n([x])(\frac{Y}{\sigma(\pi_K)\lambda})^n\allowbreak\in (\bold{B}^+_{\mathrm{max},K}\hat{\otimes}_{K,\sigma}R\otimes_E E')^{\varphi_K=
\frac{\sigma(\pi_K)\lambda}{Y}}.$$ 
This element converges because we have
$$|\frac{\sigma(\pi_K)\lambda}{Y}|\leqq |\sigma(\pi_K)|_p|\lambda|_p|Y^{-1}|<
|\sigma(\pi_K)|_p|\lambda|_p|Y|^{-1}|\pi_K|_p^{-1}\leqq 1$$ and $\varphi_K^n([x])(\frac{Y}{\sigma(\pi_K)\lambda})^n\rightarrow 0$ ($n\rightarrow +\infty$) (see Corollary 4.4 of \cite{Ki03}). For any $\sigma\in\mathcal{P}$ and any $R$-linear $G_K$-equivariant map 
$h:M^{\vee}\rightarrow \bold{B}^+_{\mathrm{dR}}/t^k\bold{B}^+_{\mathrm{dR}}\hat{\otimes}_{K,\sigma}R$, we consider 
the composition of this map with the maps
$$\bold{B}^+_{\mathrm{dR}}/t^k\bold{B}^+_{\mathrm{dR}}\hat{\otimes}_{K,\sigma}R\rightarrow 
\bold{B}^+_{\mathrm{dR}}/t^k\bold{B}^+_{\mathrm{dR}}\hat{\otimes}_{K,\sigma}R\otimes_E E':v\mapsto P(x,\frac{Y}{\sigma(\pi_K)\lambda})v$$ 
and $\bold{B}^+_{\mathrm{dR}}/t^k\bold{B}^+_{\mathrm{dR}}\hat{\otimes}_{K,\sigma}R\otimes_E E'\rightarrow 
U_{k,\sigma,\lambda}\hat{\otimes}_{E'}(R\otimes_E E')$ which is the base change of the surjection 
$\bold{B}^+_{\mathrm{dR}}/t^k\bold{B}^+_{\mathrm{dR}}\hat{\otimes}_{K,\sigma}E'\rightarrow U_{k,\sigma,\lambda}$ in the above short exact sequence. 
We denote this composition by 
$$h_{x,\lambda}:M^{\vee}\rightarrow 
U_{k,\sigma,\lambda}\hat{\otimes}_{E'} (R\otimes_E E').$$ Fix an orthonormalizable $E'$-base $\{e_i\}_{i\in I}$ of $U_{k,\sigma,\lambda}$. For any 
$m\in M^{\vee}$, then we can write uniquely by
$$h_{x,\lambda}(m)=\sum_{i\in I}a_{x,\lambda,i}(m)e_i\text{ for }
\{a_{x,\lambda,i}(m)\}_{i\in I}\subseteq R\otimes_EE'.$$ We define an ideal  
$$I(h,x,\lambda,m)\subseteq R\otimes_E E'$$ which is 
generated by $a_{x,\lambda,i}(m)$ for all $i\in I$. Because we have $\tau(I(h,x,\lambda,m))
=I(h,x,\tau(\lambda),m)\subseteq R\otimes_EE'$ for any $\tau\in \mathrm{Gal}(E'/E)$,  the ideal 
$$\sum_{\tau\in \mathrm{Gal}(E'/E)}
I(h,x,\tau(\lambda),m) \subseteq R\otimes_E E'$$ descends to an ideal $I'(h,x,\lambda,m)\subseteq R$ and this ideal is 
independent of the choice of $E'$. We define an ideal $I$ of $R$  by $$I:=\sum_{h,x,\lambda,m}I'(h,x,\lambda,m)\subseteq R,$$ where the sum runs through 
all $h,x,\lambda, m$ and $\sigma\in \mathcal{P}$ as above. 

Next, we denote by $J_{(n,\{\sigma_l\}_{l=0}^n, \{i_l\}_{l=0}^n)}(\supseteq I)$ the kernel 
of the natural map
$R\rightarrow R/I[\frac{1}{\prod_{l=1}^{n}
Q_{\sigma_l}(i_l)}]$ for any triple $(n,\{\sigma_l\}_{l=0}^n, \{i_l\}_{l=0}^n)$ such that $n\geqq 1$, $\sigma_l\in \mathcal{P}, i_l\in\mathbb{Z}_{\leqq 0}$. Denote by $J(\supseteq I)$ the sum (in fact a finite union) of the ideals $J_{(n,\{\sigma_l\}_{l=0}^n, \{i_l\}_{l=0}^n)}$ for all the triples $(n,\{\sigma_l\}_{l=0}^n, \{i_l\}_{l=0}^n)$ as above. Then, $\mathrm{Spm}(R/J)$ is the largest Zariski closed subspace of $\mathrm{Spm}(R/I)$ such that 
$\mathrm{Spm}(R/J)_{Q_{\sigma}(i)}$ is scheme theoretically dense in $\mathrm{Spm}(R/J)$ for 
any $\sigma\in \mathcal{P}, i\leqq 0$.

Finally, we prove the following lemma which claims that $X_{fs}=\mathrm{Spm}(R/J)$, hence proves the theorem.
\begin{lemma}\label{5.12}The closed subspace
$\mathrm{Spm}(R/J)\subseteq \mathrm{Spm}(R)$ satisfies the conditions $(1)$ and $(2)$ in the 
theorem.
\end{lemma}

\begin{proof}
Because the map $R/J\rightarrow R/J[\frac{1}{Q_{\sigma}(i)}]$ is injective for any $\sigma\in\mathcal{P}$ and $i\in\mathbb{Z}_{\leqq 0}$ 
by the definition of $J$, 
$\mathrm{Spm}(R/J)$ satisfies the condition (1). 
We show that $\mathrm{Spm}(R/J)$ satisfies (2). 
Let $f:\mathrm{Spm}(R')\rightarrow \mathrm{Spm}(R)$ be a $Y$-small map 
which factors through $\mathrm{Spm}(R')\rightarrow \mathrm{Spm}(R)_{Q_{\sigma}(i)}$ for 
any $\sigma\in \mathcal{P}$ and $i\in\mathbb{Z}_{\leqq 0}$. 
In this situation, we first prove that (ii) implies (i). We assume that any $G_K$-equivariant map $h:M^{\vee}\rightarrow \bold{B}^+_{\mathrm{dR}}\hat{\otimes}_{K,\sigma}R'$ factors through 
$M^{\vee}\rightarrow (\bold{B}^+_{\mathrm{max},K}\hat{\otimes}_{K,\sigma}R')^{\varphi_K=Y}$ for any $\sigma\in\mathcal{P}$. 
Then, for any $h:M^{\vee}\rightarrow \bold{B}^+_{\mathrm{dR}}\hat{\otimes}_{K,\sigma}R$, $\lambda\in E'$ and 
$x\in \widetilde{\bold{E}}^+$ as in the construction of the ideal $I\subseteq R$, the map
$$P(x,\frac{Y}{\sigma(\pi_K)\lambda})h\otimes_R\mathrm{id}_{R'}:
M^{\vee}\rightarrow U_{k,\sigma,\lambda}\hat{\otimes}_{E'}(R'\otimes_EE')$$ is zero  because the multiplication by $P(x,\frac{Y}{\sigma(\pi_K)\lambda})$ sends $(\bold{B}^+_{\mathrm{max},K}\hat{\otimes}_{K,\sigma}R)^{\varphi=Y}$ to $(\bold{B}^+_{\mathrm{max},K}\hat{\otimes}_{K,\sigma}R\otimes_EE')^{\varphi_K=\sigma(\pi_K)\lambda}$. Hence, the map $R\rightarrow R'$ factors through $R/I\rightarrow R'$ by the definition of $I$. Because $Q_{\sigma}(i)\in R^{' \times}$ for any $\sigma\in\mathcal{P}$, $i\in\mathbb{Z}_{\leqq 0}$, the map $R/I\rightarrow R'$ factors through $R/J\rightarrow R'$ by the definition of $J$.

We next prove that (i) implies (ii). Assume that $f:\mathrm{Spm}(R')\rightarrow \mathrm{Spm}(R)$ factors through $\mathrm{Spm}(R')
\rightarrow \mathrm{Spm}(R/J)\rightarrow \mathrm{Spm}(R)$. Let $h:M^{\vee}\rightarrow \bold{B}^+_{\mathrm{dR}}\hat{\otimes}_{K,\sigma}R'$ be a $R'$-linear 
$G_K$-equivariant map. We want to show that the map $h$ factors through $M^{\vee}\rightarrow (\bold{B}^+_{\mathrm{max},K}\hat{\otimes}_{K,\sigma}R')^{\varphi_K=Y}$. 
By Galois descent, it suffices to show that the map $h$ factors through $M^{\vee}\rightarrow (\bold{B}^+_{\mathrm{max},K}\hat{\otimes}_{K,\sigma}R'\otimes_EE')^{\varphi_K=Y}$ for a sufficiently large finite Galois extension $E'$ of $E$. Hence, by the definition of $Y$-smallness, we may assume that there exists $\lambda\in E'$ such that $Y\lambda^{-1}-1$ is topologically nilpotent in $R'\otimes_EE'$. Moreover, because the definitions of $I$ and $J$ are compatible with 
any base change $R\mapsto R\otimes_EE'$, we may assume that $E=E'$ and $\lambda\in E$. Under these assumptions, we have $|Y^{-1}|^{-1}\leqq |f^*(Y)^{-1}|_{R'}^{-1}=|\lambda|_p=|f^*(Y)|_{R'}\leqq |Y|$ ($|-|_{R'}$ is a norm on $R'$), hence $\lambda$ satisfies the condition in the construction of $I\subseteq R$. 
By the definition of $I$, for any $m\in M^{\vee}$, $P(x,\frac{Y}{\sigma(\pi_K)\lambda})h(m)$ is an element 
in $(\bold{B}^+_{\mathrm{max},K}\hat{\otimes}_{K,\sigma}R')^{\varphi_K=\sigma(\pi_K)\lambda}$ for any 
$x\in \widetilde{\bold{E}}^+$ such that $v(x)>0$. Take an element $u\in (\widehat{K^{\mathrm{ur}}}\hat{\otimes}_{K,\sigma}R')^{\times, \varphi_K=\frac{\lambda}{Y}}$ as in Lemma \ref{5.6}. 
 Then we have
$$t_Kuh(m)\in (\bold{B}^+_{\mathrm{max},K}\hat{\otimes}_{K,\sigma}R')^{\varphi_K=
\sigma(\pi_K)\lambda}$$ 
because we have $t_Ku\in (\bold{B}^+_{\mathrm{max},K}\hat{\otimes}_{K,\sigma}R')^{\varphi_K=\frac{\sigma(\pi_K)\lambda}{Y}}$, and because
the $R'$-module generated by the sets $\{P(x,\frac{Y}{\sigma(\pi_K)\lambda})\}_{x\in \widetilde{\bold{E}}^+, v(x)>0}$ is dense in $(\bold{B}^+_{\mathrm{max},K}\hat{\otimes}_{K,\sigma}R')^{\varphi_K=\frac{\sigma(\pi_K)\lambda}{Y}}$, which can be proved in the same way as in Corollary 
4.6 of \cite{Ki03} by using Lemma 4.3.1 of \cite{Ke05}, and because $(\bold{B}^+_{\mathrm{max},K}\hat{\otimes}_{K,\sigma}R')^{\varphi_K=\sigma(\pi_K)\lambda}$ is closed in $\bold{B}^+_{\mathrm{dR}}/t^k\bold{B}^+_{\mathrm{dR}}\hat{\otimes}_{K,\sigma}R'$ by Proposition \ref{5.7}. Hence, we obtain
\[
\begin{array}{ll}
uh(m)&\in \frac{1}{t_K}(\bold{B}^+_{\mathrm{max},K}\hat{\otimes}_{K,\sigma}R')^{\varphi_K=\sigma(\pi_K)\lambda}\cap\bold{B}^+_{\mathrm{dR}}/t^k\bold{B}^+_{\mathrm{dR}}\hat{\otimes}_{K,\sigma}R'\\
&=(\bold{B}^+_{\mathrm{max},K}\hat{\otimes}_{K,\sigma}R')^{\varphi_K=\sigma(\pi_K)\lambda},
 \end{array}
 \] where the last equality follows from Proposition 
8.10 of \cite{Co02}. Hence, we obtain $h(m)\in (\bold{B}^+_{\mathrm{max},K}\hat{\otimes}_{K,\sigma}\allowbreak R')^{\varphi_K=Y}$, which proves the lemma, hence finishes to prove the theorem.

\end{proof}

We next prove some important general 
properties of $X_{fs}$, which is a generalization of Corollary 5.16 of \cite{Ki03} for general $K$.

Let $\mathrm{Spm}(R)\subseteq X_{fs}$ be an affinoid open of $X_{fs}$. We assume that 
this inclusion is $Y$-small. By Proposition \ref{5.7}, there exists $k>0$ such that, for any $\sigma\in\mathcal{P}$, there exists a short 
exact sequence of Banach $R$-modules with the property (Pr)
$$0\rightarrow (\bold{B}^+_{\mathrm{max},K}\hat{\otimes}_{K,\sigma}R)^{\varphi_K=Y}
\rightarrow \bold{B}^+_{\mathrm{dR}}/t^k\bold{B}^+_{\mathrm{dR}}\hat{\otimes}_{K,\sigma}R\rightarrow U_{k,\sigma}\rightarrow 0.$$ 
We denote by $M_R$ the restriction of $M$ to $\mathrm{Spm}(R)$.
\begin{prop}\label{5.13}
Fix $k\geqq 1$ as above. For any $\sigma\in \mathcal{P}$, let $H_{\sigma}\subseteq R$ be the smallest  ideal of $R$ such that 
any $R$-linear $G_K$-equivariant morphism 
$M^{\vee}_R\rightarrow \bold{B}^+_{\mathrm{dR}}/t^k\bold{B}^+_{\mathrm{dR}}\hat{\otimes}_{K,\sigma}R$ factors through 
$M^{\vee}_R\rightarrow \bold{B}^+_{\mathrm{dR}}/t^k\bold{B}^+_{\mathrm{dR}}\hat{\otimes}_{K,\sigma}H_{\sigma}\rightarrow \bold{B}^+_{\mathrm{dR}}/t^k\bold{B}^+_{\mathrm{dR}}\hat{\otimes}_{K,\sigma}R$. Set
$H:=\prod_{\sigma\in \mathcal{P}}H_{\sigma}\subseteq R$.
Then the following hold:
\begin{itemize}
\item[(1)]For any $\sigma\in \mathcal{P}$, the natural map $$((\bold{B}^+_{\mathrm{max},K}\hat{\otimes}_{K,\sigma}M_R)^{\varphi_K=Y})^{G_K}\rightarrow 
(\bold{B}^+_{\mathrm{dR}}/t^k\bold{B}^+_{\mathrm{dR}}\hat{\otimes}_{K,\sigma}M_R)^{G_K}$$ is isomorphism, i.e. the natural map $$K\otimes_{K_0}\bold{D}^+_{\mathrm{cris}}(M_R)^{\varphi_K=Y}\rightarrow 
(\bold{B}^+_{\mathrm{dR}}/t^k\bold{B}^+_{\mathrm{dR}}\hat{\otimes}_{\mathbb{Q}_p}\allowbreak M_R)^{G_K}$$ is 
isomorphism. 
\item[(2)] $\mathrm{Spm}(R)\setminus V(H)$ and $\mathrm{Spm}(R)\setminus V(H_{\sigma})$ ( for any $\sigma\in \mathcal{P}$) are  scheme theoretically dense in $\mathrm{Spm}(R)$, where $V(H_{*}):=\mathrm{Spm}(R/H_{*})$.
\item[(3)]For any $x\in \mathrm{Spm}(R)$, $M(x)$ is a split trianguline $E(x)$-representation. More precisely, there exists a short exact sequence of $E(x)$-$B$-pairs
$$0\rightarrow W(\delta_{Y(x)}\prod_{\sigma\in \mathcal{P}}\sigma^{-k_{\sigma}})\rightarrow W(M(x))\rightarrow 
W(\mathrm{det}(M(x))\delta^{-1}_{Y(x)}\prod_{\sigma\in \mathcal{P}}\sigma^{k_{\sigma}})\rightarrow 0$$
for some $\{k_{\sigma}\}_{\sigma\in \mathcal{P}}\in \prod_{\sigma\in \mathcal{P}}\mathbb{Z}_{\geqq 0}$, where, for any $\lambda\in E(x)^{\times}$, we define a homomorphism $\delta_{\lambda}:K^{\times}\rightarrow E(x)^{\times}$ such 
that $\delta_{\lambda}(\pi_K)=\lambda$ and $\delta_{\lambda}|_{\mathcal{O}_K^{\times}}$ is trivial.
\end{itemize}

\end{prop}
\begin{proof}
We first prove (1). Take a point $x\in \mathrm{Spm}(R)$ such that $x\in \mathrm{Spm}(R)_{Q_{\sigma}(i)}$ for any 
$\sigma\in \mathcal{P}$ and $i\in\mathbb{Z}_{\leqq 0}$. By the characterization of $X_{fs}$, any $G_K$-map $M^{\vee}_R\rightarrow \bold{B}^+_{\mathrm{dR}}/t^k\bold{B}^+_{\mathrm{dR}}\otimes_{K,\sigma}\mathcal{O}_{X,x}/\mathfrak{m}_x^n$ 
factor through $M^{\vee}_R\rightarrow (\bold{B}^+_{\mathrm{max},K}\otimes_{K,\sigma}\mathcal{O}_{X,x}/\mathfrak{m}_x^n)^{\varphi_K=Y}$ 
for any $n\geqq 1$ and $\sigma\in \mathcal{P}$. 
We denote by $V\subseteq \mathrm{Spm}(R)$ the set of points satisfying the above condition.
By the same argument as in Lemma \ref{5.11},  it suffices to show that the natural 
map $R\rightarrow \prod_{x\in V,n\geqq 1}\mathcal{O}_{X,x}/\mathfrak{m}_{x}^n$ is an injection.
Let $f\in R$ be an element in the kernel of this map. Let $W\subseteq \mathrm{Spm}(R)$ be 
the support of $f$ with the reduced structure. Then we have $W \subseteq \cup_{\sigma\in \mathcal{P},i\leqq 0}
V(Q_{\sigma}(i))$, hence there exists $Q\in R$ a finite product of $Q_{\sigma}(i)$ such that 
$W\subseteq V(Q)$ by Lemma 5.7 of \cite{Ki03}. Hence we have  $X_Q\subseteq X\setminus W\subseteq X$. This implies that $f=0\in R[\frac{1}{Q}]$, and $f=0$ in $R$ by 
the condition (1) of Theorem \ref{5.8}. Hence, the map $R\rightarrow \prod_{x\in V,n\geqq 1}\mathcal{O}_{X,x}/\mathfrak{m}_{x}^n$ is an injection.

We next prove (2).
Let $x\in \mathrm{Spm}(R)$ such that $x\in \mathrm{Spm}(R)_{Q_{\sigma}(i)}$ for any 
$\sigma\in \mathcal{P}$ and $i\in\mathbb{Z}_{\leqq 0}$. For any affinoid algebra $R'$ 
which is a quotient of $\mathcal{O}_{X,x}$, then we have an isomorphism $$(\bold{B}^+_{\mathrm{dR}}/t^k\bold{B}^+_{\mathrm{dR}}\hat{\otimes}_{K,\sigma}
M_R)^{G_K}\otimes_RR'\isom 
(\bold{B}^+_{\mathrm{dR}}/t^k\bold{B}^+_{\mathrm{dR}}\otimes_{K,\sigma}(M_R\otimes_R R'))^{G_K}$$ and this is a free $R'$-module of rank one 
by Corollary 2.6 of \cite{Ki03} for any $\sigma\in\mathcal{P}$. 
Hence we obtain an equality $H_{\sigma}\mathcal{O}_{X,x}=
\mathcal{O}_{X,x}$ for any $\sigma$ because we have $(\bold{B}^+_{\mathrm{dR}}/t^k\bold{B}^+_{\mathrm{dR}}\hat{\otimes}_{K,\sigma}M_R)^{G_K}\otimes_R\mathcal{O}_{X,x}/H_{\sigma}\mathcal{O}_{X,x}=0$  by the definition of $H_{\sigma}$, and then we also obtain an 
equality $H\mathcal{O}_{X,x}=\mathcal{O}_{X,x}$. This implies that we have an inclusion 
$V(H)\subseteq \cup_{\sigma\in \mathcal{P},i\leqq 0}V(Q_{\sigma}(i))$. Hence, 
there exists $Q\in R$ a finite product of $Q_{\sigma}(i)$ such that $\mathrm{Spm}(R)_Q\subseteq \mathrm{Spm}(R)\setminus V(H)
\subseteq \mathrm{Spm}(R)$ by Lemma 5.7 
of \cite{Ki03}. Because $\mathrm{Spm}(R)_{Q}$ is scheme theoretically dense, so
$\mathrm{Spm}(R)\setminus V(H)$ is also scheme theoretically dense. Because we have $V(H_{\sigma})\subseteq V(H)$, 
$\mathrm{Spm}(R)\setminus V(H_{\sigma})$ is also scheme theoretically dense for any $\sigma\in\mathcal{P}$.

Finally we prove (3). Let $x$ be any point of $\mathrm{Spm}(R)$. By (2), for any $\sigma\in \mathcal{P}$, there exists $n_{\sigma}\geqq 0$ such that $H_{\sigma}\subseteq \mathfrak{m}_x^{n_{\sigma}}$ and 
$H_{\sigma}\not\subseteq \mathfrak{m}_x^{n_{\sigma}+1}$. By the definition of $H_{\sigma}$, there exists a $G_K$-map $h:M^{\vee}_R\rightarrow (\bold{B}^+_{\mathrm{max},K}\hat{\otimes}_{K,\sigma}H_{\sigma})^{\varphi_K=Y}
$ which, by composing with 
$\bold{B}^+_{\mathrm{max},K}\hat{\otimes}_{K,\sigma}H_{\sigma}\rightarrow \bold{B}^+_{\mathrm{max},K}\otimes_{K,\sigma}\mathfrak{m}_x^{n_{\sigma}}/\mathfrak{m}_x^{n_{\sigma}+1}$, induces a nonzero map $M^{\vee}_R\rightarrow (\bold{B}^+_{\mathrm{max},K}\otimes_{K,\sigma}\mathfrak{m}_x^{n_{\sigma}}/\mathfrak{m}_x^{n_{\sigma}+1})^{\varphi_K=Y(x)}$. Hence, by taking a 
suitable $E(x)$-linear projection $\mathfrak{m}_x^{n_{\sigma}}/\mathfrak{m}_x^{n_{\sigma}+1}\rightarrow E(x)$, we obtain 
a non zero $G_K$-map $M^{\vee}_R\rightarrow (\bold{B}^+_{\mathrm{max},K}\otimes_{K,\sigma}E(x))^{\varphi_K=Y(x)}$. 
This implies that $(\bold{B}^+_{\mathrm{max},K}\otimes_{K,\sigma}M(x))^{\varphi_K=Y(x)}\not= 0$, and also implies that $\bold{D}^+_{\mathrm{cris}}(M(x))^{\varphi_K=Y(x)}\not= 0$, then $M(x)$ is a split trianguline 
$E(x)$-representation as in the statement of (3).

\end{proof}

\subsection{Construction of $p$-adic families of two dimensional trianguline representations}

In this subsection, we will apply our theory of $X_{fs}$ in the previous subsection to the rigid analytic space associated with 
a universal deformation ring of mod $p$ Galois representation of $G_K$, which is a slightly modified generalization of the results of $\S$10 of \cite{Ki03} for general $K$.

Let $\mathcal{C}_{\mathcal{O}}$ be the category of local Artin $\mathcal{O}$-algebras with the residue field $\mathbb{F}$.
Let $\bar{\rho}:G_K\rightarrow \mathrm{GL}_2(\mathbb{F})$ be a continuous homomorphism, we denote by $\overline{V}$ a two dimensional $\mathbb{F}$-representation defined by $\bar{\rho}$. 
 As in the case of $E$-representations, we define a functor 
 $D_{\bar{\rho}}:\mathcal{C}_{\mathcal{O}}\rightarrow Sets$ by 
 $$D_{\bar{\rho}}(A):=\{\text{ equivalent classes of deformations of }\overline{V}\text{ over }A \}$$
  for $A\in\mathcal{C}_{\mathcal{O}}$.
 In this paper, for simplicity, we assume that $\overline{V}$ satisfies that $$\mathrm{H}^0(G_K,\mathrm{ad}(\overline{V}))=\mathbb{F}.$$
 Then, $D_{\bar{\rho}}$ is pro-representable by a complete noetherian local $\mathcal{O}$-algebra $R_{\bar{\rho}}$ with the residue 
 field $\mathbb{F}$. 
 When $\overline{V}$ does not satisfy $\mathrm{H}^0(G_K,\mathrm{ad}(\overline{V}))=\mathbb{F}$, we can prove 
 the same theorems below in almost the same way if we consider the framed deformations of Kisin.
  Let $V^{\mathrm{univ}}$ be the universal deformation over $R_{\bar{\rho}}$, which is a free $R_{\bar{\rho}}$-module of rank two with a $R_{\bar{\rho}}$-linear continuous 
  $G_K$-action. 
 Let $\mathfrak{X}(\bar{\rho})$ be the rigid analytic space over $E$ associated 
 to $R_{\bar{\rho}}$. Let $\widetilde{V}^{\mathrm{univ}}$ be a free $\mathcal{O}_{\mathfrak{X}(\bar{\rho})}$-module associated to $V^{\mathrm{univ}}$, which is 
 naturally equipped with an $\mathcal{O}_{\mathfrak{X}(\bar{\rho})}$-linear continuous $G_K$-action induced from that on $V^{\mathrm{univ}}$, where ``continuous" means that $G_K$ acts continuously on $\Gamma(U, \widetilde{V}^{\mathrm{univ}})$ for any  affinoid opens $U=\mathrm{Spm}(R)\subseteq \mathfrak{X}(\bar{\rho})$. 
 \begin{rem}\label{a1}
 For a point $x\in \mathfrak{X}(\bar{\rho})$, the fiber $V_x$ of $\widetilde{V}^{\mathrm{univ}}$ at $x$ is a  two dimensional  
 $E(x)$-representation such that the reduction of a 
 $G_K$-stable $\mathcal{O}_{E(x)}$-lattice of $V_x$ is isomorphic to 
 $\overline{V}\otimes_{\mathbb{F}}\mathcal{O}_{E(x)}/\mathfrak{m}_{E(x)}$. Because we assume that $\mathrm{End}_{\mathbb{F}[G_K]}(\bar{\rho})=\mathbb{F}$, 
  we also have 
 $\mathrm{End}_{E(x)[G_K]}(V_x)=E(x)$ for any $x\in \mathfrak{X}(\bar{\rho})$.
 \end{rem}
 Let $\mathcal{W}_E$ be the rigid analytic space over $E$ which represents the functor $D_{\mathcal{W}_E}$ from 
 the category of rigid analytic spaces over $E$ to the category of groups, which is defined by
 $$D_{\mathcal{W}_E}(Y):=\{\delta: \mathcal{O}_K^{\times}\rightarrow \Gamma(Y,\mathcal{O}_Y^{\times}) \text{ continuous homomorphisms }\}$$ 
 for any rigid analytic space $Y$ over $E$, where ``continuous" is the same 
 meaning as in the definition of $\widetilde{V}^{\mathrm{univ}}$. 
 It is known that $\mathcal{W}_E$ is the rigid analytic space associated to the Iwasawa algebra $\mathcal{O}[[\mathcal{O}_{K}^{\times}]]$, which is non-canonically isomorphic to a finite (this number is equal to the number of torsion points in $\mathcal{O}_K^{\times}$) union of $[K:\mathbb{Q}_p]$-dimensional 
 open unit disc over $E$. We denote by $$\delta^{\mathrm{univ}}_0:\mathcal{O}_K^{\times}\rightarrow \Gamma(\mathcal{W}_E,\mathcal{O}_{\mathcal{W}_E}^{\times})$$ the universal 
 continuous homomorphism, which is the composition of the map $\mathcal{O}_K^{\times}\rightarrow 
 \mathcal{O}[[\mathcal{O}_K^{\times}]]^{\times}:a\mapsto [a]$ with the natural map $\mathcal{O}[[\mathcal{O}_K^{\times}]]^{\times}\rightarrow \Gamma(\mathcal{W}_E, \mathcal{O}_{\mathcal{W}_E}^{\times})$. Using a fixed $\pi_K$, we extend $\delta^{\mathrm{univ}}_0$ to $K^{\times}$ by 
 $$\delta^{\mathrm{univ}}:K^{\times}\rightarrow \Gamma(\mathcal{W}_E, \mathcal{O}_{\mathcal{W}_E}^{\times})\text{ such 
 that }\delta^{\mathrm{univ}}(\pi_K)=1, \, \delta^{\mathrm{univ}}|_{\mathcal{O}_K^{\times}}=\delta^{\mathrm{univ}}_0.$$ By local class field theory, we can uniquely extend $\delta^{\mathrm{univ}}$ to a character 
 $$\widetilde{\delta}^{\mathrm{univ}}:G^{\mathrm{ab}}_K\rightarrow \Gamma(\mathcal{W}_E, \mathcal{O}_{\mathcal{W}_E}^{\times})\text{ such that }\delta^{\mathrm{univ}}=\widetilde{\delta}^{\mathrm{univ}}\circ \mathrm{rec}_K.$$
 Set
 $$X(\bar{\rho}):=\mathfrak{X}(\bar{\rho})\times_E \mathcal{W}_E\times_E\mathbb{G}^{an}_{m,E}.$$ 
 Let $Y$ be the canonical parameter of $\mathbb{G}^{an}_{m,E}$.
 We denote the projections by $$p_1:X(\bar{\rho})\rightarrow \mathfrak{X}(\bar{\rho}),\, 
 p_2:X(\bar{\rho})\rightarrow \mathcal{W}_E,\, p_3:X(\bar{\rho})\rightarrow \mathbb{G}^{an}_{m,E}$$
 respectively. 
 We denote by  $N:=p_1^*\widetilde{V}^{\mathrm{univ}}$ and $M:=N(p_2^*(\widetilde{\delta}^{\mathrm{univ}})^{-1})$, which is the twist of $M$ by the cahacter $p_2^*(\widetilde{\delta}^{\mathrm{univ}})^{-1}:G_K^{\mathrm{ab}}\rightarrow \Gamma(X(\bar{\rho}),\mathcal{O}_{X(\bar{\rho})})^{\times}$.
 These are rank two free $\mathcal{O}_{X(\bar{\rho})}$-modules with $\mathcal{O}_{X(\bar{\rho})}$-linear continuous $G_K$-actions. 
 \begin{rem}
 In $\S$10 of \cite{Ki03}, Kisin applied his theory of $X_{fs}$ (for $K=\mathbb{Q}_p$) to the family $q_1^{*}\widetilde{V}^{\mathrm{univ}}$ on the space $Y(\bar{\rho}):=\mathfrak{X}(\bar{\rho})\times_E\mathbb{G}^{an}_{m,E}$, where $q_1:Y(\bar{\rho})\rightarrow\mathfrak{X}(\bar{\rho})$ is the natural projection. This is because he applied the results to a study of the family of $p$-adic representations associated to Coleman-Mazur eigencurve, one of whose Hodge-Tate weights is always zero. On the other hands, in this article, we want to study all the two dimensional trianguline representations without any conditions on the Hodge-Tate weights. Hence, we use the space $X(\bar{\rho})$ and the representation $M:=N(p_2^*(\widetilde{\delta}^{\mathrm{univ}})^{-1})$ instead of $Y(\bar{\rho})$ and 
 $q_1^{*}\widetilde{V}^{\mathrm{univ}}$.
 
 \end{rem}
 A point $x$ of $X(\bar{\rho})$ can be written as a triple $x=([V_x],\delta_x,\lambda_x)$, where
 $V_x$ is an $E(x)$-representation such that the reduction of a suitable $G_K$-stable 
 $\mathcal{O}_{E(x)}$-lattice of $V_x$ is isomorphic to $\bar{V}\otimes_{\mathcal{O}_{E(x)}}\mathcal{O}_{E(x)}/\mathfrak{m}_{E(x)}$, and $\delta_x:\mathcal{O}_K^{\times}\rightarrow 
 E(x)^{\times}$ is a continuous homomorphism, and $\lambda_x\in E(x)^{\times}$.
 We denote by 
 $$P_M(T)=(P_M(T)_{\sigma})_{\sigma\in \mathcal{P}}=(T^2-a_{1,\sigma}T+a_{0,\sigma})_{\sigma\in \mathcal{P}}
 \in K\otimes_{\mathbb{Q}_p}\mathcal{O}_{X(\bar{\rho})}[T]=\prod_{\sigma\in\mathcal{P}}\mathcal{O}_{X(\bar{\rho})}[T]$$ the Sen's polynomial 
 of $M$. Let $X_0(\bar{\rho})\subseteq X(\bar{\rho})$ be the Zariski closed subspace defined by the ideal generated by 
 $a_{0,\sigma}$ for all $\sigma\in\mathcal{P}$. Let $M_0:=M|_{X_0(\bar{\rho})}$ be the restriction of $M$ to $X_0(\bar{\rho})$,
 then we have $$P_{M_0}(T)=(T(T-a_{1,\sigma}))_{\sigma\in \mathcal{P}}\in \prod_{\sigma\in\mathcal{P}}\mathcal{O}_{X_0(\bar{\rho})}[T].$$ 
 We denote by $Q_{\sigma}(T):=T-a_{1,\sigma}\in\mathcal{O}_{X_0(\bar{\rho})}[T]$ for each $\sigma\in\mathcal{P}$.
 Under this situation, we apply Theorem \ref{5.8} to $X_0(\bar{\rho})$ and $M_0$ and $Y:=(p_3^*Y)|_{X_0(\bar{\rho})}$, then we obtain a Zariski closed 
 subspace 
 $$\mathcal{E}(\bar{\rho}):=X_0(\bar{\rho})_{fs}\subseteq X_0(\bar{\rho}).$$ 
 For the properties of $\mathcal{E}(\bar{\rho})$, we have a following theorem, which is a modified generalization of Proposition 10.4 of \cite{Ki03} for general $K$.
 For any $\lambda\in \overline{E}^{\times}$, we define a unramified continuous homomorphism 
 $\delta_{\lambda}:K^{\times}\rightarrow \overline{E}^{\times}$ such that $\delta_{\lambda}(\pi_K):=\lambda$ and 
 $\delta_{\lambda}|_{\mathcal{O}_K^{\times}}$ is trivial. For a point $\delta\in \mathcal{W}_E(\overline{E})$, i.e. for a
 continuous homomorphism $\delta:\mathcal{O}_K^{\times}\rightarrow \overline{E}^{\times}$, we denote by the same letter 
 $\delta:K^{\times}\rightarrow \overline{E}^{\times}$ the homomorphism such that $\delta(\pi_K)=1$ and $\delta|_{\mathcal{O}_K^{\times}}=\delta$.

 \begin{thm}\label{5.14}
 \begin{itemize}
 \item[(1)]For any point $x:=([V_x],\delta_x,\lambda_x)\in \mathcal{E}(\bar{\rho})$, 
 there exist $\{k_{\sigma}\}_{\sigma\in \mathcal{P}}\in \prod_{\sigma\in \mathcal{P}}\mathbb{Z}_{\geqq 0}$ and a short exact sequence of $E(x)$-$B$-pairs
 $$0\rightarrow W(\delta_1)\rightarrow W(V_x)\rightarrow W(\mathrm{det}(V_x)\delta_1^{-1})\rightarrow 0$$ for $\delta_1:=\delta_x\delta_{\lambda_x}\prod_{\sigma\in \mathcal{P}}\sigma^{-k_{\sigma}}$.

 \item[(2)]Conversely, if a point $x:=([V_x],\delta_x,\lambda_x)\in X(\bar{\rho})$ satisfies the following conditions (i) and (ii), 
  \begin{itemize}
 \item[(i)] $V_x$ is a split trianguline $E(x)$-representation
  with a triangulation
  $$\mathcal{T}_x:0\subseteq W(\delta_x\delta_{\lambda_x})\subseteq W(V_x),$$
 \item[(ii)]$(V_x,\mathcal{T}_x)$ satisfies all the assumptions in Proposition $\mathrm{\ref{19}}$,
 \end{itemize}
  then we have $x\in \mathcal{E}(\bar{\rho})$.

\end{itemize}
\end{thm}
\begin{proof}
The property (1) follows from (3) of Proposition \ref{5.13}.

We prove (2). Extending the scalars from $E$ to $E(x)$, we may assume that $E(x)=E$. Let $x:=([V_x],\delta_x,\lambda_x)\in X(\bar{\rho})$ be an $E$-rational  point satisfying the conditions  (i), (ii) in (2).
Then, the trianguline deformation functor $D_{V_x,\mathcal{T}_x}$ is representable by 
a formally smooth quotient $R_{V_x,\mathcal{T}_x}$ of the universal deformation ring $R_{V_x}$ of $V_x$ by Proposition \ref{19}. Moreover, we have a canonical isomorphism $R_{V_x}\isom \widehat{\mathcal{O}}_{\mathfrak{X}(\bar{\rho}),p_1(x)}$,  
and $V_x^{\mathrm{univ}}:=\widetilde{V}^{\mathrm{univ}}\otimes_{\mathcal{O}_{\mathfrak{X}(\bar{\rho})}}\widehat{\mathcal{O}}_{\mathfrak{X}(\bar{\rho}),p_1(x)}$ is 
the universal deformation of $V_x$ by Proposition 9.5 of \cite{Ki03}. Taking a quotient, we obtain a map 
$$\widehat{\mathcal{O}}_{X(\bar{\rho}),x}
\isom \widehat{\mathcal{O}}_{\mathfrak{X}(\bar{\rho}),p_1(x)}\hat{\otimes}_{E}\widehat{\mathcal{O}}_{\mathcal{W}_E,p_2(x)}\hat{\otimes}_{E}
\widehat{\mathcal{O}}_{\mathbb{G}^{an}_{m,E},p_3(x)}\rightarrow R_{V_x,\mathcal{T}_x}\hat{\otimes}_{E}\widehat{\mathcal{O}}_{\mathcal{W}_E,p_2(x)}\hat{\otimes}_{E}\widehat{\mathcal{O}}_{\mathbb{G}^{an}_{m,E},p_3(x)}.$$ By the definition of $R_{V_x,\mathcal{T}_x}$, there exists 
a continuous homomorphism $\delta_{\mathcal{T}_x}:K^{\times}\rightarrow R^{\times}_{V_x,\mathcal{T}_x}$ which gives the universal 
triangulation, i.e. we have the following compatible triangulation 
$$\mathcal{T}_{\mathrm{univ},n}: 0\subseteq W(\delta_{\mathcal{T}_x,n})\subseteq W(V_x^{\mathrm{univ}}\otimes_{R_{V_x}}R_{V_x,\mathcal{T}_x}/\mathfrak{m}^n)$$ of $V_x^{\mathrm{univ}}\otimes_{R_{V_x}}R_{V_x,\mathcal{T}_x}/\mathfrak{m}^n$ for each $n\geqq 1$, 
where $\mathfrak{m}$ is the maximal ideal of $R_{V_x,\mathcal{T}_x}$ and $\delta_{\mathcal{T}_x,n}$ is the composition of 
$\delta_{\mathcal{T}_x}$ with the natural quotient map $R_{V_x,\mathcal{T}_x}\rightarrow R_{V_x,\mathcal{T}_x}/\mathfrak{m}^n$. Set $\lambda_{\mathcal{T}_x}:=\delta_{\mathcal{T}_x}(\pi_K)\in R_{V_x,\mathcal{T}_x}^{\times}$. Denote by $\delta^{\mathrm{univ}}_{p_2(x)}:\mathcal{O}_K^{\times}\rightarrow \widehat{\mathcal{O}}_{\mathcal{W}_E,p_2(x)}^{\times}$ 
the composition 
of the universal 
homomorphism $\delta_0^{\mathrm{univ}}:\mathcal{O}_K^{\times}\rightarrow \Gamma(\mathcal{W}_E,
\mathcal{O}_{\mathcal{W}_E})^{\times}$ with the natural map
$\Gamma(\mathcal{W}_E,\mathcal{O}_{\mathcal{W}_E})^{\times}\rightarrow \widehat{\mathcal{O}}_{\mathcal{W}_E,p_2(x)}^{\times}$.  Then, the $E$-algebra 
$\widehat{\mathcal{O}}_{\mathcal{W}_E,p_2(x)}$ is topologically generated by $\{\delta^{\mathrm{univ}}_{p_2(x)}(a)-\delta_x(a) |a\in\mathcal{O}_K^{\times} \}$. 
Denote by $\overline{R}$ a quotient of $R_{V_x,\mathcal{T}_x}\hat{\otimes}_{E}\widehat{\mathcal{O}}_{\mathcal{W}_E,p_2(x)}\hat{\otimes}_{E}
\widehat{\mathcal{O}}_{\mathbb{G}^{an}_{m,E},p_3(x)}$ by the ideal generated by $\delta_{\mathcal{T}_x}(a)\otimes 1\otimes 1-1\otimes\delta^{\mathrm{univ}}_{p_2(x)}(a)\otimes 1$ (any $a\in \mathcal{O}_K^{\times}$) 
 and $\lambda_{\mathcal{T}_x}\otimes 1\otimes 1-1\otimes 1\otimes Y$.
 Then, we can see that 
the composition of  the map $R_{V_x,\mathcal{T}_x}\rightarrow R_{V_x,\mathcal{T}_x}\hat{\otimes}_{E(x)}\widehat{\mathcal{O}}_{\mathcal{W}_E,p_2(x)}
\hat{\otimes}\widehat{\mathcal{O}}_{\mathbb{G}^{an}_{m,E},E(x)}: z\mapsto z\otimes 1\otimes 1$ with the natural quotient map 
$R_{V_x,\mathcal{T}_x}\hat{\otimes}_{E(x)}\widehat{\mathcal{O}}_{\mathcal{W}_E,p_2(x)}
\hat{\otimes}_{E(x)}\widehat{\mathcal{O}}_{\mathbb{G}^{an}_{m,E},p_3(x)}\rightarrow \overline{R}$ is an isomorphism 
$R_{V_x,\mathcal{T}_x}\isom \overline{R}$, and, if we denote by $\overline{\delta}:\mathcal{O}_K^{\times}\rightarrow \overline{R}^{\times}$ 
and $\overline{Y}\in \overline{R}^{\times}$ the reduction of $1\otimes \delta^{\mathrm{univ}}_{p_2(x)}\otimes 1$ and $1\otimes 1\otimes Y$, then the universal triangulation $\mathcal{T}_{\mathrm{univ}}:=\{\mathcal{T}_{\mathrm{univ},n}\}_{n\geqq 1}$ on $R_{V_x,\mathcal{T}_x}$ 
is transformed to the following triangulation 
$$\overline{\mathcal{T}}:0\subseteq W(\overline{\delta}\delta_{\overline{Y}})\subseteq W((p_1^*\widetilde{V}^{\mathrm{univ}})\otimes_{\mathcal{O}_{X(\bar{\rho})}}\overline{R})$$
(here we drop the notation $n\in \mathbb{Z}_{\geqq 1}$ to simplify the notation).

Put $V_{\overline{R}}:=(p_1^*\widetilde{V}^{\mathrm{univ}})\otimes_{\mathcal{O}_{X(\bar{\rho})}}\overline{R}$, and put 
 $\overline{R}_n:=\overline{R}/\mathfrak{m}^n$ and $V_{\overline{R}_n}:=V_{\overline{R}}\otimes_{\overline{R}}\overline{R}_n$ for each $n\geqq 1$. Denote by the same notation 
 $\overline{\delta}:G_K^{ab}\rightarrow\overline{R}^{\times}$ the character such that $\overline{\delta}|_{\mathcal{O}_K^{\times}}=\overline{\delta}$ and $\overline{\delta}(\mathrm{rec}_K(\pi_K))=1$.
Under this situation, we first claim that  the natural map $\mathrm{Spm}(\overline{R}_n)\rightarrow X(\bar{\rho})$ factors through 
$X_0(\bar{\rho})$ for any $n\geqq 1$. This immediately follows from 
the facts that $W(V_{\overline{R}_n}(\overline{\delta}^{-1}))$ has a triangulation $0\subseteq W(\delta_{\overline{Y}_n})\subseteq 
W(V_{\overline{R}_n}(\overline{\delta}^{ -1}))$ and that $W(\delta_{\overline{Y}_n})$ is crystalline with the Hodge-Tate weight zero, where $\overline{Y}_n\in \overline{R}_n$ is the reduction of $\overline{Y}$. 

This fact also implies that $\bold{D}_{\mathrm{cris}}(W(\delta_{\overline{Y}_n}))$ which is equal to $\bold{D}_{\mathrm{cris}}(W(\delta_{\overline{Y}_n}))^{\varphi_K=\overline{Y}_n}\cap \mathrm{Fil}^0\bold{D}_{dR}(W(\delta_{\overline{Y}_n}))$ is a $\varphi$-stable
 $K_0\otimes_{\mathbb{Q}_p}\overline{R}_n$-submodule of $\bold{D}_{\mathrm{cris}}(V_{\overline{R}_n}(\overline{\delta}^{ -1}))^{\varphi_K=\overline{Y}_n}$ of rank one
contained in $\mathrm{Fil}^0\bold{D}_{\mathrm{dR}}(V_{\overline{R}_n}(\overline{\delta}^{-1}))$. 
By Lemma \ref{5.7.5}, then we have a natural inclusion 
\begin{equation}\label{ab}
\bold{D}_{\mathrm{cris}}(W(\delta_{\overline{Y}_n})) \subseteq \bold{D}^+_{\mathrm{cris}}
(V_{\overline{R}_n}(\overline{\delta}^{-1}))^{\varphi_K=\overline{Y}_n}.
\end{equation}

Next, we take an affinoid open $\mathrm{Spm}(R)\subseteq X_0(\bar{\rho})$ which 
contains $x$ and 
satisfies the condition in the construction of $X_{fs}$ (see the paragraph after the proof of Lemma \ref{5.11}). Let $J$ be 
the ideal of $R$ which defines $\mathrm{Spm}(R)_{fs}$. We claim that the natural map $R\rightarrow \overline{R}$ factors through 
$R/J\rightarrow \overline{R}$, which proves that $x\in \mathcal{E}(\bar{\rho})(E)$ because $x$ is the point corresponding to the kernel of the map $R\rightarrow \overline{R}\rightarrow \overline{R}/\mathfrak{m}$.
By construction of $J$, it suffices to show the following lemma.
\end{proof}

\begin{lemma}\label{5.15}
In the above situation, the following hold:
\begin{itemize}
\item[(i)]For any $k\geqq 1$ and $\sigma\in \mathcal{P}$, the natural map 
$$\varprojlim_n(\bold{B}^+_{\mathrm{max},K}\otimes_{K,\sigma}V_{\overline{R}_n}(\overline{\delta}^{-1}))^{\varphi_K=\overline{Y},G_K}\rightarrow 
\varprojlim_n(\bold{B}^+_{\mathrm{dR}}/t^k\bold{B}^+_{\mathrm{dR}}\otimes_{K,\sigma}V_{\overline{R}_n}(\overline{\delta}^{-1}))^{G_K}$$ is a surjection. 
\item[(ii)]For any $\sigma\in \mathcal{P}$ and $i\in\mathbb{Z}_{\leqq 0}$, $Q_{\sigma}(i)$ is nonzero in $\overline{R}$.
\end{itemize}
\end{lemma}
\begin{proof}
Because $\overline{R}\isom R_{V_x,\mathcal{T}_x}$ is domain, (i) follows from (ii) and from the above inclusion (\ref{ab}) by the same argument as in the proof of Proposition 2.8 of \cite{Ki03}.
We prove (ii). On $\overline{R}$, we have $Q_{\sigma}(T)=T-\bar{a}_{1,\sigma}$, where $\bar{a}_{1,\sigma}\in\overline{R}$ is the image of $a_{1,\sigma}\in \mathcal{O}_{X(\bar{\rho})}$ by the natural map $R\rightarrow \overline{R}$. Hence, $\bar{a}_{1,\sigma}\in \overline{R}$ is the $\sigma$-part of the Hodge-Tate weights
of $\mathrm{det}(V_{\overline{R}})(\overline{\delta}^{-2})$ for any $\sigma\in \mathcal{P}$. 
Set $\delta_0:=\mathrm{det}(V_x)|_{\mathcal{O}_K^{\times}}\cdot\delta_x^{-2}:\mathcal{O}_K^{\times}\rightarrow E^{\times}$. By Lemma \ref{5.16} below, then $\bar{a}_{1,\sigma}\in \overline{R}$ is the image of 
the $\sigma$-part of the Hodge-Tate weight $a^{\mathrm{univ}}_{\sigma}\in R_{\delta_0}$ 
of the universal deformation $\delta_{\delta_0}^{\mathrm{univ}}:\mathcal{O}_K^{\times}
\rightarrow R_{\delta_0}^{\times}$ by the injection $R_{\delta_0}\hookrightarrow R_{V_x,\mathcal{T}_x}\isom \overline{R}$
 induced by a morphism 
$f:D_{V_x,\mathcal{T}_x}\rightarrow D_{\delta_0}$ defined below, 
where the 
injectiveness follows from Lemma \ref{5.16} below. Hence, for any $i\in\mathbb{Z}_{\leqq 0}$,  
we have $Q_{\sigma}(i)=(i-\bar{a}_{1,\sigma})\not=0\in\overline{R}$ by 
Lemma \ref{5.17} below.
\end{proof}
Let $\delta_0:\mathcal{O}_K^{\times}\rightarrow E^{\times}$ be a continuous homomorphism. 
We define a functor $D_{\delta_0}:\mathcal{C}_E\rightarrow Sets$ by 
$$D_{\delta_0}(A):=\{\delta_A:\mathcal{O}_K^{\times}\rightarrow A^{\times}:\text{ continuous homomorphisms }\delta_A \text{ (mod } \mathfrak{m}_A)=\delta_0\}$$
for $A\in\mathcal{C}_E$. It is easy to show that this functor is 
pro-representable by a ring $R_{\delta_0}$ which is isomorphic to $E[[T_1,T_2,\cdots,T_d]]$ for $d:=[K:\mathbb{Q}_p]$.
Let $W$ be a split trianguline $E$-$B$-pair of rank two with a triangulation 
$\mathcal{T}:0\subseteq W(\delta_1)\subseteq W$ such that $W/W(\delta_1)\isom W(\delta_2)$ for some 
continuous homomorphisms $\delta_1,\delta_2:K^{\times}\rightarrow E^{\times}$.
We put $\delta_0:=(\delta_2/\delta_1)|_{\mathcal{O}_K^{\times}}$.
We define a morphism of functors $f:D_{W,\mathcal{T}}\rightarrow D_{\delta_0}$ as follows. 
Let $[(W_A,\mathcal{T}_A)]\in D_{W,\mathcal{T}}(A)$ be an equivalent class of trianguline deformation 
of $(W,\mathcal{T})$ over $A$ with a triangulation $\mathcal{T}_A:0\subseteq W(\delta_{1,A})\subseteq W_A$ such that  
$W_A/W(\delta_{1,A})\isom W(\delta_{2,A})$ for some $\delta_{1,A},\delta_{2,A}:K^{\times}
\rightarrow A^{\times}$, then we define $f$ by 
$$f([(W_A, \mathcal{T}_A)]):=
(\delta_{2,A}/\delta_{1,A})|_{\mathcal{O}_K^{\times}}\in D_{\delta_0}(A).$$

 \begin{lemma}\label{5.16}
Let $W$ be a two dimensional split trianguline $E$-$B$-pair  with a triangulation 
$\mathcal{T}:0\subseteq W(\delta_1)\subseteq W$ such that $W/W(\delta_1)\isom W(\delta_2)$. 
Assume that $\mathrm{H}^2(G_K, W(\delta_1/\delta_2))\allowbreak=0$, then the morphism of functors 
$f:D_{W,\mathcal{T}}\rightarrow D_{\delta_0}$ defined above is formally smooth.
\end{lemma}
\begin{proof}
Let $A\in \mathcal{C}_E$ and $I$ be an ideal of $A$ such that $I\mathfrak{m}_A=0$. 
Take any $[(W_{A/I},\mathcal{T}_{A/I})]\in D_{W,\mathcal{T}}(A/I)$ and $\delta_A\in D_{\delta_0}(A)$
 such that $f([(W_{A/I},\mathcal{T}_{A/I})])=\delta_A\otimes\mathrm{id}_{A/I}\in D_{\delta_0}(A/I)$. Then, it suffices to show that there exists  a lift $[(W_A,\mathcal{T}_A)]\in D_{W,\mathcal{T}}(A)$ 
 of $[(W_{A/I},\mathcal{T}_{A/I})]$ such that $f([(W_A,\mathcal{T}_A)])=\delta_A$. 
 Denote $\mathcal{T}_{A/I}:0\subseteq W(\delta_{1,A/I})\subseteq W_{A/I}$ and $W_{A/I}/W(\delta_{1,A/I})\isom W(\delta_{2,A/I})$. 
Because $D_{\delta_0}$ is formally smooth, there exists 
$\delta_{1,A}:K^{\times}\rightarrow A^{\times}$ such that $\delta_{1,A}\otimes_A A/I=\delta_{1,A/I}$. 
We take a lift $\lambda\in A^{\times}$ of $\delta_{2,A/I}(\pi_K)\in (A/I)^{\times}$, and define 
$\delta_{2,A}:K^{\times}\rightarrow A^{\times}$ by $\delta_{2,A}(\pi_K)=\lambda$ and $\delta_{2,A}|_{\mathcal{O}_K^{\times}}
=\delta_A\delta_{1,A}|_{\mathcal{O}_K^{\times}}$, then we have the following short exact sequence 
$$0\rightarrow W(\delta_1/\delta_2)\otimes_E I\rightarrow W(\delta_{1,A}/\delta_{2,A})\rightarrow 
W(\delta_{1,A/I}/\delta_{2,A/I})\rightarrow 0.$$
This sequence implies that the natural map 
$$\mathrm{H}^1(G_K, W(\allowbreak \delta_{1,A}/\delta_{2,A}))\rightarrow 
\mathrm{H}^1(G_K, W(\delta_{1,A/I}/\delta_{2,A/I}))$$ is a surjection because we have 
$\mathrm{H}^2(G_K, W(\delta_1/\delta_2))=0$ by the assumption. Hence, there exists 
a lift $[(W_A,\mathcal{T}_A)]\in D_{W,\mathcal{T}}(A)$ of $[(W_{A/I},\mathcal{T}_{A/I})]\in 
D_{W,\mathcal{T}}(A/I)$ 
satisfying that $f([W_A,\mathcal{T}_A])=(\delta_{2,A}/\delta_{1,A})|_{\mathcal{O}_K^{\times}}=\delta_A$.
We finish the proof of the lemma.
\end{proof}
Let $A\in \mathcal{C}_E$, and let $\delta:\mathcal{O}_K^{\times}\rightarrow A^{\times}$ be a continuous 
homomorphism, then it is known that this is locally $\mathbb{Q}_p$-analytic by Proposition 8.3 of \cite{Bu07}. Then, for any $\sigma\in \mathcal{P}$, 
we define the $\sigma$-component of Hodge-Tate weights of $\delta$ by $\frac{\partial\delta(x)}{\partial \sigma(x)}|_{x=1}\in A$, which 
is equal to the $\sigma$-part of Hodge-Tate weights of $A(\widetilde{\delta})$ where $\widetilde{\delta}:G^{\mathrm{ab}}_K\rightarrow A^{\times}$ is 
any character such that $\widetilde{\delta}\circ\mathrm{rec}_K|_{\mathcal{O}_K^{\times}}=\delta$ (see Proposition 3.3 of \cite{Na11}).
\begin{lemma}\label{5.17}
Let $\delta_0:\mathcal{O}_K^{\times}\rightarrow E^{\times}$ be a continuous homomorphism.
Let $R_{\delta_0}$  be the universal deformation ring of $D_{\delta_0}$.
Let $\delta_0^{\mathrm{univ}}:\mathcal{O}_K^{\times}\rightarrow R_{\delta_0}^{\times}$ be 
the universal deformation of $\delta_0$. For any $\sigma\in \mathcal{P}$, define by 
$a^{\mathrm{univ}}_{\sigma}:=(a_{\sigma,n})\in R_{\delta_0}=\varprojlim_nR_{\delta_0}/\frak{m}^n$ 
the $\sigma$-part of Hodge-Tate weights of $\delta_0^{\mathrm{univ}}$, where we denote by $a_{\sigma,n}$ the  $\sigma$-part of Hodge-Tate weights of $\delta_0^{\mathrm{univ}}\otimes\mathrm{id}_{R_{\delta_0}/\frak{m}^n}$ for each $n\geqq 1$. Then, $a^{\mathrm{univ}}_{\sigma}$ 
is not constant, i.e. not contained in $E$, for any $\sigma\in\mathcal{P}$.

\end{lemma}
\begin{proof}
Let $a:=\{a_{\sigma}\}_{\sigma\in \mathcal{P}}\in \prod_{\sigma\in \mathcal{P}}E$ be any element, then 
we define a deformation of $\delta_0$ over $E[\varepsilon]$ by 
$$\delta_a:\mathcal{O}_K^{\times}
\rightarrow E[\varepsilon]^{\times}: \delta_a(x):=\delta_0(x)(1+(\sum_{\sigma\in \mathcal{P}}a_{\sigma}
\mathrm{log}(\sigma(x)))\varepsilon).$$ The $\sigma$-part of Hodge-Tate weights
of $\delta_a$ is $\frac{\partial\delta_0(x)}{\partial\sigma(x)}|_{x=1}+a_{\sigma}\varepsilon$. The lemma follows from this.
\end{proof}
\begin{corollary}\label{5.16.5}
Let $x=[V_x]\in \mathfrak{X}(\bar{\rho})$ be a point such that $V_x$ is a crystabelline $E(x)$-trianguline representation satisfying the conditions 
$(1)$ of Definition $\mathrm{\ref{26}}$.
Then, the point  $x_{\tau}:=([V_x],\delta_{\tau,1}|_{\mathcal{O}_K^{\times}},\delta_{\tau,1}(\pi_K))\in X(\bar{\rho})$ is 
contained in $\mathcal{E}(\bar{\rho})$ for any $\tau\in \mathfrak{S}_2$, where 
we denote the triangulation $\mathcal{T}_{\tau}$ 
by $0\subseteq W(\delta_{\tau,1})\subseteq W(V_x)$. 

\end{corollary}
\begin{proof}
This follows from (2) of Theorem \ref{5.14} and from Lemma \ref{21.5}.
\end{proof}

Next, we describe the local structure of $\mathcal{E}(\bar{\rho})$ at the points satisfying the 
conditions (i), (ii) in (2) of Theorem \ref{5.14} by the universal trianguline deformation rings.
 We prove the following theorem, which is a generalization of Proposition 10.6 of \cite{Ki03}.
\begin{thm}\label{5.18}
Let $x:=([V_x],\delta_x,\lambda_x)\in \mathcal{E}(\bar{\rho})$ be a  point such that  the conditions (i), (ii) in $(2)$ of Theorem 
$\mathrm{\ref{5.14}}$ hold. 
Then, we have a canonical $E(x)$-algebra isomorphism
$$\widehat{\mathcal{O}}_{\mathcal{E}(\bar{\rho}),x}\isom R_{V_x,\mathcal{T}_x}.$$
In particular, $\mathcal{E}(\bar{\rho})$ is smooth of its dimension $3[K:\mathbb{Q}_p] +1$ at $x$.
\end{thm}
\begin{proof}
We may assume that $E=E(x)$. 

In the proof of Theorem \ref{5.14},we have already  showed that 
the natural map $\widehat{\mathcal{O}}_{X(\bar{\rho}),x}\rightarrow \overline{R}\isom R_{V_x,\mathcal{T}_x}$ factors through $\widehat{\mathcal{O}}_{\mathcal{E}(\bar{\rho}),x}\rightarrow 
\overline{R}\isom R_{V_x,\mathcal{T}_x}$. 

We prove the existence of the  inverse map $R_{V_x,\mathcal{T}_x}\isom \overline{R}\rightarrow \hat{\mathcal{O}}_{\mathcal{E}(\bar{\rho}),x}$.
Because $x$ is an $E$-rational point, we can take a $Y$-small 
affinoid neighborhood $\mathrm{Spm}(R)$ of $x$ in $\mathcal{E}(\bar{\rho})$. By Proposition \ref{5.7} and Proposition 
\ref{5.13}, for any sufficiently large $k>0$, there exists a short exact sequence of Banach $R$-modules with the property (Pr) $$0\rightarrow (\bold{B}^+_{\mathrm{max},K}\hat{\otimes}_{K,\sigma}R)^{\varphi_K=Y} \rightarrow \bold{B}^+_{\mathrm{dR}}/t^k\bold{B}^+_{\mathrm{dR}}\hat{\otimes}_{K,\sigma}R \rightarrow U_{k,\sigma}\rightarrow 0$$  for any $\sigma\in\mathcal{P}$, and we have a natural isomorphism $$K\otimes_{K_0}\bold{D}^+_{\mathrm{cris}}(V_R(\widetilde{\delta}^{\mathrm{univ}  -1}))^{\varphi^f =Y}\isom  
(\bold{B}^+_{\mathrm{dR}}/t^k\bold{B}^+_{\mathrm{dR}}\allowbreak \hat{\otimes}_{\mathbb{Q}_p}V_R(\widetilde{\delta}^{\mathrm{univ} -1}))^{G_K}.$$ Fix such a $k>0$, then we defined an ideal $H_{\sigma}\subseteq R$ for each $\sigma\in \mathcal{P}$ in Proposition \ref{5.13} such 
that $\mathrm{Spm}(R)\setminus V(H)$ ($H:=\prod_{\sigma\in \mathcal{P}}H_{\sigma}$) and $\mathrm{Spm}(R)\setminus V(H_{\sigma})$ are scheme theoretically dense in $\mathrm{Spm}(R)$.
Under this situation,  we prove the existence of the inverse $R_{V_x,\mathcal{T}_x}\isom \overline{R}\rightarrow \hat{\mathcal{O}}_{\mathcal{E}(\bar{\rho}),x}$. 
First, we claim that $\bold{D}^+_{\mathrm{cris}}(V_x(\widetilde{\delta}_x^{-1}))^{\varphi^f=\lambda_x}$ is a free $K_0\otimes_{\mathbb{Q}_p}E$-module of rank one. By the definition and by Lemma \ref{5.7.5}, this module contains a submodule $\bold{D}_{\mathrm{cris}}(W(\delta_{\lambda_x}))=\bold{D}_{\mathrm{cris}}(W(\delta_{\lambda_x}))^{\varphi^f=\lambda_x}\cap \mathrm{Fil}^0\bold{D}_{\mathrm{dR}}(W(\delta_{\lambda_x}))$ which 
is of rank one. Hence $\bold{D}^+_{\mathrm{cris}}(V_x(\widetilde{\delta}_x^{-1}))^{\varphi^f=\lambda_x}$ is of rank one or two. If 
this is of rank two, then $V_x(\widetilde{\delta}_x^{-1})$ is crystalline with the Hodge-Tate weights $\{0,k_{\sigma}\}_{\sigma\in \mathcal{P}}$ such that  
$k_{\sigma}\in \mathbb{Z}_{\leqq 0}$ for any $\sigma\in \mathcal{P}$ with a unique $\varphi_K$-eigenvalue $\lambda_x$. These conditions imply that 
$\delta_2/\delta_1=\prod_{\sigma\in \mathcal{P}}\sigma^{k_{\sigma}}$, which contradicts the assumption on $(V_x,\mathcal{T}_x)$. 
Hence 
$\bold{D}^+_{\mathrm{cris}}(V_x(\widetilde{\delta}_x^{-1}))^{\varphi^f=\lambda_x}$ is of rank one and the inclusion $\bold{D}_{\mathrm{cris}}(W(\delta_{\lambda_x}))\hookrightarrow \bold{D}^+_{\mathrm{cris}}(V_x(\widetilde{\delta}_x^{-1}))^{\varphi^f=\lambda_x}$ is isomorphism.

In the same way as in the proof of 
Proposition 10.6 of \cite{Ki03}, we take the blow up $\widetilde{T}$ of  $\mathrm{Spm}(R)$ along $H$. 
By the definition of blow up, for any point $\widetilde{x}\in \widetilde{T}$ above $x\in\mathrm{Spm}(R)$ and for any $\sigma\in \mathcal{P}$, 
there exists 
$f_{\sigma}\in H_{\sigma}$ such that $f_{\sigma}$ is a non zero divisor of $\widehat{\mathcal{O}}_{\widetilde{T},\widetilde{x}}$ and that
$H_{\sigma}\widehat{\mathcal{O}}_{\widetilde{T},\widetilde{x}}=f_{\sigma}\widehat{\mathcal{O}}_{\widetilde{T},\widetilde{x}}$. By the definition of $H_{\sigma}$, for any $\sigma\in \mathcal{P}$ and for any $\widetilde{x}\in \widetilde{T}$ above $x$, there exists a $G_K$-equivariant map $V_R(\widetilde{\delta}^{\mathrm{univ}  -1})^{\vee}\rightarrow (\bold{B}^+_{\mathrm{max},K}\hat{\otimes}_{K,\sigma}R)^{\varphi_K=Y}$ such that the composite with the map

\begin{multline*}
(\bold{B}^+_{\mathrm{max},K}\hat{\otimes}_{K,\sigma}R)^{\varphi_K=Y}\rightarrow (\bold{B}^+_{\mathrm{max},K}\hat{\otimes}_{K,\sigma}
f_{\sigma}\widehat{\mathcal{O}}_{\widetilde{T},\widetilde{x}})^{\varphi_K=Y} \\
\isom (\bold{B}^+_{\mathrm{max},K}\hat{\otimes}_{K,\sigma}\widehat{\mathcal{O}}_{\widetilde{T},\widetilde{x}})^{\varphi_K=Y}\rightarrow (\bold{B}^+_{\mathrm{max},K}\otimes_{K,\sigma}E(\widetilde{x}))^{\varphi_K=Y(\widetilde{x})}
\end{multline*}
is non zero, where the isomorphism
 $$(\bold{B}^+_{\mathrm{max},K}\hat{\otimes}_{K,\sigma}f_{\sigma}\widehat{\mathcal{O}}_{\widetilde{T},\widetilde{x}})^{\varphi_K=Y}\isom (\bold{B}^+_{\mathrm{max},K}\hat{\otimes}_{K,\sigma}\widehat{\mathcal{O}}_{\widetilde{T},\widetilde{x}} )^{\varphi_K=Y}$$ 
is given by $a\mapsto \frac{a}{f_{\sigma}}$. Using this map and using 
the fact that $\bold{D}^+_{\mathrm{cris}}(V_x(\widetilde{\delta}_x^{-1}))^{\varphi^f=\lambda_x}$ is rank one, we can see by induction on $n$ that $\bold{D}^+_{\mathrm{cris}}(V_R(\widetilde{\delta}^{\mathrm{univ} -1})\otimes_R
\widehat{\mathcal{O}}_{\widetilde{T},\widetilde{x}}/\mathfrak{m}^n_{\widetilde{x}})^{\varphi^f=Y}$ is a free $K_0\otimes_{\mathbb{Q}_p}\widehat{\mathcal{O}}_{\widetilde{T},\widetilde{x}}/\mathfrak{m}^n_{\widetilde{x}}$-module of rank one and that the natural base change 
map $$\bold{D}^+_{\mathrm{cris}}(V_R(\widetilde{\delta}^{\mathrm{univ} -1})\otimes_R
\widehat{\mathcal{O}}_{\widetilde{T},\widetilde{x}}/\mathfrak{m}^n_{\widetilde{x}})^{\varphi^f=Y}\otimes_{\mathcal{O}_{\widetilde{T},\widetilde{x}}/\mathfrak{m}^n_{\widetilde{x}}} E(\widetilde{x})\isom \bold{D}^+_{\mathrm{cris}}(V_x(\widetilde{\delta}_x^{-1})\otimes_{E}E(\widetilde{x}))^{\varphi^f=Y(x)}$$ 
is isomorphism for any $n\geqq 1$. Because we have an equality
$$\mathrm{Fil}^1(K\otimes_{K_0}\bold{D}^+_{\mathrm{cris}}(V_x(\widetilde{\delta}_x^{-1}))^{\varphi^f=Y(x)})
=\mathrm{Fil}^1\bold{D}_{\mathrm{dR}}(W(\delta_{\lambda_x}))=0,$$ then $\bold{D}^+_{\mathrm{cris}}(V_R(\widetilde{\delta}^{\mathrm{univ} -1})\otimes_R
\widehat{\mathcal{O}}_{\widetilde{T},\widetilde{x}}/\mathfrak{m}^n_{\widetilde{x}})^{\varphi^f=Y}$ is a ($\widehat{\mathcal{O}}_{\widetilde{T},\widetilde{x}}/\mathfrak{m}^n_{\widetilde{x}}$)-filtered $\varphi$-module of rank one such that $\mathrm{Fil}^0=K\otimes_{K_0}\bold{D}^+_{\mathrm{cris}}(V_R(\delta^{\mathrm{univ} -1})\otimes_R
\widehat{\mathcal{O}}_{\widetilde{T},\widetilde{x}}/\mathfrak{m}^n_{\widetilde{x}})^{\varphi^f=Y}$ and $\mathrm{Fil}^1=0$.
By Lemma \ref{s4},
this shows that $V_R\otimes_R \widehat{\mathcal{O}}_{\widetilde{T},\widetilde{x}}$ is 
the projective limit of split trianguline $(\mathcal{O}_{\widetilde{T},\widetilde{x}}/\mathfrak{m}^n_{\widetilde{x}})$-representations with triangulations
$0\subseteq W(\bar{\delta}_n^{\mathrm{univ}}\delta_{\overline{Y}_n})\subseteq W(V_R\otimes_R \widehat{\mathcal{O}}_{\widetilde{T},\widetilde{x}}/\mathfrak{m}^n_{\widetilde{x}})$ which are 
trianguline deformations of $(V_x,\mathcal{T}_x)\otimes_{E} E(\widetilde{x})$ (for $n\in\mathbb{Z}_{\geqq 1}$), hence the natural map 
$R_{V_x}\rightarrow \widehat{\mathcal{O}}_{\mathcal{E}(\bar{\rho}),x}\rightarrow \widehat{\mathcal{O}}_{\widetilde{T},\widetilde{x}}$ factors 
through $R_{V_x}\rightarrow R_{V_x,\mathcal{T}_x}$ for any $\widetilde{x}\in \widetilde{T}$ above $x$. Moreover, because 
the natural map $$\widehat{\mathcal{O}}_{\mathcal{E}(\bar{\rho}),x}\hookrightarrow \prod_{\widetilde{x}\in \widetilde{T}, p(\widetilde{x})=x}\widehat{\mathcal{O}}_{\widetilde{T},\widetilde{x}}$$ is an injection by Lemma 10.7 of \cite{Ki03} and by (2) of Proposition \ref{5.13} (where $p:\widetilde{T}\rightarrow \mathrm{Spm}(R)$ is the projection), 
the map $R_{V_x}\rightarrow \widehat{\mathcal{O}}_{\mathcal{E}(\bar{\rho}),x}$ also 
factors through $R_{V_x,\mathcal{T}_x}\rightarrow \widehat{\mathcal{O}}_{\mathcal{E}(\bar{\rho}),x}$. 
By this natural construction, we can easily check that this is the inverse of the map giving the above.
We finish to prove the existence of the isomorphism $\widehat{\mathcal{O}}_{\mathcal{E}(\bar{\rho}),x}
\isom R_{V_x,\mathcal{T}_x}$ for such points. Because this isomorphism is preserved by the base change from $E$ to any 
finite extension $E'$ by Lemma \ref{5.9}, the smoothness around these points 
follows from this isomorphism and from Lemma 2.8 of \cite{BLR95}.
\end{proof}

\section{Zariski density of two dimensional crystalline representations.}
In this final section, as an application of Theorem \ref{33} (in the two dimensional case) and of Theorem \ref{5.18}, 
we prove the Zariski density of two dimensional crystalline representations for any $p$-adic field. 
\smallskip

We define a map $\pi:\mathcal{E}(\bar{\rho})\rightarrow \mathcal{W}_E\times_E\mathcal{W}_E$ by 
$([V_x],\delta_x,\lambda_x)\mapsto (\delta_x, \mathrm{det}(V_x)|_{\mathcal{O}_K^{\times}}\cdot\delta_x^{-1})$.

\begin{prop}\label{6.1}
For any point $x\in \mathcal{E}(\bar{\rho})$ which satisfies all the conditions of Theorem $\ref{5.18}$, 
the map $\pi:\mathcal{E}(\bar{\rho})\rightarrow \mathcal{W}_E\times_E\mathcal{W}_E$ is smooth at $x$.
\end{prop}
\begin{proof}
Let $x:=([V_x],\delta_x,\lambda_x)\in\mathcal{E}(\bar{\rho})$ be such a point. Set $\delta'_x:=\mathrm{det}(V_x)|_{\mathcal{O}_K^{\times}}\cdot\delta_x^{-1}$. 
By the same argument as in Proposition 9.5 of \cite{Ki03}, we have a natural isomorphism
isomorphism 
$$\widehat{\mathcal{O}}_{\mathcal{W}_E\times_E\mathcal{W}_E,(\delta_x,\delta'_x)}\isom \widehat{\mathcal{O}}_{\mathcal{W}_E,\delta_x}\hat{\otimes}_{E(x)}
\widehat{\mathcal{O}}_{\mathcal{W}_E,\delta'_x}\isom R_{\delta_x}\hat{\otimes}_{E(x)}R_{\delta'_x}.$$ Hence, by Theorem \ref{5.18}, the completion of 
$\pi$ at $x$ is the morphism
$$\pi:\mathrm{Spf}(\widehat{\mathcal{O}}_{\mathcal{E}(\bar{\rho}),x})\rightarrow \mathrm{Spf}(\widehat{\mathcal{O}}_{\mathcal{W}_E\times_E\mathcal{W}_E,(\delta_x,\delta'_x)})$$ induced by the morphism of functors 
$$\pi_x:D_{V_x,\mathcal{T}_x}\rightarrow D_{\delta_x}\times D_{\delta'_x}:[(W_A,\mathcal{T}_A)]\mapsto (\delta_{1,A}|_{\mathcal{O}_K^{\times}},\delta_{2,A}|_{\mathcal{O}_K^{\times}}),$$ where $\mathcal{T}_A:0\subseteq W(\delta_{1,A})\subseteq W_A$ and $W_A/W(\delta_{1,A})\isom W(\delta_{2,A})$ for $A\in\mathcal{C}_{E(x)}$. Then, we can prove the formal smoothness of this morphism of functor in the  same way as in the proof of Lemma \ref{5.16}.
 Hence, $\pi$ is smooth at $x$  by Proposition  \ref{19} and by Proposition 2.9 of \cite{BLR95}.
\end{proof}

Let $x:=([V_x],\delta_x,\lambda_x)\in \mathcal{E}(\bar{\rho})$ be an $E$-rational  point such 
that $V_x$ is a 
crystalline split trianguline $E$-representation with a 
triangulation $\mathcal{T}_x:0\subseteq W(\delta_x\delta_{\lambda_x})\subseteq W(V_x)$ 
satisfying the condition (1) of Definition \ref{26} (see Corollary \ref{5.16.5}).
By Proposition \ref{5.13}, for any $Y$-small 
affinoid open neighborhood $U=\mathrm{Spm}(R)$ of $x$ in $\mathcal{E}(\bar{\rho})$, there exits $k>0$ and 
there exists a short exact sequence of Banach $R$-modules with the property (Pr)
$$0\rightarrow K\otimes_{K_0}(\bold{B}^+_{\mathrm{max}}\hat{\otimes}_{\mathbb{Q}_p}R)^{\varphi^f=Y}
\rightarrow \bold{B}^+_{\mathrm{dR}}/t^k\bold{B}^+_{\mathrm{dR}}\hat{\otimes}_{\mathbb{Q}_p}R\rightarrow U_{k}\rightarrow 0$$
and we have a natural isomorphism $$K\otimes_{K_0}\bold{D}^+_{\mathrm{cris}}(V_R(\widetilde{\delta}_R^{-1}))^{\varphi^f=Y}\isom 
(\bold{B}^+_{\mathrm{dR}}/t^k\bold{B}^+_{\mathrm{dR}} \hat{\otimes}_{\mathbb{Q}_p}V_R(\widetilde{\delta}_R^{-1}))^{G_K}$$
 and, for any $\sigma\in \mathcal{P}$, there exists the smallest 
ideal $H_{\sigma}\subseteq R$ satisfying that $$(\bold{B}^+_{\mathrm{dR}}/t^k\bold{B}^+_{\mathrm{dR}}\hat{\otimes}_{K,\sigma}H_{\sigma}V_R(\widetilde{\delta}_R^{-1}))^{G_K}
\isom (\bold{B}^+_{\mathrm{dR}}/t^k\bold{B}^+_{\mathrm{dR}}\hat{\otimes}_{K,\sigma}V_R(\widetilde{\delta}_R^{-1}))^{G_K},$$
where we put $V_R:=\Gamma(U, p_1^*(\widetilde{V}^{\mathrm{univ}}))$ and $\delta_R:\mathcal{O}_K^{\times}\rightarrow R^{\times}$ is the restriction of $p_2^*(\delta^{\mathrm{univ}})$ to $U$. Moreover, 
if we put $Q:=\prod_{\sigma\in \mathcal{P},0\leqq i\leqq k}Q_{\sigma}(-i)\in R$, then we have inclusions 
$\mathrm{Spm}(R)_Q\subseteq \mathrm{Spm}(R)\setminus V(H_{\sigma})\subseteq \mathrm{Spm}(R)$ by the proof of 
Proposition \ref{5.13}. Moreover, shrinking $U$ suitably, we may assume that $v_p(\lambda_y)=v_p(\lambda_x)$ for any 
$y=(V_y,\delta_y,\lambda_y)\in U$ and that $\pi|_U$ is smooth by Proposition \ref{6.1}.

Under this situation, we study the map 
$\pi|_U:U \rightarrow \mathcal{W}_E\times_E\mathcal{W}_E$
 around $x$ in detail. 
Because $V_x$ is crystalline 
, we can write $$\pi(x)=(\prod_{\sigma\in\mathcal{P}}\sigma^{k_{1,\sigma}},\prod_{\sigma\in \mathcal{P}}\sigma^{k_{2,\sigma}})\in \mathcal{W}_E\times_E\mathcal{W}_E$$
for some integers $\{k_{1,\sigma},k_{2,\sigma}\}_{\sigma\in \mathcal{P}}$. Define a subset 
\begin{multline*}
(\mathcal{W}_E\times_E\mathcal{W}_E)_{cl,x}:=\{
(\prod_{\sigma\in \mathcal{P}}\sigma^{n_{\sigma}},\prod_{\sigma\in\mathcal{P}}\sigma^{n_{\sigma}-m_{\sigma}})\in \mathcal{W}_E\times_E
\mathcal{W}_E| n_{\sigma}\in \mathbb{Z}, \\
m_{\sigma}\in \mathbb{Z}_{\geqq k+1} \,\,\text{for any} \,\,\sigma\in\mathcal{P} \,\,\text{and} \,\,
\sum_{\sigma\in\mathcal{P}}m_{\sigma}\geqq 2e_Kv_p(\lambda_x)+[K:\mathbb{Q}_p]+1\}, 
\end{multline*}
where $e_K$ is the absolute ramified index of $K$.
Then, for any  admissible open neighborhood $V$ of $\pi(x)$ in $\mathcal{W}_E\times_E\mathcal{W}_E$, there exists an affinoid open 
$V'\subseteq V$ which contains $\pi(x)$ such that $V'_{\mathrm{cl},x}:=(\mathcal{W}_E\times_E\mathcal{W}_E)_{\mathrm{cl},x}\cap V'$ is Zariski dense 
in $V'$. Under this situation, we prove the following lemma.

\begin{lemma}\label{6.2}
Let $y:=(\prod_{\sigma\in\mathcal{P}}\sigma^{n_{\sigma}},\prod_{\sigma\in\mathcal{P}}\sigma^{n_{\sigma}-m_{\sigma}})$ be an element in $(\mathcal{W}_E\times_E\mathcal{W}_E)_{\mathrm{cl},x}$ and let $z:=([V_z],\delta_z,\lambda_z)$ be a point in $U\cap \pi^{-1}(y)$, then 
$V_z$ is crystalline and split trianguline $E(z)$-representation with a 
triangulation 
$\mathcal{T}_z:0\subseteq W(\delta_z\delta_{\lambda_z})\subseteq W(V_z)$ which satisfies the conditions $(1)$ and $(2)$ 
of Definition $\mathrm{\ref{26}}$.
\end{lemma}
\begin{proof}
Let $z$ be such a point. By Corollary 2.6 of \cite{Ki03}, we have a natural isomorphism 
 $$(\bold{B}^+_{\mathrm{dR}}/t^k\bold{B}^+_{\mathrm{dR}}\hat{\otimes}_{\mathbb{Q}_p}V_R(\widetilde{\delta}_R^{ -1}))^{G_K}\otimes_R E(z)
\isom (\bold{B}^+_{\mathrm{dR}}/t^k\bold{B}^+_{\mathrm{dR}}\otimes_{\mathbb{Q}_p}V_z(\widetilde{\delta}_z^{-1}))^{G_K}$$
 and this is 
a free $K\otimes_{\mathbb{Q}_p}E(z)$-module of rank one. Because we 
have an isomorphism  $$K\otimes_{K_0}\bold{D}^+_{\mathrm{cris}}(V_R(\widetilde{\delta}_R^{ -1}))^{\varphi^f=Y}
\isom (\bold{B}^+_{\mathrm{dR}}/t^k\bold{B}^+_{\mathrm{dR}}\hat{\otimes}_{\mathbb{Q}_p}V_R(\widetilde{\delta}_R^{-1}))^{G_K}$$ 
and an injection 
$$K\otimes_{K_0}\bold{D}^+_{\mathrm{cris}}(V_z(\widetilde{\delta}_z^{-1}))^{\varphi^f=\lambda_z}\hookrightarrow (\bold{B}^+_{\mathrm{dR}}/t^k\bold{B}^+_{\mathrm{dR}}\otimes_{\mathbb{Q}_p}V_z(\widetilde{\delta}_z^{-1}))^{G_K}$$ induced from the injection 
$K\otimes_{K_0}(\bold{B}^+_{\mathrm{cris}}\otimes_{\mathbb{Q}_p}E(z))^{\varphi^f=\lambda_z}
\hookrightarrow \bold{B}^+_{\mathrm{dR}}/t^k\bold{B}^+_{\mathrm{dR}}\otimes_{\mathbb{Q}_p}E(z),$ we obtain  an isomorphism 
$$K\otimes_{K_0}\bold{D}^+_{\mathrm{cris}}(V_z(\widetilde{\delta}_z^{-1}))^{\varphi^f=\lambda_z}\isom (\bold{B}^+_{\mathrm{dR}}/t^k\bold{B}^+_{\mathrm{dR}}\otimes_{\mathbb{Q}_p}V_z(\widetilde{\delta}_z^{-1}))^{G_K} .$$ On the other hand, because the Hodge-Tate weights of $V_z(\widetilde{\delta}_z^{-1})$ 
are $\{0, -m_{\sigma}\}_{\sigma\in\mathcal{P}}$ and $m_{\sigma}\geqq k+1\geqq 1$, 
$(t^{k+1}\bold{B}^+_{\mathrm{dR}}\otimes_{\mathbb{Q}_p}V_z(\widetilde{\delta}_z^{-1}))^{G_K}$ is also a free $K\otimes_{\mathbb{Q}_p}E(z)$-module 
of rank one. These implies that $\bold{D}_{\mathrm{dR}}(V_z(\widetilde{\delta}_z^{-1}))$ is a free of rank two $K\otimes_{\mathbb{Q}_p}E(z)$-module, i.e. 
$V_z(\widetilde{\delta}_z^{-1})$ is potentially semi-stable and split trianguline with a triangulation 
$\mathcal{T}'_z:0\subseteq W(\delta_{\lambda_z})\hookrightarrow W(V_z(\widetilde{\delta}_z^{-1}))$. Moreover, if we set
$\delta_2:=\mathrm{det}(V_z)\cdot\delta_z^{-2}\delta^{-1}_{\lambda_z}:K^{\times}\rightarrow E(z)^{\times}$, then we have $W(V_z(\widetilde{\delta}_z^{-1}))/W(\delta_{\lambda_z})\isom W(\delta_2)$ such that $\delta_2|_{\mathcal{O}_K^{\times}}=\prod_{\sigma\in\mathcal{P}}\sigma^{-m_{\sigma}}$ because $z\in \pi^{-1}(y)$, which implies that $V_z(\widetilde{\delta}_z^{-1})$ is semi-stable. Finally, we claim that $V_z(\widetilde{\delta}_z^{-1})$ is crystalline. If we assume that $V_z(\widetilde{\delta}_z^{-1})$ is semi-stable but not crystalline, then the $\varphi^f$-eigenvalue of $W(\delta_2)$ is $\lambda_zp^{f}$ or $\lambda_zp^{-f}$. By the weakly admissibility 
of $\bold{D}_{\mathrm{st}}(V_z(\widetilde{\delta}_z^{-1}))$, we have an equality $t_N(V_z(\widetilde{\delta}_z^{-1}))= t_H(V_z(\widetilde{\delta}_z^{-1}))$. On the other hand, because we have an isomorphism $W(\mathrm{det}(V_z(\widetilde{\delta}_z^{-1})))\isom W(\delta_{\lambda_z}\delta_2)$, we obtain the equalities
$$t_N(V_z(\widetilde{\delta}_z^{-1}))=\frac{2}{f}v_p(\lambda_z)\pm 1=\frac{2}{f}v_p(\lambda_x)\pm 1,\text{ and } t_H(V_z(\widetilde{\delta}_z^{-1}))=\frac{1}{[K:\mathbb{Q}_p]}(\sum_{\sigma\in \mathcal{P}}m_{\sigma}),$$ hence we obtain $t_N(V_z(\widetilde{\delta}_z^{-1}))<t_H(V_z(\widetilde{\delta}_z^{-1}))$ because 
$y\in (\mathcal{W}\times_E\mathcal{W}_E)_{\mathrm{cl},x}$, which is a contradiction.
Hence, $V_z(\widetilde{\delta}_z^{-1})$ is crystalline, and $V_z$ is 
also crystalline because $\widetilde{\delta}_z$ is crystalline. Finally, twisting $\mathcal{T}'_z$ by $\delta_z$, we obtain a triangulation $\mathcal{T}_z:0\subseteq W(\delta_z\delta_{\lambda_z})
\subseteq W(V_z)$ which satisfies (1) and (2) of Definition \ref{26}.
\end{proof}

\begin{lemma}\label{6.3}
Let $z=([V_z],\delta_z,\lambda_z)$ be a point  in $U\cap\pi^{-1}(y)$ as in the above lemma.
 Then we have a natural isomorphism 
$\widehat{\mathcal{O}}_{\pi^{-1}(y),z}\isom R^{\mathrm{cris}}_{V_z}$.
\end{lemma}
\begin{proof}
First, by Lemma \ref{21.5} and Theorem \ref{5.18} and Lemma \ref{6.2}, we have a triangulation 
$\mathcal{T}_z:0\subseteq W(\delta_z\delta_{\lambda_z})\subseteq W(V_z)$, and 
the functor $D_{V_z,\mathcal{T}_z}$ is representable by $R_{V_z,\mathcal{T}_z}$, and we have an isomorphism 
$\widehat{\mathcal{O}}_{\mathcal{E}(\bar{\rho}),z}\isom R_{V_z,\mathcal{T}_z}$. 
Then, the completion at $z$ and $\pi(z)$ of the morphism $\pi:\mathcal{E}(\bar{\rho})\rightarrow \mathcal{W}_E\times_E\mathcal{W}_E$ 
is the morphism 
$$\pi_z:\mathrm{Spf}(R_{V_z,\mathcal{T}_z})\rightarrow \mathrm{Spf}(R_{\delta_z}\hat{\otimes}_{E(z)}R_{\delta'_z})$$
 induced by 
$$D_{V_z,\mathcal{T}_z}\rightarrow D_{\delta_z}\times D_{\delta'_z}:[(V_A,\mathcal{T}_A)]\mapsto 
(\delta_{1,A}|_{\mathcal{O}_K^{\times}},\delta_{2,A}|_{\mathcal{O}_K^{\times}}),$$
 where we set $\delta'_z:=\mathrm{det}(V_z)|_{\mathcal{O}_K^{\times}}\cdot\delta^{-1}_z$. 
Under this interpretation, we have an equality $\mathrm{Spf}(\widehat{\mathcal{O}}_{\pi^{-1}(y),z})
=\pi_z^{-1}((\delta_z,\delta'_z))$, and this corresponds to 
the subfunctor $D'$ of $D_{V_z,\mathcal{T}_z}$ defined by
$$D'(A):=\{[(V_A,\mathcal{T}_A)]\in D_{V_z,\mathcal{T}_z}(A)| \delta_{1,A}|_{\mathcal{O}_K^{\times}}=\delta_z\otimes_{E(z)}\mathrm{id}_A, \,\,\delta_{2,A}|_{\mathcal{O}_K^{\times}}=\delta'_z\otimes_{E(z)}\mathrm{id}_A\}$$ for $A\in\mathcal{C}_{E(z)}$.
 Because $V_z$ is crystalline, this is equivalent to that $V_A$ is crystalline by Lemma \ref{29}. Therefore we have $D'=D^{\mathrm{cris}}_{V_z}$, hence we obtain an isomorphism $R^{\mathrm{cris}}_{V_z}
\isom \widehat{\mathcal{O}}_{\pi^{-1}(y),z}$.
\end{proof}

In the situation of Lemma \ref{6.3}, for any $y:=(\prod_{\sigma\in \mathcal{P}}\sigma^{n_{\sigma}},\prod_{\sigma\in\mathcal{P}}
\sigma^{n_{\sigma}-m_{\sigma}})\in (\mathcal{W}_E\times_E\mathcal{W}_E)_{\mathrm{cl},x}$, 
we set $U_y:=\pi^{-1}(y)\cap U$, which is smooth over $E(y)$ by the assumption on $U$, and define a subset 
$$U_{y,b}:=\{z=([V_z],\delta_z,\lambda_z)\in U_y| V_z \,\text{is benign} \} .$$
\begin{prop}\label{6.5}
In the above situation, if $U_y$ is not empty, then $U_{y,b}$ is an admissible open which is scheme theoretically dense 
in $U_y$, in particular $U_{y,b}$ is non-empty.
\end{prop}
\begin{proof}
Set $U_y:=\mathrm{Spm}(R')$. By Lemma \ref{6.2}, any point $z\in U_y$ satisfies the condition (1) and (2) of Definition \ref{26} 
and $V_z$ is crystalline 
with the Hodge-Tate weights
$\{n_{\sigma}, n_{\sigma}-m_{\sigma}\}_{\sigma\in\mathcal{P}}$. Because $U_y$ is smooth, so in particular $U_y$ is reduced. 
Hence, by Corollary 6.3.3 of \cite{Be-Co08} and Corollary 3.19 of \cite{Ch09a}, 
$$\bold{D}_{\mathrm{cris}}(V_{R'}(\widetilde{\delta}_{R'}^{-1})):=\varinjlim_n(\frac{1}{t^n}\bold{B}^+_{\mathrm{max}}\hat{\otimes}_{\mathbb{Q}_p}V_{R'}(\widetilde{\delta}_{R'}^{-1}))^{G_K}$$ is a locally 
free $K_0\otimes_{\mathbb{Q}_p}R'$-module of rank two,  and we have natural isomorphisms 
$$\bold{D}_{\mathrm{cris}}(V_{R'}(\widetilde{\delta}_{R'}^{-1}))\otimes_{R'}E(z)\isom \bold{D}_{\mathrm{cris}}(V_z(\widetilde{\delta}_z^{-1})) \text{ for any } z\in U_y$$ and 
$$K\otimes_{K_0}\bold{D}_{\mathrm{cris}}(V_{R'}(\widetilde{\delta}_{R'}^{-1}))\isom (\bold{B}_{\mathrm{dR}}\hat{\otimes}_{\mathbb{Q}_p}V_{R'}(\widetilde{\delta}_{R'}^{-1}))^{G_K}
=(\bold{B}^+_{\mathrm{dR}}\hat{\otimes}_{\mathbb{Q}_p}V_{R'}(\widetilde{\delta}_{R'}^{-1}))^{G_K},$$
where the last equality follows from the assumption on the Hodge-Tate weights of $V_z$ for any $z\in U_y$.
  Because $U_y\subseteq U_Q$, we have an isomorphism $$(\bold{B}^+_{\mathrm{dR}}/t^k\bold{B}^+_{\mathrm{dR}}\hat{\otimes}_{\mathbb{Q}_p}V_{R}(\widetilde{\delta}_R^{-1}))^{G_K}\otimes_{R}R'\isom 
(\bold{B}^+_{\mathrm{dR}}/t^k\bold{B}^+_{\mathrm{dR}}\hat{\otimes}_{\mathbb{Q}_p}V_{R'}(\widetilde{\delta}_{R'}^{-1}))^{G_K},$$
which is a locally free $K\otimes_{\mathbb{Q}_p}R'$-module of rank one by Corollary 2.6 of \cite{Ki03}. 
Because the natural map $K\otimes_{K_0}(\bold{B}^+_{\mathrm{max}}\hat{\otimes}_{\mathbb{Q}_p}R')^{\varphi^f=Y}
\hookrightarrow \bold{B}^+_{\mathrm{dR}}/t^k\bold{B}^+_{\mathrm{dR}}\hat{\otimes}_{\mathbb{Q}_p}R'$ is an injection, hence we obtain an isomorphism $$K\otimes_{K_0}\bold{D}^+_{\mathrm{cris}}(V_{R'}(\widetilde{\delta}_{R'}^{-1}))^{\varphi^f=Y}\isom (\bold{B}^+_{\mathrm{dR}}/t^k\bold{B}^+_{\mathrm{dR}}\hat{\otimes}_{\mathbb{Q}_p}V_{R'}(\widetilde{\delta}_{R'}^{-1}))^{G_K}.$$ 

From these facts, we can see that the natural map $$(\bold{B}^+_{\mathrm{dR}}\hat{\otimes}_{\mathbb{Q}_p}V_{R'}(\widetilde{\delta}_{R'}^{-1}))^{G_K}\rightarrow 
(\bold{B}^+_{\mathrm{dR}}/t^k\bold{B}^+_{\mathrm{dR}}\hat{\otimes}_{\mathbb{Q}_p}V_{R'}(\widetilde{\delta}_{R'}^{-1}))^{G_K}$$ is a surjection. Hence,
we obtain a short exact sequence
$$0\rightarrow \mathrm{Fil}^k\bold{D}_{\mathrm{dR}}(V_{R'}(\widetilde{\delta}_{R'}^{-1}))\rightarrow \bold{D}^+_{\mathrm{dR}}(V_{R'}(\widetilde{\delta}_{R'}^{-1}))
\rightarrow (\bold{B}^+_{\mathrm{dR}}/t^k\bold{B}^+_{\mathrm{dR}}\hat{\otimes}_{\mathbb{Q}_p}V_{R'}(\widetilde{\delta}_{R'}^{-1}))^{G_K}\rightarrow 0,$$
where we define $\mathrm{Fil}^k\bold{D}_{\mathrm{dR}}(V_{R'}(\widetilde{\delta}_{R'}^{-1})):=(t^k\bold{B}^+_{\mathrm{dR}}\hat{\otimes}_{\mathbb{Q}_p}V_{R'}(\widetilde{\delta}_{R'}^{-1}))^{G_K}$ which is a locally free $K\otimes_{\mathbb{Q}_p}R'$-module 
of rank one. If we set
$$D_2:=\bold{D}_{\mathrm{cris}}(V_{R'}(\widetilde{\delta}_{R'}^{-1}))/\bold{D}^+_{\mathrm{cris}}(V_{R'}(\widetilde{\delta}_{R'}^{-1}))^{\varphi^f=Y},$$
 then 
the above facts imply that 
$D_2$ is also a locally free $K_0\otimes_{\mathbb{Q}_p}R'$-module of rank one. By taking a sufficiently fine affinoid covering of $\mathrm{Spm}(R')$, we may assume 
that all these modules are free over $K_0\otimes_{\mathbb{Q}_p}R'$ or $K\otimes_{\mathbb{Q}_p}R'$. 
If we decompose $\bold{D}_{\mathrm{cris}}(V_{R'}(\widetilde{\delta}_{R'}^{-1}))=\oplus_{\tau:K_0\rightarrow K_0}D_{\tau}$ etc, then 
we obtain
a short exact sequence 
$$0\rightarrow D_{\tau}^{+, \varphi^f=Y}\rightarrow D_{\tau}\rightarrow D_{2,\tau}\rightarrow 0$$ 
of free $R'$-modules with an $R'$-linear $\varphi^f$-action for any $\tau$. Define $Y_1\in R^{' \times}$ by
$\varphi^f(e)=Y_1e$ for a $R'$-base $e$ of $D_{2,\tau}$. Because $Y_1$ is a lift of 
the other Frobenius eigenvalue of $\bold{D}_{\mathrm{cris}}(V_z(\widetilde{\delta}_z^{-1}))$ (one is $\lambda_z$) for 
any $z\in U_y=\mathrm{Spm}(R')$ and because $\bold{D}_{\mathrm{cris}}(V_z(\widetilde{\delta}_z^{-1}))$ is weakly admissible, 
the condition $\sum_{\sigma\in \mathcal{P}}m_{\sigma}\geqq 2e_Kv_p(\lambda_z)+[K:\mathbb{Q}_p]+1$ for any $z\in U_y$ implies
that 
$$Y-Y_1 (\text{ and } Y-p^{\pm f}Y_1)  \in R^{' \times}.$$ 
Then, an easy linear algebra implies that there exists a decomposition $D_{\tau}=R'e'_1\oplus  R'e'_2$ such that 
$R'e'_1=D_{\tau}^{\varphi^f=Y}=D^{+,\varphi^f=Y}_{\tau}$ and $R'e'_2=D_{\tau}^{\varphi^f=Y_1}$.  Twisting these by $\varphi^i$ for 
any $0\leqq i\leqq f-1$, we obtain a
decomposition 
$$\bold{D}_{\mathrm{cris}}(V_{R'}(\widetilde{\delta}_{R'}^{-1}))=\bold{D}^+_{\mathrm{cris}}(V_{R'}(\widetilde{\delta}_{R'}^{-1}))^{\varphi^f=Y}\oplus \bold{D}_{\mathrm{cris}}(V_{R'}(\widetilde{\delta}_{R'}^{-1}))^{\varphi^f=Y_1}.$$ We denote by $e_1$ (resp. $e_2$)  a $K_0\otimes_{\mathbb{Q}_p}R'$-basis of $\bold{D}^+_{\mathrm{cris}}(V_{R'}(\widetilde{\delta}_{R'}^{-1}))^{\varphi^f=Y}$ (resp. $\allowbreak\bold{D}_{\mathrm{cris}}(V_{R'}(\widetilde{\delta}_{R'}^{-1}))^{\varphi^f=Y_1}$). 
For any $\sigma\in \mathcal{P}$, we denote by $e_{1,\sigma}, e_{2,\sigma}$ the $R'$-basis of the $\sigma$-component of 
$\bold{D}^+_{\mathrm{dR}}(V_{R'}(\widetilde{\delta}_{R'}^{-1}))\isom K\otimes_{K_0}\bold{D}_{\mathrm{cris}}(V_{R'}(\widetilde{\delta}_{R'}^{-1}))$ naturally 
induced from $e_1, e_2$. Under this situation, we write the $\sigma$-component $\mathrm{Fil}^k\bold{D}_{\mathrm{dR}}(V_{R'}(\widetilde{\delta}_{R'}^{-1}))_{\sigma}$ by using the basis
$e_{1,\sigma}, e_{2,\sigma}$ as follows. Because the natural map $\bold{D}^+_{\mathrm{cris}}(V_{R'}(\widetilde{\delta}_{R'}^{-1}))^{\varphi^f=Y}\isom (\bold{B}^+_{\mathrm{dR}}/t^k\bold{B}^+_{\mathrm{dR}}\hat{\otimes}_{\mathbb{Q}_p}V_{R'}(\widetilde{\delta}_{R'}^{-1}))^{G_K}$ is isomorphism, the natural map 
$$\mathrm{Fil}^k\bold{D}_{\mathrm{dR}}(V_{R'}(\widetilde{\delta}_{R'}^{-1}))\rightarrow K\otimes_{K_0}D_2,$$
 which is the composition 
of the natural inclusion
$$\mathrm{Fil}^k\bold{D}_{\mathrm{dR}}(V_{R'}(\widetilde{\delta}_{R'}^{-1}))\rightarrow \bold{D}_{\mathrm{dR}}(V_{R'}(\widetilde{\delta}_{R'}^{-1}))=
K\otimes_{K_0}\bold{D}_{\mathrm{cris}}(V_{R'}(\widetilde{\delta}_{R'}^{-1}))$$ with the natural projection
$K\otimes_{K_0}\bold{D}_{\mathrm{cris}}(V_{R'}(\widetilde{\delta}_{R'}^{-1}))\rightarrow K\otimes_{K_0}D_2$, is an isomorphism. Hence, for any $\sigma\in\mathcal{P}$, 
 we can take a $R'$-basis of $\mathrm{Fil}^k\bold{D}_{\mathrm{dR}}(V_{R'}(\widetilde{\delta}_{R'}^{-1}))_{\sigma}$ of the form 
$e_{2,\sigma}+a_{\sigma}e_{1,\sigma}$ for some $a_{\sigma}\in R'$. Then, by the definition of benign representations, 
for any $z\in U_y$, $V_z$ is benign if and only if $\prod_{\sigma\in \mathcal{P}}a_{\sigma}(z)\not=0 \in E(z)$ because we have isomorphisms
$\bold{D}_{\mathrm{dR}}(V_{R'}(\widetilde{\delta}_{R'}^{-1}))\otimes_{R'}E(z)\isom \bold{D}_{\mathrm{dR}}(V_z(\widetilde{\delta}_z^{-1}))$ and 
$\bold{D}_{\mathrm{cris}}(V_{R'}(\delta_{R'}^{-1}))\otimes_{R'}E(z)\isom \bold{D}_{\mathrm{cris}}(V_z(\widetilde{\delta}_z^{-1}))$ etc.
Hence, to finish the proof of the proposition, it is enough to show that $\prod_{\sigma\in \mathcal{P}}a_{\sigma}$ is a 
non-zero divisor in $R'$. To prove this claim, it is enough to show that $a_{\sigma}\in \widehat{\mathcal{O}}_{U_y,z}\isom 
R_{V_z}^{\mathrm{cris}}$ is non-zero for any $\sigma\in \mathcal{P}$ and $z\in U_y$  because $R_{V_z}^{\mathrm{cris}}$ is domain. 
To prove this claim, we first note that we have isomorphisms  $$\bold{D}_{\mathrm{cris}}(V_{R'}(\widetilde{\delta}_{R'}^{-1}))\otimes_{R'}R_{V_z}^{\mathrm{cris}}\isom 
\bold{D}_{\mathrm{cris}}(V_{R_{V_z}^{\mathrm{cris}}}(\widetilde{\delta}_{R_{V_z}^{\mathrm{cris}}}^{-1})):=\varprojlim_{n\geqq 1} \bold{D}_{\mathrm{cris}}(V_{R_{V_z}^{\mathrm{cris}}/\mathfrak{m}^n}(\widetilde{\delta}_{R_{V_z}^{\mathrm{cris}},n}^{-1}))$$ and $$\bold{D}_{\mathrm{dR}}(V_{R'}(\widetilde{\delta}_{R'}^{-1}))\otimes_{R'}R_{V_z}^{\mathrm{cris}}\isom \bold{D}_{\mathrm{dR}}(V_{R_{V_z}^{\mathrm{cris}}}(\widetilde{\delta}_{R_{V_z}^{\mathrm{cris}}}^{-1})):=\varprojlim_{n\geqq 1} \bold{D}_{\mathrm{dR}}(V_{R_{V_z}^{\mathrm{cris}}/\mathfrak{m}^n}(\widetilde{\delta}_{R_{V_z}^{\mathrm{cris}},n}^{-1}))$$ by construction and by Corollary 6.3.3 of \cite{Be-Co08}, where we denote by $\mathfrak{m}$ the maximal ideal of $R_{V_z}^{\mathrm{cris}}$ and denote 
by $\delta_{R_{V_z}^{\mathrm{cris}}}:\mathcal{O}_K^{\times}\rightarrow (R^{\mathrm{cris}}_{V_z})^{\times}$ the homomorphism induced from 
$\delta_{R'}$ and denote by $\delta_{R_{V_z}^{\mathrm{cris}}, n}:\mathcal{O}_K^{\times}\rightarrow (R_{V_z}^{\mathrm{cris}}/\mathfrak{m}^n)^{\times}$ the reduction 
of $\delta_{R_{V_z}^{\mathrm{cris}}}$ for $n\geqq 1$. Hence, the claim follows from the following lemma.

\end{proof}

\begin{lemma}\label{6.6}
Let $V$ be a crystalline $E$-representation with Hodge-Tate weights $\{0, -k_{\sigma}\}_{\sigma\in \mathcal{P}}$ such 
that $k_{\sigma}\in\mathbb{Z}_{\geqq 1}$ for any $\sigma\in\mathcal{P}$. 
Assume that $\bold{D}_{\mathrm{cris}}(V_{R_{V}^{\mathrm{cris}}})=K_0\otimes_{\mathbb{Q}_p}R_{V}^{\mathrm{cris}}e_1\oplus K_0\otimes_{\mathbb{Q}_p}R_{V}^{\mathrm{cris}}e_2$ such that $\varphi^f(e_1)=\lambda_1'e_1, \varphi^f(e_2)=\lambda_2'e_2$ for some 
$\lambda'_1,\lambda'_2\in (R_{V}^{\mathrm{cris}})^{\times}$ and that $\mathrm{Fil}^{k_{\sigma}}\bold{D}_{\mathrm{dR}}(V_{R^{\mathrm{cris}}_V})_{\sigma}$ is generated by $e_{2,\sigma}+a_{\sigma}e_{1,\sigma}$ for any $\sigma\in\mathcal{P}$.
Then, we have $a_{\sigma}\not= 0\in R_{V}^{\mathrm{cris}}$ for any $\sigma\in \mathcal{P}$.
\end{lemma} 
\begin{proof}
Denote by $\lambda_i:=\overline{\lambda'}_i\in E^{\times}$ and $\bar{a}_{\sigma}\in E$, the images 
of $\lambda'_i$ and $a_{\sigma}$ by the natural quotient map
$R_{V}^{\mathrm{cris}}\rightarrow E$. Then we have $\bold{D}_{\mathrm{cris}}(V)=K_0\otimes_{\mathbb{Q}_p}E\bar{e}_1\oplus 
K_0\otimes_{\mathbb{Q}_p}E\bar{e}_2$ such that $\varphi^f(\bar{e}_i)=\lambda_i\bar{e}_i$ and 
$\mathrm{Fil}^{k_{\sigma}}\bold{D}_{\mathrm{dR}}(V)_{\sigma}=E(\bar{e}_{2,\sigma}+\bar{a}_{\sigma}\bar{e}_{1,\sigma})$. 
For any $b:=\{b_{\sigma}\}_{\sigma\in \mathcal{P}}\in \prod_{\sigma\in \mathcal{P}}E$, we construct a deformation $D(b)$ of 
$\bold{D}_{\mathrm{cris}}(V)$ over $E[\varepsilon]$ by $D(b):=\bold{D}_{\mathrm{cris}}(V)\otimes_E E[\varepsilon]$ as a $\varphi$-module 
and $\mathrm{Fil}^{0}(K\otimes_{K_0}D(b))=K\otimes_{K_0}D(b)$ and 
$$\mathrm{Fil}^1(K\otimes_{K_0}D(b))_{\sigma}=
\mathrm{Fil}^{k_{\sigma}}(K\otimes_{K_0}D(b))_{\sigma}:=E[\varepsilon](\bar{e}_{2,\sigma}+ (\bar{a}_{\sigma}+b_{\sigma}\varepsilon)\bar{e}_{1,\sigma}),$$ 
$\mathrm{Fil}^{k_{\sigma}+1}(K\otimes_{K_0}D(b))_{\sigma}=0$. For any $b$ as above, $D(b)$ is a deformation of $\bold{D}_{\mathrm{cris}}(V)$ over $E[\varepsilon]$.
Here, we remark that $D(b)$ is automatically  weakly admissible because, as an $E$-filtered $\varphi$-module,  $D(b)$ is an extension of $\bold{D}_{\mathrm{crys}}(V)$ by 
$\bold{D}_{\mathrm{crys}}(V)$ and because the weakly admissibility is closed under extensions.
The existence of such deformations implies that $a_{\sigma}\not=0$ for any $\sigma\in\mathcal{P}$.

\end{proof}

Next, we will prove a proposition concerning  the Zariski density of benign points in $\mathcal{E}(\bar{\rho})$. 
Before proving this proposition, we first prove some lemmas concerning general (maybe well-known easy) 
facts about rigid geometry.

\begin{lemma}\label{6.7}
Let $T_n$ be the $n$-dimensional closed unit disc defined over $E$. Then, for any admissible open $U$ of $T_n$ 
which contains the origin $0:=(0,\cdots,0)\in T_n$, there exists $m>>0$ such that 
$\{(x_1,\cdots,x_n)\in T_n| |x_i|\leqq 1/p^m $ for any $1\leqq i\leqq n \} \subseteq U$.
\end{lemma}
\begin{proof}
Because $U$ is admissibly covered by rational subdomains, we may assume that $U$ is itself 
a rational subdomain, namely, we may assume that there exist $f_1,\cdots,f_d, g\in E\{\{T_1,\cdots,T_n\}\}$ such that 
$(f_1,\cdots,f_d,g)=E\{\{T_1,\cdots,T_n\}\}$ and $U=\{x=(x_1,\cdots,x_n)\in T_n| |f_i(x)|\leqq |g(x)| $ for any $1\leqq i\leqq d \}$.
Then, the condition $0\in U$ means that $|f_{i,0}|\leqq |g_0|$ for any $i$, where $f_{i,0}, g_0\in E$ are the constant terms of $f_i$  and $g$. 
If $g_0=0$, then $f_{i,0}=0$ for any $i$, and this implies that $(f_1,\cdots,f_d,g)\subseteq (T_1,T_2\cdots,T_n)$, which is a contradiction. 
Hence we have $g_0\not=0$ and then, because the norms of coefficients of $f_i$ and $g$ are bounded, there exits $m>>0$ large enough 
such that $|f_i(x)|\leqq \mathrm{max}\{|f_{i,0}|, |g_0|\}=|g_0|$ and $ |g(x)|=|g_0|$ for any $x=(x_1,\cdots,x_n)\in T_n$ such that $|x_i|\leqq 1/p^m$ for any $i$, i.e. $\{x\in T_n| |x_i|\leqq 1/p^m $ for any $i \}\subseteq U$.
\end{proof}


\begin{lemma}\label{6.8}
Let $x:=([V_x],\delta_x,\lambda_x)\in \mathcal{E}(\bar{\rho})$ be an $E$-rational  point such that $V_x$ is crystalline trianguline 
as in Lemma $\mathrm{\ref{6.2}}$, and let $U\subseteq \mathcal{E}(\bar{\rho})$ be an admissible open neighborhood of $x$. Then, there exists an admissible open neighborhood
$U'\subseteq U$ of $x$ such that $U'_{\mathrm{cl},x}:=U'\cap \pi^{-1}((\mathcal{W}_E\times_E\mathcal{W}_E)_{\mathrm{cl},x})$ 
is Zariski dense 
in $U'$.
\end{lemma}
\begin{proof}
Re-taking smaller $U$, we may assume that  $U$ satisfies the properties as in before Lemma \ref{6.2} and 
that the morphism $\pi|_U:U\rightarrow \mathcal{W}_E\times_E\mathcal{W}_E$ is smooth 
and $U$ is irreducible smooth of its dimension $3[K:\mathbb{Q}_p]+1$ by Theorem \ref{5.18} and Lemma \ref{6.1}. In particular, we may assume that 
$\pi(U)\subseteq \mathcal{W}_E\times_E\mathcal{W}_E$
is an admissible open by Corollary 5.11 of \cite{BL93}. By definition of $\mathcal{W}_E\times_E\mathcal{W}_E$ and $(\mathcal{W}_E\times_E\mathcal{W}_E)_{\mathrm{cl},x}$ and by Lemma \ref{6.7}, if we re-take $U$ smaller, then we may assume that there exists an admissible open 
neighborhood $V$ of $y:=\pi(x)$ which is isomorphic to $T_n\isom V$ where $n:=2[K:\mathbb{Q}_p]$ such that $y$ corresponds to the origin $0\in T_n$ and that, for any $m\geqq 1$, 
the set $V_{\mathrm{cl},m}:=\{x\in T_n| |x_i|\leqq 1/p^m $ for any $1\leqq i\leqq n\}\cap (\mathcal{W}_E\times_E\mathcal{W}_E)_{\mathrm{cl},x}$ is Zariski 
dense in $V$ and that $\pi(U)\subseteq V$ and that $\pi|_U:U\rightarrow V$ factors through an \'etale morphism $\pi':U\rightarrow V\times_E T_{n'}$ 
satisfying $\pi'(x)=(y,0)$ for $n':=[K:\mathbb{Q}_p]+1$. Because $V_{\mathrm{cl},m}$ is Zariski dense in $V$ for any $m$, the set $(V\times_ET_{n'})_{\mathrm{cl},m}
:=\{(y',z)\in V_{\mathrm{cl},m}\times_ET_{n'}| |z_i|\leqq 1/p^m $ for any $1\leqq i\leqq n' \}$ is also Zariski dense in 
$V\times_E T_{n'}$. Because $\pi'(U)$ is an admissible open neighborhood of $(y,0)\in V\times_ET_{n'}$, there exists 
$m>>0$ such that $(V\times_ET_{n'})_{\mathrm{cl},m}$
is contained in $\pi'(U)$ by Lemma \ref{6.7}. Then, we have $\pi^{' -1}((V\times_E T_{n'})_{\mathrm{cl},m})\subseteq 
\pi^{-1}(V_{\mathrm{cl},m})\subseteq U_{\mathrm{cl},x}$, then the lemma follows from the following lemma.

\end{proof}

\begin{lemma}\label{6.9}
Let $f:U:=\mathrm{Spm}(B)\rightarrow V:=\mathrm{Spm}(E\{\{T_1,\cdots,T_n\}\})$ be an \'etale morphism between $E$-affinoids for 
some $n$. 
We assume that $U$ is irreducible and reduced. 
If $V_{\mathrm{cl}}\subseteq V$ is a Zariski dense subset of $V$ such that $V_{\mathrm{cl}}\subseteq f(U)$, then 
$f^{-1}(V_{\mathrm{cl}})$ is also Zariski dense in $U$.
\end{lemma}
\begin{proof}
By the assumption, the natural map $A:=E\{\{T_1,\cdots,T_n\}\}\rightarrow \prod_{x\in V_{\mathrm{cl}}}E(x)$ is an injection. To prove the lemma, it suffices to show that the kernel of the natural map 
$B\rightarrow \prod_{y\in f^{-1}(V_{\mathrm{cl}})}E(y)$ is zero. If $I$ is the kernel of this map, then the map 
$A\rightarrow B/I\hookrightarrow \prod_{y\in f^{-1}(V_{\mathrm{cl}})}E(y)$ is equal to the 
map $A\hookrightarrow \prod_{x\in V_{\mathrm{cl}}}E(x)\rightarrow \prod_{y\in f^{-1}(V_{\mathrm{cl}})}E(y)$. 
Because we have $V_{\mathrm{cl}}\subseteq f(U)$ by the assumption, the map $\prod_{x\in V_{\mathrm{cl}}}E(x)\rightarrow 
\prod_{y\in f^{-1}(V_{\mathrm{cl}})}E(y)$ is an injection. Therefore, the map $A\hookrightarrow B/I$ is also an injection.
Then, we have $\mathrm{dim}(A)\leqq \mathrm{dim}(B/I)(\leqq \mathrm{dim}(B))$ by Lemma \ref{6.10} below. 
From this, we have $\mathrm{dim}(B/I)=\mathrm{dim}(B)$ because $B$ is \'etale over $A$. 
Because $U$ is irreducible and reduced, we obtain the equality  $I=0$.

\end{proof}

\begin{lemma}\label{6.10}
Let $f:Z:=\mathrm{Spm}(B')\rightarrow \mathrm{Spm}(E\{\{T_1,\cdots,T_n\}\})$ be a morphism of 
affinoids over $E$. 
We assume that the induced map $A:=E\{\{T_1,\cdots,T_n\}\}\rightarrow B'$ is an injection.
Then, we have $\mathrm{dim}(A)\leqq \mathrm{dim}(B')$.
\end{lemma}
\begin{proof}
Because $A\hookrightarrow B'$ is an injection, 
the base change $\mathrm{Frac}(A)\hookrightarrow \mathrm{Frac}(A)\otimes_A B'$ is 
also an injection, in particular, the generic fiber of the morphism of schemes
 $f_0:\mathrm{Spec}(B')\rightarrow \mathrm{Spec}(A)$ induced from the injection $A\hookrightarrow B'$ 
 is not empty. We denote by $x$ the generic point of $\mathrm{Spec}(A)$ and take a point 
 $y\in f_0^{-1}(x)$. 
 By Proposition 2.1.1 of \cite{Berk93}, if we denote by $\kappa(x)$ and 
 $\kappa(y)$ the residue fields (in the sense of scheme) at $x$ and $y$, then the natural 
 inclusion $\kappa(x)\hookrightarrow \kappa(y)$ is an inclusion of valuation fields which induces an 
 inclusion $\widetilde{\kappa}(x)\hookrightarrow \widetilde{\kappa}(y)$, where $\widetilde{\kappa}(-)$ is the residue 
 field of the valuation field $\kappa(-)$. Form this inclusion, we obtain 
 $(\mathrm{dim}(A)=n=)s(\widetilde{\kappa}(x)/E)\leqq s(\widetilde{\kappa}(y)/E)$, where $s(\widetilde{\kappa}(-)/E)$ is the transcendence
  degree of $\widetilde{\kappa}(-)$ over $E$. By Lemma 2.5.2 of \cite{Berk93}, then
  we also have $s(\widetilde{\kappa}(y)/E)\leqq \mathrm{dim}(B')$, hence we obtain $\mathrm{dim}(A)\leqq \mathrm{dim}(B')$.

\end{proof}

Set 
$$\mathcal{E}(\bar{\rho})_{\mathrm{b}}:=\{x\in \mathcal{E}(\bar{\rho})| V_x\text{ is benign and crystalline }\}.$$

\begin{prop}\label{6.11}
Let $x$ be an $E$-rational  point  in $\mathcal{E}(\bar{\rho})$ as in Lemma $\mathrm{\ref{6.8}}$, and 
let $U$ be an admissible open neighborhood of $x$.
If we take an affinoid neighborhood $U':=\mathrm{Spm}(R)$ of $x$ as in Lemma $\mathrm{\ref{6.8}}$. 
Then, $U'_{\mathrm{b}}:=\mathcal{E}(\bar{\rho})_{\mathrm{b}}\cap U'$ is also Zariski dense in $U'$.
\end{prop}
\begin{proof}
Consider any element $f\in R$ in the kernel of the natural map 
$R\rightarrow \prod_{z\in U_{\mathrm{b}}}E(z)$. Then, for any 
$y\in (\mathcal{W}_E\times_E\mathcal{W}_E)_{\mathrm{cl},x}\cap \pi(U)$, $f|_{\pi^{-1}(y)\cap U}\in \mathcal{O}_{\pi^{-1}(y)\cap U}$ is equal to zero by Proposition \ref{6.5} because $\mathcal{O}_{\pi^{-1}(y)\cap U}$ is reduced. Hence,  we obtain $f=0\in R$ by 
Lemma \ref{6.8}.
\end{proof}

\begin{corollary}\label{6.11.1}
Let $Y$ be the Zariski closure of $\mathcal{E}(\bar{\rho})_{\mathrm{b}}$ 
in $\mathcal{E}(\bar{\rho})$. Then, $Y$ is a union of irreducible components 
of $\mathcal{E}(\bar{\rho})$.
\end{corollary}
\begin{proof}
This follows from Proposition \ref{6.11}.
\end{proof}

Set 
\begin{multline*}
\mathfrak{X}(\bar{\rho})_{\mathrm{reg}-\mathrm{cris}}:=\{x\in \mathfrak{X}(\bar{\rho})| V_x \text{ is crystalline and the Hodge-Tate weights }\\
\text{ of }V_x\text{ are }\{k_{1,\sigma},k_{2,\sigma}\}_{\sigma\in \mathcal{P}}\text{ such that }k_{1,\sigma}\not= k_{2,\sigma}\text{ for any }\sigma\in\mathcal{P} \},
\end{multline*}
and
$$
\mathfrak{X}(\bar{\rho})_{\mathrm{b}}:=\{x\in \mathfrak{X}(\bar{\rho})| V_x\text{ is benign and crystalline }\}.$$

\begin{lemma}\label{6.12}
If $\mathfrak{X}(\bar{\rho})_{\mathrm{reg}-\mathrm{cris}}$ is not empty, then 
$\mathfrak{X}(\bar{\rho})_{\mathrm{b}}$ is also not empty.
\end{lemma}
\begin{proof} 
If $\mathfrak{X}(\bar{\rho})_{\mathrm{reg}-\mathrm{cris}}$ is not empty, then it is easy to show that 
there exists some $x\in \mathfrak{X}(\bar{\rho})_{\mathrm{reg}-\mathrm{cris}}$ which satisfies the condition (1) of Definition \ref{26}. Then, the lemma follows from Proposition \ref{6.5}.

\end{proof}


For a rigid analytic space $Y$ over $E$ and for a point $y\in Y$ , we denote by 
$$t_{Y,y}:=\mathrm{Hom}_{E(y)}(\mathfrak{m}_y/\mathfrak{m}_y^2, E(y))$$
 the tangent space at $y$, where 
$\mathfrak{m}_y$ is the maximal ideal of $\mathcal{O}_{Y,y}$. The following three 
theorems are the 
main theorems of this article concerning the Zariski density of two dimensional crystalline representations.

We denote by $\overline{\mathfrak{X}(\bar{\rho})}_{\mathrm{b}}$ the Zariski closure of $\mathfrak{X}(\bar{\rho})_{\mathrm{b}}$ in $\mathfrak{X}(\bar{\rho})$. The following is a generalization of 
Corollary 1.10 of \cite{Ki10} for general $K$.

\begin{thm}\label{6.14}
If $\mathfrak{X}(\bar{\rho})_{\mathrm{reg}-\mathrm{cris}}$ is non empty, 
then $\overline{\mathfrak{X}(\bar{\rho})}_{\mathrm{b}}$ is non empty and a union of irreducible components 
of $\mathfrak{X}(\bar{\rho})$.
\end{thm}
\begin{proof}
By Lemma \ref{6.12}, $Z:=\overline{\mathfrak{X}(\bar{\rho})}_{\mathrm{b}}$ is non empty.

To show that $Z$ is a union of irreducible components of 
$\mathfrak{X}(\bar{\rho})$, we first claim that the dimension of any irreducible component of $\mathfrak{X}(\bar{\rho})$ is at most $4[K:\mathbb{Q}_p]+1$. 
Take any point  $x:=[V_x]\in \mathfrak{X}(\bar{\rho})$. Under the assumption that 
$\mathrm{End}_{\mathbb{F}}(\bar{\rho})=\mathbb{F}$, we also 
have $\mathrm{End}_{E(x)[G_K]}(V_x)=E(x)$ and we have a canonical isomorphism 
$R_{V_x}\isom \widehat{\mathcal{O}}_{\mathfrak{X}(\rho),x}$ by Proposition 9.5 of \cite{Ki03}. 
Under the condition $\mathrm{End}_{E(x)[G_K]}(V_x)=E(x)$, it is easy to show that 
the dimension of $\mathrm{H}^2(G_K, \mathrm{ad}(V_x))\isom \mathrm{H}^0(G_K, \mathrm{ad}(V_x)(\chi_p))^{\vee}$ 
is at most one. By deformation theory, then the dimension of $R_{V_x}$ is $4[K:\mathbb{Q}_p]+1$, from which the claim follows.

By this claim, 
it suffices to show that the dimension of any irreducible component of $Z$ is at least $4[K:\mathbb{Q}_p]+1$. 
Let $Z'$ be an irreducible component of $Z$. 
Because the singular locus $Z'_{\mathrm{sing}}\subseteq Z'$ is a proper Zariski closed set in $Z'$, there exists a benign point 
$x\in \mathfrak{X}(\bar{\rho})_{\mathrm{b}}\cap Z'$ such that $Z'$ is smooth at $x$. 
By the definition of benign representation and by Theorem \ref{5.14}, there exist
the different two points 
$$x_1:=([V_x],\delta_{x_1},\lambda_{x_1}), x_2:=([V_x],\delta_{x_2},\lambda_{x_2})\in \mathcal{E}(\bar{\rho})$$ 
such that $p_1(x_i)=x$ and satisfy the property (ii) in the Theorem \ref{5.14}.  We denote by $Y'_i$ an irreducible component of 
$p_1^{-1}(Z)$ containing $x_i$ for $i=1,2$ respectively. These are also irreducible components 
of $\mathcal{E}(\bar{\rho})$ by Corollary \ref{6.11.1}, and $Y'_i$ is unique for each $i=1,2$ by Theorem \ref{5.18}.
Because the natural morphism $p_1|_{Y'_i}:Y'_i\rightarrow \mathfrak{X}(\bar{\rho})$ factors through $Z'$ for $i=1,2$, 
we obtain a map
$$t_{\mathcal{E}(\bar{\rho}),x_i}=t_{Y'_i,x_i}\rightarrow t_{Z',x}\hookrightarrow t_{\mathfrak{X}(\bar{\rho}),x}$$
 for $i=1,2$.
Hence, we obtain a map 
$$\bigoplus_{i=1,2}t_{\mathcal{E}(\bar{\rho}),x_i}\rightarrow t_{Z',x}\hookrightarrow t_{\mathfrak{X}(\bar{\rho}),x}.$$ 
By Theorem \ref{33} and Theorem \ref{5.18}, this map is surjective, hence we obtain an equality 
$$t_{Z',x}=t_{\mathfrak{X}(\bar{\rho}),x}. $$ Because $x$ is smooth at $Z'$, hence $Z'$ has dimension $4[K:\mathbb{Q}_p]+1$, which proves the theorem.

\end{proof}

Concerning the assumption that $\mathfrak{X}(\bar{\rho})_{\mathrm{reg}-\mathrm{cris}}$ is non empty, in this paper we prove the following (maybe well-known) lemma.

\begin{lemma}\label{4.16}
If $\bar{\rho}\otimes_{\mathbb{F}}\bar{\mathbb{F}}\not\sim\begin{pmatrix} \omega & \ast \\ 0 & 1 \end{pmatrix}\otimes \chi$ and 
$\bar{\rho}\otimes_{\mathbb{F}}\bar{\mathbb{F}}\not\sim\begin{pmatrix} 1 & \ast \\ 0 & 1\end{pmatrix} \otimes \chi$ for any 
character $\chi:G_K\rightarrow \bar{\mathbb{F}}^{\times}$, where $\omega$ is the mod $p$ cyclotomic 
character. Then, $\mathfrak{X}(\bar{\rho})_{\mathrm{reg}-\mathrm{cris}}$ is non empty.
\end{lemma}
\begin{proof}
First, we prove the lemma for the absolutely reducible case. Extending $\mathbb{F}$, we may assume that $\bar{\rho}$ is reducible. 
Because any character $\chi:G_K\rightarrow \mathbb{F}^{\times}$ has 
a crystalline lift, we may assume that $\bar{\rho}=\begin{pmatrix}\eta & \ast \\ 0 & 1\end{pmatrix}$ for 
a character $\eta:G_K\rightarrow \mathbb{F}^{\times}$ such that $\eta\not= 1$ and $\eta\not= \omega$. Using 
twists of a Lubin-Tate character of $K$ by $\sigma\in\mathcal{P}$ and a unramified character, we can take a crystalline 
lift $\widetilde{\eta}:G_K\rightarrow \mathcal{O}^{\times}$ of $\eta$ whose Hodge-Tate weights  are
$\{k_{\sigma}\}_{\sigma\in \mathcal{P}}$ such that $k_{\sigma}\geqq 1$ for any $\sigma\in \mathcal{P}$. 
Under the assumption $\eta\not= 1, \omega$, $\mathrm{H}^1(G_K,\mathcal{O}(\widetilde{\eta}))$ is a free 
$\mathcal{O}$-module of rank $[K:\mathbb{Q}_p]$ and the natural map 
$\mathrm{H}^1(G_K,\mathcal{O}(\widetilde{\eta}))\rightarrow \mathrm{H}^1(G_K,\mathbb{F}(\eta))$ is a surjection. 
Because $k_{\sigma}\geqq 1$ for any $\sigma\in \mathcal{P}$, we have an equality 
$\mathrm{H}^1_f(G_K,E(\widetilde{\eta}))=\mathrm{H}^1(G_K,E(\widetilde{\eta}))$. 
These imply that any extension class in $\mathrm{H}^1(G_K,\mathbb{F}(\eta))$ lifts to 
an extension class in $\mathrm{H}^1(G_K,\mathcal{O}(\widetilde{\eta}))$ which is crystalline.

Next, we prove the lemma for the absolutely irreducible case. Denote by $K_2$ the unramified extension 
of $K$ of degree $2$, and denote by $\chi_2:G^{\mathrm{ab}}_{K_2}\rightarrow \mathbb{F}_{p^{2f}}^{\times}$ 
the reduction of the Lubin-Tate character $\chi_{2,\mathrm{LT}}:G_{K_2}^{ab}\rightarrow \mathcal{O}_{K_2}^{\times}$ of $K_2$ associated 
to the uniformizer $\pi_{K_2}:=\pi_K$ of $K_2$. Then, it is known that there exists an isomorphism 
$\bar{\rho}\isom (\mathrm{Ind}^{G_{K}}_{G_{K_2}}\chi_2^i)\otimes \chi$ (possibly, after extending scalars) for a character 
$\chi:G_K\rightarrow \mathbb{F}^{\times}$ and for some $i\in \mathbb{Z}$ such that $i\not\equiv 0 $ (mod $p^f+1$), where we also denote by the same letter 
$\chi_2:G_{K_2}^{ab}\rightarrow \mathbb{F}_{p^{2f}}^{\times}\hookrightarrow \mathbb{F}^{\times}$ for 
a fixed embedding $\mathbb{F}_{p^{2f}}^{\times}\hookrightarrow \mathbb{F}^{\times}$.
Hence, it suffices to show that $\mathrm{Ind}^{G_{K}}_{G_{K_2}}\chi_2^i$ has a crystalline lift. Because 
$\chi_2$ is the reduction of $\chi_{2,\mathrm{LT}}$, we can take a lift of $\chi_2^i$ of the form 
$\prod_{\sigma\in \mathcal{P}}\widetilde{\sigma}(\chi_{2,\mathrm{LT}})^{k_{\sigma}}$ such that $k_{\sigma}\geqq 1$ for all 
$\sigma\in \mathcal{P}$, where $\widetilde{\sigma}:K_2\hookrightarrow \overline{E}$  is an extension of $\sigma$. Then, 
$\mathrm{Ind}^{G_{K}}_{G_{K_2}}\chi_2^i$ has a crystalline lift $\mathrm{Ind}^{G_{K}}_{G_{K_2}}(\prod_{\sigma\in \mathcal{P}}
\widetilde{\sigma}(\chi_{2,\mathrm{LT}})^{k_{\sigma}})$ whose Hodge-Tate weights are $\{0,k_{\sigma}\}_{\sigma\in \mathcal{P}}$.

\end{proof}

Finally, we prove the following two theorems on 
the density of $\mathfrak{X}(\bar{\rho})_{\mathrm{b}}$ in $\mathfrak{X}(\bar{\rho})$ 
under the following assumptions. In particular, we need to exclude the case $p=2$. 
Under these conditions, we will show below that 
$\mathfrak{X}(\bar{\rho})$ is a finite union of smooth irreducible components.
Let $\zeta_p\in \overline{K}$ be a primitive root of unity.
The difficulty of the 
proof to the theorems depends on whether $\zeta_p\in K$ or not, which corresponds 
to whether $\mathfrak{X}(\bar{\rho})$ is irreducible or not respectively. 

We first prove the density 
when 
$\zeta_p\not\in K$,

\begin{thm}\label{4.17}
Assume that $\zeta_p\not\in K$. Moreover, assume the following conditions,
\begin{itemize}
\item[(0)]$\mathrm{End}_{G_K}(\bar{\rho})=\mathbb{F}$,
\item[(1)]$\mathfrak{X}(\bar{\rho})_{\mathrm{reg}-\mathrm{cris}}$ is not empty,
\item[(2)]if $\bar{\rho}$ is absolutely reducible, then $\bar{\rho}\otimes_{\mathbb{F}}\overline{\mathbb{F}}\not\sim \begin{pmatrix} 1 & \ast \\ 0 & \omega \end{pmatrix}\otimes \chi$  
for 
any $\chi:G_K\rightarrow \overline{\mathbb{F}}^{\times}$,
\item[(3)]if $\bar{\rho}$ is absolutely irreducible, then $[K(\zeta_p):K]\not=2$ or 
$\bar{\rho}|_{I_K}\otimes_{\mathbb{F}}\overline{\mathbb{F}}\not\sim \begin{pmatrix} \chi_2^i & 0 \\ 0 & \chi_2^{ip^f}\end{pmatrix}$ for any $i$ such that $\chi_2^{i(p^f-1)}|_{I_K}=\omega|_{I_K}$,
\end{itemize}
then we have an equality $\overline{\mathfrak{X}(\bar{\rho})}_{\mathrm{b}}=\mathfrak{X}(\bar{\rho})$.

\end{thm}
\begin{proof}
We claim that $\mathfrak{X}(\bar{\rho})$ is isomorphic to $(4[K:\mathbb{Q}_p]+1)$-dimensional open unit disc under the above conditions, 
from which  the theorem follows by Theorem \ref{6.14}. 

To show the claim, it suffices to show that $\mathrm{H}^2(G_K,\mathrm{ad}(\bar{\rho}))=0$, hence  suffices to show that 
$\mathrm{Hom}_{G_K}(\bar{\rho}, \bar{\rho}\otimes \omega)=0$ by the Tate duality. 
When $\bar{\rho}$ is absolutely reducible, it is easy to see that the conditions (0), (2) imply that 
$\mathrm{Hom}_{G_K}(\bar{\rho}, \bar{\rho}\otimes \omega)=0$. When 
$\bar{\rho}$ is absolutely irreducible, then $\bar{\rho}$ is of the form $\mathrm{Ind}^{G_{K}}_{G_{K_2}}(\chi_2^i)\otimes \chi$ for some 
$i$ and $\chi$ after extending scalars. If $\mathrm{Hom}_{G_K}(\bar{\rho}, \bar{\rho}\otimes \omega)\not=0$, then 
there exists an isomorphism $\bar{\rho}\isom \bar{\rho}\otimes \omega$ by Schur's lemma.
The latter implies that $\mathrm{det}(\bar{\rho})=\mathrm{det}(\bar{\rho})\omega^2$ and 
$\begin{pmatrix} \chi_2^i|_{I_K}& 0 \\ 0 & \chi_2^{i p^f}|_{I_K}\end{pmatrix}\isom 
\begin{pmatrix} \chi_2^i\omega|_{I_K}& 0 \\ 0 & \chi_2^{i p^f}\omega|_{I_K}\ \end{pmatrix}$. 
Because we assume that $\zeta_p\not\in K$, these  imply that $[K(\zeta_p):K]=2$ and $\chi_2^{i(p^f-1)}|_{I_K}=\omega|_{I_K}$, which proves the 
claim, hence proves the theorem.

\end{proof}
Finally, we prove the theorem on the density when $\zeta_p\in K$ and $p\not=2$ under the following assumptions.
\begin{thm}\label{4.18}
Assume that $\zeta_p\in K$ and $p\not=2$. Moreover, assume the following conditions,
\begin{itemize}
\item[(0)]$\mathrm{End}_{G_K}(\bar{\rho})=\mathbb{F}$,
\item[(1)]$\mathfrak{X}(\bar{\rho})_{\mathrm{reg}-\mathrm{cris}}$ is not empty,
\end{itemize}
then we have an equality $\overline{\mathfrak{X}(\bar{\rho})}_{\mathrm{b}}=\mathfrak{X}(\bar{\rho})$.

\end{thm}

\begin{proof}
If $\zeta_p\in K$, then $\mathfrak{X}(\bar{\rho})$ never becomes irreducible. Hence, we first need to know
how to decompose $\mathfrak{X}(\bar{\rho})$ into irreducible components under the above 
assumptions.

Let $P\subset \mathcal{O}_{K}^{\times}$ be the subgroup of $\mathcal{O}_K^{\times}$ consisting of 
all the $p$-th power roots of unity, and let $p^n$ be the order of $P$. Fix $\zeta_{p^n}\in \mathcal{O}_K^{\times}$  a generator of $P$, i.e. a primitive $p^n$-th root of unity. For each $0\leqq i\leqq p^n-1$, we define a subfunctor $D_{\bar{\rho},i}$ of $D_{\bar{\rho}}$ 
by 
$$ D_{\bar{\rho},i}(A):=\{[V_A]\in D_{\bar{\rho}}(A)| \mathrm{det}(V_A)(\mathrm{rec}_K(\zeta_{p^n}))=\iota_A(\zeta_{p^n})^i \}$$
for $A\in \mathcal{C}_{\mathcal{O}}$, 
where $\iota_A:\mathcal{O}\rightarrow A$ is the morphism which gives an $\mathcal{O}$-algebra structure to $A$. 
It is easy to see that the canonical inclusion $D_{\bar{\rho},i}\hookrightarrow D_{\bar{\rho}}$ is relatively representable, i.e. 
this satisfies the conditions (1) and (2) and (3) in the proof of Proposition \ref{15}. For each $i$, let $R_{\bar{\rho},i}$ be the quotient of 
$R_{\bar{\rho}}$ which represents $D_{\bar{\rho},i}$, and let $\mathfrak{X}(\bar{\rho})_i\subseteq \mathfrak{X}(\bar{\rho})$ be the 
Zariski closed rigid analytic space associated to $R_{\bar{\rho},i}$. Then it is easy to see that, as rigid analytic space, 
$\mathfrak{X}(\bar{\rho})$ is the disjoint union of $\mathfrak{X}(\bar{\rho})_i$ for $0\leqq i\leqq p^n-1$, 
$$ \mathfrak{X}(\bar{\rho})=\coprod_{0\leqq i\leqq p^n-1} \mathfrak{X}(\bar{\rho})_i.$$ 
We claim that 
each $\mathfrak{X}(\bar{\rho})_i$ is isomorphic to the  $(4[K:\mathbb{Q}_p]+1)$-dimensional open unit disc. 
To prove this claim, 
it suffices to show that the functor $D_{\bar{\rho},i}$ is formally smooth. 

We prove the formal smoothness of $D_{\bar{\rho},i}$ as follows. 
Let $A$ be an object of $\mathcal{C}_{\mathcal{O}}$ and $I\subseteq A$ be a non zero ideal such that $I\mathfrak{m}_A=0$. 
Let $[V_{A/I}]\in D_{\bar{\rho},i}(A/I)$ be a deformation of $\bar{\rho}$ over $A/I$. Then, it suffices to show that 
$[V_{A/I}]$ lifts to $D_{\bar{\rho},i}(A)$. Fixing a $A/I$-basis of $V_{A/I}$, we represent $V_{A/I}$ by 
a continuous homomorphism $\rho_{A/I}:G_K\rightarrow \mathrm{GL}_2(A/I)$. Because the obstruction of the liftings 
of $\mathrm{det}(\bar{\rho})$ comes only from that of $\mathrm{det}(\bar{\rho})|_{\mathrm{rec}_K(P)}$, we can take a continuous 
 character 
$c_A:G_K^{\mathrm{ab}}\rightarrow A^{\times}$ which is a lift of $\mathrm{det}(\rho_{A/I})$ and $c_{A}(\mathrm{rec}_K(\zeta_{p^n}))
=\iota_A(\zeta_{p^n})^i$. We take a continuous lift $\widetilde{\rho}_A:G_K\rightarrow \mathrm{GL}_2(A)$ of 
$\rho_{A/I}$ such that $\mathrm{det}(\widetilde{\rho}_A(g))=c_A(g)$ for any $g\in G_K$ and then 
we define a $2$-cocycle $f:G_K\times G_K\rightarrow I\otimes_{\mathbb{F}}\mathrm{ad}(\bar{\rho})$ by 
$$\widetilde{\rho}_A(g_1g_2)\widetilde{\rho}_A(g_2)^{-1}\widetilde{\rho}_A(g_1)^{-1}:=1+f(g_1,g_2)\in 1+I\otimes_A\mathrm{M}_2(A)=1+I\otimes_{\mathbb{F}}\mathrm{ad}(\bar{\rho}).$$ 
Because $\mathrm{det}(\widetilde{\rho}_A)=c_A$ is a homomorphism, $f(g_1,g_2)$ is contained in $I\otimes_{\mathbb{F}}\mathrm{ad}^0(\bar{\rho})$, where we denote by $\mathrm{ad}^0(\bar{\rho}):=\{a\in \mathrm{ad}(\bar{\rho})| \mathrm{trace}(a)=0\}$. 
Hence, we obtain 
a class of $2$-cocycle $[f]\in \mathrm{H}^2(G_K,\mathrm{ad}^0(\bar{\rho}))$. Under the assumption (0) and the assumption that $\zeta_p\in K$ and $p\not=2$, we have 
$$\mathrm{H}^2(G_K, \mathrm{ad}^0(\bar{\rho}))\isom \mathrm{H}^0(G_K, \mathrm{ad}^0(\bar{\rho})(\omega))^{\vee}=
\mathrm{H}^0(G_K, \mathrm{ad}^0(\bar{\rho}))^{\vee}=0$$
(we remark that we have $\mathrm{H}^0(G_K, \mathrm{ad}^0(\bar{\rho}))=\mathrm{H}^0(G_K,\mathrm{ad}(\bar{\rho}))=\mathbb{F}$ when $p=2$). Hence, twisting $\widetilde{\rho}_A$ by using a suitable continuous one cochain 
$d:G_K\rightarrow I\otimes_{\mathbb{F}}\mathrm{ad}^0(\bar{\rho})$, we obtain a continuous homomorphism 
 $\rho_A:G_K\rightarrow \mathrm{GL}_2(A)$ such that $\rho_A$ is a lift of $\rho_{A/I}$ and $\mathrm{det}(\rho_{A})=c_A$,
  which proves the formally smoothness of $D_{\bar{\rho},i}$.
 
 By this claim and by Theorem \ref{6.14}, to prove the theorem, it suffices to show that $\mathfrak{X}(\bar{\rho})_i\cap 
 \mathfrak{X}(\bar{\rho})_{\mathrm{b}}$ is non empty for any $i$ under the assumption (1). We prove this claim as follows. 
 First, there exists some $i$ such that $\mathfrak{X}(\bar{\rho})_i\cap \mathfrak{X}(\bar{\rho})_{\mathrm{b}}$ is 
 non empty by the assumption (1). We take a point $x=[V_x]\in \mathfrak{X}(\bar{\rho})_i\cap \mathfrak{X}(\bar{\rho})_{\mathrm{b}}$.
 
The twist $V_x(\chi_{\mathrm{LT}}^{j(p^f-1)})$ of $V_x$ for any $j\in \mathbb{Z}$ is contained in $\mathfrak{X}(\bar{\rho})_{\mathrm{b}}\cap \mathfrak{X}(\bar{\rho})_{i_j}$, where 
we define $i_j$ such that $0\leqq i_j\leqq p^n-1$ and $i_j\equiv i+2j(p^f-1)$ ( mod $p^n$). Because we assume $p\not=2$, $i_j$ runs through all 
$0\leqq i'\leqq p^n-1$, hence $\mathfrak{X}(\bar{\rho})_{\mathrm{b}}\cap \mathfrak{X}(\bar{\rho})_{i'}$ is non empty 
for any $i'$. Hence, $\mathfrak{X}(\bar{\rho})_{\mathrm{b}}$ is Zariski dense in $\mathfrak{X}(\bar{\rho})$.

\end{proof}
\begin{rem}
We remark that, from $\S$ 3.3, we assume $\mathrm{End}_{G_K}(\bar{\rho})=\mathbb{F}$. However, even if 
$\mathrm{End}(\bar{\rho})\not=\mathbb{F}$,
it may be possible to prove Theorem \ref{5.14} and Theorem \ref{5.18} and Theorem \ref{6.14} without any additional 
difficulties if we use the universal framed deformations 
instead of usual deformations. 
But, up to now,  the author does not know whether the density is satisfied or not when
$\mathrm{End}_{G_K}(\bar{\rho})\not=\mathbb{F}$.

\end{rem}

\section{Appendix : Continuous cohomology of $B$-pairs}
In \cite{Na09}, we defined a cohomology $\mathrm{H}^i(G_K,W)$ by using continuous cochains of $G_K$ which 
we review below. On the other hand,  Liu \cite{Li08} defined another cohomology which we write by $\mathrm{H}^i_{\mathrm{Liu}}(G_K, W):=
\mathrm{H}^i_{\varphi, \Gamma}(D(W))$ 
by using a complex defined from the $(\varphi,\Gamma)$-module $D(W)$ associated to $W$ (see 2.1 of \cite{Li08} for the definition). Moreover, he proved that this cohomology 
satisfies the Euler-Poincar\'e formula and the Tate duality. In this appendix, we first prove  that  $\mathrm{H}^i(G_K, W)$ also 
satisfies the Euler-Poincar\'e  formula and the Tate duality, and finally prove that $\mathrm{H}^i(G_K, W)$ is canonically isomorphic to  $\mathrm{H}^i_{\mathrm{Liu}}(G_K, W)$.

We first recall the definition of $\mathrm{H}^i(G_K, W)$. Let $G$ be a topological group. For a continuous $G$-module $M$ and $i\in\mathbb{Z}_{\geqq 0}$, we define the group of 
 $i$-th continuous cochains of $G$ with values in $M$ by $$\mathrm{C}^i(G, M):=\{c:G^{\times i}\rightarrow M| c\text{ is a continuous map }\}.$$ 
 As usual, we define the boundary map 
  $$\partial^i:\mathrm{C}^i(G, M)\rightarrow \mathrm{C}^{i+1}(G, M)$$ by 
  \[
\begin{array}{ll}
\partial^i(c)(g_1,g_2,\cdots, g_{i+1}):=& g_1c(g_2,\cdots, g_{i+1}) +(-1)^{i+1}c(g_1,g_2,\cdots, g_i)\\
 &+\sum_{s=1}^{i}(-1)^i c(g_1,\cdots, g_{s-1}, g_s g_{s+1},g_{s+2},\cdots, g_{i+1}) .
\end{array}
\]
 Let $W=(W_e, W^+_{\mathrm{dR}})$ be a $B$-pair. Set $W_{\mathrm{dR}}:=W_e\otimes_{\bold{B}_{e}}\bold{B}_{\mathrm{dR}}$. For $W$, we define a complex $\mathrm{C}^{\bullet}(G_K, W)$ of $\mathbb{Q}_p$-vector spaces as 
 the mapping cone of the map $$\mathrm{C}^{\bullet}(G_K, W_e)\oplus \mathrm{C}^{\bullet}(G_K, W^+_{\mathrm{dR}})\rightarrow 
 \mathrm{C}^{\bullet}(G_K, W_{\mathrm{dR}}): (c_e, c_{\mathrm{dR}})\mapsto c_e-c_{\mathrm{dR}},$$ i.e. defined by 
 $$\mathrm{C}^0(G_K, W):=\mathrm{C}^0(G_K, W_e)\oplus \mathrm{C}^0(G_K, W^+_{\mathrm{dR}})$$
 and 
 $$\mathrm{C}^i(G_K, W):=\mathrm{C}^{i}(G_K, W_e)\oplus \mathrm{C}^i(G_K, W^+_{\mathrm{dR}})\oplus \mathrm{C}^{i-1}(G_K, W_{\mathrm{dR}})$$ for $i\geqq 1$ and the differential
 $$ \partial^0:\mathrm{C}^0(G_K, W_e)\oplus \mathrm{C}^0(G_K,  W^+_{\mathrm{dR}}) \rightarrow \mathrm{C}^1(G_K, W_e)\oplus \mathrm{C}^1(G_K, W^+_{\mathrm{dR}})\oplus \mathrm{C}^0(G_K, W_{\mathrm{dR}})$$ is defined by 
 $$\partial^0(c_e, c_{\mathrm{dR}}):=(\partial^0(c_e),  \partial^0(c_{\mathrm{dR}}), c_e-c_{\mathrm{dR}})$$
and, for $i\geqq 1$, the differential
 \begin{multline*}
 \partial^i:\mathrm{C}^i(G_K, W_e)\oplus \mathrm{C}^i(G_K, W^+_{\mathrm{dR}})\oplus \mathrm{C}^{i-1}(G_K, W_{\mathrm{dR}})\\
 \rightarrow \mathrm{C}^{i+1}(G_K, W_e) \oplus \mathrm{C}^{i+1}(G_K, W^+_{\mathrm{dR}})\oplus \mathrm{C}^i(G_K, W_{\mathrm{dR}})
 \end{multline*}
 is defined by 
 $$\partial^i(c_e, c_{\mathrm{dR}}, c)=(\partial^i(c_e), \partial^i(c_{\mathrm{dR}}), c_e-c_{\mathrm{dR}}-\partial^{i-1}(c)).$$
   We define  the cohomology of $W$ by 
 $$\mathrm{H}^i(G_K, W):=\mathrm{H}^i(\mathrm{C}^{\bullet}(G_K, W)) ,$$ and also define
 $$\mathrm{H}^i(G_K, W_{e}):=\mathrm{H}^i(\mathrm{C}^{\bullet}(G_K, W_e))$$ and 
 $$\mathrm{H}^i(G_K, W^+_{\mathrm{dR}}):=
 \mathrm{H}^i(\mathrm{C}^{\bullet}(G_K, W^+_{\mathrm{dR}})), \,\,\mathrm{H}^i(G_K, \allowbreak W_{\mathrm{dR}}):=\mathrm{H}^i(\mathrm{C}^{\bullet}(G_K, W_{\mathrm{dR}})).$$
 By these definitions, we have the following long exact sequence,
 $$ \cdots \rightarrow \mathrm{H}^{i-1}(G_K, W_{\mathrm{dR}})\rightarrow \mathrm{H}^i(G_K, W)\rightarrow \mathrm{H}^i(G_K, W_e)\oplus 
 \mathrm{H}^i(G_K, W^+_{\mathrm{dR}})\rightarrow \cdots .$$
 
 Before proving the Euler-Poincar\'e formula, we recall some results of \cite{Be09} on the relationship 
 between $B$-pairs and almost $\mathbb{C}_p$-representations.
 Let $U$ be an almost $\mathbb{C}_p$-representation. Let $V_1$ and $V_2$ be 
 $\mathbb{Q}_p$-representations of $G_K$ of dimension $d_1$ and $d_2$ respectively 
 and $d\geqq 0$ be an integer 
 such that we have $V_1\subseteq U$ and $V_2\subseteq \mathbb{C}_p^{\oplus d}$ and 
 $U/V_1\isom \mathbb{C}_p^{\oplus d}/V_2$. Then, we define 
 the dimension of $U$ by $$\mathrm{dim}_{\mathcal{C}(G_K)}(U):=d$$ and the height of $U$ by $$\mathrm{ht}(U):=d_1-d_2,$$
 which are independent of the choice of $V_1,V_2$ and are additive with respect to exact sequences (\cite{Fo03}).
For a $B$-pair $W:=(W_e, W^+_{\mathrm{dR}})$, we define
 $$X_0(W):=W_e\cap W^+_{\mathrm{dR}}\text{ and } X_1(W):=W_{\mathrm{dR}}/(W_e+W^+_{\mathrm{dR}}).$$
 Concerning $X_0(W)$ and $X_1(W)$, Berger \cite{Be09} proved the following theorem.
 \begin{thm}\label{5-0}
 Let $W$ be a $B$-pair of rank $d$, then 
 \begin{itemize}
 \item[(1)]$X_0(W)$ and $X_1(W)$ are almost $\mathbb{C}_p$-representations,
 \item[(2)]if $W$ is pure of slope $s\leqq 0$, then we have $\mathrm{dim}_{\mathcal{C}(G_K)}(X_0(W))=-
 sd$, $\mathrm{ht}(X_0(W))=d$ and $X_1(W)=0$,
 \item[(3)]if $W$ is pure of slope $s>0$, then we have $X_0(W)=0$ and $\mathrm{dim}_{\mathcal{C}(G_K)}(X_1(W))=sd$, $\mathrm{ht}(X_1(W))=-d$.
 
 \end{itemize}
 \end{thm}
 \begin{proof}
 See Theorem 3.1 of \cite{Be09}.
 \end{proof}
 \begin{lemma}\label{5-0-5}
 Let $U$ be an almost $\mathbb{C}_p$-representation, then $\mathrm{H}^i(G_K, U)$ is finite dimensional over $\mathbb{Q}_p$ for $i=0,1,2$ and zero for $i\geqq 3$.
 \end{lemma}
 \begin{proof}
 This follows from the definition of almost $\mathbb{C}_p$-representations and 
 the facts that $\mathrm{H}^i(G_K, V)=0$ for $i\geqq 3$ for any $\mathbb{Q}_p$-representation $V$
 of $G_K$ and that $\mathrm{H}^{i}(G_K, \mathbb{C}_p)=0$ for $i\geqq 2$ and that $\mathrm{H}^i(G_K, V)$ and 
 $\mathrm{H}^i(G_K, \mathbb{C}_p)$ are finite dimensional over $\mathbb{Q}_p$.
 \end{proof}
 For an almost $\mathbb{C}_p$-representation $U$, set $\chi(U):=\sum_{i=0}^2(-1)^i\mathrm{dim}_{\mathbb{Q}_p}
 \mathrm{H}^i(G_K, U)$.
 \begin{lemma}\label{5-0-6}
 $\chi(U)=-[K:\mathbb{Q}_p]\mathrm{ht}(U)$.
 \end{lemma}
 \begin{proof}
 This follows from the definition of almost $\mathbb{C}_p$-representations and the Euler-Poincar\'e formula for 
 $\mathbb{Q}_p$-representations of $G_K$ and the fact that $\chi(\mathbb{C}_p)=0$.
 \end{proof}

 \begin{lemma}\label{5-1}
 Let $W=(W_e,W^+_{\mathrm{dR}})$ be a $B$-pair, then the following equalities hold,
 \begin{itemize}
 \item[(1)]$\mathrm{C}^{\bullet}(W_e)=\varinjlim_n \mathrm{C}^{\bullet}(W_e\cap \frac{1}{t^n}W^+_{\mathrm{dR}}),$
 \item[(2)]$\mathrm{C}^{\bullet}(W_{\mathrm{dR}})=\varinjlim_n\mathrm{C}^{\bullet}(G_K, \frac{1}{t^n}W^+_{\mathrm{dR}}).$
 \end{itemize}
 
 \end{lemma}
 \begin{proof}
 For any $n$, $\frac{1}{t^n}W^+_{\mathrm{dR}}$ is closed in $\frac{1}{t^{n+1}}W^+_{\mathrm{dR}}$ and 
 the topology on $\frac{1}{t^n}W^+_{\mathrm{dR}}$ is the topology induced from $\frac{1}{t^{n+1}}W^+_{\mathrm{dR}}$. 
 Hence, by Proposition 5.6 of \cite{Schn01}, we obtain the equality (2).
 For $W_e$, if we fix an isomorphism $W_e\isom \bold{B}_{e}^{\oplus d}$ as $\bold{B}_{e}$-module, the topology on $W_e$ 
 is defined by the direct sum topology of $\bold{B}_{e}$. Because we have an equality 
 $$t\bold{B}_{\mathrm{max}}^{+, \varphi=p^n}=
 \cap_{m\geqq 0} \mathrm{Ker}(\theta\circ\varphi^{m}: \bold{B}_{\mathrm{max}}^{+,\varphi=p^{n+1}}\rightarrow \mathbb{C}_p)$$ by 
 Proposition 8.10 (2) of \cite{Co02}, 
 $\frac{1}{t^n}\bold{B}_{\mathrm{max}}^{+,\varphi=p^n}$ is closed in $\frac{1}{t^{n+1}}\bold{B}_{\mathrm{max}}^{+,\varphi=p^{n+1}}$ and the 
 topology on $\frac{1}{t^n}\bold{B}_{\mathrm{max}}^{+,\varphi=p^n}$ is the topology induced from $\frac{1}{t^{n+1}}\bold{B}_{\mathrm{max}}^{+,\varphi=p^{n+1}}$.
 Hence, by Proposition 5.6 of \cite{Schn01}, we have $\mathrm{C}^{\bullet}(G_K, W_e)
 =\varinjlim_n\mathrm{C}^{\bullet}(G_K, (\frac{1}{t^n}\bold{B}_{\mathrm{max}}^{+,\varphi=p^n})^{\oplus d})=\varinjlim_n
 \mathrm{C}^{\bullet}(G_K, W_e\cap \frac{1}{t^n}W^+_{\mathrm{dR}})$.
 \end{proof}

\begin{lemma}\label{5-2}
Let $W^+_{\mathrm{dR}}$ be a finite free $\bold{B}^+_{\mathrm{dR}}$-module with 
a continuous semi-linear $G_K$-action. Then the canonical map 
$\mathrm{H}^i(G_K,W^+_{\mathrm{dR}})\rightarrow \varprojlim_{n}\mathrm{H}^i(G_K, W^+_{\mathrm{dR}}/t^n W^+_{\mathrm{dR}})$ 
is isomorphism.
\end{lemma}
\begin{proof}
Because we have $\mathrm{C}^{\bullet}(G_K, W^+_{\mathrm{dR}})\isom \varprojlim_n \mathrm{C}^{\bullet}(G_K, W^+_{\mathrm{dR}}/t^nW^+_{\mathrm{dR}})$, for any $i\geqq 0$, we have the following short exact sequence
$$0\rightarrow \mathbb{R}^1\varprojlim_n\mathrm{H}^{i-1}(G_K, W^+_{\mathrm{dR}}/t^nW^+_{\mathrm{dR}})\rightarrow 
\mathrm{H}^i(G_K, W^+_{\mathrm{dR}})\rightarrow \varprojlim_n\mathrm{H}^i(G_K, W^+_{\mathrm{dR}}/t^nW^+_{\mathrm{dR}}))\rightarrow 
0.$$
Because $\mathrm{H}^{i-1}(G_K, W^+_{\mathrm{dR}}/t^nW^+_{\mathrm{dR}})$ is finite dimensional over $\mathbb{Q}_p$, Mittag-Leffler condition implies that $\mathbb{R}^1\varprojlim_n\mathrm{H}^{i-1}(G_K, W^+_{\mathrm{dR}}/t^nW^+_{\mathrm{dR}})=0$. 
The lemma follows from this.
\end{proof}
\begin{corollary}\label{5-3}
Let $W^+_{\mathrm{dR}}$ be as above. Let $\{h_1,h_2,\cdots,h_d\}$ be the generalized Hodge-Tate weights of $W^+_{\mathrm{dR}}/tW^+_{\mathrm{dR}}$. Let $k\geqq 1$ be any integer such that $k+h_j\geqq0$ for any $h_j\in \mathbb{Z}$. 
Then the natural map $\mathrm{H}^i(G_K, W^+_{\mathrm{dR}})\rightarrow \mathrm{H}^i(G_K, W^+_{\mathrm{dR}}/t^kW^+_{\mathrm{dR}})$ is 
isomorphism and $\mathrm{H}^i(G_K, t^{k+1}W^+_{\mathrm{dR}})=0$ for any $i$.

\end{corollary}
\begin{proof}
By the assumption on $k$, we have $\mathrm{H}^i(G_K, t^lW^+_{\mathrm{dR}}/t^{l+1}W^+_{\mathrm{dR}})=0$ for any $l\geqq k+1$. 
Then the corollary follows from Lemma \ref{5-2}.
\end{proof}

\begin{corollary}\label{5-4}
Let $W^+_{\mathrm{dR}}$ be as above, then 
$\mathrm{H}^i(G_K, W^+_{\mathrm{dR}})=\mathrm{H}^i(G_K, W_{\mathrm{dR}})= 0$ for $i\geqq 2$ and $\mathrm{H}^i(G_K, W^+_{\mathrm{dR}})$ and $\mathrm{H}^i(G_K, W_{\mathrm{dR}})$ are finite dimensional 
over $\mathbb{Q}_p$ for $i=0,1$.
\end{corollary}
\begin{proof}
Because $\mathrm{H}^i(G_K, W^+_{\mathrm{dR}}/t^nW^+_{\mathrm{dR}})=0$ for $i\geqq 2$ and
$\mathrm{H}^i(G_K, W^+_{\mathrm{dR}}/t^n W^+_{\mathrm{dR}})$ is finite dimensional for $i=0,1$, we obtain the  corollary 
for $W^+_{\mathrm{dR}}$ by Lemma \ref{5-3}. We prove the corollary for $W_{\mathrm{dR}}$. Because we have an isomorphism 
$\mathrm{C}^{\bullet}(G_K, W_{\mathrm{dR}})\isom \varinjlim_n \mathrm{C}^{\bullet}(G_K, \frac{1}{t^n} W^+_{\mathrm{dR}})$ 
by Lemma \ref{5-1} (2), 
we obtain an isomorphism $\mathrm{H}^i(G_K, W_{\mathrm{dR}})\isom \varinjlim_n \mathrm{H}^i(G_K, \frac{1}{t^n}W^+_{\mathrm{dR}})$. 
Then we can show that for $n$ large enough the natural map $\mathrm{H}^i(G_K, \frac{1}{t^{n+j}}W^+_{\mathrm{dR}})\rightarrow 
\mathrm{H}^i(G_K, \frac{1}{t^{n+j+1}}W^+_{\mathrm{dR}})$ is isomorphism for any $j\geqq 0$,  then the natural map 
$\mathrm{H}^i(G_K, \allowbreak \frac{1}{t^n} W^+_{\mathrm{dR}})\rightarrow \mathrm{H}^i(G_K, W_{\mathrm{dR}})$ is isomorphism, the corollary for $W_{\mathrm{dR}}$ follows 
from this.

\end{proof}

\begin{lemma}\label{5-5}
Let $W=(W_e,W^+_{\mathrm{dR}})$ be a $B$-pair. 
Then we have $\mathrm{H}^i(G_K, W_e)=0$ for $i\geqq 3$.
\end{lemma}
\begin{proof}
Because we have $\mathrm{C}^{\bullet}(G_K, W_e)=\varinjlim_n \mathrm{C}^{\bullet}(G_K, W_e\cap \frac{1}{t^n}W^+_{\mathrm{dR}})$ 
by Lemma \ref{5-1} (1), 
we have an isomorphism $\mathrm{H}^i(G_K,W_e)\isom \varinjlim_n \mathrm{H}^i(G_K,W_e\cap \frac{1}{t^n}W^+_{\mathrm{dR}})$. 
For any $n$, because $W_e\cap\frac{1}{t^n}W^+_{\mathrm{dR}}$ is an almost $\mathbb{C}_p$-representation by Theorem \ref{5-0}, 
we have $\mathrm{H}^i(G_K, W_e\cap \frac{1}{t^n}W^+_{\mathrm{dR}})=0$ for $i\geqq 3$ by Lemma \ref{5-0-5}. 
The lemma follows from these facts.
\end{proof}

\begin{thm}\label{5-6}
Let $W$ be a $B$-pair,  then the following hold,
\begin{itemize}
\item[(1)]$\mathrm{H}^i(G_K,W)$ is zero for $i\geqq 3$ and 
$\mathrm{H}^i(G_K,W)$ is finite dimensional over $\mathbb{Q}_p$ for 
$i=0,1,2$,
\item[(2)](Euler-Poincar\'e characteristic formula)
$$\sum_{i=0}^2 \mathrm{dim}_{\mathbb{Q}_p}(-1)^i\mathrm{H}^i(G_K, W)=-[K:\mathbb{Q}_p]\mathrm{rank}(W).$$
\end{itemize}
\end{thm}
\begin{proof}
We first prove that $\mathrm{H}^i(G_K,W)=0$ for $i\geqq 3$. Because there is an exact sequence 
$$\cdots \rightarrow \mathrm{H}^{i-1}(G_K, W_{dR})\rightarrow \mathrm{H}^i(G_K, W)\rightarrow 
\mathrm{H}^i(G_K, W_e)\oplus \mathrm{H}^i(G_K, W^+_{\mathrm{dR}})\rightarrow \cdots ,$$
the claim follows from Corollary \ref{5-4} and Lemma \ref{5-5}.
Next we prove that $\mathrm{H}^i(G_K, W)$ is finite dimensional over $\mathbb{Q}_p$. 
By slope filtration theorem, it suffices to show this claim when $W$ is pure. Let 
$W$ be a $B$-pair pure of slope $s$. When $s\leqq 0$, we have the following short exact sequence, 
$$0\rightarrow W_e\cap W^+_{\mathrm{dR}}\rightarrow W_e\oplus W^+_{\mathrm{dR}}\rightarrow W_{\mathrm{dR}}
\rightarrow 0$$ by Theorem \ref{5-0} (2). Hence the natural map $\mathrm{H}^i(G_K, W_e\cap W^+_{\mathrm{dR}})\rightarrow 
\mathrm{H}^i(G_K, W)$ is isomorphism. Because $W_e\cap W^+_{\mathrm{dR}}$ is an almost $\mathbb{C}_p$-representation by 
Theorem \ref{5-0}, $\mathrm{H}^i(G_K, W_e\cap W^+_{\mathrm{dR}})$ is finite dimensional by Lemma \ref{5-0-5}, which proves the claim 
for $s\leqq 0$. When $s>0$, then we have the following short exact sequence
$$0\rightarrow W_e\oplus W^+_{\mathrm{dR}}\rightarrow W_{\mathrm{dR}}\rightarrow W_{\mathrm{dR}}/(W_e + W^+_{\mathrm{dR}})
\rightarrow 0 $$ by Theorem \ref{5-0} (3). Hence we obtain a natural isomorphism $\mathrm{H}^{i}(G_K, W)\isom 
\mathrm{H}^{i-1}(G_K, \allowbreak W_{\mathrm{dR}}/(W_e+W^+_{\mathrm{dR}}))$. Because $W_{\mathrm{dR}}/(W_e+W^+_{\mathrm{dR}})$ is 
an almost $\mathbb{C}_p$-representation by Theorem \ref{5-0}, $\mathrm{H}^{i-1}(G_K, W_{\mathrm{dR}}/(W_e+W^+_{\mathrm{dR}}))$ is finite dimensional, the claim for $s>0$ follows from this.

Next we prove (2). For $W$ a $B$-pair or an almost $\mathbb{C}_p$-representation, set 
$\chi(W):=\sum_{i=0}^2 (-1)^i\mathrm{dim}_{\mathbb{Q}_p}\mathrm{H}^i(G_K, W)$.
It suffices to show (2) when $W$ is pure of slope $s$. 
When $s\leqq 0$, then we have $\chi(W)=\chi(X_0(W))$ by the above proof. 
By Lemma \ref{5-0-6} and by Theorem \ref{5-0} (2),  we have equalities 
$$\chi(X^0(W))=-[K:\mathbb{Q}_p]\mathrm{ht}(W)
=-[K:\mathbb{Q}_p]\mathrm{rank}(W).$$  When $s>0$, then we have $\chi(W)=-\chi(X^1(W))$ by the above proof. 
By Lemma \ref{5-0-6} and by Theorem \ref{5-0} (3), we have equalities 
$$\chi(X^1(W))=-[K;\mathbb{Q}_p]\allowbreak \mathrm{ht}(X^1(W))=
[K:\mathbb{Q}_p]\mathrm{rank}(W),$$  which proves (2).

\end{proof}

Next, we define the cup product pairing for $B$-pairs $W:=(W_e, W^+_{\mathrm{dR}})$ and 
$W':=(W'_e, W^{' +}_{\mathrm{dR}})$ as follows.
First, for two continuous cochains $c\in \mathrm{C}^{i}(G_K, W_{?})$ and $c'\in \mathrm{C}^{j}(G_K, W'_{?})$ for
$W_{?}=W_e, W^+_{\mathrm{dR}}, W_{\mathrm{dR}}$, we define a continuous cochain 
$$c\cup c'\in \mathrm{C}^{i+j}(G_K, W_{?}\otimes_{\bold{B}_{?}}W'_{?})$$ by $$c\cup c'(g_1, \cdots, g_{i+j}):=
c(g_1,\cdots, g_i)\otimes g_1g_2\cdots g_i c'(g_{i+1},\cdots, g_{i+j})$$
where $\bold{B}_{?}=\bold{B}_{e}, \bold{B}^+_{\mathrm{dR}}, \bold{B}_{\mathrm{dR}}$ when 
$W_{?}=W_e, W^+_{\mathrm{dR}}, W_{\mathrm{dR}}$ respectively. 
Then, $c\cup c'$ satisfies 
$$\partial^{i+j}(c\cup c')= \partial^i(c)\cup c' + (-1)^i c\cup \partial^j(c').$$
For 
$$c=(c_e, c^+_{\mathrm{dR}}, c_{\mathrm{dR}}) \in \mathrm{C}^i(G_K, W_e)\oplus 
\mathrm{C}^i(G_K, W^+_{\mathrm{dR}})\oplus \mathrm{C}^{i-1}(G_K, W_{\mathrm{dR}})$$ and 
$$c'=(c'_e, c^{' +}_{\mathrm{dR}}, c'_{\mathrm{dR}})\in \mathrm{C}^{j}(G_K, W'_e)\oplus 
\mathrm{C}^{j}(G_K, W^{' +}_{\mathrm{dR}})\oplus \mathrm{C}^{j-1}(G_K, W'_{\mathrm{dR}}),$$ and for 
a parameter $\gamma\in \mathbb{Q}_p$, 
we define 
\begin{multline*}
c\cup_{\gamma} c'\in \mathrm{C}^{i+j}(G_K, W_{e}\otimes_{\bold{B}_{e}}W'_e)\oplus 
\mathrm{C}^{i+j}(G_K, W^+_{\mathrm{dR}}\otimes_{\bold{B}^+_{\mathrm{dR}}}W^{' +}_{\mathrm{dR}})\\
\oplus \mathrm{C}^{i+j-1}(G_K, W_{\mathrm{dR}}\otimes_{\bold{B}_{\mathrm{dR}}}W'_{\mathrm{dR}})
\end{multline*} 
by 
\begin{multline*}
c\cup_{\gamma} c':=( c_e\cup c'_e, c^+_{\mathrm{dR}}\cup c^{' +}_{\mathrm{dR}}, \\
c_{\mathrm{dR}}\cup(\gamma c'_e+(1-\gamma)c^{'+}_{\mathrm{dR}})+ (-1)^i((1-\gamma)c_e+\gamma c^+_{\mathrm{dR}})\cup c'_{\mathrm{dR}})).
\end{multline*}
Then, we can check that if $\partial^i(c)=\partial^j(c')=0$ then $\partial^{i+j}(c\cup_{\gamma} c')=0$, and 
if $\partial^i{c}=0$ and $c'=\partial^{j-1}(c'')$ (or $c=\partial^{i-1}(c'')$ and $\partial^j(c')=0$) then 
$c\cup_{\gamma}c'\in \mathrm{Im}(\partial^{i+j-1})$. Therefore, this paring induces a $\mathbb{Q}_p$-bi-linear paring 
$$\cup_{\gamma}:\mathrm{H}^i(G_K, W)\times \mathrm{H}^j(G_K, W')\rightarrow 
\mathrm{H}^{i+j}(G_K, W\otimes W').$$
Moreover, we can check that $\cup_{\gamma}$ doesn't depend on the choice of a parameter 
$\gamma$, so we just write $\cup$ instead of $\cup_{\gamma}$.

We define the paring 
$$\cup: \mathrm{H}^i(G_K, W)\times \mathrm{H}^{2-i}(G_K, W^{\vee}(\chi_p))\rightarrow \mathbb{Q}_p$$
as the composition the following maps
\begin{multline*}
\mathrm{H}^i(G_K, W)\times \mathrm{H}^{2-i}(G_K, W^{\vee}(\chi_p))\xrightarrow{\cup} 
\mathrm{H}^2(G_K, W\otimes W^{\vee}(\chi_p))\\
\rightarrow \mathrm{H}^2(G_K, W(\mathbb{Q}_p(\chi_p)))\isom \mathrm{H}^2(G_K, \mathbb{Q}_p(\chi_p))\isom\mathbb{Q}_p,
\end{multline*}
where the second map is induced from the evaluation map $W\otimes W^{\vee}(\chi_p)\rightarrow W(\mathbb{Q}_p(\chi_p))$ and the third isomorphism is the natural comparison isomorphism 
and the fourth isomorphism is Tate's trace map.
The Tate duality  theorem for $B$-pairs is following.
\begin{thm}\label{5-7}
For $i=0,1,2$, the paring 
$$\cup: \mathrm{H}^i(G_K, W)\times \mathrm{H}^{2-i}(G_K, W^{\vee}(\chi_p))\rightarrow \mathbb{Q}_p$$
is a perfect paring.
\end{thm}
\begin{proof}
We can prove this theorem in the same way as in the proof of Theorem 4.7 of \cite{Li08} if we use the Euler-Poincar\'e formula Theorem \ref{5-6}
and the facts  that 
$\mathrm{H}^0(G_K, W(\prod_{\sigma\in \mathcal{P}}\sigma))\allowbreak =0$ and $\mathrm{H}^0(G_K, W(|\prod_{\sigma\in \mathcal{P}}\sigma|))=0$, which are proved in Proposition \ref{i}.
\end{proof}
Finally, we prove that our continuous cohomology is canonically isomorphic to Liu's cohomology. 
We first define an isomorphism between $\mathrm{H}^0$ by the functoriality
$$\mathrm{H}^0_{\mathrm{Liu}}(G_K, W)\isom \mathrm{Hom}_{\varphi,\Gamma}(R, D(W))\isom \mathrm{Hom}(W(\mathbb{Q}_p), W)
\isom \mathrm{H}^0(G_K, W),$$ where $D(W)$ is the $(\varphi,\Gamma)$-module associated to $D$ and $R$ is the trivial 
$(\varphi,\Gamma)$-module, where the second isomorphism follows from the equivalence of categories between 
$B$-pairs and $(\varphi,\Gamma)$-modules.
\begin{thm}\label{5.100}
The above isomorphism $\mathrm{H}^0_{\mathrm{Liu}}(G_K, W)\isom \mathrm{H}^0(G_K, W)$ extends uniquely to 
an isomorphism of $\delta$-functors $\mathrm{H}^i_{\mathrm{Liu}}(G_K, W)\isom \mathrm{H}^i(G_K, W)$.
\end{thm}
\begin{proof}
This follows from weakly effaceabilities of functors $\mathrm{H}^i_{\mathrm{Liu}}(G_K, -)$ 
and $\mathrm{H}^i(G_K, -)$. 
For $\mathrm{H}^i_{\mathrm{Liu}}$, these facts  are proved in the proof of Theorem 8.1 of \cite{Ke09}. 
For $\mathrm{H}^i(G_K,- )$, we can also prove in the same way as in Theorem 8.1 of \cite{Ke09} because 
we have already proved the Euler-Poincar\'e formula and the Tate duality for $\mathrm{H^i}(G_K, -)$ and 
we have a natural  isomorphism $\mathrm{H}^1(G_K, W)\isom \mathrm{Ext}^1( W(\mathbb{Q}_p), W)$ which is proved in Proposition 2.2 of \cite{Na09}.

\end{proof}




\end{document}